\pdfoutput=1
\documentclass[11pt]{book}
\usepackage[a4paper]{geometry}
\usepackage[T1]{fontenc}
\usepackage[utf8]{inputenc}
\usepackage{lmodern}
\usepackage[french,english]{babel}

\usepackage{amsfonts}
\usepackage{amssymb}
\usepackage{amsmath}
\usepackage{amsthm}

\usepackage{bussproofs}
\usepackage{stmaryrd}
\usepackage{tabularx}
\usepackage{color}
\usepackage{microtype}

\usepackage{tikz}
\usetikzlibrary{cd}
\usetikzlibrary{calc}
\usetikzlibrary{decorations.pathmorphing}
\usetikzlibrary{arrows}
\usetikzlibrary{matrix}

\usepackage{enumerate}
\usepackage{graphicx}
\usepackage[toc,page]{appendix}
\usepackage{mathrsfs}
\usepackage{textgreek}
\usepackage{emptypage}

\usepackage{scalerel}

\usepackage{textcomp}

\usepackage[style=alphabetic,maxnames=4,backend=bibtex]{biblatex}
\bibliography{biblio}
\DeclareFieldFormat[article,online,unpublished,inproceedings,thesis]{title}{\emph{#1}}
\DeclareFieldFormat[article,online,unpublished,inproceedings]{journaltitle}{\textrm{#1}}
\renewbibmacro{in:}{}
\AtBeginBibliography{\small}

\usepackage[unicode]{hyperref}
\hypersetup{pdfauthor={Guillaume Brunerie},
            pdftitle={On the homotopy groups of spheres in homotopy type theory},
            pdfsubject={Mathematics}}
\usepackage[numbered]{bookmark}

\bigskip
\makeatletter
\def\operator@font{\sf}
\makeatother

\newcommand{\Q}{\mathbb{Q}}
\newcommand{\R}{\mathbb{R}}

\newcommand{\J}{\operatorname{J}}
\newcommand{\jcomp}{\J_{\operatorname{comp}}}

\newcommand{\topt}{\mathbin{\to_{\star}}}

\newcommand{\Type}{\mathsf{Type}}
\newcommand{\coe}{\operatorname{coe}}
\newcommand{\coeidp}{\operatorname{coeidp}}
\newcommand{\ua}{\operatorname{ua}}
\newcommand{\isequiv}{\operatorname{isequiv}}

\newcommand{\fib}{\operatorname{fib}}
\newcommand{\transport}{\operatorname{transport}}
\newcommand{\toeq}{\stackrel{\sim}{\to}}

\newcommand{\fst}{\operatorname{fst}}
\newcommand{\snd}{\operatorname{snd}}
\newcommand{\fsteq}{\fst^=}
\newcommand{\sndeq}{\snd^=}
\newcommand{\happly}{\operatorname{happly}}

\newcommand{\Id}{\operatorname{Id}}
\newcommand{\idpS}{\mathsf{idp}}
\newcommand{\idp}[1]{\idpS_{#1}}

\newcommand{\concat}{\mathbin{\cdot}}
\newcommand{\inv}{^{-1}}
\newcommand{\vcomp}{\concat}
\newcommand{\hcomp}{\mathbin{\star}}

\newcommand{\ids}{\mathsf{ids}}

\newcommand{\Unit}{\mathbf{1}}
\newcommand{\ttt}{\star_\Unit}

\newcommand{\Zero}{\mathbf{0}}

\newcommand{\N}{\mathbb{N}}
\newcommand{\succS}{\operatorname{S}}
\newcommand{\add}{\operatorname{add}}
\newcommand{\mul}{\operatorname{mul}}

\newcommand{\Z}{\mathbb{Z}}
\newcommand{\succZ}{\operatorname{succ}_{\Z}}
\newcommand{\posZ}{\operatorname{pos}}
\newcommand{\negZ}{\operatorname{neg}}

\newcommand{\Bool}{\mathbf{2}}
\newcommand{\true}{\mathsf{true}}
\newcommand{\false}{\mathsf{false}}

\newcommand{\trunc}[2]{\left\lVert{#2}\right\rVert_{#1}}

\newcommand{\apS}{\operatorname{ap}}
\newcommand{\ap}[1]{\apS_{#1}}
\newcommand{\apiiS}{\operatorname{ap}^2}
\newcommand{\apii}[1]{\apiiS_{#1}}
\newcommand{\defeq}{\mathrel{:=}}
\newcommand{\jueq}{\equiv}
\newcommand{\id}{\mathsf{id}}
\newcommand{\funext}{\operatorname{funext}}

\newcommand{\Sn}[1]{\mathbb{S}^{#1}}
\newcommand{\north}{\mathsf{north}}
\newcommand{\south}{\mathsf{south}}
\newcommand{\merid}{\operatorname{merid}}

\newcommand{\CP}[1]{\mathbb{C}P^{#1}}

\newcommand{\Susp}{\Sigma}

\newcommand{\proj}{\operatorname{proj}}
\newcommand{\projl}{\operatorname{proj_l}}
\newcommand{\projr}{\operatorname{proj_r}}
\newcommand{\projlr}{\operatorname{proj_{rl}}}

\newcommand{\inn}{\operatorname{in}}
\newcommand{\push}{\operatorname{push}}

\newcommand{\inl}{\operatorname{inl}}
\newcommand{\inr}{\operatorname{inr}}

\newcommand{\pt}{\bullet}
\newcommand{\sq}{\scalerel*{\boldsymbol{\square}}{\bullet}}

\newcommand{\base}{\mathsf{base}}
\newcommand{\lloop}{\mathsf{loop}}
\newcommand{\invol}{\operatorname{invol}}

\newcommand{\fold}{\nabla}
\newcommand{\contreq}{\theta}
\newcommand{\iwedge}{i^{\vee}}
\newcommand{\ft}{\widetilde{f}}
\newcommand{\Sm}[1]{\mathbf{S}^{#1}}

\newcommand{\JiA}{J_\infty A}
\newcommand{\epsiloni}{\varepsilon_\infty}
\newcommand{\alphai}{\alpha_\infty}
\newcommand{\deltai}{\delta_\infty}
\newcommand{\gammai}{\gamma_\infty}
\newcommand{\etai}{\eta_\infty}

\newcommand{\innJ}{\operatorname{in}^J}
\newcommand{\pushJ}{\push^J}

\newcommand{\tof}{\operatorname{to}}
\newcommand{\from}{\operatorname{from}}

\newcommand{\pp}{\mathop{\hat{\times}}}

\newcommand{\Ht}{\widetilde{H}}
\newcommand{\cupp}{\operatorname{\smallsmile}}
\newcommand{\Ker}{\operatorname{Ker}}
\renewcommand{\Im}{\operatorname{Im}}
\newcommand{\cc}{\mathbf{c}}
\newcommand{\xx}{\mathbf{x}}
\newcommand{\yy}{\mathbf{y}}
\newcommand{\zz}{\mathbf{z}}

\newcommand{\Hopf}{\operatorname{Hopf}}

\newcommand{\tildeE}{\widetilde{E}}
\newcommand{\tildeF}{\widetilde{F}}
\newcommand{\Phit}{\widetilde{\Phi}}

\newcommand{\lcr}{\llbracket}
\newcommand{\rcr}{\rrbracket}

\newcommand{\typ}{\ \mathsf{type}}

\newcommand{\depth}{\operatorname{depth}}
\newcommand{\cohf}{\mathsf{coh}}
\newcommand{\coh}[3]{\cohf_{#1.#2}(#3)}
\newcommand{\cohhf}[3]{\mathbf{coh}^{#1}_{#2.#3}}
\newcommand{\cohh}[4]{\cohhf{#1}{#2}{#3}(#4)}
\newcommand{\contr}{\ \mathsf{contr}}
\newcommand{\ctx}{\ \mathsf{ctx}}
\newcommand{\Glob}{\operatorname{Glob}}
\newcommand{\Ob}{\operatorname{Ob}}
\newcommand{\Hom}{\operatorname{Hom}}

\newcommand{\IID}{\mathsf{I}}
\newcommand{\iid}{\mathsf{i}}
\newcommand{\vdashi}{\vdash_{\infty}}
\newcommand{\vdashML}{\vdash_{\mathsf{ML}}}

\newcommand{\llp}{\llparenthesis}
\newcommand{\rrp}{\rrparenthesis}
\newcommand{\tinfgpd}{\mathcal{T}_\infty}
\newcommand{\tml}{\mathcal{T}_{\mathsf{ML}}}

\newcommand{\glob}{^{\mathsf{G}}}

\newcommand{\hub}{\operatorname{hub}}
\newcommand{\spoke}{\operatorname{spoke}}
\newcommand{\lift}{\operatorname{lift}}
\newcommand{\lemm}{\operatorname{lemma}}
\newcommand{\tHopf}{\operatorname{tHopf}_3}
\newcommand{\funextd}{\operatorname{funext^{\mathsf{dep}}}}
\newcommand{\apd}[1]{\operatorname{apd}_{#1}}
\newcommand{\inne}{\operatorname{in}^=}
\newcommand{\oute}{\operatorname{out}^=}
\newcommand{\innt}{\operatorname{in}^{\transport}}
\newcommand{\outt}{\operatorname{out}^{\transport}}

\numberwithin{equation}{section}

\theoremstyle{plain}
\newtheorem{maintheorem}{Theorem}
\newtheorem{maintheoremfr}{Théorème}

\newtheorem{cor}[equation]{Corollary}
\newtheorem{lemma}[equation]{Lemma}
\newtheorem{proposition}[equation]{Proposition}
\newtheorem{axiom}[equation]{Axiom}

\theoremstyle{definition}
\newtheorem{definition}[equation]{Definition}
\newtheorem{example}[equation]{Example}

\newcommand{\chaptertoc}[1]{\chapter*{#1}
\addcontentsline{toc}{chapter}{#1}
\markboth{\slshape\MakeUppercase{#1}}{\slshape\MakeUppercase{#1}}}

\tikzset{sdiag/.style={decorate,decoration={snake,amplitude=.3mm,segment length=1.5mm,post length=1mm}}}

\tikzcdset{sdiag/.style={
    arrows={decorate,decoration={snake,amplitude=.3mm,segment length=1.5mm,post length=1mm}},
    labels={font=\small},
    column sep=large,
    row sep=large}}

\begin{document}
\frontmatter

\newgeometry{hcentering}
\pdfbookmark{Front matter}{frontmatter}
\subpdfbookmark{Page de titre}{title}
\phantomsection

\begin{titlepage}
 
\centering
\vspace*{-1cm}
{\Large UNIVERSITÉ DE NICE SOPHIA ANTIPOLIS   --   UFR Sciences}\\
\vspace*{0.5cm}
École Doctorale Sciences Fondamentales et Appliquées\\
\vspace*{1.5cm}
{\Large\bf THÈSE}\\
\vspace*{0.3cm}
pour obtenir le titre de\\
\vspace*{0.1cm}
{\Large\bf Docteur en Sciences}\\
\vspace*{0.1cm}
Spécialité
{\Large {\sc Mathématiques}}\\
\vspace*{0.8cm}
présentée et soutenue par\\
\vspace*{0.1cm}
{\large\bf Guillaume BRUNERIE}\\
\vspace*{1.0cm}
{\LARGE\bf Sur les groupes d’homotopie des sphères\\en théorie des types homotopiques}\\
\vspace*{0.5cm}
{\LARGE\bf On the homotopy groups of spheres\\in homotopy type theory}\\
\vspace*{1.0cm}
Thèse dirigée par {\bf Carlos SIMPSON}\\
soutenue le 15 juin 2016\\

\vfill
Membres du jury :\\
\vspace*{0.5cm}
\begin{tabular}{lll}
M. Denis-Charles CISINSKI & Professeur des universités & \textit{Examinateur}\\
M. Thierry COQUAND & Professor & \textit{Rapporteur}\\
M. André HIRSCHOWITZ & Professeur émérite & \textit{Examinateur}\\
M. André JOYAL & Professeur émérite & \textit{Examinateur}\\
M. Paul-André MELLIÈS & Chargé de recherche CNRS & \textit{Examinateur}\\
M. Michael SHULMAN & Assistant Professor & \textit{Rapporteur}\\
M. Carlos SIMPSON & Directeur de recherche CNRS & \textit{Directeur de thèse}\\
\end{tabular}

\vfill
Laboratoire Jean-Alexandre Dieudonné, Université de Nice, Parc Valrose, 06108 NICE
\end{titlepage}


\restoregeometry

\clearpage
\thispagestyle{empty}
\vspace*{\fill}
\begin{center}
  \includegraphics{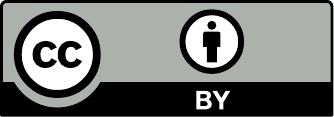}

  This work is licensed under the Creative Commons Attribution 4.0 International License. To view a copy of this license, visit \url{http://creativecommons.org/licenses/by/4.0/}.
  \vspace{1cm}
  
  © 2016 Guillaume Brunerie
\end{center}

\currentpdfbookmark{Abstract/Résumé}{abstracten}

\section*{Abstract}

The goal of this thesis is to prove that $\pi_4(\Sn3)\simeq\Z/2\Z$ in homotopy type theory. In
particular it is a constructive and purely homotopy-theoretic proof. We first recall the basic
concepts of homotopy type theory, and we prove some well-known results about the homotopy groups of
spheres: the computation of the homotopy groups of the circle, the triviality of those of the form
$\pi_k(\Sn n)$ with $k<n$, and the construction of the Hopf fibration. We then move to more advanced
tools. In particular, we define the James construction which allows us to prove the Freudenthal
suspension theorem and the fact that there exists a natural number $n$ such that
$\pi_4(\Sn3)\simeq\Z/n\Z$. Then we study the smash product of spheres, we construct the cohomology
ring of a space, and we introduce the Hopf invariant, allowing us to narrow down the $n$ to either
$1$ or $2$. The Hopf invariant also allows us to prove that all the groups of the form
$\pi_{4n-1}(\Sn{2n})$ are infinite. Finally we construct the Gysin exact sequence, allowing us to
compute the cohomology of $\CP2$ and to prove that $\pi_4(\Sn3)\simeq\Z/2\Z$ and that more generally
$\pi_{n+1}(\Sn n)\simeq\Z/2\Z$ for every $n\ge 3$.

\bigskip
\noindent
\textbf{Keywords:} homotopy type theory, homotopy theory, algebraic topology, cohomology, type
theory, logic, constructive mathematics

\vfill
\begin{otherlanguage}{french}
  \section*{Résumé}

  L’objectif de cette thèse est de démontrer que $\pi_4(\Sn3)\simeq\Z/2\Z$ en théorie des types
  homotopiques. En particulier, c’est une démonstration constructive et purement homotopique. On
  commence par rappeler les concepts de base de la théorie des types homotopiques et on démontre
  quelques résultats bien connus sur les groupes d’homotopie des sphères~: le calcul des groupes
  d’homotopie du cercle, le fait que ceux de la forme $\pi_k(\Sn n)$ avec $k<n$ sont triviaux et la
  construction de la fibration de Hopf. On passe ensuite à des outils plus avancés. En particulier,
  on définit la construction de James, ce qui nous permet de démontrer le théorème de suspension de
  Freudenthal et le fait qu’il existe un entier naturel $n$ tel que $\pi_4(\Sn3)\simeq\Z/n\Z$. On
  étudie ensuite le produit smash des sphères, on construit l’anneau de cohomologie des espaces et
  on introduit l’invariant de Hopf, ce qui nous permet de montrer que $n$ est égal soit à $1$, soit
  à $2$. L’invariant de Hopf nous permet également de montrer que tous les groupes de la forme
  $\pi_{4n-1}(\Sn{2n})$ sont infinis. Finalement, on construit la suite exacte de Gysin, ce qui nous
  permet de calculer la cohomologie de $\CP2$ et de démontrer que $\pi_4(\Sn3)\simeq\Z/2\Z$, et que
  plus généralement on a $\pi_{n+1}(\Sn n)\simeq\Z/2\Z$ pour tout $n\ge 3$.

  \bigskip
  \noindent
  \textbf{Mots-clés:} théorie des types homotopiques, théorie de l’homotopie, topologie algébrique,
  cohomologie, théorie des types, logique, mathématiques constructives
\end{otherlanguage}


\cleardoublepage
\currentpdfbookmark{Acknowledgments}{ack}

\chapter*{Acknowledgments}

Five years ago, towards the end of my first year of master studies, I wasn’t sure in which domain
of mathematics or computer science to continue. I was interested in many subjects, in particular
homotopy theory and type theory, so one day I decided to search for “homotopy type theory” on the
Internet, thinking that if something with such a name exists it is probably something for
me. Apparently I was right.
  
\smallskip
I would like to thank all the people who helped me and supported me, in particular
\begin{itemize}
\item my advisor, Carlos Simpson, for always trusting me, supporting me, and encouraging me,
\item the “rapporteurs” of my thesis, Thierry Coquand and Mike Shulman, and also Ulrik Buchholtz
  for their careful reading and comments on my thesis,
\item the university of Nice Sophia Antipolis and the LJAD for letting me work in such a great
  environment,
\item Paul-André Melliès who arranged for me to go to my first conference on homotopy type theory,
\item Thierry Coquand, Steve Awodey, and Vladimir Voevodsky for letting me take part in the
  special year on Univalent Foundations at the Institute for Advanced Study in Princeton,
\item all the people I’ve worked with, in particular Dan Licata and Thierry Coquand, but also
  Carlo Angiuli, Favonia, Eric Finster, Bob Harper, Simon Huber, André Joyal, Peter Lumsdaine,
  Egbert Rijke and many others,
\item all other PhD students in Nice for making my stay enjoyable, in particular my officemates
  Arthur, Laurence, and Byron,
\item my family for their constant support, in particular my late grandfather Alain who started
  fueling my mathematical curiosity when I was very young, teaching me how to multiply a number by
  $10$ or $100$,
\item the European forró community for the countless festivals and dances,
\item and finally my girlfriend, Monika, for her presence, her support, and for reading in detail
  early versions of parts of this text.
\end{itemize}


\cleardoublepage
\currentpdfbookmark{\contentsname}{Contents}
\tableofcontents


\mainmatter

\cleardoublepage
\chaptertoc{Introduction}

The aim of this PhD thesis is to prove the following theorem, whose statement and proof will be
explained in due time.

\begin{maintheorem}\label{theorem}
  We have a group isomorphism \[\pi_4(\Sn3)\simeq\Z/2\Z.\]
\end{maintheorem}

This is actually a well-known theorem in classical homotopy theory, originally proved by Freudenthal
in \cite{freud37} (see also \cite[corollary 4J.4]{hatcher}). The main difference is that in this
thesis we work in \emph{homotopy type theory} (also known as \emph{univalent foundations}), which is
a new framework for doing mathematics introduced by Vladimir Voevodsky in 2009 and which is
particularly well-suited for homotopy theory. From the point of view of a homotopy theorist, the
most striking difference between classical homotopy theory and homotopy type theory is that in
homotopy type theory \emph{all} constructions are invariant under homotopy equivalences. One of the
advantages is that all the constructions and proofs done in this framework are completely
independent of the definition of “spaces”. In particular, nothing depends on point-set topology or
on combinatorics of simplicial sets. Moreover, as we hope the reader will be convinced after reading
this thesis, the constructions and proofs have often a more “homotopy-theoretic feel” and are closer
to intuition.

However, this also poses a number of challenges as it is not \textit{a priori} obvious which
concepts can or cannot be defined in a purely homotopy-invariant way.  For instance, even though
singular cohomology is homotopy-invariant, the classical definition uses the set of singular
cochains which is not homotopy-invariant.  Therefore the classical definition cannot be reproduced
verbatim in homotopy type theory. An even simpler example is the universal cover of the circle which
is classically defined using the exponential function $\R\to\Sn1$, but that function is actually
homotopic to a constant function.
Homotopy type theory gives us a number of tools to work in a completely homotopy-invariant way and
in this thesis we show how to prove theorem \ref{theorem} in homotopy type theory, starting
essentially from scratch.

Another advantage of homotopy type theory over classical homotopy theory is that proofs written in
homotopy type theory are much more amenable to being formally checked by a computer. While the
present work hasn’t been formalized yet, many intermediate results (in particular from the first two
chapters) have already been formalized by various people, see for instance the libraries
\cite{HoTTCoq} and \cite{Unimath} for Coq, \cite{HoTTAgda} for Agda and \cite{HoTTLean} for Lean.

\paragraph{Content of the thesis}

The first two chapters of this thesis review basic homotopy type theory. An alternative reference is
the book \cite{hottbook}, but we tried here to be more concise and to keep in mind our end
goal. Nevertheless there might be some overlap in the style of presentation between \cite{hottbook}
and the introduction and the first two chapters of this thesis. Most of the content of the last four
chapters is new in homotopy type theory even though the concepts are well-known in classical
homotopy theory. The definition of weak $\infty$-groupoid presented in the first appendix is new as
well.

In chapter \ref{ch:hott} we introduce all the basic concepts of homotopy type theory, namely all
type constructors and in particular the univalence axiom and higher inductive types. We also state
the $3\times3$-lemma and the flattening lemma in sections \ref{threebythree} and \ref{flattening},
which are two results that we use in various places. Finally we talk about $n$-truncatedness and
truncations. The notion of $n$-truncated type corresponds to the classical notion of homotopy
$n$-types, i.e.\ spaces with no homotopical information above dimension $n$, and truncation is an
operation turning any space into an $n$-truncated space in a universal way. All this is again
standard in homotopy type theory.

In chapter \ref{ch:hopf} we define the homotopy groups of spheres. The group $\pi_k(\Sn n)$ is
defined as the $0$-truncation (i.e.\ the set of connected components) of the space of
$k$-dimensional loops in $\Sn n$. Then we show how to prove that $\pi_1(\Sn1)\simeq\Z$, which is a
result originally proven by Michael Shulman in 2011 and which appears in \cite[section
8.1]{hottbook}, cf also \cite{mikeblog:pi1s1} and \cite{mikelicata:pi1s1}. The idea is that, in
homotopy type theory, in order to define a fibration we do not give a map from the total space to
the base space. Instead we give directly the fibers over every point of the base space. In the case
of a fibration over the circle, it is enough to give the fiber over the basepoint of $\Sn1$ and the
action on the fiber of the loop going around $\Sn1$. Here the fiber is the space of integers $\Z$
and the loop of the circle acts on it by the function which adds one. This gives a fibration over
$\Sn1$ and one can show that its total space is contractible, from which the isomorphism
$\pi_1(\Sn1)\simeq\Z$ follows. We then define the notion of connectedness and prove various
properties about connected spaces and maps which allow us to prove that $\pi_k(\Sn n)$ is trivial
for all $k<n$. This result already appears in \cite[section 8.3]{hottbook} with a more complicated
proof, also due to the author. Finally we define the Hopf fibration, which is a fibration over
$\Sn2$ with fiber $\Sn1$ and total space $\Sn3$. The idea of the definition of the Hopf fibration is
as follows. In order to define a fibration over $\Sn2$ it is enough to give the fiber $N$ over the
north pole, the fiber $S$ over the south pole and, for every element $x$ of $\Sn1$, an equivalence
between $N$ and $S$ which describe what happens when we move in the fibration over the meridian
corresponding to $x$. In the case of the Hopf fibration, we take $N, S\defeq \Sn1$, and the
equivalence between $N$ and $S$ corresponding to $x$ is the operation of multiplication by $x$. The
Hopf fibration was first defined by Peter Lumsdaine, in a slightly different way, but without a
proof that its total space is equivalent to $\Sn3$. The construction presented here was first
written as \cite[section 8.5]{hottbook}.

In chapter \ref{ch:james} we define the James construction following an initial idea of André
Joyal. For every type $A$ we define a family of spaces $(J_nA)$ and we prove that their colimit is
equivalent to the loop space of the suspension of $A$. This is done by defining another space $JA$
and proving that $JA$ is equivalent to both the colimit of $(J_nA)$ and to the loop space of the
suspension of $A$. The James construction gives a sequence of approximations of the loop space of
the suspension of $A$ which, in conjunction with the Blakers--Massey theorem, allows us to prove
that there is a natural number $n$ such that $\pi_4(\Sn3)\simeq\Z/n\Z$. This number $n$ is defined
using Whitehead products, more precisely it is the image of the Whitehead product
$[\id_{\Sn2},\id_{\Sn2}]$, which is an element of $\pi_3(\Sn2)$, by the equivalence
$\pi_3(\Sn2)\simeq\Z$ constructed using the Hopf fibration.

In chapter \ref{ch:smash} we study the smash product and its symmetric monoidal structure. In
particular we construct a family of equivalences $\Sn n\wedge\Sn m\simeq\Sn{n+m}$ which is
compatible, in some sense, with associativity and commutativity of the smash product. The
construction of the symmetric monoidal structure will be essentially admitted, but we give some
intuition on how to construct it.

In chapter \ref{ch:cohomology} we first define, for every natural number $n$, the Eilenberg--MacLane
space $K(\Z,n)$ as the $n$-truncation of the sphere $\Sn n$ and the $n$-th cohomology group of a
space $X$ as the $0$-truncation of the function space $X\to K(\Z,n)$. We then define the cup product
as a map $K(\Z,n)\wedge K(\Z,m)\to K(\Z,n+m)$ by taking the smash product of the two maps
$\Sn n\to K(\Z,n)$ and $\Sn m\to K(\Z,m)$, composing with the equivalence
$\Sn n\wedge\Sn m\simeq\Sn{n+m}$, and using some properties of connectivity of maps to show that we
can essentially invert it. The properties of the smash product from chapter \ref{ch:smash} are then
used to prove that the cup product is associative and graded-commutative. We finally define the Hopf
invariant of a map $f:\Sn{2n-1}\to\Sn n$ using the cup product structure on the pushout
$\Unit\sqcup^{\Sn{2n-1}}\Sn n$, and we prove that for every even $n$, some particular map
$\Sn{2n-1}\to\Sn n$ coming from the James construction has Hopf invariant $2$. This shows that the
number $n$ defined in chapter \ref{ch:james} is equal to either $1$ or $2$ and that the group
$\pi_{4n-1}(\Sn{2n})$ is infinite for every natural number $n$.

Finally in chapter \ref{ch:gysin} we construct the Gysin exact sequence which is a long exact
sequence of cohomology groups associated to every fibration where the base space is $1$-connected
and the fibers are spheres. This exact sequence describes some part of the multiplicative structure
of the cohomology of the base space. We then define $\CP2$ as the pushout $\Unit\sqcup^{\Sn3}\Sn2$
for the Hopf map $\Sn3\to\Sn2$, we construct a fibration of circles above it in a way similar to the
construction of the Hopf fibration, and we compute its cohomology ring using the Gysin exact
sequence. This proves that the Hopf invariant of the Hopf map is equal to $\pm1$ and that
$\pi_4(\Sn3)\simeq\Z/2\Z$.

In appendix \ref{ch:infgpd} we present an elementary definition of weak $\infty$-groupoids, based on
ideas coming from homotopy type theory, together with a proof that every type in homotopy type
theory has the structure of a weak $\infty$-groupoid.

In appendix \ref{ch:defn} we give a self-contained definition of the natural number $n$ defined at
the end of chapter \ref{ch:james} which satisfies $\pi_4(\Sn3)\simeq\Z/n\Z$. The reason is that, as
we will see later, computing this number from its definition is an important open problem in
homotopy type theory, hence, for the benefit of people trying to solve it, it is convenient to have
the complete definition all in one place.

\paragraph{Analytic versus synthetic}

The main difference between classical homotopy theory and homotopy type theory is that the first one
is \emph{analytic} whereas the second one is \emph{synthetic}. To understand the difference between
analytic and synthetic homotopy theory, it is helpful to go back to elementary geometry.

Analytic geometry is geometry in the sense of Descartes. The set $\R^2$ is our object of study,
points are defined as pairs $(x,y)$ of real numbers and lines are defined as sets of points
satisfying an equation of the form $ax+by=c$. Then in order to prove something we use the properties
of $\R^2$. For instance we can determine whether two lines intersect by solving a particular system
of equations.

In contrast, synthetic geometry is geometry in the sense of Euclid. Points and lines are not defined
in terms of other notions, they are just primitive notions, and a collection of axioms stipulating
how they are supposed to behave is given. Then in order to prove something we have to use the
axioms. For instance we cannot use the equation of a line or the coordinates of a point because
lines do not have equations and points do not have coordinates.

Analytic geometry can be used to justify synthetic geometry. Indeed, analytic geometry gives a
meaning to the notions of point and line and all the axioms of synthetic geometry can be proved to
hold in analytic geometry. Therefore the axioms are consistent and everything which is true is
synthetic geometry is also true in analytic geometry. The converse doesn’t hold, so one could think
that synthetic geometry is less powerful than analytic geometry as less theorems are provable. But
from a different point of view, one can also argue that synthetic geometry is actually \emph{more}
powerful than analytic geometry because the theorems that can be proved are more general. They are
true for any interpretation of the primitive notions for which the axioms are validated, whereas a
proof in analytic geometry is by nature only valid in $\R^2$. Another disadvantage of analytic
geometry is that because it reduces geometry to the resolution of equations, it is easy to lose
track of the geometrical intuition. To sum up, in analytic geometry we give an explicit definition
to the concepts we are interested in, and we can prove a lot of things about them, but we are
restricted to this particular model, whereas in synthetic geometry we only axiomatize the basic
properties of the concepts we are interested in, less theorems are provable, but they have a wider
range of applicability and they are closer to the geometrical intuition.

The situation of homotopy theory is very similar. In analytic homotopy theory (or classical homotopy
theory), the sphere $\Sn n$ is defined as the set
$\{(x_0,\dots,x_n)\in\R^{n+1},x_0^2+\dots+x_n^2=1\}$ equipped with the appropriate topology,
continuous maps are defined as functions preserving the topology in the appropriate way, and
$\pi_4(\Sn3)$ is defined as the quotient of the set of continuous pointed maps $\Sn4\to\Sn3$ by the
relation of homotopy. We can then use various techniques to prove that $\pi_4(\Sn3)\simeq\Z/2\Z$,
i.e.\ that $\pi_4(\Sn3)$ contains exactly two elements.

In synthetic homotopy theory, which is what this thesis is about, the notion of space does
\emph{not} come from topology. Instead it is axiomatized as a primitive notion (under the name
\emph{type}) together with primitive notions of \emph{point} of a type and of \emph{path} between
two points. In particular, a path is not seen anymore as a continuous function from the interval, it
is a primitive notion. We also introduce a primitive notion of \emph{continuous function}. Note that
in classical homotopy theory, we need to define first what is a \mbox{possibly-non-continuous}
function before being able to define what a continuous function is, but here we directly take the
concept of continuous function as primitive. For us a continuous function is \emph{not} a
possibly-non-continuous function which has the additional property of being continuous, indeed there
is not any notion of possibly-non-continuous function. Therefore, the adjective “continuous” is
superfluous, and we will simply use the word “function” or “map” for what would be called
“continuous function” in classical homotopy theory.

Various basic spaces are also axiomatized, for instance the space $\N$ of natural numbers is
axiomatized together with an element $0$, a function $S:\N\to\N$ and the principle of
induction/recursion. The circle is axiomatized together with a point called $\base$, a path called
$\lloop$ from $\base$ to $\base$ and a similar principle of induction/recursion stating intuitively
that the circle is freely generated by $\base$ and $\lloop$. Similarly, we describe the
higher-dimensional spheres $\Sn n$ and the set of connected components of a space. Combining all of
that with the notion of (continuous) functions mentioned above, we can define $\pi_4(\Sn3)$ and we
will see that we can still prove that it is isomorphic to the group $\Z/2\Z$.

\paragraph{Type theory}

Homotopy type theory is a variant of type theory and more precisely of Per Martin--Löf’s
intuitionistic theory of types (called simply \emph{dependent type theory} here), which was
introduced in the 1970s as a foundation for constructive mathematics (cf \cite{ML75}). Constructive
mathematics is a philosophy of mathematics based on the idea that in order to prove that a
particular object exists, we have to give a method to construct it. It works by restricting the
logical principles we are allowed to use and only allows those which are constructive.  A proof in
constructive mathematics isn’t necessarily presented as an algorithm but an algorithm can always be
extracted from it. Therefore constructive mathematics rejects principles like the axiom of choice,
which asserts the existence of a function without giving a way to compute it, and reasoning by
contradiction, which allows us to prove that something exists simply by proving that it cannot not
exist. In particular, a proof that there exists a natural number having a specific property has to
give (at least implicitly) a method to compute this number. This isn’t true in classical
mathematics. For instance, let’s define $n\in\N$ as the smallest odd perfect number or $0$ if no odd
perfect number exists. In classical mathematics, this is a correct and complete definition of $n$,
but it doesn’t give any way to compute $n$. Indeed, at the time of writing it isn’t known whether
$n$ is equal to $0$ or not. On the other hand, this would not be considered a valid definition in
constructive mathematics because we used the principle of excluded middle (either there exists an
odd perfect number or there doesn’t exist any) which isn’t constructive. There are various flavors
of constructive mathematics and note that the one we are using here, homotopy type theory, is not
incompatible with classical logic. It would be perfectly possible to add the axiom of choice or
excluded middle, but the drawback is that constructivity, which is one of the main advantages of
type theory, would be lost.

In dependent type theory the primitive notions are \emph{types} and \emph{elements of types} (or
\emph{terms}). We write $u:A$ for the statement that $u$ is an element of type $A$. Intuitively, one
can think of a type as being something like a set, but there are several important differences with
traditional set theory. Elements of types do not exist in isolation, they are always elements
\emph{of a given type} which is an intrisic part of the nature of the element. The type of an
element is always known and it doesn’t make sense to “prove” that an element $u$ has type $A$. It is
similar to the fact that it doesn’t make sense to “prove” that $x^2+y^2=0$ is an equation. Just
looking at it we see that it is an equation and not a matrix. Moreover the type of an element is
always unique (modulo computation rules as we will see later). For instance, we cannot say that the
number $2$ has both type $\N$ and type $\Q$. Instead there are two different elements, one of which
is $2_\N$ of type $\N$ and the other is $2_\Q$ of type $\Q$ (which may both be written as $2$ in a
mathematical text if there is no risk of confusion) and they satisfy $i(2_\N)=2_\Q$ for $i:\N\to\Q$
the canonical inclusion. Similarly, if we are given a rational number $q:\Q$, we cannot ask whether
$q$ has type $\N$. By nature $q$ has type $\Q$ which is different from $\N$. What we can ask,
however, is whether there exists a natural number $k:\N$ such that $i(k)=q$. This is what proving
that $q$ is a natural number would mean.

Mathematics is traditionally based on a two-layer system: the logical layer where propositions and
proofs live and the mathematical layer where mathematical objects live. The logical layer is used to
reason about the mathematical layer. For instance, constructing a specific mathematical object is an
activity carried out in the mathematical layer, while proving a theorem happens in the logical
layer. In dependent type theory those two layers are merged into one unique layer where types and
their elements live. Apart from representing mathematical objects, types also play the role of
(logical) propositions, and their elements play the role of “proofs” or witnesses of those
propositions. Proving a given proposition is done by constructing an element of the corresponding
type. For example, proving an implication $A\implies B$ corresponds to constructing an element in
the function type $A\to B$, i.e.\ a function taking proofs of $A$ to proofs of $B$. Proving a
conjunction $A\wedge B$ corresponds to constructing an element in the product type $A\times B$,
i.e.\ a pair composed of a proof of $A$ and a proof of $B$. This correspondence between types and
propositions and between elements of types and proofs is known as the \emph{Curry--Howard
  correspondence}. We will sometimes distinguish between types “seen as propositions” and “seen as
types” in order to explain the intuition between various constructions, but the difference between
the two is often blurry. For instance, the type $A\simeq B$ can be seen both as the proposition “$A$
and $B$ are isomorphic” and as the type of all isomorphisms between them. Indeed, in constructive
mathematics proving that $A$ and $B$ are isomorphic is the same thing as constructing an isomorphism
between them.

The word “dependent” in “dependent type theory” refers to the fact that types can depend on elements
of other types. Such types are called \emph{dependent types} or \emph{families of types}. Given a
type $A$, having a dependent type $B$ over $A$ means that for every element $a:A$ there is a type
$B(a)$. Dependent types are essential for the representation of quantified propositions as we see in
chapter \ref{ch:hott}. For instance, a proposition depending on a natural number $n:\N$ is
represented by a type depending on the variable $n$. A dependent type $B$ over $A$ where all the
types $B(a)$ are seen as propositions is called a \emph{predicate on $A$}.

The constructivity property of dependent type theory enables one to see it as a programming
language. In dependent type theory all primitive constructions have \emph{computation rules} (or
\emph{reduction rules}), which essentially explain how to execute the programs of the language. All
elements of types can then be seen as programs and can be executed, simply by repeatedly applying
the computation rules. Note that in dependent type theory there are no infinite loops. All programs
terminate and therefore a result is always obtained when executing a program. From the point of view
of mathematics, the computation rules are the defining equations of the primitive constructions, and
applying a computation rule corresponds to replacing something by its definition. Two elements $u$
and $v$ of a given type $A$ are said to be \emph{definitionally equal} (or \emph{judgmentally
  equal}) if they become syntactically equal after replacing everything by their definition, i.e.\
after executing $u$ and $v$. An important rule of type theory, known as the \emph{conversion rule},
states that if $u$ is of type $A$ and $A$ is definitionally equal to $A'$, then $u$ has also type
$A'$. In particular, types are unique only up to definitional equality, but definitional equality is
decidable because it is simply a matter of repeatedly unfolding the definitions. In the same way as
it doesn’t make sense to prove that a term $u$ is of type $A$, it also doesn’t make sense to prove
that two terms or two types are definitionally equal. This is something that can simply be checked
algorithmically.

Given the correspondence between proofs and elements of types it follows that proofs themselves can
be executed, which is what gives dependent type theory its constructive nature. For instance, given
a proof that there exists a natural number having a certain property, one can execute the proof and
the final result will be a pair of the form $(n,p)$ where $n$ is a natural number of the form either
$0$, $1$, $2$, … (i.e.\ we know its value) and $p$ is a proof that $n$ does satisfy the property.
This close relation between type theory and computer science led to the development of \emph{proof
  assistants} like Coq, Agda or Lean (see \cite{coq}, \cite{agda}, \cite{lean}). They are
essentially type-checkers for dependent type theory together with various features making them
easier to use. In a proof assistant, one can state a theorem by defining the corresponding type and
then prove it by constructing a term (i.e.\ writing a program) having this type. If the proof
assistant accepts it, it means that the program representing the proof is well-typed and that,
therefore, the proof is correct.

\paragraph{Homotopy type theory}

Dependent type theory is very successful but suffers from a few problems, in particular when it
comes to the treatment of equality. Given a type $A$ and two elements $u,v:A$, the proposition “$u$
is equal to $v$” is reified as a type $u=_Av$ called the \emph{identity type} (whose elements are
proofs that $u$ is equal to $v$). Martin--Löf gave several versions of dependent type theory with
different rules for the identity types. In one of them, called \emph{extensional type theory}, the
identity types are behaving in a nice way but typing is not decidable, i.e.\ there is no algorithm
checking whether a term has a given type. This is usually an undesirable feature for a type theory.
In another one, called \emph{intensional type theory}, the rules of the identity types are different
and typing is decidable. However, the treatment of equality in intensional type theory is sometimes
unsatisfactory. For instance, two functions $f,g:A\to B$ can satisfy $f(x)=g(x)$ for every $x:A$
without being equal themselves as functions.  Defining the quotient of a set by an equivalence
relation is also quite problematic. A different issue is that the principle of \emph{uniqueness of
  identity proofs}, which states that for any $u,v:A$, any two proofs of $u=_Av$ are equal, isn’t
provable anymore, which is contrary to the intuition which was behind the identity types. Indeed,
the idea of the identity types in Martin--Löf’s type theory is that every type represents a set and
that $u=_Av$ represents the set having exactly one element if $u$ and $v$ are equal, and the empty
set if $u$ and $v$ are different.

Homotopy type theory is based on intensional type theory and resolves this last problem by changing
the intuition behind types and the identity types. In homotopy type theory, types are not seen as
sets anymore but as \emph{spaces}, dependent types are seen as \emph{fibrations}, and the identity
type $u=_Av$ is seen as the space of all \emph{continuous paths} from $u$ to $v$ in the space
$A$. Rather surprisingly, it can be shown that under this interpretation, all rules of intensional
type theory are still satisfied. Moreover, in this interpretation, uniqueness of identity proofs
isn’t a desirable property anymore. Given two points $u$ and $v$ in a space $A$ there can be many
non-homotopic paths from $u$ to $v$ and many non-homotopic homotopies between two paths, and so on.

This connection between type theory and homotopy theory was discovered around 2006 independently by
Vladimir Voevodsky and by Steve Awodey and Michael Warren in \cite{awodeywarren}. Then in 2009
Vladimir Voevodsky stated the \emph{univalence axiom}, proved its consistency in the simplicial set
model, and started the project of formalizing mathematics in this system, intensional type theory
with the univalence axiom, named \emph{univalent foundations}. Given a universe $\Type$, i.e.\ a
type whose elements are themselves types, and two elements $A$ and $B$ of $\Type$, the univalence
axiom identifies the identity type $A=_{\Type}B$ with the type of equivalences $A\simeq B$. This
axiom makes precise the idea that “isomorphic structures have the same properties”, which is often
used implicitly in mathematics. Note that it is not compatible with the principle of uniqueness of
identity proofs because, for instance, it implies that there are two different equalities
$\Bool=_{\Type}\Bool$ corresponding to the two bijections $\Bool\simeq\Bool$ (where $\Bool$ is the
type with two elements). Voevodsky also noticed that the univalence axiom implies function
extensionality, i.e.\ that if $f(x)=_Bg(x)$ for all $x:A$, then $f=_{A\to B}g$, and that it makes
the definition of quotients possible and well-behaved.

In 2011, the notion of \emph{higher inductive types} started to emerge. Ordinary inductive types are
types defined by giving some generators (the \emph{constructors}) and an induction principle making
precise the idea that the type is freely generated by the constructors. Higher inductive types are a
generalization of ordinary inductive types where we can give not only point-constructors but also
path-constructors. For instance, the circle has one point-constructor $\base$ and one
path-constructor $\lloop$ which is a path from $\base$ to $\base$.  In combination with univalence,
fibrations can be defined by induction on the base space, which is a very powerful way of defining
fibrations. For instance in order to define a fibration over the circle it is enough to give the
fiber over $\base$ and the action of $\lloop$ on this fiber (this action must be an equivalence).

One of the drawbacks of homotopy type theory is that by adding the univalence axiom or higher
inductive types, we lose the constructivity property which, as we mentioned previously, is an
essential feature of type theory. However, unlike the axiom of choice or excluded middle it is
widely believed that the univalence axiom and higher inductive types are constructive in some way,
and several people are trying to give an alternative description of homotopy type theory in which
univalence and higher inductive types compute, see in particular \cite{cubicaltt}. A related
conjecture is Voevodsky’s homotopy canonicity conjecture: for every closed term $n:\N$ constructed
using the univalence axiom, there exists a closed term $k:\N$ constructed \emph{without} using the
univalence axiom and a proof of $k=_{\N}n$.

\paragraph{Constructivity of $\pi_4(\Sn3)$}

The first major result of this thesis is corollary \ref{firstpi4s3} which states that there exists a
natural number $n$ such that $\pi_4(\Sn3)\simeq\Z/n\Z$. This statement is quite curious, because it
is a statement of the form “there exists a natural number $n$ satisfying a given property” hence
according to the constructivity conjecture it should be possible to extract from its proof the value
of $n$. However, nobody has managed to do it so far, mainly because the proof is relatively
complicated and that constructivity of the univalence axiom and of higher inductive types isn’t very
well understood yet. In chapters \ref{ch:smash}, \ref{ch:cohomology} and \ref{ch:gysin} we present a
proof that this number is equal to $2$, but note that this is a mathematical proof, as opposed to a
computation extracted from the definition of $n$, so it doesn’t address the constructivity
conjecture. However, it shows that we can define and work with cohomology and the Gysin sequence in
homotopy type theory, which is interesting in its own right.

\paragraph{Models of homotopy type theory}

We do not talk much about the relationship between homotopy type theory (synthetic homotopy theory)
and classical homotopy theory (analytic homotopy theory) in this thesis, apart from the fact that
many definitions and proofs look quite similar to their classical counterpart. A construction of a
model of homotopy type theory (minus higher inductive types) in classical homotopy theory is
presented in \cite{simplicialmodel} and a proof that they also model higher inductive types is in
preparation in \cite{ls:hit}. As we mentioned previously, one of the consequence of working
synthetically is that all the work done in this thesis is also valid in any other model of homotopy
type theory, not only the classical one. Michael Shulman gave in \cite{mikemodelsunivalence} various
other models of homotopy type theory and it is widely believed that any $\infty$-topos in the sense
of Lurie (cf \cite{htt}) gives a model of homotopy type theory.

Another very important model is the model of Thierry Coquand et al. described \cite{cubical}, which
is a constructive model of homotopy type theory in cubical sets. Note that, in theory, this model
should allow us to compute the number $n$ of chapter \ref{ch:james}, but this hasn’t been done at
the time of writing.  This model also suggests a different version of homotopy type theory, called
\emph{cubical type theory} (cf. \cite{cubicaltt}), but in this work we decided to stay with the type
theory used in \cite{hottbook}. Various squares and cubes are nevertheless used whenever convenient.


\chapter{Homotopy type theory}\label{ch:hott}

In this first chapter we give an introduction to homotopy type theory and to a few basic results
that are used throughout this work. The reader is encouraged to read \cite{hottbook} for a more
comprehensive presentation of homotopy type theory.  Unlike in \cite{hottbook}, we do not notate
definitional equalities differently from propositional equalities, we simply use the terminology
“$u=v$ by definition” when we want to insist on the fact that the equality is definitional. We also
use the standard notation $u\defeq v$ when \emph{introducing} new definitions. We use the word
“proposition” in its standard mathematical meaning. A proposition is either a statement which might
be true or false, for instance “the negation of the proposition $1+1=2$ is the proposition
$1+1\neq2$”, or a statement for which we do provide a proof. In particular when we say that a type
is “seen as a proposition” or when we state a proposition, it doesn’t mean that the type in
consideration is assumed to be $(-1)$-truncated in the sense of section
\ref{sec:truncatednesstruncations}. We reserve the expression “mere proposition” for such types.

All types are seen as elements of a particular type called $\Type$. For consistency reasons, $\Type$
cannot be an element of itself so we have an infinite sequence of universes $\Type_0$, $\Type_1$,
$\Type_2$, …, with $\Type_n:\Type_{n+1}$ for every $n$. In practice, though, we rarely need to worry
about which universe we are in, so from now on we simply write $\Type$ for any of the $\Type_n$, as
is often done in type theory.

\section{Function types}

We first present \emph{function types}. Given two types $A$ and $B$, there is a type written
\[A\to B\]
representing the type of functions from $A$ to $B$. A function can be defined by an explicit formula
as follows:
\begin{align*}
  f &: A \to B,\\
  f(x) &\defeq \Phi[x],
\end{align*}
where $\Phi[x]$ is a syntactical expression which may use the variable $x$ (and usually do, unless
the function is constant) and which is of type $B$ when we assume that $x$ is of type $A$. We can
also use the notation $\lambda x.\Phi[x]$ which is the same thing as the function $f$ above except
that it avoids the need to give it a name. Given a function $f:A\to B$ and an element $a:A$, we can
apply $f$ to $a$ and we obtain an element of type $B$,
\[f(a) : B.\]
Moreover, if $f$ is defined as above then $f(a)$ is equal to $\Phi[a/x]$ by definition, where
$\Phi[a/x]$ is the expression $\Phi[x]$ where all instances of the variable $x$ have been replaced
by $a$.

When we see $A$ and $B$ as spaces, an element of $A\to B$ should be thought of as a
\emph{continuous} function from $A$ to $B$.
When we see $A$ and $B$ as propositions, an element of $A\to B$ is a function turning a proof of $A$
into a proof of $B$. In other words, it corresponds to a proof of “$A$ implies $B$”. In particular,
this means that logical implications are translated into function types in type theory.

An element $P:A\to\Type$ is called a \emph{dependent type} over $A$ and it represents a family of
types indexed by $A$, or a fibration over $A$ if $A$ and all types $P(a)$ are seen as spaces, or a
predicate on $A$ if all types $P(a)$ are seen as propositions.

\begin{definition}
  Given a type $A$, the \emph{identity function} of $A$ is the function
  \begin{align*}
    \id_A &: A \to A,\\
    \id_A(x) &\defeq x.
  \end{align*}
\end{definition}

\begin{definition}
  Given three types $A$, $B$ and $C$ and two functions $f:A\to B$ and $g:B\to C$, the
  \emph{composition} of $f$ and $g$ is the function
  \begin{align*}
    g\circ f &: A \to C,\\
    (g\circ f)(x) &\defeq g(f(x)).
  \end{align*}
\end{definition}

\subsection{Dependent functions}

A function $f:A\to B$ always returns an element of type $B$ no matter what its argument is. It is
possible to generalize function types in order to allow the output type to depend on the value of
the input. More precisely, given a type $A$ and a dependent type $B:A\to\Type$, there is a type
written
\[(x:A)\to B(x)\quad\text{ or }\quad\prod_{x:A}B(x)\]
representing \emph{dependent functions} from $A$ to $B$, i.e.\ functions sending an element $x$ of
$A$ to an element of the corresponding type $B(x)$. Just as with regular functions, a dependent
function can be defined by an explicit formula as follows:
\begin{align*}
  f &: (x:A) \to B(x),\\
  f(x) &\defeq \Phi[x],
\end{align*}
where $\Phi[x]$ is an expression of type $B(x)$, and we can also write it $\lambda x.\Phi[x]$. When
we apply a dependent function $f:(x:A)\to B(x)$ to an element $a:A$, we get an element of type
$B(a)$ (which depends on $a$ in general) \[f(a):B(a).\]

When we see $A$ as a space and $B$ as a fibration over $A$, a dependent function $f:(x:A)\to B(x)$
should be seen as a continuous section of $B$.  When we see $A$ as a space and $B$ as a predicate on
$A$, the dependent function type $(x:A)\to B(x)$ corresponds to the universally quantified
proposition $\forall x:A, B(x)$. Indeed, proving the proposition $\forall x:A, B(x)$ corresponds to
proving $B(x)$ for every $x$ in $A$, which is exactly what a dependent function of type
$(x:A)\to B(x)$ does.

For instance, let us assume we have a type $A$ seen as a space and a dependent type $B$ over $A$
seen as a fibration over $A$. Then a theorem of the form “For every section $f$ of $B$, if $P(f)$
holds then $Q(f)$ holds” should be interpreted as the type
\[(f:(x:A) \to B(x))\to(P(f)\to Q(f)).\]
The arrow on the left represents the type of sections of $B$, the arrow on the right represents the
logical implication between $P(f)$ and $Q(f)$ and the arrow in the middle represents universal
quantification.  A proof of such a theorem is a function taking a function $f$ of type
$(x:A)\to B(x)$ (i.e.\ $f$ is a function taking an argument $x$ of type $A$ and returning a result of
type $B(x)$) and returning a function of type $P(f)\to Q(f)$, i.e.\ which takes an element of type
$P(f)$ (a proof that $f$ satisfies $P$) and returns an element of type $Q(f)$ (a proof that $f$
satisfies $Q$).

\subsection{Functions with several arguments}

There are several ways to talk about functions with several arguments. Let’s say for instance that
we are interested in a function $f$ taking two arguments, of types $A$ and $B$, and returning a
result of type $C$. One way to state it is to say that $f$ has type $(A\times B)\to C$, i.e.\ $f$
takes one argument of the product type $A\times B$ (that we define in the next section) and returns
an element of $C$. Another way to state it is to say that $f$ has type $A\to(B\to C)$, i.e.\ $f$
takes one argument of type $A$ and returns another function taking the second argument of type $B$
and returning the result of type $C$.
The two versions turn out to be equivalent, and in general we use the second version (called the
\emph{curried} form), as is common in type theory, with the syntax
\begin{align*}
  f &: A\to B\to C,\\
  f(a,b) &\defeq \Phi[a,b].
\end{align*}
Of course, the type $B$ could be a dependent type over $A$ and the type $C$ could be a dependent
type over both $A$ and $B$, and it can be generalized to functions with more than two arguments.

\section{Pair types}\label{sec:pairtypes}

We now present \emph{pair types}. Given two types $A$ and $B$, there is a type written
\[A\times B\]
representing the type of pairs consisting in one element of $A$ and one element of $B$. One can
construct an element of $A\times B$ by pairing one element $a$ of $A$ and one element $b$ of $B$:
\[(a,b):A\times B.\]
One can deconstruct an element of $A\times B$ as follows. If $P:A\times B\to\Type$ is a dependent
type over $A\times B$, then a section of it can be defined by
\begin{align*}
  f &: (z : A \times B) \to P(z),\\
  f((x,y)) &\defeq f_\times(x,y),
\end{align*}
where
\begin{align*}
  f_\times &: (x : A) (y : B) \to P((x,y)).
\end{align*}
For instance, we can define the first and the second projection by
\begin{align*}
  \fst &: A \times B \to A, & \snd &: A\times B \to B,\\
  \fst((a,b)) &\defeq a, & \snd((a,b)) &\defeq b.
\end{align*}

When we see $A$ and $B$ as propositions, the type $A\times B$ represents the \emph{conjunction} of
$A$ and $B$. Indeed proving that “$A$ and $B$” holds is equivalent to proving that both $A$ and $B$
hold, therefore a proof of “$A$ and $B$” can be seen as a pair $(a,b)$ where $a$ is a proof of $A$
and $b$ is a proof of $B$.

\subsection{Dependent pairs}

The second component of an element of $A\times B$ always has type $B$. It is possible to generalize
pair types in order to allow the type of the second component to depend on the value of the first
component. More precisely, given a type $A$ and a dependent type $B$ over $A$, there is a type
written
\[\sum_{x:A}B(x)\]
representing \emph{dependent pairs}. Such types are often called \emph{$\Sigma$-types}. Given $a:A$
and $b:B(a)$, we can construct the dependent pair
\[(a,b):\sum_{x:A}B(x).\]
One can define a function out of it in the same way as for non-dependent pair types. For instance
the first and second projections are defined by
\begin{align*}
  \fst &: \sum_{x:A}B(x) \to A, & \snd &: \left(z:\sum_{x:A}B(x)\right) \to B(\fst(z)),\\
  \fst((x,y)) &\defeq x, & \snd((x,y)) &\defeq y.
\end{align*}
Note that, this time, the second projection is a dependent function because the type of the second
component of a dependent pair depends on the first component.

It is possible to nest $\Sigma$-types in order to obtain types of arbitrary $n$-tuples. For instance
the type of semigroups can be defined as the type
\begin{align*}
  \operatorname{SemiGroup} &: \Type,\\
  \operatorname{SemiGroup} &\defeq \sum_{G:\Type}\sum_{m:G\to G\to G}((x,y,z:G)\to m(m(x,y),z) =_G m(x,m(y,z))).
\end{align*}
In other words, a semigroup is a triple $(G,(m,a))$ where $G$ is a type, $m$ is a function of type
$G\to G\to G$ (the multiplication operation), and $a$ is a proof of associativity of $m$, i.e., $a$
is a function taking three arguments $x$, $y$ and $z$ of type $G$ and returning an equality between
$m(m(x,y),z)$ and $m(x,m(y,z))$. Note that the type of $m$ depends on $G$, and that in turn the type
of $a$ depends on $m$.

When we see $B$ as a fibration over $A$, the type $\sum_{x:A}B(x)$ corresponds to the total space of
$B$. This will be used in particular in the flattening lemma in section \ref{flattening}. When we
see $B$ as a predicate on $A$, the type $\sum_{x:A}B(x)$ corresponds to the type of elements of $A$
which satisfy $B$. Note that there is a subtlety here because if for some $a:A$ there are several
distinct elements in $B(a)$, then $a$ is counted several times which isn’t what we want in
general. We can also see $\sum_{x:A}B(x)$ as corresponding to the proposition “there exists an $x:A$
satisfying $B(x)$”. Indeed, one can prove this proposition by exhibiting an $x:A$ and a proof that
$B(x)$ holds, i.e.\ an element of $\sum_{x:A}B(x)$. However this would be more accurately called
\emph{explicit existence} because it requires us to choose an explicit $x$ satisfying $B$, which
might be a too strong requirement in some cases. We will come back to both problems in section
\ref{sec:merepropsandlogic}.

\section{Inductive types}

We now present inductive types, which give a wide variety of type formers, including base types. The
general idea is that an \emph{inductive type} $T$ is presented by a list of \emph{constructors}
which describe all the different ways of constructing elements of $T$ and, in some sense which is
made precise by an \emph{induction principle} (or \emph{elimination rule}), the only elements of $T$
are those given by the constructors. In this section, all equalities (introduced by the symbol
$\defeq$) are equalities by definition. We now give various examples of inductive types.

\paragraph{Natural numbers}

The canonical example of an inductive type is the type of \emph{natural numbers} $\N$. The two
constructors are
\begin{align*}
  0 &: \N,\\
  \succS &: \N\to\N.
\end{align*}
In other words there are two ways to construct a natural number: either we take $0$ or we take the
successor of an already constructed natural number. We use the usual notation $1\defeq \succS(0)$,
$2\defeq \succS(\succS(0))$, and so on, and we write $n+1$ for $\succS(n)$.

The induction principle states that given a dependent type $P$ over $\N$, we can define a section of
it by giving $f_0$ and $f_{\succS}$ as follows:
\begin{align*}
  f &: (n : \N) \to P(n),\\
  f(0) &\defeq f_0,\\
  f(n+1) &\defeq f_{\succS}(n,f(n)),
\end{align*}
where we have
\begin{align*}
  f_0 &: P(0),\\
  f_{\succS} &: (n : \N) \to P(n) \to P(n+1).
\end{align*}
For instance, one can define addition and multiplication on natural numbers by
\begin{align*}
  \add &: \N \to \N \to \N, & \mul &: \N \to \N \to \N,\\
  \add(0, n) &\defeq n, & \mul(0, n) &\defeq 0,\\
  \add(m + 1, n) &\defeq \add(m,n) + 1, & \mul(m + 1, n) &\defeq \add(\mul(m,n),n).
\end{align*}

\paragraph{The unit type}

The \emph{unit type} is the inductive type $\Unit$ with one constructor
\[\ttt:\Unit.\]
Its induction principle states that if $P(x)$ is a dependent type over $x:\Unit$, then we can
construct a section of $P(x)$ by
\begin{align*}
  f &: (x : \Unit) \to P(x),\\
  f(\ttt) &\defeq f_{\ttt},
\end{align*}
where $f_{\ttt}:P(\ttt)$.

\paragraph{The type of booleans}

The \emph{type of booleans} or \emph{$2$-element type} is the inductive type $\Bool$ with two
constructors
\[\true,\false:\Bool.\]
Its induction principle states that if $P(x)$ is a dependent type over $x:\Bool$, then we can
construct a section of $P(x)$ by
\begin{align*}
  f &: (x : \Bool) \to P(x),\\
  f(\true) &\defeq f_{\true},\\
  f(\false) &\defeq f_{\false},
\end{align*}
where $f_{\true}:P(\true)$ and $f_{\false}:P(\false)$.

\paragraph{Disjoint sum}

Given two types $A$ and $B$, their \emph{disjoint sum} is the inductive type $A+B$ with the two
constructors
\begin{align*}
  \inl &: A \to A+B,\\
  \inr &: B \to A+B.
\end{align*}
Its induction principle states that if $P(x)$ is a dependent type over $x:A+B$, then we can
construct a section of $P(x)$ by
\begin{align*}
  f &: (x : A+B) \to P(x),\\
  f(\inl(a)) &\defeq f_{\inl}(a),\\
  f(\inr(b)) &\defeq f_{\inr}(b),
\end{align*}
where
\begin{align*}
  f_{\inl} &: (a : A) \to P(\inl(a)),\\
  f_{\inr} &: (b : B) \to P(\inr(b)).
\end{align*}

This type can also be used to represent \emph{explicit disjunction}. Indeed, when $A$ and $B$ are
seen as propositions, an element of $A+B$ is either a proof of $A$ or a proof of $B$, hence it’s a
proof of “$A$ or $B$”. As for explicit existence, the drawback is that it requires us to choose
whether we proved the left-hand side or the right-hand side which might be a too strong
requirement. We will come back to that in section \ref{sec:merepropsandlogic}.

\paragraph{Integers}

The type of \emph{integers} is the inductive type $\Z$ with the three constructors
\begin{align*}
  \negZ &: \N \to \Z,\\
  0_{\Z} &: \Z,\\
  \posZ &: \N \to \Z.
\end{align*}
An integer is either $0_{\Z}$, $\posZ(n)$ for $n:\N$ (which represents $n+1$) or $\negZ(n)$ for
$n:\N$ (which represents $-(n+1)$). There is again an induction principle derived from the
constructors. For instance one can define the operation of adding one by
\begin{align*}
  \succZ &: \Z \to \Z,\\
  \succZ(\negZ(n+1)) &\defeq \negZ(n),\\
  \succZ(\negZ(0)) &\defeq 0_{\Z},\\
  \succZ(0_{\Z}) &\defeq \posZ(0),\\
  \succZ(\posZ(n)) &\defeq \posZ(n+1).
\end{align*}
Note that here we have used both the induction principle for $\Z$ and the one for $\N$ (in the case
$\negZ$).

\paragraph{The empty type}

The \emph{empty type} is the inductive type $\bot$ (also called $\Zero$) without any constructor. In
particular there is no way to construct an element of it. Its induction principle states that any
dependent type over $\bot$ has a section. For instance there is a canonical map $\bot\to A$ for any
type $A$.

The empty type is used to represent the proposition False. Indeed, False is the proposition that
doesn’t have any proof. It is also used for defining negation. When $A$ is a type seen as a
proposition, the type representing its negation is $\neg A\defeq(A\to\bot)$. For instance, we can
prove the principle of reasoning by contraposition by
\begin{align*}
  \operatorname{contraposition} &: (A\to B) \to (\neg B \to \neg A),\\
  \operatorname{contraposition}(f,b',a) &\defeq b'(f(a)).
\end{align*}
Here $f$ is a function from $A$ to $B$, $b'$ is a function from $B$ to $\bot$ and $a$ is a element
of $A$, therefore by applying $f$ and then $b'$ to $a$ we get an element of type $\bot$ which is
what we wanted.

\paragraph{General inductive types}

More generally one can introduce new inductive types whenever needed, but there are a few
constraints that have to be satisfied in order for them to make sense and be consistent. In
particular:
\begin{itemize}
\item Every constructor must end with the type being defined, i.e.\ a constructor cannot give
  elements in a different type (we will slightly relax this condition in section \ref{sec:hit} on
  higher inductive types).
\item The recursive occurrences of the type being defined have to be in strictly positive positions,
  i.e.\ they can only appear as the codomain of an argument of the constructors.
\end{itemize}
We refer to \cite[chapter 5]{hottbook} for more details on inductive types.

\section{Identity types}\label{idtypes}

Given a type $A$ and two elements $u$ and $v$ of type $A$, there is a type $u=_Av$ (or simply $u=v$
when $A$ is clear from the context) called the \emph{identity type} or \emph{equality type}. Its
elements are called \emph{paths}, \emph{equalities} or \emph{identifications}. The idea is that when
we see $A$ as a space, the type $u=_Av$ corresponds to the space of (continuous) paths from $u$ to
$v$. To understand the notation, note that if $A$ is just a set seen as a discrete space, then there
is a path between $u$ and $v$ if and only if $u$ and $v$ are equal. Therefore, in that case the type
$u=_Av$ does correspond to the logical proposition “$u$ is equal to $v$”.

For every point $a:A$, there is a path
\[\idp a:a=_Aa\]
which we call the \emph{constant path at $a$}. Moreover, we have an “induction principle”. Given an
point $a:A$, a dependent type $P:(x:A)(p:a=_Ax)\to\Type$ and an element $d:P(a,\idp a)$, there is a
function
\[\J_P(d) : (x : A) (p : a=_Ax) \to P(x,p)\]
together with an equality (called \emph{the computation rule of $\J$})
\begin{equation}
  \tag{$\jcomp$}
  \J_P(d)(a,\idp a) = d.
\end{equation}
We refer to the use of this operator $\J$ as \emph{path induction}. The idea is that when we want to
prove/construct something depending on a path $p$ where one endpoint of $p$ is free (the $x$ above),
then it is enough to prove/construct it in the case where $p$ is the constant path. It is important
that one of the endpoints be free, for instance if we only have $P:(a=_Aa)\to\Type$ and
$d:P(\idp a)$, we \emph{cannot} deduce that $P(p)$ hold for every $p:a=_Aa$. This last statement is
known as the “$\mathsf{K}$ rule” and is equivalent to the principle of uniqueness of identity proofs
which we do not want.

There is some debate on whether the equality $(\jcomp)$ should be taken as a definitional equality
or not. On the one hand it was assumed definitional in the original definition of the identity type
by Per Martin--Löf in \cite{ML75} and most of the work done in homotopy type theory so far has been
done assuming it is definitional, in particular the reference book \cite{hottbook}. On the other
hand the original motivation for having a definitional equality was that the family of identity
types was supposed to be inhabited only by elements of the form $\idp a$, but homotopy type theory
challenges this intuition by adding new elements to the identity types via the univalence axiom and
higher inductive types. In this thesis we assume that it holds definitionally, but the only place
where we really use it is in the definition of the structure of weak $\infty$-groupoid, and we
conjecture that it is not actually required.

\subsection{The weak $\infty$-groupoid structure of types}\label{sec:infgpd}

The path induction principle allows us to equip every type $A$ with a very rich structure making $A$
into what is known as an \emph{weak $\infty$-groupoid}. In this chapter we give an intuitive
presentation by giving many examples of operations part of this structure, and we give a precise
definition of weak $\infty$-groupoids in appendix \ref{ch:infgpd}.

The first thing to notice is that the identity type operation can be iterated. If we have
$p,q:u=_Av$ we can consider the type $p=_{u=_Av}q$, and if $\alpha,\beta:p=_{u=_Av}q$ we can
consider the type $\alpha=_{p=_{u=_Av}q}\beta$, and so on. If we see $p$ and $q$ as proofs of
equality between $u$ and $v$, it seems rather odd to talk about equalities between them, let alone
equalities between equalities between them. However, when seeing $u$ and $v$ as two points in a
space $A$ and $p$ and $q$ as two paths between $u$ and $v$, it makes sense to think of $p=_{u=_Av}q$
as the type of homotopies between $p$ and $q$ and then $\alpha=_{p=_{u=_Av}q}\beta$ as the type of
homotopies between the homotopies $\alpha$ and $\beta$, and so on. Let’s now look at several
examples of operations (called \emph{coherence operations}) that we can do on those paths and higher
paths.

\paragraph{Inverse}

If $p:a=b$ is a path in $A$, then there is a path $p\inv:b=a$ called the \emph{inverse path} of
$p$. We write this operation as follows:
\[(a,b:A)(p:a=b) \mapsto (p\inv:b=a).\]
To define it, the idea is to do a path induction on $p$, where the dependent type $P$ is
$P(b,p)\defeq(b=a)$. We then need to give an element of $P(a,\idp a)$ (i.e.\ $a= a$) and we use
$\idp a$. In particular, we get the equality by definition \[\idp a\inv\defeq\idp a.\]
If we do not take $(\jcomp)$ as an equality by definition, then the equality still holds but only in
the sense that there is a path
\[\operatorname{inv-def}:\idp a\inv=\idp a.\]

\paragraph{Composition}

If $p:a=b$ and $q:b=c$ are two paths in $A$, then there is a path $p\concat q:a=c$ called the
\emph{composition} of $p$ and $q$. This operation is written as
\[(a,b:A)(p:a=b)(c:A)(q:b=c) \mapsto (p\concat q:a=c).\]
It is obtained by applying the $\J$ rule successively to $q$ and $p$ and we have the equality
\[\idp a\concat\idp a\defeq\idp a\]
by definition.
If $(\jcomp)$ is not an equality by definition, then we have instead a term
\[\operatorname{comp-def}:\idp a\concat\idp a=\idp a.\]

\paragraph{Associativity}

If we take three composable paths $p:a=b$, $q:b=c$ and $r:c=d$, then there are two ways to compose
them and we define an operation as follows:
\[(a,b:A)(p:a=b)(c:A)(q:b=c)(d:A)(r:c=d) \mapsto (\alpha_{p,q,r} : (p \concat q) \concat r = p
\concat (q \concat r)).\]
By path induction on $r$, $q$ and $p$, it suffices to give $\alpha_{\idp a,\idp a,\idp a}$ of type
$(\idp a\concat\idp a)\concat\idp a=\idp a\concat(\idp a\concat\idp a)$. But we have
$\idp a\concat\idp a=\idp a$ by definition, therefore the equality holds by definition and
$\idp{\idp a}$ fits:
\[\alpha_{\idp a,\idp a,\idp a}\defeq\idp{\idp a}.\]
Note that if we do not take $(\jcomp)$ as an equality by definition, then it is still possible to
find a term of type $(\idp a\concat\idp a)\concat\idp a=\idp a\concat(\idp a\concat\idp a)$ but it
is more complicated to define as we need to explicitly combine several uses of
$\operatorname{comp-def}$. The complexity would increase even more for more complicated coherence
operations.

\paragraph{Inverse and identity laws}

Similarly we have the inverse and identity laws:
\begin{align*}
  (a,b:A)(p:a=b) &\mapsto (\zeta_p:p\concat p\inv=_{a=a}\idp a),\\
  (a,b:A)(p:a=b) &\mapsto (\eta_p:p\inv\concat p=_{b=b}\idp b),\\
  (a,b:A)(p:a=b) &\mapsto (\rho_p:p\concat \idp b=_{a=b}p),\\
  (a,b:A)(p:a=b) &\mapsto (\lambda_p:\idp a\concat p=_{a=b}p).
\end{align*}
These four operations are defined by path induction on $p$, using the fact that $\idp a\inv=\idp a$
and $\idp a\concat\idp a=\idp a$ by definition, and we have
\begin{align*}
  \zeta_{\idp a} &\defeq \idp{\idp a}, & \eta_{\idp a} &\defeq \idp{\idp a},\\
  \rho_{\idp a} &\defeq \idp{\idp a}, & \lambda_{\idp a} &\defeq \idp{\idp a}.
\end{align*}

\paragraph{Vertical and horizontal composition}

\newcommand{\phreplace}[2]{\makebox[0pt][l]{#1}\phantom{#2}}

Higher dimensional paths can be composed in a variety of ways. For $2$-dimensional paths, we first
consider \emph{vertical composition}, which is the operation
\[(a,b:A)(p,q:a=b)(\alpha : p=_{a=b}q)(r:a=b)(\beta : q=_{a=b}r) \mapsto \alpha\vcomp\beta :
p=_{a=b}r\]
and which corresponds to the diagram
\[
\begin{tikzcd}[sdiag,column sep=5em]
  \phreplace{$a$}{b} \arrow[r,bend left=75,"p",""{anchor=base,pos=0.5,name=p}]
    \arrow[r,"q" near start,""{anchor=base,pos=0.5,name=q}]
    \arrow[r,bend right=75,"r"',""{anchor=base,pos=0.5,name=rp}]
  & b
  \arrow[Rightarrow,from=p,to=q,"\alpha"]
  \arrow[Rightarrow,from=q,to=rp,"\beta"]
\end{tikzcd}
\]
Note that vertical composition is the same thing as regular composition in the type $a=b$, therefore
we use the same notation.
We also have \emph{horizontal composition}, which is the operation
\begin{align*}
  (a,b:A)(p,q:a=b)(\alpha :p=_{a=b}q)\qquad&\\(c:A)(p',q':b=c)(\beta : p'=_{b=c}q') &\mapsto \alpha\hcomp\beta :
                                                                          (p\concat p')=_{a=c}(q\concat q')
\end{align*}
and which corresponds to the diagram
\[
\begin{tikzcd}[sdiag,column sep=5em]
  \phreplace{$a$}{b} \arrow[r,bend left=45,"p",""{anchor=base,pos=0.5,name=p1}]
    \arrow[r,bend right=45,"q"',""{anchor=base,pos=0.5,name=q1}]
  & b \arrow[r,bend left=45,"p'",""{anchor=base,pos=0.5,name=p2}]
  \arrow[r,bend right=45,"q'"',""{anchor=base,pos=0.5,name=q2}]
  & \phreplace{$c$}{b}
  \arrow[Rightarrow,from=p1,to=q1,"\alpha"]
  \arrow[Rightarrow,from=p2,to=q2,"\beta"]
\end{tikzcd}
\]
These two operations are defined by successive path inductions and we have
\begin{align*}
  \idp{\idp a}\concat\idp{\idp a} &\defeq \idp{\idp a},\\
  \idp{\idp a}\hcomp\idp{\idp a} &\defeq \idp{\idp a}.
\end{align*}

\paragraph{Exchange law}
Given four $2$-dimensional paths as follows:
\[
\begin{tikzcd}[sdiag,column sep=5em]
  \phreplace{$a$}{b} \arrow[r,bend left=75,"p",""{anchor=base,pos=0.5,name=p1}]
    \arrow[r,"q" near start,""{anchor=base,pos=0.5,name=q1}]
    \arrow[r,bend right=75,"r"',""{anchor=base,pos=0.5,name=r1}]
  & b \arrow[r,bend left=75,"p'",""{anchor=base,pos=0.5,name=p2}]
    \arrow[r,"q'" near start,""{anchor=base,pos=0.5,name=q2}]
    \arrow[r,bend right=75,"r'"',""{anchor=base,pos=0.5,name=r2}]
  & \phreplace{$c$}{b}
  \arrow[Rightarrow,from=p1,to=q1,"\alpha"]
  \arrow[Rightarrow,from=q1,to=r1,"\beta"]
  \arrow[Rightarrow,from=p2,to=q2,"\alpha'"]
  \arrow[Rightarrow,from=q2,to=r2,"\beta'"]
\end{tikzcd}
\]
we can consider the two compositions
\[(\alpha\vcomp\beta)\hcomp(\alpha'\vcomp\beta'):p\concat p'=r\concat r'\]
and
\[(\alpha\hcomp\alpha')\vcomp(\beta\hcomp\beta'):p\concat p'=r\concat r'.\]
The \emph{exchange law} states that there is a $3$-dimensional path equating them and is defined by
successive path inductions, using the fact that
\begin{align*}
  (\idp{\idp a}\vcomp\idp{\idp a})\hcomp(\idp{\idp a}\vcomp\idp{\idp a}) &= \idp{\idp a},\\
  (\idp{\idp a}\hcomp\idp{\idp a})\vcomp(\idp{\idp a}\hcomp\idp{\idp a}) &= \idp{\idp a}
\end{align*}
by definition. We use it in the Eckmann--Hilton argument in proposition \ref{eckmannhilton}.

\paragraph{Pentagon of associativities}

Given four composable paths $p$, $q$, $r$ and $s$, there are two different ways to go from
$((p\concat q)\concat r)\concat s$ to $p\concat(q\concat (r\concat s))$, and the following pentagon
states that there is a $3$-dimensional path between them:
\[
\begin{tikzpicture}[commutative diagrams/every diagram]
  \node[label=below:{\small $p\concat(q\concat(r\concat s))$}] (P0) at (-54:8em) {$\bullet$};
  \node[label=right:{\small $(p\concat q)\concat(r\concat s)$}] (P1) at (18:8em) {$\bullet$};
  \node[label=above:{\small $((p\concat q)\concat r)\concat s$}] (P2) at (90:8em) {$\bullet$};
  \node[label=left:{\small $(p\concat(q\concat r))\concat s$}] (P3) at (162:8em) {$\bullet$};
  \node[label=below:{\small $p\concat((q\concat r)\concat s)$}] (P4) at (234:8em) {$\bullet$};

  \path[commutative diagrams/.cd, every arrow, every label, labels={font=\small}]
  (P2) edge[sdiag] node[swap] {\Large $\alpha_{p\concat q,r,s}$} (P1)
       edge[sdiag] node[near end] {\Large $\alpha_{p,q,r}\,\hcomp\,\idp s$} (P3)
  (P1) edge[sdiag] node[swap] {\Large $\alpha_{p,q,r\concat s}$} (P0)
  (P3) edge[sdiag] node {\Large $\alpha_{p,q\concat r,s}$} (P4)
  (P4) edge[sdiag] node {\Large $\idp p\,\hcomp\,\alpha_{q,r,s}$} (P0);
\end{tikzpicture}
\]
It corresponds to the coherence operation
\begin{align*}
&(a,b:A)(p:a=b)(c:A)(q:b=c)(d:A)(r:c=d)(e:A)(s:d=e) \mapsto\\ &\qquad\qquad\qquad (\pi_p:(\alpha_{p\concat q,r,s}\concat\alpha_{p,q,r\concat s})=(\alpha_{p,q,r}\hcomp\idp s)\concat\alpha_{p,q\concat r,s}\concat(\idp p\hcomp\alpha_{q,r,s})).
\end{align*}
and it is defined by successive path inductions and using the fact that
\begin{align*}
  \alpha_{\idp a\concat \idp a,\idp a,\idp a} \concat \alpha_{\idp a,\idp a,\idp a\concat \idp a} &= \idp{\idp a}
\end{align*}
and
\begin{align*}
  (\alpha_{\idp a,\idp a,\idp a}\hcomp\idp{\idp a})\concat\alpha_{\idp a,\idp a\concat \idp a,\idp
  a}\concat(\idp{\idp a}\hcomp\alpha_{\idp a,\idp a,\idp a}) &= \idp{\idp a}
\end{align*}
by definition.
\paragraph{General coherence operations}

The operations we just presented are only examples of coherence operations. There are many others
and in all dimensions. The structure of a general coherence operation is as follows: the list of
arguments needs to form a “contractible” shape in the sense that one can apply $\J$ to it until
there is only one point $a:A$ left, and the return type can be any iterated identity type between
two terms built using the variables and other coherence operations. All coherence operations have
the property that they are equal by definition to the iterated constant path when applied only to
constant paths. This property allows us to define coherence operations by successive path inductions
as we did above in the examples. A precise definition of weak $\infty$-groupoids together with a
proof that path induction gives such a structure on types is presented in appendix \ref{ch:infgpd}.

\subsection{Continuity of maps}\label{sec:cont}

Given a map $f:A\to B$ and a path $p:a=_Ab$ in $A$, we can \emph{apply} $f$ to $p$, and obtain a
path in $B$ between $f(a)$ and $f(b)$:
\[\ap f(p) : f(a)=_Bf(b).\]
This path is defined by path induction on $p$ and we have the equality
\[
\ap f(\idp a) = \idp{f(a)}.
\]
Here are some useful properties of $\ap{}$, which can all be proved by path induction.
\begin{itemize}
\item If $f=(\lambda x.a)$ is a constant function, then $\ap f(p)=\idp a$ for every $p$.
\item If $f=(\lambda x.x)$ is the identity function, then $\ap f(p)=p$ for every $p$.
\item If $f=g\circ h$, then $\ap {g\circ h}(p)=\ap g(\ap h(p))$.
\end{itemize}
Note that we have defined $\ap{}$ only for non-dependent functions, we will see later how to define
it for dependent functions as well.

If we see $a=_Ab$ as the proposition “$a$ and $b$ are equal”, then $\ap{}$ simply states that
application of functions respects equality, which is something very reasonable to ask. However, if
we see $a=_Ab$ as the space of paths from $a$ to $b$, then $\ap{}$ states that we can apply
functions not only to points but also to paths. One can also apply functions to $2$-dimensional
paths by
\begin{align*}
  \ap f^2 &: (p=_{a=_Ab}q) \to (\ap f(p)=_{f(a)=_Bf(b)}\ap f(q)),\\
  \ap f^2(p) &\defeq \ap{\ap f}(p).
\end{align*}
The fact that we can apply functions to paths and to higher dimensional paths is compatible with the
intuition that, in homotopy type theory, all functions are continuous.

When we see types as weak $\infty$-groupoid as sketched previously, then any map $f:A\to B$ should be
seen as an \emph{$\infty$-functor}. Indeed , it turns out that (the iterated versions of) $\ap f$
commute with all coherence operations. For instance given any two composable paths $p:a=_Ab$ and
$q:b=_Ac$ in $A$, we have an equality
\[\ap f(p\concat q) = \ap f(p)\concat\ap f(q),\]
which can be proved by path induction on $p$ and $q$. It also holds for higher coherences but it is
more complicated to write. For instance for associativity, given $p:a=_Ab$, $q:b=_Ac$ and $r:c=_Ad$
we have
\begin{align*}
  \ap f^2(\alpha_{p,q,r}) &: \ap f((p\concat q)\concat r) = \ap f(p\concat (q\concat r)),\\
  \alpha_{\ap f(p),\ap f(q),\ap f(r)} &: ((\ap f(p)\concat\ap f(q))\concat\ap f(r)) = (\ap
                                        f(p)\concat(\ap f(q)\concat\ap f(r)))
\end{align*}
and the commutation of $\ap f^2$ with $\alpha$ states that they are equal after composition with the
two paths
\begin{align*}
  \ap f((p\concat q)\concat r) &= (\ap f(p)\concat\ap f(q))\concat\ap f(r),\\
  \ap f(p\concat (q\concat r)) &= \ap f(p)\concat(\ap f(q)\concat\ap f(r))
\end{align*}
constructed using the fact that $\ap f$ commutes with composition of paths.

\section{The univalence axiom}\label{sec:ua}

Given two types $A$ and $B$, we can consider the identity type $A=_{\Type} B$ of paths between $A$
and $B$ in $\Type$. An equality between two types allows us to transport elements from one type to
the other, we have the two functions
\begin{align*}
  \coe &: (A=_{\Type} B) \to (A \to B),\\
  \coe\inv &: (A=_{\Type} B) \to (B \to A),
\end{align*}
which are defined by path induction, and which satisfy
\begin{align*}
  \coe_{\idp A}(a) &= a,\\
  \coe\inv_{\idp A}(a) &= a.
\end{align*}
We do not take these equalities as definitional equalities because we never need it and we wish to
limit the use of the definitional $(\jcomp)$.  We can also easily check that those two functions are
inverse to each other, again by path induction:
\begin{align*}
  \coeidp &: (p : A=_{\Type}B) (a : A) \to \coe\inv_p(\coe_p(a))=_Aa,\\
  \coeidp' &: (p : A=_{\Type}B) (b : B) \to \coe_p(\coe\inv_p(b))=_Bb.
\end{align*}

Another function related to $\coe$ is the transport function. Given a type $A$ and a dependent type
$P:A\to\Type$, we have the function
\begin{align*}
  \transport^P &: (a=_Ab) \to (P(a) \to P(b)),\\
  \transport^P_p(x) &\defeq \coe_{\ap P(p)}(x)
\end{align*}
defined using the fact that $\ap P(p)$ is a path in the universe from $P(a)$ to $P(b)$. We can also
transport in the other direction and it gives an equivalence between $P(a)$ and $P(b)$.

The \emph{univalence axiom} states that conversely, any “equivalence” between two types can be
turned into a path in the universe. We first define equivalences.
\begin{proposition}\label{equivalences}
  There is a predicate $\isequiv:(A,B:\Type)(f:A\to B)\to\Type$ where $\isequiv(f)$ is read as “$f$
  is an equivalence”, which has the following properties.
  \begin{itemize}
  \item A function $f$ is an equivalence if and only if there is a function $g:B\to A$ such that
    $g(f(x))=x$ for all $x:A$ and $f(g(y))=y$ for all $y:B$.
  \item Given a function $f:A\to B$, any two elements of $\isequiv(f)$ are equal.
  \end{itemize}
\end{proposition}
\begin{proof}
  The first point makes it seem like we could define $\isequiv$ by
  \[\isequiv(f)\overset{?}\defeq\sum_{g:B\to A}\left(((x:A)\to g(f(x))=_Ax)\times((y:B)\to
    f(g(y))=_By)\right),\]
  but it turns out that this definition does not satisfy the second point of the proposition. One
  possibility to fix it is to consider separately one left-inverse and one right-inverse, instead of
  one two-sided inverse:
  \begin{align*}
    \isequiv(f)&\defeq\Bigg(\sum_{g:B\to A}((x:A)\to g(f(x))=_Ax)\Bigg)\\
               &\times\Bigg(\sum_{h:B\to A}((y:B)\to f(h(y))=_By)\Bigg).
  \end{align*}
  The two inverses turn out to be equal and that definition now satisfies the second point of the
  proposition.  We refer to \cite[section 4.3]{hottbook} for a proof that this definition of
  equivalences is suitable and to \cite[chapter 4]{hottbook} for many other equivalent definitions
  of equivalences.
\end{proof}
We can now state the univalence axiom. We first define the type of equivalences between two types by
\[A\simeq B \defeq \sum_{f:A\to B}\isequiv(f).\]
There is a map $(A=_{\Type}B)\to(A\simeq B)$ sending a path $p$ to the equivalence given by
$(\coe_p,\coe\inv_p,\coeidp_p,\coeidp'_p)$.
\begin{axiom}[Univalence axiom]
  The map \[(A=_{\Type}B)\to(A\simeq B)\] is an equivalence.
\end{axiom}
In particular, it means that there is a map
\begin{align*}
  \ua&:(A\simeq B)\to (A=_{\Type} B),
\end{align*}
which allows us to construct an equality between two types given an equivalence between them, and if
$e:A\simeq B$ is an equivalence then $\coe_{\ua(e)}$ is equal to the underlying function of $e$.

\section{Dependent paths and squares}

The notion of path we described in section \ref{idtypes} is \emph{homogeneous} in that it can only
be applied between two elements of the same type $A$. It does not make sense, at least in homotopy
type theory, to talk about a path between a point $a:A$ and a point $b:B$ for $A$ and $B$ two
unrelated types. However, if $A$ and $B$ are themselves connected by a path in $\Type$, then we can
make sense of it.
\begin{proposition}
  Given a path $p:A=_{\Type} B$ in $\Type$ between two types $A$ and $B$ and two terms $a:A$ and
  $b:B$, there is a type of \emph{heterogeneous paths} between $a$ and $b$ over $p$ written
  \[a=_pb\] and satisfying the two equivalences
  \begin{align*}
    (a=_pb) &\simeq (\coe_p(a) =_Bb),\\
    (a=_pb) &\simeq (a=_A\coe\inv_p(b)).
  \end{align*}
\end{proposition}

Either $(\coe_p(a) =_Bb)$ or $(a=_A\coe\inv_p(b))$ can be used as the definition of $a=_pb$. We
could also define it by path induction on $p$, or as an inductive family of types.
As a special case, if $a$ and $b$ have the same type $A$ and $p$ is the constant path, then the type
$a=_{\idp A}b$ is equivalent to the type $a=_Ab$. Therefore homogeneous paths can be seen as a
special case of heterogeneous paths via this equivalence.

Heterogeneous paths allow us to define dependent paths, which are useful in many different
situations as we will see.
\begin{proposition}
  Given a dependent type $B:A\to\Type$ over a type $A$ and a path $p:x=_Ay$ in $A$, for every
  $u:B(x)$ and $v:B(y)$ there is a type of \emph{dependent paths} in $B$ over $p$ from $u$ to $v$
  written
  \[u=^B_pv\]
  and satisfying the two equivalences
  \begin{align*}
    (u=^B_pv) &\simeq (\transport^B(p,u) =_{B(y)}v),\\
    (u=^B_pv) &\simeq (u=\transport^B(p\inv,v)).\\
  \end{align*}
\end{proposition}

We can define the type of dependent paths by
\[(u=^B_pv) \defeq (u=_{\ap B(p)}v).\]
This makes sense because $\ap B(p)$ is a path in the universe between $B(x)$ and $B(y)$.  As with
the type of heterogeneous paths, there are many other ways to define it.

Dependent paths are used in particular to define $\ap{}$ for dependent functions.
\begin{definition}\label{def:apd}
  Given a dependent function $f:(x:A)\to B(x)$ and a path $p:a=_Ab$ in $A$, there is a dependent
  path
  \[\ap f(p) : f(a) =^B_p f(b)\]
  defined by path induction and satisfying
  \[\ap f(\idp a)=\idp{f(a)}.\]
  Note that this last equation uses implicitly the equivalence between $u=^B_{\idp a}v$ and
  $u=_{B(a)}v$.
\end{definition}

If $B$ does not depend on $x$, then the types $f(a)=^B_pf(b)$ and $f(a)=_Bf(b)$ are equivalent and
the element $\ap f(p)$ defined here and the one defined in section \ref{sec:cont} are identified by
this equivalence. Therefore we keep the same notation for both and there will be no ambiguity.

In many cases, we can give a different characterization of $u=^B_pv$ depending on $B$. A very useful
instance of that is when $B$ is an identity type.
\begin{proposition}\label{depeqid}
  Given $A,B:\Type$, $f,g:A\to B$, $p:a=_Ab$, $u:f(a)=_Bg(a)$ and
  $v:f(b)=_Bg(b)$, the type
  \[u =^{\lambda z. f(z)=_Bg(z)}_{p}v\]
  is equivalent to the type
  \[(u\cdot \ap g(p)) = (\ap f(p)\cdot v).\]
\end{proposition}

We have the picture
\begin{center}
  \begin{tikzpicture}[scale=0.5,looseness=0.5]
    \draw (0,5.5) ellipse (4 and 3.5);

    \draw (0,0) ellipse (4 and 1);
    \draw (-2,0) node {$\bullet$} to[bend left] (0,0) to[bend right] (2,0) node {$\bullet$};

    \draw (-2,4) node {$\bullet$} to[bend left] (0,4) to[bend right] (2,4) node {$\bullet$};
    \draw (-2,7) node {$\bullet$} to[bend left] (0,7) to[bend right] (2,7) node {$\bullet$};
    \draw (-2,4) to[looseness=0] (-2,7);
    \draw (2,4)  to[looseness=0] (2,7);

    \draw (-2,5.5) ellipse (0.3 and 3.5);
    \draw (2,5.5) ellipse (0.3 and 3.5);

    \draw (-2.5,9.5) node {$f(a)=g(a)$};
    \draw (2.5,9.5) node {$f(b)=g(b)$};
    \draw (-0.2,-0.5) node {$p$};
    \draw (-2, -0.5) node {$a$};
    \draw (2, -0.5) node {$b$};
    \draw (4.5, -0.5) node {$A$};
    \draw (4.5, 5) node {$B$};
    \draw (-0.2, 3.3) node {$\ap g(p)$};
    \draw (0.2, 7.7) node {$\ap f(p)$};
    \draw (-2.7, 5.5) node {$u$};
    \draw (2.7, 5.5) node {$v$};
  \end{tikzpicture}
\end{center}
The type $(u\cdot \ap g(p)) = (\ap f(p)\cdot v)$ can be seen as the type of \emph{fillers} of the
square
\[
\begin{tikzcd}[sdiag]
  \bullet \arrow[r,"\ap f(p)"] \arrow[d,"u"'] & \bullet \arrow[d,"v"]\\
  \bullet \arrow[r,"\ap g(p)"'] & \bullet
\end{tikzcd}
\]

There are many other equivalent ways to define the type of fillers of such a square. We could use
for instance any of the following four types (which are all equivalent):
\begin{align*}
u &= (\ap f(p)\concat v\concat\ap g(p)\inv),\\
v &= (\ap f(p)\inv\concat u\concat\ap g(p)),\\
\ap f(p) &= (u\concat \ap g(p)\concat v\inv),\\
\ap g(p) &= (u\inv\concat \ap f(p)\concat v).
\end{align*}
There is also a direct inductive definition, similar to the definition of the identity type, which
states that in order to construct a section of a dependent type over all squares whose upper-left
corner is a fixed point $a:A$, it is enough to define it over the identity square $\ids_a$ (which
has $\idp a$ on all four sides).

We can generalize the notion of coherence operation to include the case of squares or other
geometrical shapes. For instance in the diagram
\[
\begin{tikzcd}[sdiag]\label{trianglesquare}
  a \arrow[r,"p"] \arrow[d,"q"'] \arrow[rd,phantom,description,"\beta" pos=0.8]
  \arrow[rd,phantom,description,"\alpha" pos=0.2] & b \arrow[d,"r"]\\
  c \arrow[r,"s"'] \arrow[ru,"t" description] & d
\end{tikzcd}
\]
one can compose the two triangles $\alpha$ and $\beta$ in order to obtain a filler of the
square. The coherence operation is the one given by
\begin{align}
  &(a,c:A)(q:a=c)(b:A)(p:a=b)(t:c=b)(\alpha:q\inv\concat p=t)\\
  &\qquad(d:A)(s:c=d)(r:b=d)(\beta:t\inv\concat s=r) \mapsto \alpha\mathbin{\boxslash}\beta:(p\concat
    r)=(q\concat s).
\end{align}

A type of squares which is often used is the naturality squares of homotopies.
\begin{definition}
  Given two functions $f,g:A\to B$, a pointwise equality $h:(x:A)\to f(x)=_Bg(x)$ between $f$ and
  $g$ (also called a \emph{homotopy}) and a path $p:a=_Aa'$ in $A$, the \emph{naturality square} of
  $h$ on $p$ is the filler of the square
  \[
  \begin{tikzcd}[sdiag]
    \bullet \arrow[r,"\ap f(p)"] \arrow[d,"h(a)"'] & \bullet \arrow[d,"h(a')"]\\
    \bullet \arrow[r,"\ap g(p)"'] & \bullet
  \end{tikzcd}
  \]
  obtained by applying proposition \ref{depeqid} to $\ap h(p)$. Note that $h$ is a dependent
  function, so this $\ap{}$ is the dependent one defined in definition \ref{def:apd}.
\end{definition}

The next proposition shows that paths in $\Sigma$-types can be seen as pairs of paths.  There is an
operation of pairing of paths and one can take the first and second component of a path in a pair
type. Note also that the second path is a dependent path over the first one.

\begin{proposition}\label{pathspairtypes}
  Given a type $A$, a dependent type $B:A\to\Type$ over $A$ and two elements $(a,b)$ and $(a',b')$
  of $\sum_{x:A}B(x)$, there is an equivalence of types
  \[\left((a,b)=_{\sum_{x:A}B(x)}(a',b')\right) \simeq \left(\sum_{p:a=_Aa'}b=^B_pb'\right).\]
\end{proposition}

\begin{proof}
  The map from the left-hand side to the right-hand side is given by the two maps
  \begin{align*}
    \fsteq &: (r:(a,b)=(a',b')) \to a =_Aa',\\
    \fsteq(r) &\defeq \ap{\fst}(r),
  \end{align*}
  \begin{align*}
    \sndeq &: (r:(a,b)=(a',b')) \to b=^B_{\fsteq(r)}b',\\
    \sndeq(r) &\defeq \ap{\snd}(r)
  \end{align*}
  and the map from the right-hand side to the left-hand side is defined by path induction twice. The
  fact that these two maps are inverse to each other is immediate by path induction.
\end{proof}

For function types we have that paths between functions correspond to homotopies (pointwise equalities).

\begin{proposition}[Function extensionality]
  Given a type $A$, a dependent type $B:A\to\Type$ and two functions $f,g:(x:A)\to B(x)$, we have an
  equivalence
  \[(f=_{(x:A)\to B(x)}g) \simeq ((x:A)\to f(x)=_{B(x)}g(x)).\]
\end{proposition}
\begin{proof}
  The map from the left-hand side to the right-hand side is defined by
  \begin{align*}
    \happly &: (f=_{(x:A)\to B(x)}g) \to ((x:A)\to f(x)=_{B(x)}g(x)),\\
    \happly(p)(x) &\defeq \ap{\lambda h.h(x)}(p).
  \end{align*}
  The map in the other direction, which turns a homotopy between $f$ and $g$ into a path between $f$
  and $g$, cannot be defined by path induction as in the previous proposition. But it turns out that
  we can define it using the univalence axiom, and we can also prove that it is an inverse to
  $\happly$, see \cite[section 4.9]{hottbook}.
\end{proof}

\section{Higher inductive types}\label{sec:hit}

In ordinary inductive types, the constructors only generate elements of the type $T$ we are
defining. But in homotopy type theory, we are seeing paths and higher paths in $T$ as being somehow
still part of $T$. This leads to the notion of \emph{higher inductive types}, which are similar to
inductive types with the difference that there might be \emph{path-constructors} which give new
paths or new equalities in the type being defined. Here are some of the most important examples of
higher inductive types.

\paragraph{Circle}

The \emph{circle} $\Sn1$ is the simplest non-trivial higher inductive type. It is generated by the
two constructors

\begin{minipage}{0.47\linewidth}
  \begin{align*}
    \base &: \Sn1, \\
    \lloop &: \base=_{\Sn1}\base.\\
  \end{align*}
\end{minipage}
\begin{minipage}{0.47\linewidth}
  \begin{center}
    \begin{tikzpicture}
      \draw (0, 0) circle (.7cm); \draw (0.7, 0.05) -- (0.65, -0.05); \draw (0.7, 0.05) -- (0.75,
      -0.05); \draw (-0.7, 0) node {$\bullet$}; \draw (-1.3, 0) node {$\base$}; \draw (1.2, 0) node
      {$\lloop$}; \draw (1.4, -0.6) node {$\Sn1$};
    \end{tikzpicture}
  \end{center}
\end{minipage}

\noindent
Just as with inductive types, there is an induction principle which states that given a dependent
type $P:\Sn1\to\Type$, a function $f:(x:\Sn1)\to P(x)$ can be defined by
\begin{align*}
  f &: (x : \Sn1) \to P(x),\\
  f(\base) &\defeq f_{\base},\\
  \ap f(\lloop) &\defeq f_{\lloop},
\end{align*}
where the terms $f_{\base}$ and $f_{\lloop}$ satisfy
\begin{align*}
  f_{\base} &: P(\base),\\
  f_{\lloop} &: f_{\base} =^P_{\lloop} f_{\base}.
\end{align*}
Note that we need to give the value of $\ap f(\lloop)$ and not of $f(\lloop)$ (which would not make
sense given that $\lloop$ is not an element of $\Sn1$) and that $f_{\lloop}$ is a dependent path
in $P$ over $\lloop$.

There is a subtlety here in that, in the standard presentation of homotopy type theory, the second
equation $\ap f(\lloop) \defeq f_{\lloop}$ is \emph{not} taken as a definitional equality but only
as an equality in the sense of section \ref{idtypes}. There are two main reasons for that. The first
one is that most of the known models of homotopy type theory do not model it as a definitional
equality, and the second one is that most proof assistants do not allow it to be a definitional
equality either. These difficulties are being resolved. For instance the cubical model described in
\cite{cubical} models them as definitional equalities and one can now add custom new definitional
equalities in Agda as well (see \cite{rewritingagda}), but this is all very much work in
progress. There are a few places (for instance in the $3\times3$-lemma in section
\ref{threebythree}) where it would be very useful to have it as a definitional equality, together
with various other equalities, but for proofs written on paper (as opposed to proofs written in a
proof assistant) one can usually gloss over these kinds of issues. We use the same symbol $\defeq$
for aesthetic reasons.

Here are two examples of usage of the induction principle. We first define a function
$i:\Sn1\to\Sn1$ by
\begin{align*}
  i &: \Sn1\to\Sn1,\\
  i(\base) &\defeq \base,\\
  \ap i(\lloop) &\defeq \lloop\inv.
\end{align*}
Note that $i$ is a non-dependent function, therefore in the last line we just need to give a path in
$\Sn1$ from $i(\base)$ to itself, and we choose $\lloop\inv$.
We now prove that $i$ is involutive by
\begin{align*}
  \invol &: (x : \Sn1) \to i(i(x))=_{\Sn1}x,\\
  \invol(\base) &\defeq \idp\base,\\
  \ap{\invol}(\lloop) &\defeq \invol_{\lloop},
\end{align*}
with $\invol_{\lloop}$ to be defined. This time, the dependent type is
$P(x)\defeq(i(i(x)) =_{\Sn1}x)$.  In the case for $\base$ we need to prove that $i(i(\base))=\base$,
but using the fact that $i(\base)$ is equal to $\base$ by definition we have
$\idp\base:i(i(\base))=\base$. In the case for $\lloop$, we need to give a dependent path
$\invol_{\lloop}$ in $P$ over $\lloop$ from $\idp\base$ to $\idp\base$. Using proposition
\ref{depeqid} we see that $\invol_{\lloop}$ has to be a filler of the square
\[
\begin{tikzcd}[sdiag]
  \bullet \arrow[r,"\idp\base"] \arrow[d,"\ap{i\circ i}(\lloop)"'] & \bullet \arrow[d,"\ap{\lambda x.x}(\lloop)"]\\
  \bullet \arrow[r,"\idp\base"'] & \bullet
\end{tikzcd}
\]
Using now the fact that $\ap{\lambda x.x}(p)$ is equal to $p$, that $\ap{i\circ i}(p)$ is equal to
$\ap i(\ap i(p))$ and that $\ap i(\lloop)$ is equal to $\lloop\inv$, we have to fill the square
\[
\begin{tikzcd}[sdiag]
  \bullet \arrow[r,"\idp\base"] \arrow[d,"(\lloop\inv)\inv"'] & \bullet \arrow[d,"\lloop"]\\
  \bullet \arrow[r,"\idp\base"'] & \bullet
\end{tikzcd}
\]
and this follows from a coherence operation.

\paragraph{Pushouts}

Let’s consider three types $A$, $B$, $C$ and two functions $f:C\to A$, $g:C\to B$,
\[
\begin{tikzcd}
  A & C \arrow[l,"f"'] \arrow[r,"g"] & B.
\end{tikzcd}\]
Such a diagram is called a \emph{span}. The \emph{pushout} of this span is the higher inductive type
$A\sqcup^CB$ generated by the constructors
\begin{align*}
\inl&:A\to A\sqcup^CB,\\
\inr&:B\to A\sqcup^CB,\\
\push&:(c:C)\to\inl(f(c))=_{A\sqcup^CB}\inr(g(c)).
\end{align*}
In particular, we have the commutative square
\[
\begin{tikzcd}
  C \arrow[r,"g"] \arrow[d,"f"'] \arrow[rd, phantom, "\scriptstyle\push"] & B \arrow[d,dashed,"\inr"]\\
  A \arrow[r,dashed,"\inl"'] & A\sqcup^CB
\end{tikzcd}
\]
The idea is that we start with the disjoint sum $A+B$ and for every element $c$ of $C$, we add a new
path from $\inl(f(c))$ to $\inr(g(c))$.  What we obtain would rather be called a “homotopy pushout”
in classical homotopy theory, but in homotopy type theory this is the only sort of pushout which is
possible to define therefore we call them simply “pushouts”.

The induction principle states that, given a dependent type $P:A\sqcup^CB\to\Type$, we can define a
function $h:(x:A\sqcup^CB)\to P(x)$ by
\begin{align*}
  h &: (x:A\sqcup^CB)\to P(x),\\
  h(\inl(a)) &\defeq h_{\inl}(a),\\
  h(\inr(b)) &\defeq h_{\inr}(b),\\
  \ap h(\push(c)) &\defeq h_{\push}(c),
\end{align*}
where we have
\begin{align*}
  h_{\inl} &: (a : A) \to P(\inl(a)),\\
  h_{\inr} &: (b : B) \to P(\inr(b)),\\
  h_{\push} &: (c : C) \to h_{\inl}(f(c)) =^P_{\push(c)} h_{\inr}(g(c)).
\end{align*}
This induction principle looks similar to the usual universal property of pushouts, but there are
two differences. On the one hand this is only an existence result, we say nothing about uniqueness a
priori. On the other hand we state it for every dependent type whereas the universal property talks
only about non-dependent types. It turns out that one can prove the universal property from the
induction principle using the fact that uniqueness-up-to-homotopy can be seen as existence of a
particular equality which is constructed using the induction principle (see \cite[section
6.8]{hottbook} and \cite{kristina}). However the universal property only implies a weak version of
the induction principle, where the equalities of the form $h(\inl(a))=h_{\inl}(a)$ are only
propositional equalities. Moreover, the statement of the induction principle is very natural from
the point of view of type theory, while the statement of the universal property looks more
contrived.

Many interesting constructions are defined as pushouts, in particular we have the following
definitions.
\begin{itemize}
\item The \emph{suspension} $\Susp A$ of a type $A$ is the pushout of the span
\[
\begin{tikzcd}
  \Unit & A \arrow[l] \arrow[r] & \Unit.
\end{tikzcd}
\]
We write $\north$, $\south$ and $\merid(a)$ for the terms $\inl(\ttt)$, $\inr(\ttt)$ and
$\push(a)$. If $A$ is a pointed type (i.e.\ a type $A$ equipped with a basepoint $\star_A$), we
write $\Omega A$ for the type $\star_A=\star_A$ and we define the map
\begin{align*}
  \varphi_A &: A \to \Omega\Susp A,\\
  \varphi_A(a) &\defeq \merid(a)\concat\merid(\star_A)\inv,
\end{align*}
(where $\Sigma A$ is pointed by $\north$) which satisfies
\[\varphi_A(\star_A)=\idp{\star_A}.\]
\item The spheres $\Sn n$ for $n:\N$ are defined by induction on $n$. For $n=0$ we define
  $\Sn0\defeq\Bool$ and for $n+1$ we define
  \[\Sn{n+1}\defeq\Susp\Sn n.\]
  Note that $\Susp\Bool$ is equivalent to the type $\Sn1$ defined above, so we could alternatively
  define the spheres by iterated suspensions of $\Sn1$. In practice, we often use the direct
  inductive definition of $\Sn1$ using $\base$ and $\lloop$ instead of $\Susp\Bool$.
\item The \emph{join} $A*B$ of two types $A$ and $B$ is the pushout of the span
\[
\begin{tikzcd}
  A & A\times B \arrow[l,"\fst"'] \arrow[r,"\snd"] & B.
\end{tikzcd}
\]
\item The \emph{wedge sum} $A\vee B$ of two pointed types $A$ and $B$ is the pushout of the span
\[
\begin{tikzcd}
  A & \Unit \arrow[l] \arrow[r] & B,
\end{tikzcd}
\]
where the two maps pick the basepoints of $A$ and $B$
\item The \emph{smash product} $A\wedge B$ of two pointed types $A$ and $B$ is the pushout of the
  span
\[
\begin{tikzcd}
  \Unit & A\vee B \arrow[l] \arrow[r,"\iwedge_{A,B}"] & A\times B,
\end{tikzcd}
\]
where the map on the right is defined by
\begin{align*}
  \iwedge_{A,B} &: A\vee B \to A\times B,\\
  \iwedge_{A,B}(\inl(a)) &\defeq (a,\star_B),\\
  \iwedge_{A,B}(\inr(b)) &\defeq (\star_A,b),\\
  \ap{\iwedge_{A,B}}(\push(\ttt)) &\defeq \idp{(\star_A,\star_B)}.
\end{align*}
The smash product is studied in more detail in chapter \ref{ch:smash}.
\end{itemize}

\section{The \texorpdfstring{$3\times3$}{3×3}-lemma}\label{threebythree}

The $3\times3$-lemma is a technical lemma that we will use a few times to construct equivalences
between various nested pushouts. Its usefulness has been suggested to me by Eric Finster.

Let’s consider the diagram
\begin{equation}\label{diag:3x3}
\begin{tikzcd}[sep=large]
  A_{00} & A_{02} \arrow[l,"f_{01}"'] \arrow[r,"f_{03}"] \arrow[dr,Rightarrow,"H_{13}",shorten <= 1.5ex, shorten >= 1.5ex] \arrow[dl,Rightarrow,"H_{11}"',shorten <= 1.5ex, shorten >= 1.5ex] & A_{04}\\
  A_{20} \arrow[d,"f_{30}"'] \arrow[u,"f_{10}"] & A_{22} \arrow[l,near start,"f_{21}"] \arrow[r,near
  start,"f_{23}"]
  \arrow[d,near start,"f_{32}"] \arrow[u,near start,"f_{12}"] & A_{24} \arrow[d,"f_{34}"] \arrow[u,"f_{14}"']\\
  A_{40} & A_{42} \arrow[l,"f_{41}"] \arrow[r,"f_{43}"'] \arrow[ur,Rightarrow,"H_{33}"',shorten <= 1.5ex, shorten >= 1.5ex]
  \arrow[ul,Rightarrow,"H_{31}",shorten <= 1.5ex, shorten >= 1.5ex] & A_{44}
\end{tikzcd}
\end{equation}
where the $A_{ij}$ are types, the $f_{ij}$ are maps and the $H_{ij}$ are \emph{fillers} of the
squares, i.e.\ homotopies between compositions of maps (for instance $H_{11}$ has type
$(x:A_{22})\to f_{01}(f_{12}(x))=f_{10}(f_{21}(x))$).
The pushouts of the three columns fit in the diagram
\begin{equation}\label{diag:3x3fst}
\begin{tikzcd}[sep=large]
  A_{\pt0} & A_{\pt2} \arrow[l,"f_{\pt1}"'] \arrow[r,"f_{\pt3}"] & A_{\pt4},
\end{tikzcd}
\end{equation}
where the map $f_{\pt1}$ is defined by
\begin{align*}
  f_{\pt1} &: A_{\pt2} \to A_{\pt0},\\
  f_{\pt1}(\inl(x_{02})) &\defeq \inl(f_{01}(x_{02})),\\
  f_{\pt1}(\inr(x_{42})) &\defeq \inr(f_{41}(x_{42})),\\
  \ap{f_{\pt1}}(\push(x_{22})) &\defeq \ap{\inl}(H_{11}(x_{22})) \concat \push(f_{21}(x_{22}))
                                  \concat \ap{\inr}(H_{31}(x_{22}))\inv
\end{align*}
and the map $f_{\pt3}$ is defined in a similar way. We denote by $A_{\pt\sq}$ the pushout of diagram
\ref{diag:3x3fst}. We can also first consider the pushout of all three rows of diagram
\ref{diag:3x3} and then the pushout $A_{\sq\pt}$ of the resulting column.

\begin{lemma}[$3\times3$-lemma]
  There is an equivalence
  \[A_{\pt\sq}\simeq A_{\sq\pt}.\]
\end{lemma}

\begin{proof}[Sketch of proof]
  The map $f:A_{\pt\sq}\to A_{\sq\pt}$ basically takes an element of $A_{\pt\sq}$ and swaps the two
  constructors. For instance, it sends $\inl(\inr(x_{40}))$ to $\inr(\inl(x_{40}))$, and so on. It
  is defined by
  \begin{align*}
    f &: A_{\pt\sq}\to A_{\sq\pt},\\
    f(\inl(x_0)) &\defeq f_{\inl}(x_0),\\
    f(\inr(x_4)) &\defeq f_{\inr}(x_4),\\
    \ap f(\push(x_2)) &\defeq f_{\push}(x_2),
  \end{align*}
  where
  \begin{align*}
    f_{\inl} &: A_{\pt0} \to A_{\sq\pt}, & f_{\inr} &: A_{\pt4} \to A_{\sq\pt},\\
    f_{\inl}(\inl(x_{00})) &\defeq \inl(\inl(x_{00})), & f_{\inr}(\inl(x_{04})) &\defeq \inl(\inr(x_{04})),\\
    f_{\inl}(\inr(x_{40})) &\defeq \inr(\inl(x_{40})), & f_{\inr}(\inr(x_{44})) &\defeq \inr(\inr(x_{44})),\\
    \ap{f_{\inl}}(\push(x_{20})) &\defeq \push(\inl(x_{20})), & \ap{f_{\inr}}(\push(x_{24})) &\defeq \push(\inr(x_{24})),
  \end{align*}
  and
  \begin{align*}
    f_{\push} &: (x_2 : A_{\pt2}) \to f_{\inl}(f_{\pt1}(x_2)) = f_{\inr}(f_{\pt3}(x_2)),\\
    f_{\push}(\inl(x_{02})) &\defeq \ap{\inl}(\push(x_{02})),\\
    f_{\push}(\inr(x_{42})) &\defeq \ap{\inr}(\push(x_{42})),\\
    \ap{f_{\push}}(\push(x_{22})) &\defeq f_{\push,\push}.
  \end{align*}
  The last term $f_{\push,\push}$ is trickier. We want something filling the middle square of the
  diagram
  \[
  \begin{tikzcd}[sdiag]
    \bullet \arrow[r,"\push(\inl(f_{21}(x_{22})))"] & \bullet \\
    \bullet \arrow[r,dotted,decoration={name=none}] \arrow[u,"\ap{\inl\circ\inl}(H_{11}(x_{22}))"]
    \arrow[d,"\ap\inl(\push(f_{12}(x_{22})))"'] & \bullet
    \arrow[u,"\ap{\inr\circ\inl}(H_{31}(x_{22}))"']
    \arrow[d,"\ap\inr(\push(f_{32}(x_{22})))"]\\
    \bullet \arrow[r,dotted,decoration={name=none}] \arrow[d,"\ap{\inl\circ\inr}(H_{13}(x_{22}))"'] &
    \bullet \arrow[d,"\ap{\inr\circ\inr}(H_{33}(x_{22}))"]\\
    \bullet \arrow[r,"\push(\inr(f_{23}(x_{22})))"'] & \bullet
  \end{tikzcd}
  \]
  where the dotted paths are the composites of the upper and lower squares. What we have is
  $\ap{\push}(\push(x_{22}))$ which fills the middle square of the diagram
  \[
  \begin{tikzcd}[sdiag,column sep=huge]
    \bullet \arrow[r,"\ap{\inl\circ\inl}(H_{11}(x_{22}))",at start]
    \arrow[d,"\ap\inl(\push(f_{12}(x_{22})))"'] &
    \bullet \arrow[r,"\push(\inl(f_{21}(x_{22})))"] \arrow[d,dotted,decoration={name=none}] & \bullet \arrow[d,dotted,decoration={name=none}] &
    \bullet \arrow[l,"\ap{\inr\circ\inl}(H_{31}(x_{22}))"',at start]
    \arrow[d,"\ap\inr(\push(f_{32}(x_{22})))"] \\
    \bullet \arrow[r,"\ap{\inl\circ\inr}(H_{13}(x_{22}))"',at start] &
    \bullet \arrow[r,"\push(\inr(f_{23}(x_{22})))"'] & \bullet &
    \bullet \arrow[l,"\ap{\inr\circ\inr}(H_{33}(x_{22}))",at start]
  \end{tikzcd}
  \]
  where the dotted paths are the composites of the left and right square. We notice that the eight
  paths around are the same in both diagrams, therefore there is a coherence operation going from
  one middle square to the other. The inverse map $g$ is then defined in a similar way.

  In order to prove that they are inverse to each other, we use again the induction principle twice
  in order to construct two functions
  \begin{align*}
    (x:A_{\pt\sq})&\to g(f(x))=x,\\
    (y:A_{\sq\pt})&\to f(g(y))=y.
  \end{align*}
  For elements of the form $\inl(\inl(x_{00}))$, $\inl(\inr(x_{40}))$, $\inr(\inl(x_{04}))$ and
  $\inr(\inr(x_{44}))$ the equality is true by definition, so we use $\idpS$. For paths of the form
  $\push(\inl(x_{02}))$, $\push(\inr(x_{42}))$, $\ap\inl(\push(x_{20}))$ and
  $\ap\inr(\push(x_{24}))$ it is not true definitionally but not difficult to prove. Finally, for
  squares of the form $\ap\push(\push(x_{22}))$ it is more tricky because we have to prove an
  equality along the equalities proved in the previous step, but it can be done. A proof checked in
  Agda is available at \cite{threebythree}.
\end{proof}

Here are two propositions which are very useful when using the $3\times3$-lemma.
\begin{proposition}\label{pushoutid}
  Given a map $f:A\to B$, the pushout of the diagram
  \[
  \begin{tikzcd}
    A & A \arrow[l,"\ \id_A"'] \arrow[r,"f"] & B
  \end{tikzcd}
  \]
  is equivalent to $B$.
\end{proposition}
\begin{proof}
  We define $g:B\to A\sqcup^AB$ by $g(b)\defeq\inr(b)$ and $h:A\sqcup^AB\to B$ by
  \begin{align*}
    h &: A\sqcup^AB \to B,\\
    h(\inl(a)) &\defeq f(a),\\
    h(\inr(b)) &\defeq b,\\
    \ap h(\push(a)) &\defeq \idp{f(a)}.
  \end{align*}
  It is straightforward to show that these functions are inverse to each other.
\end{proof}

\begin{proposition}\label{pushoutprod}
  Given two maps $f:C\to A$ and $g:C\to B$ and a type $X$, the pushout of the diagram
  \[
  \begin{tikzcd}[column sep=huge]
    A\times X & C\times X \arrow[l,"{(c,x)\mapsto(f(c),x)}"'] \arrow[r,"{(c,x)\mapsto(g(c),x)}"] &
    B\times X
  \end{tikzcd}
  \]
  is equivalent to $(A\sqcup^CB)\times X$.
\end{proposition}
\begin{proof}
  We define the two maps
  \begin{align*}
    h_1 &: (A\sqcup^CB)\times X \to (A\times X)\sqcup^{C\times X}(B\times X),\\
    h_1(\inl(a),x) &\defeq \inl(a,x),\\
    h_1(\inr(b),x) &\defeq \inr(b,x),\\
    \ap{h_1(-,x)}(\push(c)) &\defeq \push(c,x),
  \end{align*}
  \begin{align*}
    h_2 &: (A\times X)\sqcup^{C\times X}(B\times X) \to (A\sqcup^CB)\times X,\\
    h_2(\inl(a,x)) &\defeq (\inl(a),x),\\
    h_2(\inr(b,x)) &\defeq (\inr(b),x),\\
    \ap{h_2}(\push(c,x)) &\defeq \ap{(-,x)}(\push(c)).
  \end{align*}
  It is again easy to prove that those two functions are inverse to each other.
\end{proof}

We now present a proposition that we can prove using the $3\times3$-lemma and which will be used in
the next chapter. It has also been formally proved in Agda by Evan Cavallo in \cite{cavallo3x3},
directly, without using the $3\times3$-lemma.
\begin{proposition}\label{joinassoc}
  Given three types $A$, $B$ and $C$, there is an equivalence
  \[(A*B)*C \simeq A*(B*C).\]
\end{proposition}
\begin{proof}
  Let’s consider the diagram
  \[
  \begin{tikzcd}
    A & A\times B \arrow[l] \arrow[r] & B \\
    A\times C \arrow[u] \arrow[d] & A\times B\times C \arrow[l] \arrow[r] \arrow[u] \arrow[d] &
    B\times C \arrow[u] \arrow[d]\\
    A\times C & A\times C \arrow[l] \arrow[r] & C
  \end{tikzcd}
  \]
  where the arrows are the obvious projections and the four squares are filled by $\idpS$. The
  pushout of the first row is $A*B$, the pushout of the second row is $(A*B)\times C$ according to
  proposition \ref{pushoutprod} and the pushout of the third row is $C$ according to proposition
  \ref{pushoutid}. Moreover, the two maps $(A*B)\times C\to A*B$ and $(A*B)\times C\to C$ are equal
  to the two projections, as can be checked by hand by induction. Therefore the pushout of the
  pushouts of the rows is equivalent to $(A*B)*C$.

  Similarly, the pushout of the first column is $A$, the pushout of the second column is
  $A\times(B*C)$, the pushout of the third column is $B*C$, and the two maps are the two
  projections, hence the pushout of the pushouts of the columns is equivalent to
  $A*(B*C)$. Therefore, we have
  \[(A*B)*C \simeq A*(B*C).\qedhere\]
\end{proof}

Here is another simple application, which could also be proved directly.
\begin{proposition}\label{booljoinsusp}
  For any type $A$ there is an equivalence $\Bool*A\simeq\Susp A$.
\end{proposition}

\begin{proof}
  We apply the $3\times3$-lemma to the diagram
  \[
  \begin{tikzcd}
    \Unit & A \arrow[l] \arrow[r] & A\\
    \Zero \arrow[u] \arrow[d] & \Zero \arrow[l] \arrow[r] \arrow[u] \arrow[d] & A \arrow[u]
    \arrow[d]\\
    \Unit & A \arrow[l] \arrow[r] & A
  \end{tikzcd}
  \]
  and we conclude by noticing that $\Unit+\Unit\simeq\Bool$ and that $A+A\simeq \Bool\times A$.
\end{proof}

Combining the two previous results, we obtain
\begin{proposition}\label{joinspheres}
  For every $m,n:\N$, there is an equivalence
  \[\Sn m*\Sn n\simeq\Sn{m+n+1}.\]
\end{proposition}

\begin{proof}
  We proceed by induction on $m$. For $m=0$ it is proposition \ref{booljoinsusp}, and for $m+1$ we have
  \begin{align*}
    \Sn{m+1}*\Sn n &\simeq (\Susp\Sn{m})*\Sn n\\
                   &\simeq (\Bool*\Sn{m})*(\Sn n)\quad\text{by proposition }\ref{booljoinsusp}\\
                   &\simeq \Bool*(\Sn{m}*\Sn n)\quad\text{by proposition }\ref{joinassoc}\\
                   &\simeq \Bool*\Sn{m+n+1}\quad\text{by induction hypothesis}\\
                   &\simeq \Susp\Sn{m+n+1}\quad\text{by proposition }\ref{booljoinsusp}\\
                   &\simeq \Sn{(m+1)+n+1}.\qedhere
  \end{align*}
\end{proof}

Another simple application is the following, which we will use in chapter \ref{ch:james}.
\begin{proposition}\label{pushoutangle}
  Given four types $A$, $B$, $C$ and $D$ and three functions $f:B\to A$, $g:C\to B$ and $h:C\to D$,
  there is an equivalence
  \[A\sqcup^B(B\sqcup^C D)\simeq A\sqcup^CD\]
  where the map $B\to B\sqcup^C D$ on the left is $\inl$ and the map $C\to A$ on the right is
  $f\circ g$.
\end{proposition}

\begin{proof}
  We apply the $3\times3$-lemma to the diagram
  \[
  \begin{gathered}[b]
    \begin{tikzcd}
      A & \Zero \arrow[l] \arrow[r] & \Zero\\
      B \arrow[u,"f"] \arrow[d,"\id_B"'] & \Zero \arrow[l] \arrow[r] \arrow[u] \arrow[d] & \Zero
      \arrow[u] \arrow[d]\\
      B & C \arrow[l,"g"] \arrow[r,"h"'] & D
    \end{tikzcd}\\[-8pt]
  \end{gathered}\qedhere
  \]
\end{proof}

All the examples of use of the $3\times3$-lemma presented so far are simpler to prove than the full
version either because it is degenerate or because $H_{11}(x)$, $H_{13}(x)$, $H_{31}(x)$ and
$H_{33}(x)$ are all equal to $\idpS$. However in proposition \ref{whiteheadab} we will use one
instance of the $3\times3$-lemma where $H_{11}(x)$ is not a constant path.

\section{The flattening lemma}\label{flattening}

In classical homotopy theory, the usual way to construct a fibration is to define a continuous map
$f:E\to B$ between two spaces $E$ and $B$ and prove that it has the property of being a
fibration. In homotopy type theory the story is quite different because there is no predicate “being
a fibration”. Indeed, being a fibration is not a property invariant under homotopy of maps. Instead,
fibrations are replaced by dependent types, which encode directly the fact that every point of the
base space has a fiber over it. One way to understand why we can see dependent types as fibrations
is with the function $\transport$ of section \ref{sec:ua} which shows that one can transport
elements between fibers of a dependent type along a path in the base type, and we could similarly
lift any homotopy in the base space. Constructing a fibration in homotopy type theory consists
simply in defining a map $P:B\to\Type$. There is no need to “prove” that it is a fibration, but what
is non-trivial is to understand its total space. In general, the total space of a fibration
$P:B\to\Type$ is the type
\[\sum_{x:B}P(x),\]
but this is not a description which is easy to work with in general.

A special case of fibrations in homotopy type theory are fibrations over higher inductive types
which are defined by induction using the univalence axiom. For instance, one can define a fibration
$P$ over $\Sn1$ by
\begin{align*}
  P &: \Sn1\to\Type,\\
  P(\base) &\defeq P_{\base},\\
  \ap P(\lloop) &\defeq \ua(P_{\lloop}),
\end{align*}
where $P_{\base}$ is a type and $P_{\lloop}$ is an equivalence between $P_{\base}$ and
itself. Similarly, one can define a fibration over a pushout $A\sqcup^CB$ by
\begin{align*}
  P &: A\sqcup^C B\to\Type,\\
  P(\inl(a)) &\defeq P_{\inl}(a),\\
  P(\inr(b)) &\defeq P_{\inr}(b),\\
  \ap P(\push(c)) &\defeq \ua(P_{\push}(c)),
\end{align*}
where
\begin{align*}
  P_{\inl} &: A\to\Type,\\
  P_{\inr} &: B\to\Type,\\
  P_{\push} &: (c:C)\to P_{\inl}(f(c)) \simeq P_{\inr}(g(c)).
\end{align*}
The \emph{flattening lemma} gives us a nice description of the total space of such fibrations.  We
refer to \cite[section 6.12]{hottbook} for the general statement and proof, and we state here only
these two special cases.

\begin{proposition}\label{flatteningS1}
  The total space of a fibration $P:\Sn1\to\Type$ defined by $P_{\base}$ and $P_{\lloop}$ as above
  is equivalent to the higher inductive type $T$ generated by
  \begin{align*}
    \widetilde{\base} &: P_{\base} \to T,\\
    \widetilde{\lloop} &: (x : P_{\base}) \to \widetilde{\base}(x) =_T \widetilde{\base}(P_{\lloop}(x)).
  \end{align*}
\end{proposition}

\begin{proposition}\label{flatteningpushout}
  Given the pushout $D$ of the diagram
  \[
  \begin{tikzcd}
    A & C \arrow[l,"f"'] \arrow[r,"g"] & B
  \end{tikzcd}
  \]
  and $P:D\to\Type$ defined by $P_{\inl}$, $P_{\inr}$ and $P_{\push}$ as above, the total space of
  $P$ is equivalent to the pushout of the diagram
  \[
  \begin{tikzcd}\displaystyle
    \sum_{x:A}P_{\inl}(x) & \displaystyle\sum_{z:C}P_{\inl}(f(z)) \arrow[l] \arrow[r] &\displaystyle
    \sum_{y:B}P_{\inr}(y),
  \end{tikzcd}
  \]
  where the left map sends the pair $(z,u)$ to $(f(z),u)$ and the right map sends the pair $(z,v)$
  to $(g(z),P_{\push}(z)(v))$.
\end{proposition}

\section{Truncatedness and truncations}\label{sec:truncatednesstruncations}

Finally we introduce $n$-truncated types and truncations, which are an essential part of this work
and of homotopy type theory in general. The notion of $n$-truncated type was introduced in 2009 by
Vladimir Voevodsky under the name “type of h-level $n+2$” and proved invaluable in the understanding
of homotopy type theory. Intuitively, a type is $n$-truncated if it doesn’t contain any interesting
information in its $k$-iterated identity types for $k>n$. From the point of view of homotopy theory
they correspond to homotopy $n$-types. For instance, $0$-truncated types corresponds to \emph{sets}
(i.e.\ discrete spaces). We can also make sense of $(-1)$-truncated types which are types whose
elements are all equal, and of $(-2)$-truncated types which are the contractible types.

\subsection{Truncatedness of types}

\begin{definition}
  A type $A$ is \emph{contractible} or \emph{$(-2)$-truncated} if there is a point $a:A$ (sometimes
  called the \emph{center of contraction}) and an equality from $a$ to $x$ for every $x:A$.  In
  other words, being contractible is defined by the predicate
  \[\operatorname{is-contr}(A) \defeq \sum_{a:A}((x:A)\to a=_Ax).\]
\end{definition}

Note that if we read the definition of contractibility as “there exists $a:A$ such that for every
$x:A$ there exists a path from $a$ to $x$”, we might think that it is only a definition of
connectedness. The error is that the existential quantifier is a \emph{explicit} existential
quantifier, which means that the path from $a$ to $x$ should not only exist but it should depend
continuously on $x$, which is much stronger than connectedness.

An important example of contractible type is the space of all paths in a type with one endpoint
fixed.
\begin{proposition}\label{pathfibcontr}
  Given a type $A$ and a point $a:A$, the type
  \[\sum_{x:A}(a=_Ax)\]
  is contractible.
\end{proposition}
\begin{proof}
  We first notice that the pair $(a,\idp a)$ is an element of this type. We take it as the center of
  contraction and we now have to prove that for every $x:A$ and $p:a=_Ax$, there is a path
  $(a,\idp a)=(x,p)$ in $\sum_{x:A}(a=_Ax)$. Using propositions \ref{pathspairtypes} and
  \ref{depeqid}, we see that we need to find a path $q:a=_Ax$ together with a filler of the square
  \[
  \begin{tikzcd}[sdiag]
    \bullet \arrow[r,"\idp a"] \arrow[d,"\idp a"'] & \bullet \arrow[d,"p"] \\
    \bullet \arrow[r,"q"'] & \bullet
  \end{tikzcd}
  \]
  We take $q\defeq p$ and the filler coming from the equality
  $\idp{\idp a\concat p}:\idp a\concat p=\idp a\concat q$.
\end{proof}

\begin{definition}
  For $n\ge -2$, a type $A$ is \emph{$(n+1)$-truncated} if for every $x,y:A$, the type $x=_Ay$ is $n$-truncated.
\end{definition}

As a special case, a type which is $(-1)$-truncated is called a \emph{mere proposition} and a type
which is $0$-truncated is called a \emph{set}.

\begin{proposition}\label{ntruncprop}
  We have the following properties.
  \begin{itemize}
  \item Being $n$-truncated is a mere proposition.
  \item A type $A$ satisfies $x=_Ay$ for every $x,y:A$ if and only if $A$ is a mere proposition.
  \item If $B:A\to\Type$ is a family of $n$-truncated types, then $(x:A)\to B(x)$ is itself
    $n$-truncated.
  \item If $A$ is $n$-truncated and $B:A\to\Type$ is a family of $n$-truncated types, then
    $\sum_{x:A}B(x)$ is itself $n$-truncated.
  \item If $A$ is $n$-truncated and $m\ge n$, then $A$ is $m$-truncated as well.
  \item If $A$ is $n$-truncated and $x,y:A$, then $x=_Ay$ is $n$-truncated.
  \end{itemize}
\end{proposition}

We refer to \cite[section 7.1]{hottbook} for a proof of these properties.

\begin{proposition}
  The unit type is contractible and the empty type is a proposition.
\end{proposition}

\begin{proof}
  For the unit type, the center of contraction is $\ttt$ and for every $x:\Unit$ we have $x=\ttt$
  because we only need to check it for $x\defeq\ttt$ and in that case we use $\idp{\ttt}$. Hence,
  the type $\Unit$ is contractible.

  For the empty type, we have to prove that given $x,y:\bot$ we have $x=_\bot y$, but the induction
  principle of $\bot$ says that given $x:\bot$ everything is true. Hence, the type $\bot$ is a
  proposition.
\end{proof}

A large class of sets ($0$-truncated types) is given by types with decidable equality.
\begin{proposition}[Hedberg’s theorem]
  If a type $A$ is such that for every $x,y:A$, either $x=_Ay$ or $\neg(x=_Ay)$ (explicit
  disjunction), then $A$ is a set.
\end{proposition}
We say that a type $A$ \emph{has decidable equality} if it satisfies the hypothesis of the proposition.
\begin{proof}
  Given $x,y:A$ and $p,q:x=_Ay$, we want to prove that $p=_{x=_Ay}q$. Let’s call $d$ the witness of
  decidable equality on $A$. We have
  \[d : (x,y:A)\to((x=_Ay)+\neg(x=_Ay)).\]
  Decidable equality at $(x,x)$ gives us either a loop at $x$ or an element of $\neg(x=_Ax)$, but
  this last case cannot hold because there exists at least one loop at $x$, namely $\idp x$.
  Therefore, for every $x:A$ we have a loop $p_x:x=_Ax$, with $d(x,x)=\inl(p_x)$. Let’s now consider
  $\ap{d(x,-)}(p\concat q\inv)$ which is a dependent path from $\inl(p_x)$ to itself over
  $p\concat q\inv$. Using proposition \ref{depeqid}, we obtain a filler of the square
  \[
  \begin{tikzcd}[sdiag]
    \bullet \arrow[r,"\idp x"] \arrow[d,"p_x"'] & \bullet \arrow[d,"p_x"]\\
    \bullet \arrow[r,"p\concat q\inv"'] & \bullet
  \end{tikzcd}
  \]
  Therefore, we have
  \[p_x=p_x\concat(p\concat q\inv),\] which shows that $p=q$.
\end{proof}

\begin{cor}\label{Zisset}
  The type $\Bool$, the type $\N$ of natural numbers and the type $\Z$ of integers are sets.
\end{cor}

\begin{proof}
  In all cases, we can easily define by induction on $x$ and $y$ a function of type
  \[(x,y:A)\to((x=_Ay)+\neg(x=_Ay)),\] and then we apply Hedberg’s theorem.
\end{proof}

Note that the circle $\Sn1$ does not have decidable equality because we cannot give a path from $x$
to $y$, for every $x,y:\Sn1$, which depends continuously on $x$ and $y$.

\subsection{Truncations}

Not every type is an $n$-type for some $n$, but we can turn any type $A$ into an $n$-type
$\trunc nA$ in a universal way. This operation is called the \emph{truncation}. There is a map
$|-|:A\to\trunc nA$ and the type $\trunc nA$ satisfies the following induction principle: given a
dependent type $P:\trunc nA\to\Type$ such that $P(x)$ is $n$-truncated for every $x:\trunc nA$, we
can define a section of $P$ by defining it only on elements of the form $|a|$:
\begin{align*}
  g&:(x:\trunc nA)\to P(x),\\
  g(|a|) &\defeq g_{\operatorname{|-|}}(a),
\end{align*}
where
\[g_{\operatorname{|-|}} : (a : A) \to P(|a|).\]

Note that we can only apply this induction principle if all fibers of $P$ are $n$-truncated. In
particular, it is not possible to use it directly to define a function from $\trunc nA$ to a type
$B$ which is not $n$-truncated. When we need to go around this limitation, the usual technique is to
construct an $n$-truncated type $\widetilde B$ together with a map $\widetilde{B}\to B$ and to use
the induction principle for $\widetilde{B}$ instead of $B$.

We can implement truncations using higher inductive types as described in \cite[section
7.3]{hottbook} and to some extent in appendix \ref{ch:defn} of this thesis, but in practice we
always use the induction principle given above.

Note that truncating a type which is already truncated gives back an equivalent type.
\begin{proposition}
  If $A$ is $n$-truncated, then the map $|-|:A\to\trunc nA$ is an equivalence.
\end{proposition}

\begin{proof}
  We define $f:\trunc nA\to A$ using the induction principle of $\trunc nA$ by
  \begin{align*}
    f &: \trunc nA \to A,\\
    f(|a|) &\defeq a.
  \end{align*}
  Note that this is allowed because $A$ is $n$-truncated. We then have $f\circ|-|=\id_A$ by
  definition and we prove that $|-|\circ f=\id_{\trunc nA}$ by
  \begin{align*}
    h &: (x : \trunc nA) \to |f(x)|=_{\trunc nA}x,\\
    h(|a|) &\defeq \idp{|a|}.
  \end{align*}
  This is allowed because for every $x:\trunc nA$ the type $|f(x)|=_{\trunc nA}x$ is $n$-truncated,
  and we use the fact that $|f(|a|)|=|a|$ by definition.
  Therefore the map $|-|:A\to\trunc nA$ is an equivalence of inverse $f$.
\end{proof}

We can truncate maps as well.

\begin{proposition}\label{truncmap}
  Given an integer $n$ and a map $f:A\to B$, there is a map $\trunc nf:\trunc nA\to\trunc nB$ and a
  filler of the diagram
  \[
  \begin{tikzcd}
    A \arrow[r,"|-|"] \arrow[d,"f"'] & \trunc nA \arrow[d,"\trunc nf"]\\
    B \arrow[r,"|-|"'] & \trunc nB
  \end{tikzcd}
  \]
\end{proposition}

\begin{proof}
  The map $\trunc nf$ is defined by
  \begin{align*}
    \trunc nf &: \trunc nA \to \trunc nB,\\
    \trunc nf(|a|) &\defeq |f(a)|,
  \end{align*}
  which is allowed because $\trunc nB$ is $n$-truncated.
\end{proof}

An important property of truncations is the fact that loop spaces “go under” truncations (see
\cite[theorem 7.3.12]{hottbook} for a proof)
\begin{proposition}\label{looptrunc}
  Given a type $A$, two elements $x,y:A$, and $n\ge-2$, the map
  \begin{align*}
    f &: \trunc n{x=_Ay}\to|x|=_{\trunc {n+1}A}|y|,\\
    f(|p|) &\defeq \ap{|-|}(p)
  \end{align*}
  is an equivalence.
\end{proposition}
Note that the right-hand side is an identity type in $\trunc{n+1}A$, hence $|x|=_{\trunc {n+1}A}|y|$
is $n$-truncated so the induction principle applies.

\subsection{Mere propositions and logic}
\label{sec:merepropsandlogic}

We saw in section \ref{sec:pairtypes} that given a type $A$ and a predicate $B:A\to\Type$, the type
\[\sum_{x:A}B(x)\]
can be seen as the type of elements of $A$ satisfying $B$. However, if $B$ is arbitrary, there could
be some $a:A$ and $b_0,b_1:B(a)$ distinct, and $(a,b_0)$ and $(a,b_1)$ would be two different
elements of $\sum_{x:A}B(x)$ which means that $a$ would be counted twice. In order for
$\sum_{x:A}B(x)$ to accurately represent the type of elements of $A$ satisfying $B$, we need to
assume additionally that $B(x)$ is a mere proposition for every $x:A$. This ensures that each
element of $A$ is counted at most once.

The type $\sum_{x:A}B(x)$ can also be seen as the proposition “there exists $x:A$ such that $B(x)$
holds”. Now the problem is that even if $B$ is a family of mere propositions, then $\sum_{x:A}B(x)$
may not be a mere proposition anymore because there could be several elements of $A$ satisfying
$B$. For instance a natural number $k$ is \emph{composite} if there exists $n,m\ge2$ such that
$nm=_{\N}k$, but defining the \emph{type} of composite numbers by
\[\sum_{k:\N}\sum_{n,m:\N}(n\ge2\text{ and }m\ge2\text{ and }nm=_{\N}k)\]
is incorrect because, for instance, $6$ is counted twice as it is equal to both $2\times3$ and
$3\times2$ (note that the $\sum_{k:\N}$ is interpreted in the sense of the previous paragraph while
the $\sum_{n,m:\N}$ is interpreted in the sense of this paragraph). One way to solve this issue is
to use the $(-1)$-truncation and to consider the type
\[\sum_{k:\N}\trunc{-1}{\sum_{n,m:\N}(n\ge2\text{ and }m\ge2\text{ and }nm=_{\N}k)},\]
which forces the inner $\Sigma$-type to be a mere-proposition. Now the number $6$ is counted only
once because the two proofs that it is composite have been identified in the truncation. Therefore,
a statement of the form “there exists $x:A$ such that $B(x)$ holds” can be interpreted in two
different ways: either as the type $\sum_{x:A}B(x)$ if we care about the witness $x$ (\emph{explicit
  existence}) or as the type $\trunc{-1}{\sum_{x:A}B(x)}$ if we want a mere proposition (\emph{mere
  existence}).

The same phenomenon happens for disjunctions. Given two types $A$ and $B$, the disjunction “$A$ or
$B$” can be interpreted as either $A+B$ if we care about which of $A$ or $B$ holds (\emph{explicit
  disjunction}) or as $\trunc{-1}{A+B}$ if we want a mere proposition (\emph{mere disjunction}). For
the universal quantifier and conjunction, however, there is no need to truncate given that a product
of mere propositions is a mere proposition.

A related issue arises with equalities. When defining a type of algebraic structures, for instance
groups, we usually want to assume that some equalities hold between some elements. But in homotopy
type theory, that means that we have to choose \emph{specific} paths, and there might be several
such paths. For instance a type together with a binary operation could be seen as group in several
different ways if there are several paths witnessing associativity. Instead of using truncations as
above, we instead require the carrier type to be a set so that there cannot be several different
equalities between its elements. Therefore we take the following definition of group.
\begin{definition}
  A \emph{group} is a set $G$ (i.e.\ a $0$-truncated type) together with a multiplication operation
  $m:G\to G\to G$, an inverse operation $i:G\to G$, and a neutral element $e:G$,
  such that the following equalities hold:
  \begin{align*}
    m(m(x,y),z) &=_G m(x,m(y,z))\text{ for all }x,y,z:G,\\
    m(x,i(x)) &=_G e\text{ for all }x:G,\\
    m(i(x),x) &=_G e\text{ for all }x:G,\\
    m(x,e) &=_G x\text{ for all }x:G,\\
    m(e,x) &=_G x\text{ for all }x:G.
  \end{align*}
\end{definition}

It is easy to see that given a set $G$ and a function $m:G\to G\to G$, the type corresponding to
“$(G,m)$ is a group” is a mere proposition. In other words, there is at most one proof that $(G,m)$
is a group. This wouldn’t be the case if we hadn’t assumed $G$ to be a set.


\chapter{First results on homotopy groups of spheres}\label{ch:hopf}

For every pointed type $A$ and every integer $k\ge1$, we define a group $\pi_k(A)$ called the $k$-th
homotopy group of $A$. In some sense, it describes the structure of $k$-dimensional loops in
$A$. The definition of $\pi_k(A)$ alone isn’t very informative, so we usually want to \emph{compute}
it, i.e.\ to prove that it is equivalent to some well-known group, for instance by giving an
explicit presentation. Computing homotopy groups is notoriously difficult even for very simple types
like the spheres $\Sn n$. In classical homotopy theory, one can show that the first few homotopy
groups of spheres are those given in table \ref{hgs} on page \pageref{hgs}, and in this thesis we
show that we can compute some of them in homotopy type theory.

\newcommand{\Zx}{\Z_{12}\times\Z_2}
\newcommand{\Zy}{\Z\times\Z_{12}}
\newcommand{\Zz}{\Z_{24}\times\Z_3}
\newcommand{\Zt}{\Z_{120}\times\Z_{12}\times\Z_2}
\newcommand{\Zu}{\Z_{84}\times\Z_2^2}
\newcommand{\Zv}{\Z_{24}\times\Z_2}
\begin{figure}
  \centering
  \begin{tikzpicture}
    \matrix(m)[ampersand replacement=\&,matrix of math nodes,nodes in empty cells, minimum
    width=3em, text height=1em] {
      \& \Sn1 \& \Sn2 \& \Sn3 \& \Sn4 \& \Sn5 \& \Sn6 \& \Sn7 \& \Sn8 \& \Sn9 \& \Sn{10}\\
      \pi_1    \& \Z \& 0       \& 0       \& 0       \& 0       \& 0       \& 0       \& 0       \& 0       \& 0\\
      \pi_2    \& 0  \& \Z      \& 0       \& 0       \& 0       \& 0       \& 0       \& 0       \& 0       \& 0\\
      \pi_3    \& 0  \& \Z      \& \Z      \& 0       \& 0       \& 0       \& 0       \& 0       \& 0       \& 0\\
      \pi_4    \& 0  \& \Z_2    \& \Z_2    \& \Z      \& 0       \& 0       \& 0       \& 0       \& 0       \& 0\\
      \pi_5    \& 0  \& \Z_2    \& \Z_2    \& \Z_2    \& \Z      \& 0       \& 0       \& 0       \& 0       \& 0\\
      \pi_6    \& 0  \& \Z_{12} \& \Z_{12} \& \Z_2    \& \Z_2    \& \Z      \& 0       \& 0       \& 0       \& 0\\
      \pi_7    \& 0  \& \Z_2    \& \Z_2    \& \Zy     \& \Z_2    \& \Z_2    \& \Z      \& 0       \& 0       \& 0\\
      \pi_8    \& 0  \& \Z_2    \& \Z_2    \& \Z_2^2  \& \Z_{24} \& \Z_2    \& \Z_2    \& \Z      \& 0       \& 0\\
      \pi_9    \& 0  \& \Z_3    \& \Z_3    \& \Z_2^2  \& \Z_2    \& \Z_{24} \& \Z_2    \& \Z_2    \& \Z      \& 0\\
      \pi_{10} \& 0  \& \Z_{15} \& \Z_{15} \& \Zz     \& \Z_2    \& 0       \& \Z_{24} \& \Z_2    \& \Z_2    \& \Z\\
      \pi_{11} \& 0  \& \Z_2    \& \Z_2    \& \Z_{15} \& \Z_2    \& \Z      \& 0       \& \Z_{24} \& \Z_2    \& \Z_2\\
      \pi_{12} \& 0  \& \Z_2^2  \& \Z_2^2  \& \Z_2    \& \Z_{30} \& \Z_2    \& 0       \& 0       \& \Z_{24} \& \Z_2\\
      \pi_{13} \& 0  \& \Zx     \& \Zx     \& \Z_2^3  \& \Z_2    \& \Z_{60} \& \Z_2    \& 0       \& 0       \& \Z_{24}\\
    };
  
    \draw[thick](m-1-1.south west)--(m-1-11.south east); \draw[thick](m-1-1.north
    east)--(m-14-1.south east);

    \draw[line width = 2pt] (m-2-2.north east) --(m-4-2.north east) --(m-4-3.north east)
    --(m-6-3.north east) --(m-6-4.north east) --(m-8-4.north east) --(m-8-6.north west)
    --(m-10-6.north west) --(m-10-6.north east) --(m-12-6.north east) --(m-12-7.north east)
    --(m-14-7.north east) --(m-14-8.north east) --(m-14-8.south east);
  \end{tikzpicture}
  \caption{Homotopy groups of spheres (where $\Z_n$ denotes the group $\Z/n\Z$)}
  \label{hgs}
\end{figure}

In section \ref{hgss1} we compute those in the first column (the homotopy groups of $\Sn1$), in
section \ref{hgs0} we compute all the $0$s in the upper-right part ($\pi_k(\Sn n)$ for $k<n$), and
in section \ref{hgshopf} we compute $\pi_2(\Sn2)$ and show that the second and third column are
identical apart for $\pi_2$. In the next chapter we prove that the diagonals are constant above the
zigzag line (Freudenthal suspension theorem) and that there exists an $n:\N$ such that
$\pi_4(\Sn3)\simeq\Z/n\Z$, and in the rest of this thesis we prove that this $n$ is equal to $2$.
 
\section{Homotopy groups}

\begin{definition}
  A \emph{pointed type} is a type $A$ equipped with a point $\star_A$.
  Given two pointed types $A$ and $B$, a \emph{pointed function} between $A$ and $B$ is a function
  $f:A\to B$ equipped with a path $\star_f:f(\star_A)=\star_B$.  We write $A\topt B$ for the type of
  pointed functions between $A$ and $B$. The type $\Sn1$ is pointed by $\base$, and the type
  $\Susp A$ is pointed by $\north$ for every type $A$.
\end{definition}

\begin{definition}
  Given a pointed type $A$, its \emph{loop space} is the type
  \[\Omega A\defeq(\star_A=\star_A)\] of paths
  from $\star_A$ to $\star_A$, pointed by the constant path. We can iterate this construction by
  defining the $n$-fold iterated loop space of $A$, for $n:\N$, by
  \begin{align*}
    \Omega^0A &\defeq A,\\
    \Omega^{n+1}A &\defeq \Omega(\Omega^n A).
  \end{align*}
\end{definition}

\begin{definition}
  Given a pointed function $f:A\topt B$ between pointed types, its \emph{looping} is
  \begin{align*}
    \Omega f &: \Omega A \topt \Omega B,\\
    (\Omega f)(p) &\defeq \star_f\inv\concat\ap f(p)\concat\star_f.
  \end{align*}
  Note that $\ap f(p)$ has type $f(\star_A)=f(\star_A)$, therefore we need to compose it with
  $\star_f$ on both sides in order to obtain a loop at $\star_B$.  The function $\Omega f$ is
  pointed by the proof of \[\star_f\inv\concat\ap f(\idp{\star_A})\concat\star_f=\idp{\star_B}\]
  obtained by canceling $\ap f(\idp{\star_A})$ (which is equal to $\idp{f(\star_A)}$) and then
  canceling both $\star_f$ together.
\end{definition}

We can now define the homotopy groups of a pointed type.  The idea is that $\pi_n(A)$ is the set of
connected components of the type of $n$-dimensional loops in $A$. The first homotopy group
$\pi_1(A)$ is called the \emph{fundamental group} of $A$ and the $\pi_n(A)$ are called the
\emph{higher homotopy groups} of $A$.
\begin{definition}
  Given a pointed type $A$ and $n\ge1$, the \emph{$n$-th homotopy group of $A$} is the set
  \[\pi_n(A) \defeq \trunc0{\Omega^nA}.\]
  It is equipped with the group structure induced by composition of paths, inversion of paths and
  the constant path.
\end{definition}

The homotopy groups of the loop space of a type $A$ are the homotopy groups of $A$ shifted by one.
\begin{proposition}\label{piomega}
  Given a pointed type $A$ and $n\ge1$, we have an isomorphism of groups
  \[\pi_n(\Omega A)\simeq\pi_{n+1}(A).\]
\end{proposition}

\begin{proof}
  We first construct an equivalence $\Omega^n(\Omega A)\simeq\Omega(\Omega^nA)$ by induction on
  $n$. For $n=0$ we take the identity equivalence and for $n+1$ we take the composition
  \begin{align*}
    \Omega^{n+1}(\Omega A) &\simeq \Omega(\Omega^n(\Omega A))\quad\text{ by definition}\\
                           &\simeq \Omega(\Omega(\Omega^n A))\quad\text{ by induction hypothesis}\\
                           &\simeq \Omega(\Omega^{n+1}A).
  \end{align*}
  We then have
  \begin{align*}
    \pi_n(\Omega A) &\simeq\trunc0{\Omega^n(\Omega A)}\\
                    &\simeq\trunc0{\Omega(\Omega^n A)}\\
                    &\simeq\trunc0{\Omega^{n+1}A}\\
                    &\simeq\pi_{n+1}(A).
  \end{align*}
  It is a group homomorphism because the equivalence $\Omega^n(\Omega A)\simeq\Omega(\Omega^nA)$
  preserves composition of paths.
\end{proof}

It turns out that the higher homotopy groups are always abelian, as we prove now.

\begin{proposition}[Eckmann--Hilton argument]\label{eckmannhilton}
  For any pointed type $A$ and $n\ge2$, the group $\pi_n(A)$ is abelian.
\end{proposition}

\newcommand{\ccc}{\mathsf{c}}
\begin{proof}
  We prove by induction on $n\ge2$ that for every $A$ the group $\pi_n(A)$ is abelian. For $n=2$ we
  have to prove that $\alpha\concat\beta=\beta\concat\alpha$ for
  $\alpha,\beta:\idp{\star_A}=_{\Omega A}\idp{\star_A}$ and the idea is to use horizontal
  composition $\hcomp$ and the exchange law. Note that for an arbitrary $2$-dimensional path
  $\gamma:p=_{a=_Ab}q$, we do not have $\idp{\idp a}\hcomp\gamma=\gamma$. Indeed this is not
  well-typed as the left-hand side has type $(\idp b\concat p)=(\idp b\concat q)$. In order to get
  something having the same type as $\gamma$ we need to compose on both sides with the paths
  $\lambda_p:\idp b\concat p=p$ and $\lambda_q:\idp b\concat q=q$. We get the following two unit
  laws for horizontal composition, which are proved like other coherence operations:
  \begin{align*}
  (a,b:A)(p,q:a=b)(\gamma:p=q) &\mapsto \gamma =
  \lambda_p\inv\concat(\idp{\idp{a}}\hcomp\gamma)\concat\lambda_q,\\
  (a,b:A)(p,q:a=b)(\gamma:p=q) &\mapsto \gamma =
  \rho_p\inv\concat(\gamma\hcomp\idp{\idp{b}})\concat\rho_q.
  \end{align*}
  In the case of $\alpha$ and $\beta$, both $p$ and $q$ are equal to $\idpS$ and $\lambda_{\idpS}$
  and $\rho_{\idpS}$ are both equal. We write $\ccc$ for short for both of them. Note that $\ccc$ is
  equal to $\idp{\idp{\star_A}}$ by definition if we take $(\jcomp)$ definitional, but the
  computation below does not depend on this fact.
  Using these two unit laws and the exchange law twice we have
  \begin{align*}
    \alpha\vcomp \beta &= \ccc\inv\vcomp(\idpS\hcomp \alpha)\vcomp\ccc\vcomp\ccc\inv\vcomp(\beta\hcomp\idpS)\vcomp\ccc\\
                       &= \ccc\inv\vcomp(\idpS\hcomp \alpha)\vcomp(\beta\hcomp\idpS)\vcomp\ccc\\
                       &= \ccc\inv\vcomp((\idpS\vcomp \beta)\hcomp(\alpha\vcomp\idpS))\vcomp\ccc\\
                       &= \ccc\inv\vcomp((\beta\vcomp\idpS)\hcomp(\idpS\vcomp \alpha))\vcomp\ccc\\
                       &= \ccc\inv\vcomp(\beta\hcomp\idpS)\vcomp(\idpS\hcomp \alpha)\vcomp\ccc\\
                       &= \ccc\inv\vcomp(\beta\hcomp\idpS)\vcomp\ccc\vcomp\ccc\inv\vcomp(\idpS\hcomp\alpha)\vcomp\ccc\\
                       &= \beta\vcomp \alpha.
  \end{align*}
  Therefore $\pi_2(A)$ is abelian.

  For $n>2$, we know from proposition \ref{piomega} that $\pi_{n+1}(A)\simeq\pi_n(\Omega A)$ and
  $\pi_n(\Omega A)$ is abelian by induction hypothesis, hence $\pi_{n+1}(A)$ is abelian as well.
\end{proof}

We can also define homotopy groups using pointed maps from spheres. We first show the following proposition.
\begin{proposition}\label{adjsuspomega}
  Given two pointed types $A$ and $B$, there is an equivalence
  \[(\Susp A \topt B) \simeq (A \topt \Omega B).\]
\end{proposition}

\begin{proof}
  We remind that we defined in section \ref{sec:hit} a pointed map $\varphi_A:A\to\Omega\Sigma A$
  for every type $A$. We have to define two maps between $\Susp A\topt B$ and $A\topt\Omega B$ and
  prove that they are inverse to each other.
  
  The map from left to right sends $f:\Susp A \topt B$ to $\lambda a.(\Omega f)(\varphi_A(a))$. That
  map is pointed because $\varphi_A$ is.  The map from right to left sends $g:A\topt \Omega B$ to
  the map $\Susp A\topt B$ sending both $\north$ and $\south$ to $\star_B$ and the path $\merid(a)$
  to $g(a)$. That map is pointed by $\idp{\star_B}$.

  Starting from $f:\Susp A\topt B$, its image by the composite is the map $\widetilde{f}$ sending
  both $\north$ and $\south$ to $\star_B$ and the path $\merid(a)$ to
  $\star_f\inv\concat\ap{f}(\varphi_A(a))\concat \star_f$. That function is equal to $f$ via
  $\star_f\inv$ for $\north$ and $\ap f(\merid(\star_A))\inv\concat\star_f$ for $\south$.

  In the other direction, if we start from $g:A\topt\Omega B$, its image is the map
  \[\lambda a.(\idpS\inv\concat g(a)\concat g(\star_A)\inv\concat\idpS),\]
  which is pointwise equal to $g$ because $g$ is pointed.
\end{proof}

It is easy to see that $\Bool\topt A$ is equivalent to $A$ and that $\Susp\Bool$ is equivalent to
$\Sn1$, therefore given that $\Sn{n+1}$ is defined to be $\Susp\Sn n$, we have the following by
induction on $n$.

\begin{proposition}
  For every pointed type $A$ and $n:\N$, we have
  \[\Omega^n A\simeq(\Sn{n}\topt A)\]
  and in particular
  \[\pi_n(A) \simeq \trunc0{\Sn{n}\topt A}.\]
\end{proposition}

Moreover, one can define the group operation directly on $\trunc0{\Sn{n}\topt A}$ as follows.
\begin{definition}
  For every type $A$, we define the map
  \begin{align*}
    \contreq_A &: \Susp A \to \Susp A \vee \Susp A,\\
    \contreq_A(\north) &\defeq \inl(\north),\\
    \contreq_A(\south) &\defeq \inr(\north),\\
    \ap{\contreq_A}(\merid(a)) &\defeq \ap\inl(\varphi_A(a)) \concat \push(\ttt) \concat \ap\inr(\varphi_A(a)).
  \end{align*}
  Given two pointed maps $f:X\topt A$ and $g:Y\topt A$, we define the map
  \begin{align*}
    \langle f,g\rangle &: X\vee Y \topt A,\\
    \langle f,g\rangle(\inl(x)) &\defeq f(x),\\
    \langle f,g\rangle(\inr(y)) &\defeq g(y),\\
    \ap{\langle f,g\rangle}(\push(\ttt)) &\defeq \star_f\concat\star_g\inv.
  \end{align*}
\end{definition}

\begin{proposition}
  Given two pointed maps $f,g:\Sn{n}\topt A$, their multiplication when seen as elements of
  $\pi_n(A)$ corresponds to the composition $\langle f,g\rangle\circ\contreq_{\Sn{n-1}}$.
\end{proposition}

\begin{proof}
  Using proposition \ref{adjsuspomega} we have the equivalence
  \begin{align*}
    e &: (\Sn n\topt A)\simeq(\Sn{n-1}\topt\Omega A),\\
    e(h) &\defeq \lambda x.(\star_h\inv\concat\ap h(\varphi_{\Sn{n-1}}(x))\concat\star_h).
  \end{align*}
  There is a group structure on the right-hand side given by composition of paths in $\Omega A$, and
  by proposition \ref{piomega} it is the same as the group structure on $\pi_n(A)$.
  We then have to check that $e(\langle f,g\rangle\circ\contreq_{\Sn{n-1}})$ is the composition of
  $e(f)$ and $e(g)$. For every $x:\Sn{n-1}$ we have
  \begin{align*}
    e(\langle f,g\rangle\circ\contreq_{\Sn{n-1}})(x) &= \star_f\inv\concat\ap{\langle f,g\rangle\circ\contreq_{\Sn{n-1}}}(\varphi_{\Sn{n-1}}(x))\concat\star_f\\
    &= \star_f\inv\concat\ap{\langle f,g\rangle}(\ap{\contreq_{\Sn{n-1}}}(\varphi_{\Sn{n-1}}(x)))\concat\star_f\\
    &= \star_f\inv\concat\ap{\langle
      f,g\rangle}(\ap\inl(\varphi_{\Sn{n-1}}(x))\concat\push(\ttt)\concat\phantom{a}\\
    &\qquad\qquad\qquad\quad\ap\inr(\varphi_{\Sn{n-1}}(x))\concat\push(\ttt)\inv)\concat\star_f\\
    &= \star_f\inv\concat\ap{\langle f,g\rangle}(\ap\inl(\varphi_{\Sn{n-1}}(x)))\concat\ap{\langle
      f,g\rangle}(\push(\ttt))\\
    &\qquad\concat\ap{\langle f,g\rangle}(\ap\inr(\varphi_{\Sn{n-1}}(x)))\concat\ap{\langle
      f,g\rangle}(\push(\ttt))\inv\concat\star_f\\
    &= \star_f\inv\concat\ap f(\varphi_{\Sn{n-1}}(x))\concat\star_f\concat\star_g\inv\concat\ap
      g(\varphi_{\Sn{n-1}}(x))\concat\star_g\\
    &= e(f)(x) \concat e(g)(x).
  \end{align*}
  which concludes the proof.
\end{proof}

\section{Homotopy groups of the circle}\label{hgss1}

In order to compute the homotopy groups of the circle we define a map $U:\Sn1\to\Type$, which
corresponds classically to the universal cover of $\Sn1$, by
\begin{align*}
  U &: \Sn1\to\Type,\\
  U(\base) &\defeq \Z,\\
  \ap U(\lloop) &\defeq \ua(\succZ).
\end{align*}
This definition is valid because the function $\succZ:\Z\to\Z$, which adds one to its argument, is
an equivalence as is easy to check. Intuitively it means that $U$ is a fibration over $\Sn1$ whose
fiber over $\base$ is the type of integers $\Z$ and such that transporting an integer along $\lloop$
in that fibration corresponds to adding $1$ to it.
We have the picture
\begin{center}
  \begin{tikzpicture}[scale=0.5,antidashed/.style={dashed,dash phase=3pt}]
    \draw (0,-0.5) circle (2 and 1);
    \draw (-2,-0.5) node {$\bullet$};
    \draw (-3,-1) node {$\base$};
    \draw (3,-1) node {$\lloop$};

    \draw[antidashed] (-1,2.5) arc (90:0:3 and 1);
    \draw[dashed] (-1,2.5) arc (90:180:1 and 0.25);
    \draw (-2,2.25) arc (180:270:1 and 0.25);
    \draw[color=white,line width=5pt] (-1,2) arc (-90:90:3 and 1);
    \draw (-1,2) arc (-90:90:3 and 1);
    \draw (-1,4) arc (90:270:1 and 0.25);
    \draw[color=white,line width=5pt] (-1,3.5) arc (-90:90:3 and 1);
    \draw (-1,3.5) arc (-90:90:3 and 1);
    \draw (-1,5.5) arc (90:270:1 and 0.25);
    \draw[color=white,line width=5pt] (-1,5) arc (-90:90:3 and 1);
    \draw (-1,5) arc (-90:90:3 and 1);
    \draw (-1,7) arc (90:270:1 and 0.25);
    \draw[color=white,line width=5pt] (-1,6.5) arc (-90:90:3 and 1);
    \draw (-1,6.5) arc (-90:90:3 and 1);
    \draw (-1,8.5) arc (90:180:1 and 0.25);
    \draw[color=white,line width=5pt] (-1,8) arc (-90:90:3 and 1);
    \draw[dashed] (-2,8.25) arc (180:270:1 and 0.25);
    \draw[antidashed] (-1,8) arc (-90:0:3 and 1);

    \draw (-2,2.25) node {$\bullet$};
    \draw (-2,3.75) node {$\bullet$};
    \draw (-2,5.25) node {$\bullet$};
    \draw (-2,6.75) node {$\bullet$};
    \draw (-2,8.25) node {$\bullet$};
    \draw[anchor=east] (-2.2,2.25) node {$-2$};
    \draw[anchor=east] (-2.2,3.75) node {$-1$};
    \draw[anchor=east] (-2.2,5.25) node {$0$};
    \draw[anchor=east] (-2.2,6.75) node {$1$};
    \draw[anchor=east] (-2.2,8.25) node {$2$};
  \end{tikzpicture}
\end{center}

According to the flattening lemma (proposition \ref{flatteningS1}), the total space of this
fibration is equivalent to the higher inductive type $E$ generated by $\Z$-many points $c_n:E$, and
$\Z$-many paths $p_n:c_n=_Ec_{n+1}$ relating each point with the next one.
\[
\begin{tikzcd}[sdiag]
  \dots \arrow[r,"p_{-2}"] & \bullet \arrow[r,"p_{-1}"] \arrow[r,phantom,"c_{-1}"' at start, yshift=-8pt] & \bullet \arrow[r,"p_0"] \arrow[r,phantom,"c_0"' at start, yshift=-8pt] & \bullet \arrow[r,"p_1"] \arrow[r,phantom,"c_1"' at start, yshift=-8pt] & \bullet \arrow[r,"p_2"] \arrow[r,phantom,"c_2"' at start, yshift=-8pt] & \dots
\end{tikzcd}
\]
It is easy to check that this type is contractible. We first construct a path $\ell_n:c_0=c_n$ by
induction on $n:\Z$ as a composition of the paths $p_i$ or their inverses, and then for every $n:\Z$
we prove that $\ell_n\concat p_n=\ell_{n+1}$, which follows from associativity of composition of
paths and from the fact that $p\inv\concat p=\idpS$ for any path $p$. Therefore $U$ is a fibration
over $\Sn1$ with fiber $\Z$ and a contractible total space and we have the following proposition
(which is a consequence of \cite[theorem 4.7.7]{hottbook} and of proposition \ref{pathfibcontr})
\begin{proposition}\label{fibtotcontr}
  Given a pointed type $A$ and a dependent type $P:A\to\Type$ whose total space is contractible, for
  every $a:A$ and $x:P(\star_A)$ the map
  \[
  \transport^P(-,x) : (\star_A=a) \to P(a)
  \]
  is an equivalence
\end{proposition}

Therefore we get
\begin{proposition}\label{prop:pi1s1}
  We have an equivalence $\Omega\Sn1 \simeq \Z$ and
  \[\pi_n(\Sn1)\simeq
  \begin{cases}
    \Z &\text{for $n=1$},\\
    0  &\text{for $n>1$}.
  \end{cases}
\]
\end{proposition}

\begin{proof}
  Applying proposition \ref{fibtotcontr} to $U$, $a\defeq\base$ and $x\defeq 0$ gives us an
  equivalence between $\Omega\Sn1$ and $U(\base)$, and $U(\base)$ is equal to $\Z$ by
  definition. Therefore $\Omega\Sn1\simeq\Z$. Given that $\Z$ is a set by proposition \ref{Zisset},
  we have $\trunc 0\Z\simeq\Z$ and $\Omega^k\Z$ is contractible for every $k\ge1$, therefore we get
  the equivalences of types
  \[\pi_n(\Sn1)\simeq
  \begin{cases}
    \Z &\text{for $n=1$},\\
    0  &\text{for $n>1$}.
  \end{cases}
  \]
  Moreover $\transport$ commutes with composition of paths, therefore for every $n:\Z$ the map
  $\Omega\Sn1\to\Z$ sends $\lloop^n$ to $n$, where $\lloop^n$ is defined by
  \begin{align*}
    \lloop^0 &\defeq \idp{\base},\\
    \lloop^{n+1} &\defeq \lloop\concat\lloop^n,\\
    \lloop^{-n-1} &\defeq \lloop\inv\concat\lloop^{-n}.
  \end{align*}
  This shows that the equivalence $\pi_1(\Sn1)\simeq\Z$ is a group isomorphism for the usual group
  structure on $\Z$.
\end{proof}

\begin{proposition}
  The circle $\Sn1$ is $1$-truncated.
\end{proposition}

\begin{proof}
  We have to prove that for every $x,y:\Sn1$, the type $x=_{\Sn1}y$ is a set
  (i.e.\ $0$-truncated). We proceed by induction on $x$. Note that we only have to do the case
  $x=\base$ because for $\lloop$ it follows immediately from the fact that being a set is a mere
  proposition. We proceed now by induction on $y$. For the same reason we only have to do the case
  $y=\base$. Therefore we now have to prove that $\base=_{\Sn1}\base$ is a set, but this follows
  immediately from the fact that $\Omega\Sn1\simeq\Z$ and that $\Z$ is a set.
\end{proof}

\section{Connectedness}

We introduce now the notion of $n$-connectedness of types and maps, which can be thought of as dual
to the notion of $n$-truncatedness. A type is $n$-truncated if it has no homotopical information
above dimension $n$ while a type in $n$-connected if it has no homotopical information \emph{below}
dimension $n$. We take the same convention as in \cite{hottbook}, which differs from the one in
classical homotopy theory. Connectivity of spaces coincides with the classical notion while an
$n$-connected map for us corresponds to an $(n+1)$-connected map for the classical definition.

\begin{definition}
  Given $f:A\to B$ and $b:B$, the (homotopy) \emph{fiber} of $f$ at $b$ is the type
  \[\fib_f(b)\defeq\left(\sum_{a:A}(f(a)=b)\right).\]
\end{definition}
For instance the fiber of the unique map $A\to\Unit$ is equivalent to $A$ (because $\Unit$ is
contractible), and the fiber at $a$ of a map $f:\Unit\to A$ is equivalent to the type $f(\ttt)=_Aa$.

\begin{definition}
  A type $A$ is said \emph{$n$-connected} if its $n$-truncation $\trunc nA$ is contractible.  A map
  $f:A\to B$ is said \emph{$n$-connected} if all of its fibers are $n$-connected. For short we
  simply say \emph{connected} for $0$-connected.
\end{definition}

A type $A$ is $n$-connected if and only if the canonical map $A\to\Unit$ is $n$-connected, given
that the only fiber of that map is equivalent to $A$. On the other hand if $A$ is pointed and
$n$-connected, then the canonical map $\Unit\to A$ is only $(n-1)$-connected as we show below.

\begin{proposition}
  Every type and every map is $(-2)$-connected and every pointed type is $(-1)$-connected.
\end{proposition}

\begin{proof}
  The $(-2)$-truncation of a type is contractible by definition, hence every type is
  $(-2)$-connected and then every map is $(-2)$-connected as well.

  In order to prove that a type $A$ is $(-1)$-connected, we have to prove that $\trunc{-1}A$ is
  contractible. Given that $\trunc{-1}A$ is a mere proposition, it is enough to prove that it is
  pointed, and it is the case if $A$ is pointed.
\end{proof}

\begin{proposition}\label{connectedpointedmap}
  A pointed type $A$ is $n$-connected if and only if the canonical map $\Unit\to A$ is
  $(n-1)$-connected.
\end{proposition}

\begin{proof}
  For every $a:A$, the fiber of $\Unit\to A$ over $a$ is the type $(a=_A\star_A)$ and according to
  proposition \ref{looptrunc} we have
  \[\trunc{n-1}{a=_A\star_A}\simeq|a|=_{\trunc nA}|\star_A|.\]
  Therefore if $A$ is $n$-connected, then the right-hand side is contractible which shows that the
  fiber $(a=_A\star_A)$ is $(n-1)$-connected and that the map $\Unit\to A$ is
  $(n-1)$-connected. Conversely, if the map $\Unit\to A$ is $(n-1)$-connected, i.e.\
  $|a|=_{\trunc nA}|\star_A|$ is contractible for every $a:A$, let’s prove that $\trunc nA$ is
  contractible. We need to prove that for every $x:\trunc nA$, we have $x=_{\trunc
    nA}|\star_A|$.
  This type is $n$-connected, hence it’s enough to prove it for $x$ of the form $|a|$, but it’s true
  for such $x$ by assumption.
\end{proof}

\begin{proposition}
  Given a pointed type $A$, if $A$ is $(n+1)$-connected then $\Omega A$ is $n$-connected.
\end{proposition}

\begin{proof}
  We have $\trunc n{\Omega A} \simeq \Omega\trunc{n+1}{A}$ which is contractible.
\end{proof}

\begin{proposition}
  Given a type $A$ and $m\le n$, we have
  \[\trunc m{\trunc nA}\simeq\trunc mA.\]
\end{proposition}

\begin{proof}
  We define two maps $f$ and $g$ by
  \begin{align*}
    f &: \trunc m{\trunc nA} \to \trunc mA, & g &: \trunc mA \to \trunc m{\trunc nA},\\
    f(||a||) &\defeq |a|, & g(|a|) &\defeq ||a||.
  \end{align*}
  Note that in the definition of $f$ we use the induction principle for truncations twice and the
  fact that $\trunc mA$ is $n$-truncated (which follows from $m\le n$ and proposition
  \ref{ntruncprop}). Then we define two homotopies $h$ and $k$ by
  \begin{align*}
    h &: (x:\trunc m{\trunc nA}) \to g(f(x))=x, & k &: (y : \trunc mA)\to f(g(y))=y, \\
    h(||a||) &\defeq \idp{||a||}, & k(|a|) &\defeq \idp{|a|},
  \end{align*}
  which proves that $f$ and $g$ are inverse to each other.
\end{proof}

\begin{proposition}
  Given a type $A$ and $m\le n$, if $A$ is $n$-connected then $A$ is $m$-connected as well.
\end{proposition}

\begin{proof}
  We assume that $A$ is $n$-connected and we want to prove that $\trunc mA$ is contractible. From
  the previous proposition we know that $\trunc mA$ is equivalent to $\trunc m{\trunc nA}$, and
  $\trunc nA$ is contractible by assumption, therefore $\trunc mA$ is contractible as well.
\end{proof}

The following very important property can be thought of as an “induction principle” for
$n$-connected maps. It is proved in \cite[lemma 7.5.7]{hottbook}. For readability we use the
notation $\prod_{b:B}P(b)$ for the function type $(b:B)\to P(b)$.
\begin{proposition}\label{inductionconnected}
  For $f:A\to B$ and $P:B\to \Type$, consider the map
  \[\lambda s.s\circ f:\prod_{b:B}P(b) \to \prod_{a:A}P(f(a)).\]

  Then the following are equivalent:
  \begin{itemize}
  \item $f$ is $n$-connected,
  \item for every family of $n$-types $P$, the map $(\lambda s.s\circ f)$ is an equivalence,
  \item for every family of $n$-types $P$, the map $(\lambda s.s\circ f)$ has a section.
  \end{itemize}
\end{proposition}

One way to use it is as follows. In order to define a map $g:(b:B)\to P(b)$, where $P$ is a family
of $n$-types, it is enough to define $g$ on elements of the form $f(a)$ for $f:A\to B$ an
$n$-connected map. For instance if $B$ is an $(n+1)$-connected pointed type and $P$ is a family of
$n$-types over $B$, then in order to construct a section of $P$ it is enough to give an element of
$P(\star_B)$.

\begin{proposition}
  Given two $n$-connected maps $f:A\to B$ and $g:B\to C$, the composition $g\circ f:A\to C$ is also
  $n$-connected.
\end{proposition}

\begin{proof}
  We consider $P:C\to\Type$ a family of $n$-truncated types over $C$ together with
  $d:(a:A)\to P(g(f(a)))$. The composition $P\circ g$ is a family of $n$-truncated types over $B$
  and we have $d:(a:A)\to (P\circ g)(f(a))$ therefore, using the fact that $f$ is $n$-connected,
  there is a function $d':(b:B)\to P(g(b))$ such that $d'(f(a))=d(a)$ for every $a:A$. Using now the
  fact that $g$ is $n$-connected, there is a function $d'':(c:C)\to P(c)$ such that
  $d''(g(b))=d'(b)$. We have in particular $d''(g(f(a)))=d'(f(a))=d(a)$, therefore $g\circ f$ is
  $n$-connected.
\end{proof}

\begin{proposition}\label{connpushout}
  Given two maps $f:C\to A$ and $g:C\to B$ such that $f$ is $n$-connected, the map $\inr:B\to D$ is
  $n$-connected, where $D\defeq A\sqcup^CB$.
  \[
  \begin{tikzcd}
    C \arrow[r,"g"] \arrow[d,"f"'] & B \arrow[d,"\inr"]\\
    A \arrow[r,"\inl"'] & D \arrow[lu,phantom,"\ulcorner",at start]
  \end{tikzcd}
  \]
\end{proposition}

\begin{proof}
  Let’s consider $P:D\to\Type$ a family of $n$-truncated types and $h:(b:B)\to P(\inr(b))$. We want
  to construct a map $k:(d:D)\to P(d)$. We define $Q:A\to\Type$ by $Q(a)\defeq P(\inl(a))$. It’s a
  fibration of $n$-types and there is an element of $(c:C)\to Q(f(c))$ given by transporting
  $(\lambda c.h(g(c)))$ in $P$ backwards along the equality
  $\push(c):\inl(f(c))=\inr(g(c))$. Therefore, given that $f$ is $n$-connected, there is a section
  $\ell:(a:A)\to Q(a)$ of $Q$ together with a dependent equality $e(c)$ between $\ell(f(c))$ and
  $h(g(c))$ in $P$ over $\push(c)$.

  We can now define $k$ by
  \begin{align*}
    k &: (d:D)\to P(d),\\
    k(\inl(a)) &\defeq \ell(a),\\
    k(\inr(b)) &\defeq h(b),\\
    \ap k(\push(c)) &\defeq e(c),
  \end{align*}
  and it agrees with $h$ on element of the form $\inr(b)$.
  Therefore, $\inr$ is $n$-connected.
\end{proof}

\begin{proposition}\label{truncconn}
  For any $X$, the map $|-|:X\to\trunc nX$ is $n$-connected.
\end{proposition}

\begin{proof}
  It’s an immediate consequence of proposition \ref{inductionconnected} and of the induction
  principle for truncations.
\end{proof}

\section{Lower homotopy groups of spheres}\label{hgs0}

We can now prove that all homotopy groups of spheres of the form $\pi_k(\Sn n)$ for $k<n$ are
trivial.
\begin{proposition}\label{lowerhgs}
  For every $k,n:\N$ such that $k<n$, we have
  \[\pi_k(\Sn n)\simeq \Unit.\]
\end{proposition}
It follows from the following proposition, as shown below.
\begin{proposition}\label{sphereconn}
  For every $n:\N$, the $n$-sphere $\Sn n$ is $(n-1)$-connected.
\end{proposition}
\begin{proof}
  For $n=0$ this is true because $\Sn 0\defeq\Bool$ is pointed hence $(-1)$-connected.  If we assume
  now that $\Sn n$ is $(n-1)$-connected, we have the following pushout diagram where the top map is
  $(n-1)$-connected, hence the bottom map is $(n-1)$-connected as well by proposition
  \ref{connpushout}.
  \[
  \begin{tikzcd}
    \Sn n \arrow[r] \arrow[d] & \Unit \arrow[d] \\
    \Unit \arrow[r] & \Sn{n+1} \arrow[lu,phantom,"\ulcorner",at start]
  \end{tikzcd}
  \]
  Therefore by proposition \ref{connectedpointedmap} the type $\Sn{n+1}$ is $n$-connected, which
  concludes the proof.
\end{proof}
\begin{proof}[Proof of proposition \ref{lowerhgs}]
  Given $k,n:\N$ such that $k<n$, the type $\Sn n$ is $(n-1)$-connected and taking $\Omega^k$
  decreases the connectivity by $k$ hence $\Omega^k(\Sn n)$ is $(n-1-k)$-connected and we have
  $(n-1-k)\ge0$ because $k<n$. Therefore $\Omega^k(\Sn n)$ is $0$-connected, which shows that
  $\pi_k(\Sn n)$ is contractible.
\end{proof}

\section{The Hopf fibration}\label{hgshopf}

Given a fibration $P:B\to\Type$ we often use the notation
\[
\begin{tikzcd}
  F\arrow[r,"i"] & E \arrow[r,"p"] & B,
\end{tikzcd}
\]
where $F\defeq P(\star_B)$ is the \emph{fiber}, $E\defeq\sum_{x:B}P(x)$ is the \emph{total space},
$i$ sends $y$ to $(\star_B,y)$ and $p$ is the first projection.

In this section we introduce H-spaces and the Hopf construction which is a fibration
\[
\begin{tikzcd}
  A \arrow[r] & A*A \arrow[r] & \Susp A
\end{tikzcd}
\]
for any connected H-space $A$. Applying this construction to the circle gives us the Hopf fibration
\[
\begin{tikzcd}
  \Sn1 \arrow[r] & \Sn3 \arrow[r] & \Sn2.
\end{tikzcd}
\]
Using the long exact sequence of homotopy groups, it will imply that $\pi_2(\Sn2)\simeq\Z$ and that
$\pi_n(\Sn2)\simeq\pi_n(\Sn3)$ for all $n\ge3$.

\begin{definition}
  An \emph{H-space} is a pointed type $A$ together with a binary operation $\mu:A\to A\to A$ and
  equalities
  \begin{align*}
    \mu_l &:(a : A) \to \mu(\star_A,a)=_Aa,\\
    \mu_r &:(a : A) \to \mu(a,\star_A)=_Aa,\\
    \mu_{lr} &:\mu_l(\star_A)=\mu_r(\star_A).
  \end{align*}
\end{definition}

For instance every group has an obvious structure of H-space, where $\mu$ is given by the
multiplication of the group, $\mu_l$ and $\mu_r$ follow from the unit laws, and $\mu_{lr}$ follows
from the fact that a group is a set. In the classical definition of H-spaces, the homotopies $\mu_l$
and $\mu_r$ are required to fix the basepoint (which cannot be expressed in homotopy type theory),
and as a consequence the two-dimensional path $\mu_{lr}$ is not required anymore.

\begin{proposition}\label{hspacemapsequiv}
  If $A$ is a connected H-space, then for every $a:A$ the functions $\mu(a,-):A\to A$ and
  $\mu(-,a):A\to A$ are equivalences.
\end{proposition}

\begin{proof}
  What we want to prove is a mere proposition depending on $a:A$. Given that $A$ is pointed and
  $0$-connected, it is enough to prove it for the basepoint $\star_A$. But the functions
  $\mu(\star_A,-)$ and $\mu(-,\star_A)$ are both equal to the identity function (using $\mu_l$ and
  $\mu_r$), hence they are equivalences.
\end{proof}

\begin{proposition}
  Given a connected H-space $A$, there is a fibration over $\Susp A$ with fiber $A$ and whose total
  space is equivalent to $A*A$.
\end{proposition}

\begin{proof}
  We define $H:\Susp A\to\Type$ by
  \begin{align*}
    H(\north) &\defeq A,\\
    H(\south) &\defeq A,\\
    \ap H(\merid(a)) &\defeq \ua(\mu(-,a)).
  \end{align*}
  The definition is valid because $\mu(-,a)$ is an equivalence by proposition \ref{hspacemapsequiv}.

  By the flattening lemma (proposition \ref{flatteningpushout}), the total space of this fibration
  is equivalent to the pushout of the span
  \[
  \begin{tikzcd}
    A & A\times A \arrow[l,"\fst"'] \arrow[r,"\mu"] & A.
  \end{tikzcd}
  \]

  The map $\alpha:A\times A\to A\times A$ defined by $\alpha(x,y)=(x,\mu(x,y))$ is an
  equivalence, its inverse is the map $\alpha^{-1}(u,v)=(u,\mu(u,-)^{-1}(v))$, hence we have the
  equivalence of spans
  \[
  \begin{tikzcd}
    A \arrow[d, equals] & A\times A \arrow[l,"\fst"'] \arrow[r,"\mu"] \arrow[d,"\alpha"] & A
    \arrow[d,equals]\\
    A & A\times A \arrow[l,"\fst"'] \arrow[r,"\snd"] & A
  \end{tikzcd}
  \]
  This shows that the total space of the fibration we constructed is equivalent to the join $A * A$.
\end{proof}

\begin{proposition}\label{hspacecircle}
  There is a structure of H-space on the circle.
\end{proposition}

\begin{proof}
  We define the map $\mu:\Sn1\to\Sn1\to\Sn1$ by induction on the first argument.
  \begin{itemize}
  \item For $\base$, we take $\mu(\base,x)\defeq x$.
  \item For $\lloop$, we have to construct an equality between the identity function and itself in
    the type $\Sn1\to\Sn1$. By function extensionality, it is equivalent to construct a function
    $\lloop_-:(x:\Sn1)\to x=_{\Sn1}x$, i.e.\ a path $\lloop_x$ from $x$ to itself for every point
    $x:\Sn1$. We proceed again by induction on $x$.
    \begin{itemize}
    \item For $\base$, we take $\lloop_{\base}\defeq\lloop$.
    \item For $\lloop$, we have to fill the square
      \[
      \begin{tikzcd}[sdiag]
        \bullet \arrow[r,"\ap{\lambda x.x}(\lloop)"] \arrow[d,"\lloop_{\base}"'] & \bullet
        \arrow[d,"\lloop_{\base}"] \\
        \bullet \arrow[r,"\ap{\lambda x.x}(\lloop)"'] & \bullet
      \end{tikzcd}
      \]
      which is equal to the square
      \[
      \begin{tikzcd}[sdiag]
        \bullet \arrow[r,"\lloop"] \arrow[d,"\lloop"'] & \bullet
        \arrow[d,"\lloop"] \\
        \bullet \arrow[r,"\lloop"'] & \bullet
      \end{tikzcd}
      \]
      which is filled by $\idp{\lloop\concat\lloop}$.
    \end{itemize}
  \end{itemize}

  We construct $\mu_l$, $\mu_r$ and $\mu_{lr}$:

  \begin{itemize}
  \item For $\mu_l$ we have $\mu(\base,x)=_{\Sn1}x$ by definition, so we take $\mu_l(x)\defeq\idp{\base}$ for
    every $x$.
  \item For $\mu_r:(x:\Sn1)\to\mu(x,\base)=_{\Sn1}x$, we proceed by induction on $x$.
    \begin{itemize}
    \item For $\base$, it’s true by definition, so we take $\mu_r(\base)\defeq\idp{\base}$.
    \item For $\lloop$, we have to fill the square
      \[
      \begin{tikzcd}[sdiag]
        \bullet \arrow[r,"\ap{\lambda x.\mu(x,\base)}(\lloop)"] \arrow[d,"\idp{\base}"'] & \bullet
        \arrow[d,"\idp{\base}"] \\
        \bullet \arrow[r,"\ap{\lambda x.x}(\lloop)"'] & \bullet
      \end{tikzcd}
      \]
      which is equal to the square
      \[
      \begin{tikzcd}[sdiag]
        \bullet \arrow[r,"\lloop"] \arrow[d,"\idp{\base}"'] & \bullet
        \arrow[d,"\idp{\base}"] \\
        \bullet \arrow[r,"\lloop"'] & \bullet
      \end{tikzcd}
      \]
      which is again filled by some coherence operation.
    \end{itemize}
  \item Finally, both $\mu_l(\base)$ and $\mu_r(\base)$ are equal to $\idp\base$ by definition,
    therefore we take $\mu_{lr}\defeq\idp{\idp\base}$. \qedhere
  \end{itemize}
\end{proof}

\begin{proposition}
  There is a fibration over $\Sn2$ with fiber $\Sn1$ and whose total space is equivalent to $\Sn3$.
\end{proposition}
\begin{proof}
  We apply the Hopf construction to the circle, which is a connected H-space according to the
  previous proposition. We get a fibration over $\Susp\Sn1$ with fiber $\Sn1$ and total space
  $\Sn1*\Sn1$, and we have constructed in proposition \ref{joinspheres} an equivalence
  $\Sn1*\Sn1\simeq\Sn3$ so we get the Hopf fibration.
\end{proof}

This fibration is called the \emph{Hopf fibration} and the induced map $\Sn3\to\Sn2$ is called the
\emph{Hopf map}.

\section{The long exact sequence of a fibration}

The \emph{long exact sequence of homotopy groups} of a fibration is a long exact sequence relating
the homotopy groups of the base space, the fiber space and the total space of a fibration. A
construction of it in homotopy type theory is also presented in \cite[section 8.4]{hottbook}. We
present here a different construction.

\begin{definition}
  A sequence of pointed sets ($0$-truncated types) and pointed maps
  \[
  \begin{tikzcd}
    \dots \arrow[r] & A_{n-1} \arrow[r,"f_{n-1}"] & A_n \arrow[r,"f_n"] & A_{n+1} \arrow[r] & \dots
  \end{tikzcd}
  \]
  is called a \emph{long exact sequence (of sets)} if for every $n$ and for every $x:A_n$, then $x$
  is merely in the image of $f_{n-1}$ if and only if it is in the kernel of $f_n$. In other words,
  it means that we have for every $x:A_n$ an equivalence
  \[\trunc{-1}{\sum_{a:A_{n-1}}f_{n-1}(a)=_{A_n}x} \!\!\!\!\!\simeq \left(f_n(x)=_{A_{n+1}}\star_{A_{n+1}}\right).\]
\end{definition}
Note that the left-hand side is a mere proposition by definition, and the right-hand side as well
because $A_{n+1}$ is assumed to be a set. Therefore having an equivalence is equivalent to having
two functions going back and forth. When all the $A_n$ are groups and all the $f_n$ are group
homomorphisms (as is usually the case), we talk about a \emph{long exact sequence of groups}.

When the $A_n$ are not assumed to be sets, we introduce instead the notion of fiber sequence.
\begin{definition}
  A sequence of pointed types and pointed maps
  \[
  \begin{tikzcd}
    \dots \arrow[r] & A_{n-1} \arrow[r,"f_{n-1}"] & A_n \arrow[r,"f_n"] & A_{n+1} \arrow[r] & \dots
  \end{tikzcd}
  \]
  is called a \emph{fiber sequence} if for every $n$ and for every $x:A_n$, the untruncated types
  “$x$ is in the image of $f_{n-1}$” and “$x$ is the kernel of $f_n$” are equivalent. In other
  words, it means that we have for every $x:A_n$ an equivalence
  \[\left(\sum_{a:A_{n-1}}f_{n-1}(a)=_{A_n}x\right) \simeq \left(f_n(x)=_{A_{n+1}}\star_{A_{n+1}}\right).\]
\end{definition}

The relationship between long exact sequences and fiber sequences is given by the following
proposition.

\begin{proposition}
  Given a fiber sequence
  \[
  \begin{tikzcd}
    \dots \arrow[r] & A_{n-1} \arrow[r,"f_{n-1}"] & A_n \arrow[r,"f_n"] & A_{n+1} \arrow[r] & \dots,
  \end{tikzcd}
  \]
  then its $0$-truncation
  \[
  \begin{tikzcd}
    \dots \arrow[r] & \trunc0{A_{n-1}} \arrow[r,"\trunc0{f_{n-1}}"] & \trunc0{A_n} \arrow[r,"\trunc0{f_n}"] & \trunc0{A_{n+1}} \arrow[r] & \dots
  \end{tikzcd}
  \]
  is a long exact sequence.
\end{proposition}

\begin{proof}
  Given $x':\trunc0{A_n}$, we want to construct an element of the type
  \[\trunc{-1}{\sum_{a':\trunc0{A_{n-1}}}\trunc0{f_{n-1}}(a')=_{\trunc0{A_n}}x'} \!\!\!\!\!\simeq
  \left(\trunc0{f_n}(x')=_{\trunc0{A_{n+1}}}|\star_{A_{n+1}}|\right).\]
  This type is a set, therefore we can assume that $x'$ is of the form $|x|$ with $x:A_n$. Applying
  proposition \ref{looptrunc}, we see that the right-hand side is equivalent to the type
  $\trunc{-1}{f_n(x)=_{A_{n+1}}\star_{A_{n+1}}}$, and by assumption we have
  \[\left(\sum_{a:A_{n-1}}f_{n-1}(a)=_{A_n}x\right) \simeq
  \left(f_n(x)=_{A_{n+1}}\star_{A_{n+1}}\right).\]
  Therefore it is enough to construct an equivalence
  \[\trunc{-1}{\sum_{a:A}B(a)}\simeq\trunc{-1}{\sum_{a':\trunc0A}{B'(a')}}\]
  given a family of equivalences $e:(a:A)\to \trunc{-1}{B(a)}\simeq B'(|a|)$, where $A$ is $A_{n-1}$,
  $B(a)$ is $(f_{n-1}(a)=_{A_n}x)$, and $B'(a')$ is $(\trunc0{f_{n-1}}(a')=_{\trunc0{A_n}}|x|)$.

  From left to right we do an induction on the outer truncation, and we send $|(a,b)|$ to
  $|(|a|,e_a(|b|))|$. From right to left we do an induction on the outer truncation, an induction
  on $\trunc0A$, and, using the equivalence $e_a$, an induction on $\trunc{-1}{B(a)}$, and we send
  $|(|a|,e_a(|b|))|$ to $|(a,b)|$.
\end{proof}

\begin{definition}
  Given a pointed type $B$ and a dependent type $P:B\to\Type$ together with a pointing of
  $F\defeq P(\star_B)$, we define
  \begin{align*}
    P^{\Omega} &: \Omega B\to\Type,\\
    P^{\Omega}(p) &\defeq (\star_F=^P_p\star_F).
  \end{align*}
  Note that the fiber of $P^{\Omega}$ is $\Omega F$ and that the total space of $P^{\Omega}$ is the
  loop space of the total space of $P$, according to proposition \ref{pathspairtypes}. Moreover
  there is a map
  \begin{align*}
    d &: \Omega B \to F,\\
    d(p) &\defeq \transport^P(p,\star_F).
  \end{align*}
\end{definition}

Let $E$ be the total space of $P$. By iterating the previous construction, we get the diagram
\[
\begin{tikzcd}
  & & \dots \arrow[dll,out=280,in=80,looseness=0.40]\\
  \Omega^3 F \arrow[r]& \Omega^3 E \ar[r]& \Omega^3 B\arrow[dll,out=280,in=80,looseness=0.40]\\
  \Omega^2 F \arrow[r]& \Omega^2 E \ar[r]& \Omega^2 B\arrow[dll,out=280,in=80,looseness=0.40]\\
  \Omega F \arrow[r]& \Omega E \ar[r]& \Omega B,
\end{tikzcd}
\]
and after $0$-truncating, we get
\begin{equation}
\begin{tikzcd}
  & & \dots \arrow[dll,out=280,in=80,looseness=0.40]\\
  \pi_3(F) \arrow[r]& \pi_3(E) \ar[r]& \pi_3(B)\arrow[dll,out=280,in=80,looseness=0.40]\\
  \pi_2(F) \arrow[r]& \pi_2(E) \ar[r]& \pi_2(B)\arrow[dll,out=280,in=80,looseness=0.40]\\
  \pi_1(F) \arrow[r]& \pi_1(E) \ar[r]& \pi_1(B).
\end{tikzcd}\label{leshg}
\end{equation}
It is easy to see that the maps $\pi_n(F)\to\pi_n(E)$ and $\pi_n(E)\to\pi_n(B)$ are group
homomorphisms and for the maps $\pi_{n+1}(B)\to\pi_n(F)$ it follows from the following proposition.
\begin{proposition}
  The map
  \begin{align*}
    d &: \Omega^2 B \to \Omega F,\\
    d(p) &\defeq \transport^{P^\Omega}(p,\idp{\star_F})
  \end{align*}
  satisfies $d(p\concat q)=d(p)\concat d(q)$ for every $p,q:\Omega^2B$.
\end{proposition}

\begin{proof}
  The key observation is that for every $\ell:\Omega B$, $\alpha:\idp{\star_B}=\ell$ and
  $r:\Omega F$, we have
  \[\transport^{P^\Omega}(\alpha,r) =_{P^\Omega(\ell)}
  \left(r\concat\transport^{P^\Omega}(\alpha,\idp{\star_F})\right),\]
  where on the right-hand side we compose the homogeneous path $r$ with the dependent path
  $\transport^{P^\Omega}(\alpha,\idp{\star_F})$. This observation is simply proved by path induction
  on $\alpha$.

  Then we have
  \begin{align*}
    d(p\concat q) &= \transport^{P^\Omega}(p\concat q,\idp{\star_F})\\
                  &= \transport^{P^\Omega}(q,\transport^{P^\Omega}(p,\idp{\star_F}))\\
                  &= \transport^{P^\Omega}(q,d(p))\\
                  &= d(p)\concat\transport^{P^\Omega}(q,\idp{\star_F})\\
                  &= d(p)\concat d(q).\qedhere
  \end{align*}
\end{proof}

\begin{proposition}
  The sequence \ref{leshg} is a long exact sequence of groups.
\end{proposition}
\begin{proof}
  It suffices to prove that the sequence
  \[
  \begin{tikzcd}
    \Omega E \arrow[r]& \Omega B \arrow[r] & F \arrow[r] & E \arrow[r] & B
  \end{tikzcd}
  \]
  is a fiber sequence at $E$, $F$ and $\Omega B$.
  
  An element $(b,f)$ of $E$ is in the kernel of the map $E\to B$ if and only if $b=\star_B$ and in
  the image of the map $F\to E$ if and only if there exists $f':F$ such that
  $(b,f)=(\star_B,f')$. They are equivalent because if $(b,f)=(\star_B,f')$ then $b=\star_B$, and
  conversely if we have $p:b=\star_B$ then we can define $f'$ by transporting $f$ in $P$ along $p$
  and we get an equality $(b,f)=(\star_B,f')$ in $E$.

  An element $f$ of $F$ is in the kernel of the map $F\to E$ if and only if
  $(\star_B,f)=(\star_B,\star_F)$, i.e.\ that there is a path $p:\star_B=\star_B$ such that
  $f=\transport^P(p,\star_F)$, which is exactly the condition that $f$ is in the image of $\Omega
  B\to F$.

  A loop $p:\Omega B$ is in the kernel of the map $\Omega B\to F$ if and only if
  $\transport^P(p,\star_F)=\star_F$, and is in the image of
  $\Omega E\to\Omega B$ if and only if there exists a dependent path of type
  $\star_F=^P_p\star_F$. These two conditions are equivalent.
\end{proof}

\begin{proposition}
  We have $\pi_2(\Sn2)\simeq\Z$ and for every $k\ge3$ we have $\pi_k(\Sn2)\simeq\pi_k(\Sn3)$.
\end{proposition}

\begin{proof}
  Using the results already obtained in this chapter, the long exact sequence of homotopy groups of
  the Hopf fibration is
  \[
  \begin{tikzcd}
    & & \dots \arrow[dll,out=280,in=80,looseness=0.40]\\
    0 \arrow[r]& \pi_4(\Sn3) \ar[r]& \pi_4(\Sn2)\arrow[dll,out=280,in=80,looseness=0.40]\\
    0 \arrow[r]& \pi_3(\Sn3) \ar[r]& \pi_3(\Sn2)\arrow[dll,out=280,in=80,looseness=0.40]\\
    0 \arrow[r]& 0 \ar[r]& \pi_2(\Sn2)\arrow[dll,out=280,in=80,looseness=0.40]\\
    \Z \arrow[r]& 0 \ar[r]& 0
  \end{tikzcd}
  \]
  therefore the result follows.
\end{proof}

We can also use the long exact sequence of homotopy groups to describe how $n$-connected maps act on
homotopy groups.
\begin{proposition}
  Given two pointed types $A$ and $B$ and an $n$-connected pointed map $f:A\to B$, the map
  $\pi_k(f):\pi_k(A)\to\pi_k(B)$ is an isomorphism for $k\le n$ and a surjection for $k=n+1$.
\end{proposition}

\begin{proof}
  The dependent type
  \begin{align*}
    P &: B\to\Type,\\
    P(b) &\defeq \left(\sum_{a:A}f(a)=_Bb\right)
  \end{align*}
  corresponds to a fibration
  \[
  \begin{tikzcd}
    \fib_f(\star_B) \arrow[r] & A \arrow[r,"f"] & B.
  \end{tikzcd}
  \]
  Moreover, $\fib_f(\star_B)$ is $n$-connected, therefore for $k\le n$ we have
  \begin{align*}
    \pi_k(\fib_f(\star_B)) &\simeq \Omega^k\trunc k{\fib_f(\star_B)}\\
                           &\simeq \Omega^k\trunc k{\trunc n{\fib_f(\star_B)}},
  \end{align*}
  which is contractible.
  Therefore, the long exact sequence of homotopy groups for the fibration above has $n$ zeros on the
  left column, which shows that $\pi_{n+1}(f)$ is surjective and that $\pi_k(f)$ is bijective for
  $k\le n$.
\end{proof}


\chapter{The James construction}\label{ch:james}

Given a pointed type $A$, the James construction is a sequence of approximations of $\Omega\Susp A$
which are often easier to study than $\Omega\Susp A$. We define the James construction and use it to
prove the Freudenthal suspension theorem and the fact that there exists a natural number $n$ such
that $\pi_4(\Sn3)\simeq\Z/n\Z$.

\section{Sequential colimits}

In this section we define sequential colimits and we prove a few properties about them.
\begin{definition}
  Given a family of types $(A_n)_{n:\N}$ and a family of functions $(i_n)_{n:\N}$ with
  $i_n:A_n\to A_{n+1}$, as in the diagram
  \[
  \begin{tikzcd}
    A_0 \arrow[r,"i_0"] & A_1 \arrow[r,"i_1"] & A_2 \ar[r,"i_2"] & A_3 \ar[r] &\dots,
  \end{tikzcd}
  \]
  we define the \emph{sequential colimit} of $(A_n)_{n:\N}$ as the higher inductive type $A_\infty$
  generated by the two constructors
  \begin{align*}
    \inn&:(n:\N)\to A_n\to A_\infty,\\
    \push&:(n:\N)(x:A_n)\to\inn_n(x)=_{A_\infty}\inn_{n+1}(i_n(x)).
  \end{align*}
\end{definition}
The induction principle for sequential colimits states that given a dependent type
$P:A_\infty\to\Type$, a function $f:(x:A_\infty)\to P(x)$ can be defined by
\begin{align*}
  f &: (x:A_\infty)\to P(x),\\
  f(\inn_n(x)) &\defeq f_{\inn}(n,x),\\
  \ap{f}(\push_n(x)) &\defeq f_{\push}(n,x),
\end{align*}
where we have
\begin{align*}
  f_{\inn}&:(n:\N)(x:A_n)\to P(\inn_n(x)),\\
  f_{\push}&:(n:\N)(x:A_n)\to f_{\inn}(n,x) =^P_{\push_n(x)}f_{\inn}(n+1,i_n(x)).
\end{align*}
We have the diagram
\[
\begin{tikzcd}[column sep=large]
  A_0 \arrow[r,"i_0"] \arrow[rd,"\inn_0"'] & A_1 \arrow[r,"i_1"] \arrow[d,"\inn_1"] & A_2
  \arrow[r,"i_2"] \arrow[ld,"\inn_2"] & A_3 \arrow[r] \arrow[lld,bend left=15,"\inn_3" pos=0.4] &\dots\\
  &A_\infty
\end{tikzcd}
\]
and the map $\inn_0:A_0\to A_\infty$ is called the \emph{transfinite composition} of the family
$(i_n)_{n:\N}$. The following proposition shows (among other things) that $\inn_k$ is the
transfinite composition of the family $(i_{n+k})_{n:\N}$.
\begin{proposition}\label{cropdiag}
  There is an equivalence between $A_\infty$ and the sequential colimit $A_\infty'$ of the diagram
  \[
  \begin{tikzcd}
    A_1 \arrow[r,"i_1"] & A_2 \arrow[r,"i_2"] & A_3 \ar[r,"i_3"] & A_4 \ar[r] &\dots,
  \end{tikzcd}
  \]
  and the maps $\inn'_n:A_{n+1}\to A_\infty'$ correspond to the maps
  $\inn_{n+1}:A_{n+1}\to A_\infty$.
\end{proposition}

\begin{proof}
  We write $\inn'$ and $\push'$ for the constructors of $A'_\infty$.
  The map from $A_\infty'$ to $A_\infty$ sends $\inn'_n(x)$ to $\inn_{n+1}(x)$ and $\push'_n(x)$ to
  $\push_{n+1}(x)$, and the map from $A_\infty$ to $A_\infty'$ sends $\inn_0(x)$ to
  $\inn'_0(i_0(x))$, $\inn_{n+1}(x)$ to $\inn'_n(x)$, $\push_0(x)$ to $\idp{i_0(x)}$ and
  $\push_{n+1}(x)$ to $\push'_n(x)$. One can easily check that those two functions are inverse to
  each other.
\end{proof}
The main result of this section is the fact that connectedness is stable under transfinite
composition.

\begin{proposition}
  If all the maps $i_0$, $i_1$, \dots are $k$-connected, then the transfinite composition of
  $(i_n)_{n:\N}$ is also $k$-connected.
\end{proposition}

\begin{proof}
  Let’s consider $P:A_\infty\to\Type$ a family of $k$-truncated types and
  $d_0:(x:A_0)\to P(\inn_0(x))$. Using proposition \ref{inductionconnected}, it is enough to
  construct a section $d$ of $P$ which is equal to $d_0$ on $A_0$ to conclude that $\inn_0$ is
  $k$-connected. We define a family of maps $d_n:(x:A_n)\to P(\inn_n(x))$ by induction on $n$,
  starting with the given $d_0$ for $n=0$, as follows. Let’s consider
  \begin{align*}
    P_{n+1} &: A_{n+1}\to\Type,\\
    P_{n+1}(x) &\defeq P(\inn_{n+1}(x)).
  \end{align*}
  It is a family of $k$-truncated types, the map $i_n$ is $k$-connected, and we have
  \begin{align*}
    \widetilde{d}_n &: (x : A_n) \to P_{n+1}(i_n(x)),\\
    \widetilde{d}_n(x) &\defeq \transport^P(\push_n(x),d_n(x)),
  \end{align*}
  therefore, using proposition \ref{inductionconnected} again, there is a map
  $d_{n+1}:(x:A_{n+1})\to P(\inn_{n+1}(x))$ satisfying
  \[
  d_{n+1}(i_n(x)) =^P_{\push_n(x)} d_n(x).
  \]
  The family $(d_n)_{n:\N}$ together with those equalities gives a section of $P$ which is equal to
  $d_0$ on $A_0$. Therefore, the map $\inn_0$ is $k$-connected, which is what we wanted to prove.
\end{proof}

\section{The James construction}

\begin{proposition}
  Given a $k$-connected pointed type $A$, with $k\ge0$, there is a sequence of types
  $(J_nA)_{n:\N}$ and maps $i_n : J_nA \to J_{n+1}A$, as in the diagram
  \[
  \begin{tikzcd}
    J_0A \arrow[r,"i_0"] & J_1A \arrow[r,"i_1"] & J_2A \arrow[r,"i_2"] & J_3A \arrow[r,"i_3"] &
    \dots,
  \end{tikzcd}
  \]
  such that for every $n:\N$, the map $i_n$ is $(n(k+1)+(k-1))$-connected, and such that the
  sequential colimit of $(J_nA)_{n:\N}$ is equivalent to $\Omega\Susp A$.
  Moreover, for low values of $n$ we have
  \begin{itemize}
  \item $J_0A \simeq \Unit$,
  \item $J_1A \simeq A$, the map $i_0$ is the inclusion of the basepoint, and the induced map from
    $A$ to $\Omega\Susp A$ is equal to $\varphi_A$,
  \item $J_2A \simeq (A\times A) \sqcup^{A\vee A} A$, where the two maps are the folding map
    \begin{align*}
      \fold_A &: A\vee A\to A,\\
      \fold_A(\inl(x)) &\defeq x,\\
      \fold_A(\inr(y)) &\defeq y,\\
      \ap{\fold_A}(\push(\ttt)) &\defeq\idp{\star_A}
    \end{align*}
    and the inclusion of the wedge in the product
    \begin{align*}
      \iwedge_{A,A} &: A\vee A \to A\times A,\\
      \iwedge_{A,A}(\inl(a)) &\defeq (a,\star_A),\\
      \iwedge_{A,A}(\inr(a)) &\defeq (\star_A,a),\\
      \ap{\iwedge_{A,A}}(\push(\ttt)) &\defeq \idp{(\star_A,\star_A)},
    \end{align*}  
    and $i_1:A\to J_2A$ is equal to $\inr$.
  \end{itemize}
\end{proposition}
Given that connectedness of maps is downward-closed and stable by transfinite composition, the map
$\inn_n$ from $J_nA$ to the colimit is $(n(k+1)+(k-1))$-connected, hence there is an
$(n(k+1)+(k-1))$-connected map from $J_nA$ to $\Omega\Susp A$ which allows us to study
low-dimensional homotopy groups of $\Omega\Susp A$ by studying those of some $J_nA$ instead. We have
in particular the following two corollaries, which are the only two properties of the James
construction that we use here.
\begin{cor}[Freudenthal suspension theorem]\label{freud}
  Given a $k$-connected pointed type $A$, the map
  \[\varphi_A:A\to\Omega\Susp A\]
  is $2k$-connected.
\end{cor}

\begin{cor}\label{corjames}
  Given a $k$-connected pointed type $A$, there is a $(3k+1)$-connected map
  \[(A\times A)\sqcup^{A\vee A}A\to\Omega\Susp A.\]
\end{cor}

Note that both corollaries are also true in the case $k=-1$ because every map is $(-2)$-connected.

We will define the types $(J_nA)$ as pushouts, by induction on $n$, and prove that their colimit
$\JiA$ is equivalent to another type $JA$ which has a much simpler definition. Intuitively, $JA$
corresponds to the free monoid on $A$ with unit $\star_A$, while $J_nA$ corresponds to the “subtype” of
$JA$ consisting of products of length at most $n$. We then prove that the space $\Omega\Sigma A$ is
equivalent to $JA$, where we crucially use the fact that $A$ is connected.

\subsection{Definition of the family $(J_nA)$}

We define the types $J_nA$ together with three functions
\begin{align*}
  i_n &: J_nA\to J_{n+1}A,\\
  \alpha_n &: A\times J_nA\to J_{n+1}A,\\
  \beta_n &: (x:J_nA)\to\alpha_n(\star_A,x)=_{J_{n+1}A}i_n(x),
\end{align*}
by induction on $n$ as follows.
\begin{itemize}
\item $J_0A\defeq\Unit$ (whose unique element is called $\epsilon$),
\item $J_1A\defeq A$, $i_0(\epsilon)\defeq\star_A$, $\alpha_0(a,\epsilon)\defeq a$ and
  $\beta_0(\epsilon)\defeq\idp{\star_A}$,
\item $J_{n+2}A$, $i_{n+1}$, $\alpha_{n+1}$ and $\beta_{n+1}\defeq\push\circ\inr$ are
  defined by the pushout diagram
  \begin{equation}
    \begin{tikzcd}
      (A\times J_nA)\sqcup^{J_nA}J_{n+1}A \arrow[r,"g"] \arrow[d,"f"'] & J_{n+1}A
      \ar[d,"i_{n+1}",dashed]\\
      A\times J_{n+1}A \ar[r,"\alpha_{n+1}"',dashed] & J_{n+2}A \arrow[lu,phantom,"\ulcorner",at start]
    \end{tikzcd}\label{eq:jn+2a}
  \end{equation}
  where the pushout at the top-left of the diagram is defined by the maps $x\mapsto(\star_A,x)$ and
  $i_n$, and the maps $f$ and $g$ are defined by
  \begin{align*}
    f(\inl(a,x))&\defeq(a, i_n(x)), & g(\inl(a,x))&\defeq\alpha_n(a,x),\\
    f(\inr(y))&\defeq (\star_A, y),     & g(\inr(y))&\defeq y,\\
    \ap f(\push(x))&\defeq \idpS,  & \ap g(\push(x))&\defeq \beta_n(x).
  \end{align*}
\end{itemize}
The pushout also consists of
\begin{align*}
  \gamma_n &: (a : A) (x : J_nA) \to \alpha_{n+1}(a,i_n(x))=_{J_{n+2}A}i_{n+1}(\alpha_n(a,x)),\\
  \gamma_n(a,x) &\defeq \push(\inl(a,x))
\end{align*}
and
\begin{align*}
  \eta_n &: (x : J_nA) \to (\gamma_n(\star_A,x)\concat\ap{i_{n+1}}(\beta_n(x)))=\beta_{n+1}(i_n(x)),
\end{align*}
which follows from $\ap\push(\push(x))$ and which is an equality in the type
\[\alpha_{n+1}(\star_A,i_n(x))=_{J_{n+2}A}i_{n+1}(i_n(x)).\]

Note that, in order to define a map out of $J_{n+2}A$, it is enough to define it for elements of the
form $i_{n+1}(x)$ and $\alpha_{n+1}(a,x)$, for paths of the form $\beta_{n+1}(x)$ and
$\gamma_n(a,x)$, and for 2-paths of the form $\eta_n(x)$.

\subsection{Colimit of the family $(J_nA)$}

We now prove that the colimit $\JiA$ of the family $(J_nA)$ is equivalent to the higher inductive
type $JA$ with constructors
\begin{align*}
  \varepsilon &: JA, \\
  \alpha &: A\to JA\to JA, \\
  \delta &: (x:JA)\to x=_{JA}\alpha(\star_A, x).
\end{align*}
In $JA$ we just have a unit called $\varepsilon$, an action of $A$ on $JA$ called $\alpha$ and a
proof that multiplying by $\star_A$ is the identity, called $\delta$.
The induction principle for $JA$ states that given a dependent type $P:JA\to\Type$, a function
$f:(x:JA)\to P(x)$ can be defined by
\begin{align*}
  f &: (x:JA)\to P(x),\\
  f(\varepsilon) &\defeq f_{\varepsilon},\\
  f(\alpha(a,x)) &\defeq f_{\alpha}(a,x,f(x)),\\
  \ap{f}(\delta(x)) &\defeq f_{\delta}(x,f(x)),
\end{align*}
where we have
\begin{align*}
  f_{\varepsilon}&:P(\varepsilon),\\
  f_{\alpha}&:(a:A)(x:JA)\to P(x)\to P(\alpha(a,x)),\\
  f_{\delta}&:(x:JA)(y:P(x)) \to y=^P_{\delta(x)}f_{\alpha}(\star_A,x,y).
\end{align*}
Note that $f$ is used recursively in $f(\alpha(a,x))$ and in $\ap{f}(\delta(x))$, because $JA$ is
a recursive higher inductive type.

In order to prove that $JA$ and $\JiA$ are equivalent, we first define equivalents of
$\gamma_n$, $\eta_n$, $\inn_n$ and $\push_n$ in $JA$ and then equivalents of $\varepsilon$,
$\alpha$, $\delta$, $\gamma$ and $\eta$ in $\JiA$.

\paragraph{Structure on $JA$}

We define the map $\gamma$, where we simply apply $\delta$ twice, by
\begin{align*}
  \gamma &: (a : A) (x : JA) \to \alpha(a,\alpha(\star_A,x)) =
                  \alpha(\star_A,\alpha(a, x)),\\
  \gamma(a,x) &\defeq (\ap{\alpha(a,-)}{(\delta(x))})\inv\concat\delta(\alpha(a,x)).
\end{align*}
Then we define the map
\begin{align*}
  \eta &: (x : JA) \to \gamma(\star_A, x) = \idpS
\end{align*}
using naturality of $\delta$ on $\delta(x)$, which is the square
\[
\begin{tikzcd}[sdiag,column sep=huge]
  x \arrow[d,"{\delta(x)}"'] \arrow[r,"\delta(x)"] &
  \alpha(\star_A,x) \arrow[d,"{\delta(\alpha(\star_A,x))}"]\\
  \alpha(\star_A,x) \arrow[r,"\ap{\alpha(\star_A,-)}(\delta(x))"'] & \alpha(\star_A,\alpha(\star_A, x))
\end{tikzcd}
\]
which shows that $\delta(\alpha(\star_A,x))=\ap{\alpha(\star_A,-)}(\delta(x))$, which is what we
wanted to prove.
We define now $(\innJ_n)$ and $(\pushJ_n)$ by
\begin{align*}
  \innJ_n &: J_nA \to JA, &  \pushJ_n &: (x : J_nA) \to \innJ_n(x) = \innJ_{n+1}(i_n(x)),\\
  \innJ_0(\epsilon) &\defeq \varepsilon, &  \pushJ_n(x) &\defeq \delta(\innJ_n(x)).\\
  \innJ_1(a) &\defeq \alpha(a,\varepsilon),&&\\
  \innJ_{n+2}(i_{n+1}(x)) &\defeq \alpha(\star_A,\innJ_{n+1}(x)),&&\\
  \innJ_{n+2}(\alpha_{n+1}(a,x)) &\defeq \alpha(a,\innJ_{n+1}(x)),&&\\
  \ap{\innJ_{n+2}}{(\beta_n(x))} &\defeq \idpS,&&\\
  \ap{\innJ_{n+2}}{(\gamma_n(a,x))} &\defeq \gamma(a,\innJ_n(x)),&&\\
  \apii{\innJ_{n+2}}{(\eta_n(x))} &\defeq \eta(\innJ_n(x)),&&
\end{align*}

\paragraph{Structure on $\JiA$}

The equivalent of $\varepsilon$ is the term $\epsiloni\defeq\inn_0(\epsilon)$ of type $\JiA$. We
then define the action of $A$ on $\JiA$ by
\begin{align*}
  \alphai &: A \to \JiA \to \JiA,\\
  \alphai(a,\inn_n(x)) &\defeq \inn_{n+1}(\alpha_n(a,x)),\\
  \ap{\alphai(a,-)}{(\push_n(x))} &\defeq \push_{n+1}(\alpha_n(a,x))\cdot\ap{\inn_{n+2}}{(\gamma_n(a,x))}\inv.
\end{align*}
The equivalent of $\delta$ is $\deltai$ defined by
\begin{align*}
  \deltai &: (x : \JiA) \to x=\alphai(\star_A,x), \\
  \deltai(\inn_n(x)) &\defeq \push_n(x) \concat \ap{\inn_{n+1}}(\beta_n(x))\inv, \\
  \ap\deltai(\push_n(x)) &\defeq \deltai^{\push_n}(x),
\end{align*}
where $\deltai^{\push_n}(x)$ is the filler of the square
\begin{equation}
\begin{tikzcd}[sdiag,sep=6.5em]
  \bullet \arrow[rr,"\push_n(x)"] \arrow[d,"\push_n(x)"'] & & \bullet
  \arrow[d,"\push_{n+1}(i_n(x))"]\\
  \bullet \arrow[rru,"\idpS"'] & & \bullet\\
  \bullet \arrow[u,"\ap{\inn_{n+1}}(\beta_n(x))"] \arrow[r,"{\push_{n+1}(\alpha_n(\star_A,x))}"']
  & \bullet \arrow[ru,"\ap{\inn_{n+2}}(\ap{i_{n+1}}(\beta_n(x)))"] & \bullet \arrow[l,"{\ap{\inn_{n+2}}(\gamma_n(\star_A,x))}"] \arrow[u,"\ap{\inn_{n+2}}(\beta_{n+1}(i_n(x)))"']
\end{tikzcd}\label{eq:deltaipush}
\end{equation}
where the lower right triangle is filled using $\eta_n(x)$ and the pentagon in the middle is filled
using the naturality square of $\push_{n+1}$ on $\beta_n(x)$.

We finally define 
\begin{align*}
  \gammai &: (a : A) (x : \JiA) \to \alphai(a,\alphai(\star_A,x)) =
                  \alphai(\star_A,\alphai(a, x)),\\
  \etai &: (x : \JiA) \to \gammai(\star_A, x)=\idpS
\end{align*}
in the same way as we defined $\gamma$ and $\eta$, but using
$\deltai$ instead of $\delta$. In the case of $\gammai(a,\inn_n(x))$ we note that
\begin{align*}
  \gammai(a,\inn_n(x))
  &= \ap{\alphai(a,-)}(\deltai(\inn_n(x)))\inv\concat\deltai(\alphai(a,\inn_n(x)))\\
  &= \ap{\alphai(a,-)}(\push_n(x)\concat
    \ap{\inn_{n+1}}(\beta_n(x))\inv)\inv\concat\deltai(\inn_{n+1}(\alpha_n(a,x)))\\
  &= (\push_{n+1}(\alpha_n(a,x))\concat\ap{\inn_{n+2}}(\gamma_n(a,x))\inv\\
  &\qquad\concat\ap{\inn_{n+2}}(\ap{\alpha_{n+1}(a,-)}(\beta_n(x)))\inv)\inv\\
  &\quad\concat(\push_{n+1}(\alpha_n(a,x))\concat\ap{\inn_{n+2}}(\beta_{n+1}(\alpha_n(a,x)))\inv)\\
  &= \ap{\inn_{n+2}}(\ap{\alpha_{n+1}(a,-)}(\beta_n(x)))\concat\ap{\inn_{n+2}}(\gamma_n(a,x))\\
  &\qquad\concat\ap{\inn_{n+2}}(\beta_{n+1}(\alpha_n(a,x)))\inv.
\end{align*}
Therefore $\gammai(a,\inn_n(x))$ fits in the square
\begin{equation}
\begin{tikzcd}[sdiag,sep=huge]
  \bullet \arrow[r,"{\ap{\inn_{n+2}}(\ap{\alpha_{n+1}(a,-)}(\beta_n(x)))}"]
  \arrow[d,"{\gammai(a,\inn_n(x))}"'] & \bullet \arrow[d,"{\ap{\inn_{n+2}}(\gamma_n(a,x))}"]\\
  \bullet \arrow[r,"{\ap{\inn_{n+2}}(\beta_{n+1}(\alpha_n(a,x)))}"'] & \bullet
\end{tikzcd}\label{eq:gammai}
\end{equation}
For $\etai(\inn_n(x))$, it is defined using $\ap{\deltai}(\deltai(\inn_n(x)))$ and we have
\begin{align*}
  \ap{\deltai}(\deltai(\inn_n(x))) &= \ap{\deltai}(\push_n(x)\concat
                                     \ap{\inn_{n+1}}(\beta_n(x)\inv))\\
                                   &= \deltai^{\push_n}(x)\concat\ap{\lambda x.\push_{n+1}(x)\concat\ap{\inn_{n+2}}(\beta_{n+1}(x)\inv)}(\beta_n(x)\inv).
\end{align*}
The $\ap{\push_{n+1}}(\beta_n(x)\inv)$ part cancels with the naturality square of $\push_{n+1}$ on
$\beta_n(x)$ used in $\deltai^{\push_n}(x)$ and the remaining part
$\ap{\ap{\inn_{n+2}}(\beta_{n+1}(-)\inv)}(\beta_n(x)\inv)$ is the naturality square of $\beta_{n+1}$
on $\beta_n(x)$. Therefore $\etai(\inn_n(x))$ fits in the three-dimensional diagram

\begin{equation}
\begin{tikzcd}[sdiag,sep=6em,row sep=normal]
  \bullet \arrow[r,"{\ap{\inn_{n+2}}(\ap{\alpha_{n+1}(\star_A,-)}(\beta_n(x)))}"]
  \arrow[dd,"{\gammai(\star_A,\inn_n(x))}"' description] \arrow[dd,bend right=90,looseness=2,"\idpS"']& \bullet
  \arrow[dd,"{\ap{\inn_{n+2}}(\gamma_n(\star_A,x))}" description] \arrow[rrd,"{\ap{\inn_{n+2}}(\beta_{n+1}(i_n(x)))}"]\\
  & & & \bullet\\
  \bullet \arrow[r,"{\ap{\inn_{n+2}}(\beta_{n+1}(\alpha_n(\star_A,x)))}"'] & \bullet \arrow[rru,"\ap{\inn_{n+2}}(\ap{i_{n+1}}(\beta_n(x)))"']
\end{tikzcd}
\label{diagetai}
\end{equation}
where the half-disc on the left is $\etai(\inn_n(x))$, the square in the middle is square
\ref{eq:gammai}, the triangle on the right is the application of $\inn_{n+2}$ to $\eta_n(x)$ and the
outer diagram is the application of $\inn_{n+2}$ to the naturality square of $\beta_{n+1}$ on
$\beta_n(x)$, which is
\[
\begin{tikzcd}[sdiag,sep=huge]
  \bullet \arrow[r,"{\ap{\alpha_{n+1}(\star_A,-)}(\beta_n(x))}"]
  \arrow[d,"{\beta_{n+1}(\alpha_n(\star_A,x))}"'] & \bullet \arrow[d,"{\beta_{n+1}(i_n(x))}"]\\
  \bullet \arrow[r,"{\ap{i_{n+1}}(\beta_n(x))}"'] & \bullet
\end{tikzcd}
\]

\paragraph{The two maps}

We can now define the maps back and forth by
\begin{align*}
  \tof &: \JiA \to JA, & \from &: JA \to \JiA,\\
  \tof(\inn_n(x)) &\defeq \innJ_n(x), & \from(\varepsilon) &\defeq \epsiloni,\\
  \ap{\tof}{(\push_n(x))} &\defeq \pushJ_n(x), & \from(\alpha(a,x)) &\defeq \alphai(a,\from(x)),\\
  & & \ap{\from}(\delta(x)) &\defeq \deltai(\from(x)).
\end{align*}

\paragraph{First composite}

Let’s prove first that $\from\circ\tof = \id_{\JiA}$.

By induction on the argument, it is enough to show that for every $n:\N$ and $x:J_nA$, we have
$\from(\innJ_n(x))=\inn_n(x)$ and $\ap\from(\pushJ_n(x))=\push_n(x)$ (in the appropriate dependent
path type). We proceed by induction on $n$, and then by induction on $x$.
\begin{itemize}
\item For $0$ and $\epsilon$, we have
  \begin{align*}
    \from(\innJ_0(\epsilon)) &= \from(\varepsilon)\\
                             &= \epsiloni\\
    &= \inn_0(\epsilon).
  \end{align*}
\item For $1$ and $a:A$, we have
  \begin{align*}
    \from(\innJ_1(a)) &= \from(\alpha(a,\varepsilon))\\
                      &= \alphai(a,\from(\varepsilon))\\
                      &= \inn_1(a).
  \end{align*}
\item For $n+2$ and $i_{n+1}(x)$, we have
  \begin{align*}
    \from(\innJ_{n+2}(i_{n+1}(x))) &= \from(\alpha(\star_A,\innJ_{n+1}(x)))\\
                                   &= \alphai(\star_A,\from(\innJ_{n+1}(x)))\\
                                   &= \alphai(\star_A,\inn_{n+1}(x))\quad\text{ by induction hypothesis}\\
                                   &= \inn_{n+2}(\alpha_{n+1}(\star_A,x))\\
                                   &= \inn_{n+2}(i_{n+1}(x))\quad \text{ using }\beta_{n+1}(x).
  \end{align*}
\item For $n+2$ and $\alpha_{n+1}(a,x)$, we have
  \begin{align*}
    \from(\innJ_{n+2}(\alpha_{n+1}(a,x))) &= \from(\alpha(a,\innJ_{n+1}(x)))\\
                                   &= \alphai(a,\from(\innJ_{n+1}(x)))\\
                                   &= \alphai(a,\inn_{n+1}(x))\quad\text{ by induction hypothesis}\\
                                   &= \inn_{n+2}(\alpha_{n+1}(a,x)).
  \end{align*}
\item For $n+2$ and $\beta_{n+1}(x)$, we have
  \begin{align*}
    \ap\from(\ap{\innJ_{n+2}}(\beta_{n+1}(x))) &= \ap\from(\idp{\alpha(\star_A,\innJ_{n+1}(x))})\\
                                               &= \idp{\from(\alpha(\star_A,\innJ_{n+1}(x)))}\\
                                               &= \idp{\alphai(\star_A,\from(\innJ_{n+1}(x)))}\\
                                               &= \idp{\alphai(\star_A,\inn_{n+1}(x))}\quad\text{
                                                 by induction hypothesis}\\
                                               &= \idp{\inn_{n+2}(\alpha_{n+1}(\star_A,x))},
  \end{align*}
  hence it follows from the fact that the square
  \[
  \begin{tikzcd}[sdiag,column sep=huge]
    \bullet \arrow[r,"{\idp{\inn_{n+2}(\alpha_{n+1}(\star_A,x))}}"] \arrow[d,"{\idp{\inn_{n+2}(\alpha_{n+1}(\star_A,x))}}"'] &
    \bullet \arrow[d,"\ap{\inn_{n+2}}(\beta_{n+1}(x))"] \\
    \bullet \arrow[r,"\ap{\inn_{n+2}}(\beta_{n+1}(x))"'] & \bullet
  \end{tikzcd}
  \]
  can be filled. Here the left side is $\ap\from(\ap{\innJ_{n+2}}(\beta_{n+1}(x)))$ as computed
  above, the right side is $\ap{\inn_{n+2}}(\beta_{n+1}(x))$ and the top and bottom side are the
  equalities $\from(\innJ_{n+2}(x))=\inn_{n+2}(x)$ for $x\defeq\alpha_{n+1}(\star_A,x)$ and
  $x\defeq i_{n+1}(x)$ respectively (the endpoints of $\beta_{n+1}(x)$) that we constructed
  above. Note that the paths corresponding to the induction hypothesis do not appear on the diagram
  because they were taken into account in the computation of
  $\ap\from(\ap{\innJ_{n+2}}(\beta_{n+1}(x)))$.
\item For $n+2$ and $\gamma_n(a,x)$, we have
  \begin{align*}
    \ap\from(\ap{\innJ_{n+2}}(\gamma_n(a,x))) &= \ap\from(\gamma(a,\innJ_n(x)))\\
                                              &= \gammai(a,\from(\innJ_n(x)))\\
                                              &= \gammai(a,\inn_n(x))\quad\text{ by induction
                                   hypothesis}
  \end{align*}
  hence it follows from diagram \ref{eq:gammai}.
\item For $n+2$ and $\eta_n(x)$, we have
  \begin{align*}
    \ap\from^2(\ap{\innJ_{n+2}}^2(\eta_n(x))) &= \ap\from^2(\eta(\innJ_n(x))) \\
                                              &= \etai(\from(\innJ_n(x))) \\
                                              &= \etai(\inn_n(x))\quad\text{ by induction hypothesis}
  \end{align*}
  hence there is some three-dimensional diagram to fill and it follows from diagram \ref{diagetai}.
\end{itemize}

We then have to show that for every $n:\N$ and $x:J_nA$, we have an equality between
$\ap\from(\pushJ_n(x))$ and $\push_n(x)$ along the equalities $\from(\innJ_n(x))=\inn_n(x)$ and
$\from(\innJ_{n+1}(i_n(x)))=\inn_{n+1}(i_n(x))$ that we just constructed. We have
\begin{align*}
  \ap\from(\pushJ_n(x)) &= \deltai(\from(\innJ_n(x))) \\
                        &= \deltai(\inn_n(x))\quad\text{ by induction hypothesis}\\
                        &= \push_n(x)\concat\ap{\inn_{n+1}}(\beta_n(x))\inv,
\end{align*}
hence we have a filler of the square
\[
\begin{tikzcd}[sdiag,sep=huge]
  \bullet \arrow[rr,"p"] \arrow[d,"\ap\from(\pushJ_n(x))"'] & & \bullet \arrow[d,"\push_n(x)"]\\
  \bullet \arrow[r,"\ap{\alphai(\star_A,-)}(p)"'] & \bullet & \bullet \arrow[l,"\ap{\inn_{n+1}}(\beta_n(x))"]
\end{tikzcd}
\]
where $p$ is the equality $\from(\innJ_n(x))=\inn_n(x)$, which concludes this part of the proof.

\paragraph{Second composite}

Let’s now prove that $\tof\circ\from = \id_{JA}$.

The idea is very similar, we have to prove that $\tof(\epsiloni)=\varepsilon$ (which is true by
definition), that $\tof(\alphai(a,x))=\alpha(a,\tof(x))$, and that
$\ap\tof(\deltai(x))=\delta(\tof(x))$ along some equalities constructed before. Let’s first do the
case of $\alphai$ by induction on $x$. There are two cases.
\begin{itemize}
\item For an element of the form $\inn_n(x)$, we have
  \begin{align*}
    \tof(\alphai(a,\inn_n(x))) &= \tof(\inn_{n+1}(\alpha_n(a,x)))\\
                               &= \alpha(a,\innJ_n(x))\\
                               &= \alpha(a,\tof(\inn_n(x))),
  \end{align*}
  which is what we wanted.
\item For a path of the form $\push_n(x)$, we have
  \begin{align*}
    \ap\tof(\ap{\alphai(a,-)}(\push_n(x))) &= \ap\tof(\push_{n+1}(\alpha_n(a,x))\concat
                                              \ap{\inn_{n+2}}(\gamma_n(a,x))\inv)\\
                                            &= \delta(\innJ_{n+1}(\alpha_n(a,x)))
                                              \concat\gamma(a,\innJ_n(x))\inv\\
                                            &= \delta(\alpha(a,\innJ_n(x)))
                                              \concat\gamma(a,\innJ_n(x))\inv\\
                                            &= \ap{\alpha(a,-)}(\delta(\innJ_n(x)))\quad\text{ by
                                              definition of $\gamma$}\\
                                            &= \ap{\alpha(a,-)}(\pushJ_n(x))\\
                                            &= \ap{\alpha(a,-)}(\ap\tof(\push_n(x))),
  \end{align*}
  which is again what we wanted.
\end{itemize}

We now prove that the path $\ap\tof(\deltai(x)):\tof(x)=\tof(\alphai(\star_A,x))$ composed with the
path from $\tof(\alphai(\star_A,x))$ to $\alpha(\star_A,\tof(x))$ that we have just constructed is
equal to the path $\delta(\tof(x)):\tof(x)=\alpha(\star_A,\tof(x))$.
\begin{itemize}
\item For an element of the form $\inn_n(x)$, we have
  \begin{align*}
    \ap\tof(\deltai(\inn_n(x))) &= \ap\tof(\push_n(x)\concat\ap{\inn_{n+1}}(\beta_n(x))\inv)\\
                                &= \pushJ_n(x)\\
                                &= \delta(\innJ_n(x))\\
                                &= \delta(\tof(\inn_n(x))),
  \end{align*}
  which proves the result, as the path from $\tof(\alphai(\star_A,\inn_n(x)))$ to
  $\alpha(\star_A,\tof(\inn_n(x)))$ is the constant path.
\item For a path of the form $\push_n(x)$, we have to compare $\ap\tof^2(\ap\deltai(\push_n(x)))$
  and $\ap\delta(\ap\tof(\push_n(x)))$. For the first one, we just apply $\tof$ to diagram
  \ref{eq:deltaipush}. We obtain
  \[
  \begin{tikzcd}[sdiag,sep=5.4em]
    \bullet \arrow[rr,"\pushJ_n(x)"] \arrow[d,"\pushJ_n(x)"'] & & \bullet
    \arrow[d,"{\delta(\alpha(\star_A,\innJ_n(x)))}"]\\
    \bullet \arrow[rru,"\idpS"'] & & \bullet\\
    \bullet \arrow[u,"\idpS"] \arrow[r,"{\delta(\alpha(\star_A,\innJ_n(x)))}"']
    & \bullet \arrow[ru,"\idpS"] & \bullet \arrow[l,"{\gamma(\star_A,\innJ_n(x))}"] \arrow[u,"\idpS"']
  \end{tikzcd}
  \]
  where the bottom right triangle is filled using $\eta(\innJ_n(x))$ and the rest is degenerate.
  On the other hand, we have
  \begin{align*}
    \ap\delta(\ap\tof(\push_n(x))) &= \ap\delta(\pushJ_n(x))\\
    &= \ap\delta(\delta(\innJ_n(x)))
  \end{align*}
  and $\eta(\innJ_n(x))$ is defined from $\ap\delta(\delta(\innJ_n(x)))$ by a coherence
  operation. Therefore some coherence operation proves that $\ap\tof^2(\ap\deltai(\push_n(x)))$ and
  $\ap\delta(\ap\tof(\push_n(x)))$ are equal.
\end{itemize}

This concludes the proof that $\JiA$ is equivalent to $JA$.

\subsection{Equivalence between $JA$ and $\Omega\Susp A$}

We now prove that $JA$ is equivalent to $\Omega\Susp A$. The function $\delta$ shows that the map
$\alpha(\star_A, -)$ is equal to the identity function, hence $\alpha(\star_A,-)$ is an equivalence. Given
that $A$ is $0$-connected it follows that $\alpha(a,-)$ is an equivalence for every $a$. We define
$F:\Susp A\to\Type$ by
\begin{align*}
  F(\north) &\defeq JA,\\
  F(\south) &\defeq JA,\\
  \ap F(\merid(a)) &\defeq \ua(\alpha(a,-)),
\end{align*}
which is a valid definition because $\alpha(a,-)$ is an equivalence.

We now prove that the total space of $F$ is contractible. According to the flattening lemma
(proposition \ref{flatteningpushout}), the total space of $F$ is equivalent to the type
\[T\defeq JA\sqcup^{A\times JA}JA,\]
where the two maps $A\times JA\to JA$ are $\snd$ and $\alpha$ respectively.
We want to construct, for every $x:T$, a path $p(x)$ from $\inl(\varepsilon)$ to $x$.
\begin{itemize}
\item For $\inl(\varepsilon)$, we take the constant path $\idp{\inl(\varepsilon)}$
\item For an element of the form $\inl(\alpha(a,x))$, we take the composition
  \[
  \begin{tikzcd}[sdiag]
    \inl(\alpha(a,x)) \arrow[rrr,"{\push(\star_A,\alpha(a,x))}"] &&& \inr(\alpha(\star_A,\alpha(a,x)))
    \arrow[d,"{\ap\inr(\delta(\alpha(a,x)))}"] \\
    \inl(x) \arrow[d,"p(\inl(x))"'] \arrow[rrr,"{\push(a,x)}"] &&&\inr(\alpha(a,x))\\
    \inl(\varepsilon).
  \end{tikzcd}
  \]
\item For a path of the form $\ap\inl(\delta(x))$, we need to fill the diagram
  \[
  \begin{tikzcd}[sdiag]
    \inl(\alpha(\star_A,x)) \arrow[rrr,"{\push(\star_A,\alpha(\star_A,x))}"]
    \arrow[d,"\ap\inl(\delta(x))"'] &&& \inr(\alpha(\star_A,\alpha(\star_A,x)))
    \arrow[d,"{\ap\inr(\delta(\alpha(\star_A,x)))}"] \\
    \inl(x) \arrow[rrr,"{\push(\star_A,x)}"'] &&& \inr(\alpha(\star_A,x))
  \end{tikzcd}
  \]
  By naturality of $\push(\star_A,-)$ on the path $\delta(x)$, we get a filler of the similar
  diagram which has $\ap\inr(\ap{\alpha(\star_A,-)}(\delta(x)))$ on the right side. Moreover, we
  know that $\ap\inr(\ap{\alpha(\star_A,-)}(\delta(x)))$ and
  $\ap\inr(\delta(\alpha(\star_A,x)))$ are equal via $\eta(x)$, which concludes.
\item For a point of the form $\inr(x)$, we take the composition
  \[
  \begin{tikzcd}[sdiag]
    \inr(x) \arrow[r,"\ap\inr(\delta(x))"]& \inr(\alpha(\star_A,x))&&
    \inl(x) \arrow[rr,"p(\inl(x))"] \arrow[ll,"{\push(\star_A,x)}"'] &&\inl(\varepsilon).
  \end{tikzcd}
  \]
\item Finally for a path of the form $\push(a,x)$, it is enough to notice that the path from
  $\inr(\alpha(a,x))$ to $\inl(\varepsilon)$ constructed as above is equal to the composition
  \[
  \begin{tikzcd}[sdiag]
    \inr(\alpha(a,x)) && \inl(x) \arrow[r,"p(\inl(x))"] \arrow[ll,"{\push(a,x)}"'] &\inl(\varepsilon).
  \end{tikzcd}
  \]
\end{itemize}

This concludes the proof that $T$ is contractible, and by the same argument as for $\Omega\Sn1$ we
obtain that $\Omega\Susp A$ is equivalent to $F(\star_{\Susp A})$, which is equal to $JA$ by definition.

\subsection{Connectivity of the maps $i_n$}

Given two maps $f:X\to A$ and $g:Y\to B$, the \emph{pushout-product} of $f$ and $g$ is the map
\begin{align*}
  f\pp g &: (X\times B)\sqcup^{X\times Y}(A\times Y)\to (A\times B),\\
  (f\pp g)(\inl(x,b)) &\defeq (f(x),b),\\
  (f\pp g)(\inr(a,y)) &\defeq (a,g(y)),\\
  \ap{f\pp g}(\push(x,y)) &\defeq \idp{(f(x),g(y))}.
\end{align*}
We have the commutative square
\[
\begin{tikzcd}
  X\times Y \arrow[rr,"f\times1_Y"] \arrow[dd,"1_X\times g"'] && A\times Y \arrow[dd,"1_A\times g"]
  \arrow[ld,"\inr"'] \\
  & (X\times B)\sqcup^{X\times Y}(A\times Y) \arrow[rd,"f\pp g",dashed]
  \arrow[lu,phantom,"\ulcorner",at start] & \\
  X\times B \arrow[rr,"f\times1_B"] \arrow[ru,"\inl"] && A \times B
\end{tikzcd}
\]

The main result of this subsection is the following.
\begin{proposition}\label{pushoutproduct}
  If $f$ is $m$-connected and $g$ is $n$-connected, then $f\pp g$ is $(m+n+2)$-connected.
\end{proposition}

\begin{proof}
  Let’s first recall the following proposition which is a generalization of proposition
  \ref{inductionconnected}. It is proved in \cite[lemma 8.6.1]{hottbook}.
  \begin{proposition}\label{inductionconnectedtruncated}
    If $f:A \to B$ is $n$-connected and $P:B\to\Type$ is a family of $(n+k)$-types, then the map
    \[\lambda s.s\circ f:\prod_{b:B}P(b) \to \prod_{a:A}P(f(a))\] is $(k-2)$-truncated (in the sense that
    all its fibers are $(k-2)$-truncated).
  \end{proposition}

  We now prove that $f\pp g$ is $(m+n+2)$-connected. We consider $P:A\times B\to\Type$ a family of
  $(n+m+2)$-types together with
  \[k:(u:(X\times B)\sqcup^{X\times Y}(A\times Y)) \to P((f\pp g)(u)).\]
  By splitting $k$ in three parts and currying, we have to prove the following proposition.
  
  \begin{proposition}
    Suppose we have $P : A \to B \to \Type$ a family of $(n+m+2)$-types together with
    \begin{align*}
      u &: (x:X)(b:B)\to P(f(x),b),\\
      v &: (a:A)(y:Y)\to P(a, g(y)),\\
      w &: (x:X)(y:Y)\to u(x,g(y))=_{P(f(x),g(y))}v(f(x),y).
    \end{align*}
    Then there exists a map
    \[h : (a:A)(b:B)\to P(a,b)\] together with homotopies
    \begin{align*}
      p &: (x:X)(b:B)\to h(f(x),b) = u(x,b),\\
      q &: (a:A)(y:Y)\to h(a,g(y)) = v(a,y),\\
      r &: (x:X)(y:Y)\to p(x,g(y))\inv\cdot q(f(x),y) = w(x,y).
    \end{align*}
  \end{proposition}
  
  \begin{proof}
    Let’s define $F:A\to\Type$ by
    \[F(a)\defeq \sum_{k:(b:B) \to P(a,b)}((y:Y) \to k(g(y)) = v(a,y)).\]
    For a given $a:A$, the type $F(a)$ is the fiber of the map
    \[\lambda s.s\circ g:\prod_{b:B}P(a,b) \to \prod_{y:Y}P(a,g(y))\]
    at $v(a,-)$.
    Given that $g$ is $n$-connected and $P$ is a family of $(n+m+2)$-truncated types, proposition
    \ref{inductionconnectedtruncated} shows that $F(a)$ is $m$-truncated.
  
    For every $x:X$ we have an element of $F(f(x))$ given by $(u(x, -), w(x, -))$.  Hence, using the
    fact that $f$ is $m$-connected, there is a map $k:(a:A)\to F(a)$ together with a homotopy
    $\varphi$ between $k\circ f$ and $\lambda x.(u(x,-),w(x,-))$.
    We can now define $h$, $p$, $q$, and $r$ by
    \begin{align*}
      h(a,b) &\defeq \fst(k(a))(b),\\
      p(x,b) &\defeq \fst(\varphi(x))(b),\\
      q(a,y) &\defeq \snd(k(a))(y),\\
      r(x,y) &\defeq \snd(\varphi(x))(y).\qedhere
    \end{align*}
  \end{proof}
  This concludes the proof that $f\pp g$ is $(n+m+2)$-connected.
\end{proof}

We can now compute the connectivity of the maps $i_n$.

\begin{proposition}
  For every $n:\N$, the map $i_n$ is $(n(k+1)+(k-1))$-connected.
\end{proposition}

\begin{proof}
  We proceed by induction on $n$. For $0$, the map $i_0$ is the inclusion of the basepoint of $A$,
  hence $i_0$ is $(k-1)$-connected because $A$ is $k$-connected.

  For $n+1$, the map $f$ in diagram \ref{eq:jn+2a} (the diagram defining $J_{n+2}A$) is the
  pushout-product of $i_n$ and of the map $\Unit\to A$ (which is $(k-1)$-connected). Hence $f$ is
  $((n+1)(k+1)+(k-1))$-connected by proposition \ref{pushoutproduct}. Therefore, by proposition
  \ref{connpushout}, the map $i_{n+1}$ is $((n+1)(k+1)+(k-1))$-connected as well.
\end{proof}

This concludes the James construction.

\section{Whitehead products}

In proposition \ref{whiteheadmap} we give a decomposition of a product of spheres into a pushout of
spheres. This allows us to define Whitehead products, which are used in the next section in the
definition of the natural number $n$ such that $\pi_4(\Sn3)\simeq\Z/n\Z$.
\begin{proposition}\label{whiteheadmap}
  Given $n,m:\N^*$, there is a map $W_{n,m}:\Sn{n+m-1}\to\Sn{n}\vee\Sn{m}$ such that
  \[\Sn n\times\Sn m \simeq 1\sqcup^{\Sn{n+m-1}}(\Sn n\vee\Sn m)\]
  and such that the induced map $\Sn n\vee\Sn m\to\Sn n\times\Sn m$ is $\iwedge_{\Sn n,\Sn m}$.
\end{proposition}
We first prove the following more general version which isn’t more complicated to prove.
\begin{proposition}\label{whiteheadab}
  Given two types $A$ and $B$, there is a map $W_{A,B}:A*B\to\Susp A\vee\Susp B$ such that
  \[\Susp A\times\Susp B \simeq 1\sqcup^{A*B}(\Susp A\vee\Susp B)\]
  and such that the induced map $\Susp A\vee\Susp B\to\Susp A\times\Susp B$ is $\iwedge_{\Susp
    A,\Susp B}$.
\end{proposition}
\begin{proof}
  We use the $3\times3$-lemma with the diagram
  \[
  \begin{tikzcd}
    \Susp A & B \arrow[l,"\north"'] \arrow[r] \arrow[ld,"\alpha"',Rightarrow,shorten >= 1.5ex,
    shorten <= 1.5ex] & 1 \\
    B \arrow[u,"\south"] \arrow[d,"\id"'] & A\times B \arrow[l,"\snd"] \arrow[r,"\fst"']
    \arrow[u,"\snd"']
    \arrow[d,"\snd"] & A \arrow[u] \arrow[d] \\
    B & B \arrow[l,"\id"] \arrow[r] & 1
  \end{tikzcd}
  \]
  where $\alpha:A\times B\to\north=_{\Susp A}\south$ is defined by
  $\alpha(x,y)\defeq\merid(x)$.

  The pushout of the top row is equivalent to $\Susp A\vee\Susp B$, the pushout of the middle row
  is equivalent to the join $A*B$ and the pushout of the bottom row is contractible, so the pushout
  of the pushouts of the rows is equivalent to $1\sqcup^{A*B}(\Susp A\vee\Susp B)$ for the map
  $A*B\to\Susp A\vee\Susp B$ defined by
  \begin{align*}
    W_{A,B} &: A*B \to \Susp A\vee\Susp B,\\
    W_{A,B}(\inl(a)) &\defeq \inr(\north),\\
    W_{A,B}(\inr(b)) &\defeq \inl(\north),\\
    \ap{W_{A,B}}(\push(a,b)) &\defeq \ap\inr(\varphi_B(b))\concat\push(\ttt)\concat\ap\inl(\varphi_A(a)).
  \end{align*}

  The pushouts of the left and of the right columns are both equivalent to $\Susp A$, and the
  pushout of the middle column is equivalent to $\Susp A\times B$. Moreover, the horizontal map on
  the left between $\Susp A\times B$ and $\Susp A$ is equal to $\fst$, as can be proved by
  induction using the definition of $\alpha$. The horizontal map on the right is also equal to
  $\fst$. Hence the pushout of the pushout of the columns is equivalent to $\Susp A\times\Susp
  B$. Therefore we have
  \[\Susp A\times\Susp B \simeq 1\sqcup^{A*B}(\Susp A\vee\Susp B)\]
  and it can be checked that the induced map $\Susp A\vee\Susp B\to\Susp A\times\Susp B$ is equal to
  $i^\vee_{\Susp A,\Susp B}$.
\end{proof}

\begin{proof}[Proof of proposition \ref{whiteheadmap}]
  We apply proposition \ref{whiteheadab} to $A\defeq\Sn{n-1}$ and $B\defeq\Sn{m-1}$, and we obtain
  \[\Sn n\times\Sn m \simeq 1\sqcup^{\Sn{n-1}*\Sn{m-1}}(\Sn n\vee\Sn m).\]
  Moreover, we have $\Sn{n-1}*\Sn{m-1}\simeq\Sn{n+m-1}$ by proposition \ref{joinspheres}, which
  concludes.
\end{proof}

This allows us to define the following operation on homotopy groups.
\begin{definition}
  Given a pointed type $X$ and two positive integers $n$ and $m$, the \emph{Whitehead product} is
  the function
  \[[-,-] : \pi_n(X) \times \pi_m(X) \to \pi_{n+m-1}(X)\]
  defined by composition with $W_{n,m}$ when representing elements of homotopy groups as maps from
  the spheres.
\end{definition}

\section{Application to homotopy groups of spheres}

The sphere $\Sn n$ is $(n-1)$-connected, therefore by the Freudenthal suspension theorem
\ref{freud}, the map $\varphi_{\Sn n}:\Sn n\to\Omega\Sn{n+1}$ is $(2n-2)$-connected. On homotopy
groups it gives the following result.

\begin{proposition}
  For $k,n:\N$, the map $\pi_{n+k}(\Sn n)\to\pi_{n+k+1}(\Sn{n+1})$ is an isomorphism if $n\ge k+2$
  and surjective if $n=k+1$.
\end{proposition}
On the table \ref{hgs} of homotopy groups of spheres this corresponds to the fact that the diagonals
$\pi_{n+k}(\Sn n)$ (for a fixed $k$) are constant above the zigzag line.  In particular, for $k=0$
and $k=1$ we have the following result. I already obtained the result for $\pi_n(\Sn n)$ in
collaboration with Dan Licata in \cite{pinsn} but with a different technique.
\begin{cor}\label{pinsn}
  For $n\ge2$ we have $\pi_n(\Sn n)\simeq\pi_2(\Sn2)\simeq\Z$.
  For $n\ge3$ we have $\pi_{n+1}(\Sn n)\simeq\pi_4(\Sn3)$ and the map $\pi_3(\Sn2)\to\pi_4(\Sn3)$ is
  surjective.
\end{cor}

Note that as we are working constructively this does \emph{not} imply that $\pi_4(\Sn3)$ is of the
form $\Z/n\Z$ for some $n:\Z$. Indeed, it cannot be proved constructively that every subgroup of
$\Z$ is of the form $n\Z$ for some $n:\Z$, as there is no way to compute this $n$ in general. In
this case, however, we can use the James construction to give an explicit definition of the kernel of
that map. We will need the Blakers--Massey theorem (\cite{blakersmassey}), which is stated as
\cite[theorem 8.10.2]{hottbook} and which is proved in Agda in \cite{blakersmasseyagda}.

\begin{proposition}[Blakers--Massey theorem]
  Given two maps $f:C\to A$ and $g:C\to B$, we consider $D\defeq A\sqcup^CB$,
  \[E\defeq\sum_{a:A}\sum_{b:B}(\inl(a)=_D\inr(b)),\]
  and $h : C\to E$ defined by $h(c) \defeq (f(c),g(c),\push(c))$.
  \[
  \begin{tikzcd}
    C \arrow[rrd,bend left,"g"] \arrow[rdd,bend right,"f"'] \arrow[rd,dashed,"h"] & & \\
    & E \arrow[r] \arrow[d] \arrow[rd,phantom,"\lrcorner",at start] & B \arrow[d,"\inr"]\\
    & A \arrow[r,"\inl"'] & D
  \end{tikzcd}
  \]
  If $f$ is $n$-connected and $g$ is $m$-connected, then $h$ is $(n+m)$-connected.
\end{proposition}

We now prove the following proposition.
\begin{proposition}\label{kerneljames}
  For $n\ge2$, the map $\pi_{2n-1}(\Sn n)\to \pi_{2n}(\Sn{n+1})$ induced by $\varphi_{\Sn n}$ is
  surjective and its kernel is generated by the Whitehead product $[i_n,i_n]$, where $i_n$ is the
  generator of $\pi_n(\Sn n)$.
\end{proposition}

\begin{proof}
  Applying corollary \ref{corjames} to $\Sn n$ which is $(n-1)$-connected, we get a
  $(3n-2)$-connected map from $J_2(\Sn n)$ to $\Omega\Sn{n+1}$. In particular, given that
  $2n-1<3n-2$, it means that
  \[\pi_{2n-1}(J_2(\Sn n))\simeq\pi_{2n-1}(\Omega\Sn{n+1})\simeq\pi_{2n}(\Sn{n+1}),\] so we now
  study the map $\pi_{2n-1}(\Sn n)\to \pi_{2n-1}(J_2(\Sn n))$.  We know from the James construction
  that \[J_2(\Sn n)\simeq(\Sn n\times\Sn n)\sqcup^{\Sn n\vee\Sn n}\Sn n,\]
  hence using the decomposition of $\Sn n\times\Sn n$ given in proposition \ref{whiteheadmap}, we
  get
  \[J_2(\Sn n)\simeq(\Unit\sqcup^{\Sn{2n-1}}(\Sn n\vee\Sn n))\sqcup^{\Sn n\vee\Sn n}\Sn n\]
  where the map from $\Sn n\vee\Sn n$ to the pushout on the left is $\inr$ (i.e.\ it’s the identity
  on the second component). Using proposition \ref{pushoutangle}, this reduces to
  \[J_2(\Sn n)\simeq\Unit\sqcup^{\Sn{2n-1}}\Sn n,\]
  where the map $\Sn{2n-1}\to\Sn n$ is the Whitehead map $W_{n,n}:\Sn{2n-1}\to\Sn n\vee\Sn n$
  composed with the folding map $\fold_{\Sn n}:\Sn n\vee\Sn n\to\Sn n$.

  We now take the fiber $P$ of the map $\Sn n\to J_2(\Sn n)$, which is the pullback of the two
  maps from $\Sn n$ and $\Unit$ to $J_2(\Sn n)$
  \[
  \begin{tikzcd}
    \Sn{2n-1} \arrow[rrr,"{\fold_{\Sn n}\circ W_{n,n}}"] \arrow[ddd] \arrow[rd, dashed] &&& \Sn n \arrow[ddd] \\
    & P \arrow[rru] \arrow[ddl] \arrow[rrdd,phantom,"\lrcorner",at start] && \\ \\
    \Unit \arrow[rrr] &&& J_2(\Sn n) \arrow[lluu,phantom,"\ulcorner",at start]
  \end{tikzcd}
  \]
  The relevant part of the long exact sequence of homotopy groups for $P\to \Sn n\to J_2(\Sn n)$ is
  \[
  \begin{tikzcd}
    \pi_{2n-1}(P) \arrow[r]& \pi_{2n-1}(\Sn n) \arrow[r]& \pi_{2n-1}(J_2(\Sn n)) \arrow[r]& \pi_{2n-2}(P).
  \end{tikzcd}
  \]
  The map from $\Sn{2n-1}$ to $\Unit$ is $(2n-2)$-connected and the map from $\Sn{2n-1}$ to $\Sn n$
  is $(n-2)$-connected (indeed, every map between two $(n-1)$-connected types is $(n-2)$-connected),
  hence using the Blakers-Massey theorem, the dashed map from $\Sn{2n-1}$ to $P$ is
  $(3n-4)$-connected. Given that $2n-2\le3n-4$, it follows that
  $\pi_{2n-2}(P)\simeq\pi_{2n-2}(\Sn{2n-1})\simeq0$.

  Hence $\pi_{2n-1}(J_2(\Sn n))$ is the quotient of $\pi_{2n-1}(\Sn n)$ by the image of the map
  $\pi_{2n-1}(P)\to\pi_{2n-1}(\Sn n)$. But the dashed map is surjective on $\pi_{2n-1}$, so it’s the
  same as the image of the map $\pi_{2n-1}(\Sn{2n-1})\to\pi_{2n-1}(\Sn n)$, which is generated by
  $[i_n,i_n]$, by definition of the Whitehead product.

  Therefore, the kernel of the map $\pi_{2n-1}(\Sn n)\to\pi_{2n}(\Sn{n+1})$ is generated by $[i_n,i_n]$.
\end{proof}
In particular, applying this result to $n=2$ and using the fact that $\pi_3(\Sn2)\simeq\Z$, we get the
following corollary.
\begin{cor}\label{firstpi4s3}
  We have
  \[\pi_4(\Sn3)\simeq\Z/n\Z,\]
  where $n$ is the absolute value of the image of $[i_2,i_2]$ by the equivalence
  $\pi_3(\Sn2)\stackrel\sim\longrightarrow\Z$.
\end{cor}

This result is quite remarkable in that even though it is a constructive proof, it is not at all
obvious how to actually compute this $n$. At the time of writing, we still haven’t managed to
extract its value from its definition. A complete and concise definition of this number $n$ is
presented in appendix \ref{ch:defn}, for the benefit of someone wanting to implement it in a
prospective proof assistant. In the rest of this thesis, we give a mathematical proof in homotopy
type theory that $n=2$.


\chapter{Smash products of spheres}\label{ch:smash}

In this chapter we prove various properties of the smash product, in particular concerning its
symmetric monoidal structure, the smash product of spheres, and the connectedness of the smash
product of two maps. The results of this chapter are essential for the definition and the basic
properties of the multiplicative structure on cohomology, which we will see in the next chapter.

We recall that the \emph{smash product} of two pointed types $A$ and $B$ is the pushout of the
span
\[
\begin{tikzcd}
  \Unit & A\vee B \arrow[l] \arrow[r,"\iwedge_{A,B}"] & A\times B.
\end{tikzcd}
\]
The smash product $A\wedge B$ is pointed by $\star_{A\wedge B}\defeq \inl(\ttt)$ and there is a map
\begin{align*}
  \proj &: A \to B \to A\wedge B,\\
  \proj(a,b) &\defeq \inr((a,b))
\end{align*}
together with two equalities
\begin{align*}
  \projr &: (a : A) \to \proj(a,\star_B) = \star_{A\wedge B}, & \projl &: (b : B) \to
                                                                        \proj(\star_A,b) =
                                                                        \star_{A\wedge B},\\
  \projr(a) &\defeq \push(\inl(a)), & \projl(b) &\defeq \push(\inr(b))
\end{align*}
and a $2$-dimensional path
\begin{align*}
  \projlr &: \projr(\star_A) = \projl(\star_B),\\
  \projlr &\defeq \ap\push(\push(\ttt)).
\end{align*}
In particular, we can also see $A\wedge B$ as the higher inductive type with constructors
$\star_{A\wedge B}$, $\proj$, $\projr$, $\projl$ and $\projlr$.

\section{The monoidal structure of the smash product}\label{sec:monoidal}

Let’s first define the notion of $1$-coherent symmetric monoidal product on
pointed types.

\begin{definition}\label{defmonoidal}
  A \emph{$1$-coherent symmetric monoidal product on pointed types} is a binary operation $\otimes$
  between pointed types equipped with
  \begin{itemize}
  \item a pointed map $f\otimes g:A\otimes B\to A'\otimes B'$ given any two pointed maps $f:A\to A'$ and
    $g:B\to B'$,
  \item a pointed equality $\id_A\otimes\id_B=\id_{A\otimes B}$ given any two pointed types $A$ and $B$,
  \item a pointed equality $(f'\circ f)\otimes(g'\circ g)=(f'\otimes g')\circ(f\otimes g)$ given any four pointed
    maps $f:A\to A'$, $f':A'\to A''$, $g:B\to B'$ and $g':B'\to B''$,
  \item a pointed equivalence $\alpha_{A,B,C}:(A\otimes B)\otimes C\toeq A\otimes(B\otimes C)$ given
    any three pointed types $A$, $B$ and $C$, and a pointed filler of the square
    \[
    \begin{tikzcd}
      (A\otimes B)\otimes C \arrow[r,"\alpha_{A,B,C}"] \arrow[d,"(f\otimes g)\otimes h"'] &
      A\otimes(B\otimes C) \arrow[d,"f\otimes(g\otimes h)"]\\
      (A'\otimes B')\otimes C'\arrow[r,"\alpha_{A',B',C'}"'] & A'\otimes(B'\otimes C')
    \end{tikzcd}
    \]
    given any three pointed maps $f:A\to A'$, $g:B\to B'$ and $h:C\to C'$,
  \item a pointed filler of the pentagon
    \[
    \begin{tikzpicture}[commutative diagrams/every diagram]
      \node (P0) at (90:8em) {$((A\otimes B)\otimes C)\otimes D$};
      \node (P1) at (162:8em) {$(A\otimes(B\otimes C))\otimes D$};
      \node (P2) at (224:8em) {$A\otimes((B\otimes C)\otimes D)$};
      \node (P3) at (316:8em) {$A\otimes(B\otimes(C\otimes D))$};
      \node (P4) at (18:8em) {$(A\otimes B)\otimes(C\otimes D)$};

      \path
      (P0.210) edge[commutative diagrams/.cd, every arrow, every label] node[swap] {$\alpha_{A,B,C}\otimes\id_D$} (P1)
      (P1) edge[commutative diagrams/.cd, every arrow, every label] node[swap] {$\alpha_{A,B\otimes C,D}$} (P2)
      (P2) edge[commutative diagrams/.cd, every arrow, every label] node[swap,yshift=-1ex] {$\id_A\otimes\alpha_{B,C,D}$} (P3)
      (P0.330) edge[commutative diagrams/.cd, every arrow, every label] node {$\alpha_{A\otimes B,C,D}$} (P4)
      (P4) edge[commutative diagrams/.cd, every arrow, every label] node {$\alpha_{A,B,C\otimes D}$} (P3);
    \end{tikzpicture}
    \]
    given any four pointed types $A$, $B$, $C$ and $D$, 
  \item a pointed map $\sigma_{A,B}:A\otimes B\to B\otimes A$ (which will turn out to be an
    equivalence) given any two pointed types $A$ and $B$, and a pointed filler of the square
    \[
    \begin{tikzcd}
      A\otimes B \arrow[r,"\sigma_{A,B}"] \arrow[d,"f\otimes g"'] & B \otimes A \arrow[d,"g\otimes
      f"] \\
      A'\otimes B' \arrow[r,"\sigma_{A',B'}"] & B'\otimes A'
    \end{tikzcd}
    \]
    given any two pointed maps $f:A\to A'$ and $g:B\to B'$,
  \item a pointed filler of the hexagon
    \[
    \begin{tikzpicture}[commutative diagrams/every diagram]
      \node (P0) at (0:8em) {$(B\otimes C)\otimes A$};
      \node (P1) at (60:8em) {$A\otimes(B\otimes C)$};
      \node (P2) at (120:8em) {$(A\otimes B)\otimes C$};
      \node (P3) at (180:8em) {$(B\otimes A)\otimes C$};
      \node (P4) at (240:8em) {$B\otimes(A\otimes C)$};
      \node (P5) at (300:8em) {$B\otimes(C\otimes A)$};

      \path[commutative diagrams/.cd, every arrow, every label]
      (P2) edge node[yshift=1ex] {$\alpha_{A,B,C}$} (P1)
      (P1) edge node {$\sigma_{A,B\otimes C}$} (P0)
      (P0) edge node {$\alpha_{B,C,A}$} (P5)
      (P2) edge node[swap] {$\sigma_{A,B}\otimes \id_C$} (P3)
      (P3) edge node[swap] {$\alpha_{B,A,C}$} (P4)
      (P4) edge node[swap,yshift=-1ex] {$\id_B\otimes\sigma_{A,C}$} (P5);
    \end{tikzpicture}
    \]
    given any three pointed types $A$, $B$ and $C$,
  \item a pointed filler of the triangle
    \[
    \begin{tikzcd}
      A\otimes B \arrow[rr,"\id_{A\otimes B}"] \arrow[rd,"\sigma_{A,B}"'] && A\otimes B\\
      &B\otimes A\arrow[ru,"\sigma_{B,A}"']&
    \end{tikzcd}
    \]
    given any two pointed types $A$ and $B$,
  \item a pointed type $I:\Type$,
  \item a pointed equivalence $\lambda_A:I\otimes A\toeq A$ given any pointed type $A$ and a pointed
    filler of the square
    \[
    \begin{tikzcd}
      I\otimes A \arrow[r,"\lambda_A"] \arrow[d,"\id_I\otimes f"'] & A \arrow[d,"f"]\\
      I\otimes A' \arrow[r,"\lambda_{A'}"] & A'
    \end{tikzcd}
    \]
    given any pointed map $f:A\to A'$,
  \item a pointed filler of the triangle
    \[
    \begin{tikzcd}
      (I\otimes A)\otimes B \arrow[rr,"\alpha_{I,A,B}"] \arrow[rd,"\lambda_A\otimes\id_B"'] &&
      I\otimes(A\otimes B) \arrow[ld,"\lambda_{A\otimes B}"]\\
      &A\otimes B&
    \end{tikzcd}
    \]
    given any two pointed types $A$ and $B$.
  \end{itemize}
\end{definition}

We call this a \emph{$1$-coherent} symmetric monoidal structure because we do not ask the fillers of
the diagrams to satisfy any further coherence condition. It’s an open question to give a definition
in homotopy type theory of the notion of fully coherent (or even only $n$-coherent) symmetric
monoidal structure, but here we only need the $1$-coherent structure of the smash product. The
following result is the main result of this section even though we essentially admit it.

\begin{proposition}\label{smashismonoidal}
  The smash product is a $1$-coherent symmetric monoidal product on pointed types.
\end{proposition}

\begin{proof}[Sketch of proof]
  Putting the unit aside for a moment, we have to define six functions of the form
  \[(x:A\wedge B)\to P(x),\]
  four of the form
  \[(x:(A\wedge B)\wedge C)\to P(x),\]
  two of the form
  \[(x:A\wedge(B\wedge C))\to P(x),\]
  and one of the form
  \[(x:((A\wedge B)\wedge C)\wedge D)\to P(x),\]
  where each time $P(x)$ is either a smash product like $B\wedge A$ or $A\wedge(B\wedge C)$, or an
  equality in a smash product between combinations of some of those functions.
  
  The idea is that the smash product $A\wedge B$ can be seen as the product $A\times B$ where all
  elements of the form $(a,\star_B)$ and $(\star_A,b)$ have been identified together. Therefore, in
  order to define a map out of $A\wedge B$ it should be enough to define it on elements of the form
  $\proj(a,b)$ in such a way that the image of elements of the form $\proj(a,\star_B)$ and
  $\proj(\star_A,b)$ are identified to the basepoint in the codomain. But it is not enough to simply
  give paths from the images of $\proj(a,\star_B)$ and $\proj(\star_A,b)$ to the basepoint of the
  codomain, because we also need to check that the two induced paths from the image of
  $\proj(\star_A,\star_B)$ to the basepoint are equal, and this is often the most technical
  part. Intuitively, however, the idea (that we do not make precise here) is the following. In the
  type $A\wedge B$, the “subtype” generated (using coherence operations) by $\star_{A\wedge B}$, all
  $\proj(a,\star_B)$, all $\proj(\star_A,b)$, all $\projr(a)$, all $\projl(b)$, and $\projlr$ is
  “contractible” in the sense that there is a path between any two points, a $2$-dimensional path
  between any two parallel paths, and so on. Similarly, the subtype of $(A\wedge B)\wedge C$
  generated by all the $\proj(\proj(a,b),c)$ where either $a$, $b$, or $c$ is the basepoint and by
  all the combinations of $\projr$, $\projl$, and $\projlr$ that one can write is contractible in
  the same sense. Therefore, in order to define (say) a map from $(A\wedge B)\wedge C$ to
  $A\wedge(B\wedge C)$, it is enough to define it on elements of the form $\proj(\proj(a,b),c)$ in
  such a way that when either $a$, $b$, or $c$ is the basepoint, the image is in the contractible
  part of $A\wedge(B\wedge C)$. The contractibility takes care of all the other constructors.

  For instance, given two pointed maps $f:A\to A'$ and $g:B\to B'$, we define their smash product by
  \begin{align*}
    f\wedge g &: A\wedge B\to A'\wedge B',\\
    (f\wedge g)(\proj(x,y)) &\defeq \proj(f(x),g(y)).
  \end{align*}
  The images of the other constructors is as follows. For $\star_{A\wedge B}$ we just take
  $\star_{A'\wedge B'}$. For $\projr(x)$ we need to give a path from $\proj(f(x),g(\star_B))$ to
  $\star_{A'\wedge B'}$ and the most obvious choice is
  \[\ap{\proj(f(x),-)}(\star_g) \concat \projr(f(x))\]
  and similarly for $\projl(y)$. For $\projlr$ we need to fill the square
  \[
  \begin{tikzcd}[sdiag]
    \proj(f(\star_A),g(\star_B)) \arrow[r] \arrow[d] & \proj(f(\star_A),\star_{B'}) \arrow[d] \\
    \proj(\star_{A'},g(\star_B)) \arrow[r] & \star_{A'\wedge B'}
  \end{tikzcd}
  \]
  and this is done by combining the two squares
  \[
  \begin{tikzcd}[sdiag]
    \proj(f(\star_A),g(\star_B)) \arrow[r] \arrow[d] & \proj(f(\star_A),\star_{B'}) \arrow[d] \\
    \proj(\star_{A'},g(\star_B)) \arrow[r] & \proj(\star_{A'},\star_{B'})
  \end{tikzcd}
  \begin{tikzcd}[sdiag]
    \proj(\star_{A'},\star_{B'}) \arrow[r] \arrow[d] \arrow[rd,bend right=10,"\projl(\star_{B'})"' description]
    \arrow[rd,bend left=10,"\projr(\star_{A'})" description] & \proj(f(\star_A),\star_{B'}) \arrow[d] \\
    \proj(\star_{A'},g(\star_B)) \arrow[r] & \star_{A'\wedge B'}
  \end{tikzcd}
  \]
  where the first one and the two triangles of the second one are filled by naturality squares and
  the oval shape in the middle of the second one is filled by $\projlr$.

  Similarly, we define associativity by
  \begin{align*}
    \alpha_{A,B,C} &: (A\wedge B)\wedge C \to A\wedge(B\wedge C),\\
    \alpha_{A,B,C}(\proj(\proj(a,b),c)) &\defeq \proj(a,\proj(b,c)),
  \end{align*}
  and all generating paths and higher-dimensional paths on the left are sent to appropriate paths
  and higher-dimensional paths on the right. It works because the image of $\proj(\proj(a,b),c)$
  where either $a$, $b$ or $c$ is the basepoint is in the contractible part of $A\wedge(B\wedge C)$.
  The commutator $\sigma_{A,B}$ is defined analogously by
  \begin{align*}
    \sigma_{A,B} &: A\wedge B \to B\wedge A,\\
    \sigma_{A,B}(\proj(a,b)) &\defeq \proj(b,a).
  \end{align*}

  For functoriality, naturality of associativity, naturality of commutativity, the pentagon, the
  hexagon, and the triangle, the idea is that we can easily check by hand that it holds by definition
  for elements of the form $\proj(a,b)$ (or $\proj(\proj(a,b),c)$ or
  $\proj(\proj(\proj(a,b),c),d)$), therefore by a similar argument of contractibility, they hold
  also for an arbitrary element of type $A\wedge B$ (or $(A\wedge B)\wedge C$ or
  $((A\wedge B)\wedge C)\wedge D$).

  Finally the unit is $I\defeq\Bool$, the unitor and its inverse are defined by
  \begin{align*}
    \lambda_A &: \Bool\wedge A \to A, & \lambda\inv_A &: A\to \Bool\wedge A,\\
    \lambda_A(\true, a) &\defeq \star_A, & \lambda\inv_A(a) &\defeq (\false,a),\\
    \lambda_A(\false, a) &\defeq a,
  \end{align*}
  and a similar reasoning applies for the remaining diagrams.
\end{proof}

\section{Smash product of spheres}

We now study the structure of the smash product on spheres. We first notice that taking the smash
product with $\Sn1$ is equivalent to taking the suspension.

\begin{proposition}\label{suspsmashs1}
  For every pointed type $A$, there is an equivalence
  \[\Sn1\wedge A \simeq \Susp A.\]
\end{proposition}

\begin{proof}
  We define the map $f:\Sn1\wedge A\to\Susp A$ by
  \begin{align*}
    f &: \Sn1\wedge A \to \Susp A,\\
    f(\proj(x,a)) &\defeq \ft(x,a),
  \end{align*}
  where
  \begin{align*}
    \ft &: \Sn1\to A\to\Susp A,\\
    \ft(\base,a) &\defeq \north,\\
    \ap{\ft(-,a)}(\lloop) &\defeq \varphi_A(a).
  \end{align*}
  We also need to give the images by $f$ of the other constructors of $\Sn1\wedge A$. We send the
  basepoint $\star_{\Sn1\wedge A}$ to $\north$. For $\projl(a)$ we can take $\idp{\north}$ because
  $f(\proj(\base,a))$ is equal to $\north$ by definition. For $\projr(x)$ we need to construct a
  path from $\ft(x,\star_A)$ to $\north$ for every $x:\Sn1$ and we proceed by induction on $x$. For
  $\base$ they are equal by definition so we take $\idp\north$, and for $\lloop$ we use the
  fact that $\varphi_A(\star_A)=\idp{\north}$. Finally, we have to prove that the two paths from
  $\ft(\base,\star_A)$ to $\north$ that we just constructed are equal, but it is true because they
  are both equal to $\idp{\north}$.

  The inverse function is defined by
  \begin{align*}
    g &: \Susp A\to\Sn1\wedge A,\\
    g(\north) &\defeq \star_{\Sn1\wedge A},\\
    g(\south) &\defeq \star_{\Sn1\wedge A},\\
    \ap g(\merid(a)) &\defeq \projl(a)\inv\concat\ap{\proj(-,a)}(\lloop)\concat\projl(a).
  \end{align*}

  The composition $f\circ g:\Susp A\to\Susp A$ sends both $\north$ and $\south$ to $\north$ and
  sends $\merid(a)$ to $\varphi_A(a)$. Therefore we have $f(g(x))=x$ via $\idp\north$ for
  $x=\north$, $\merid(\star_A)$ for $x=\south$, and the filler of the square
  \[
  \begin{tikzcd}[sdiag]
    \bullet \arrow[r,"\varphi_A(a)"] \arrow[d,"\idp\north"'] & \bullet \arrow[d,"\merid(\star_A)"]
  \\
  \bullet \arrow[r,"\merid(a)"'] & \bullet
  \end{tikzcd}
  \]
  given by the definition of $\varphi_A(a)$ for $\merid(a)$.

  For the composition $g\circ f:\Sn1\wedge A\to\Sn1\wedge A$, we first prove that
  $g(\ft(x,a))=\proj(x,a)$ for all $x:\Sn1$ and $a:A$, by induction on $x$. For $\base$ we have
  $g(\ft(\base,a))=g(\north)=\star_{\Sn1\wedge A}$ by definition, which is equal to $\proj(\base,a)$
  along $\projl(a)$, and for $\lloop$ we have
  \begin{align*}
  \ap g(\ap{\ft(-,a)}(\lloop)) &= \ap g(\varphi_A(a)) \\
                               &= \projl(a)\inv\concat\ap{\proj(-,a)}(\lloop)\concat\projl(a)\\
    &\quad\concat\projl(\star_A)\inv\concat\ap{\proj(-,\star_A)}(\lloop)\inv\concat\projl(\star_A).
  \end{align*}
  Moreover the naturality square of $\projr$ on $\lloop$ is
  \[
  \begin{tikzcd}[sdiag]
    \bullet \arrow[r,"{\ap{\proj(-,\star_A)}(\lloop)}"] \arrow[d,"\projr(\base)"'] & \bullet
    \arrow[d,"\projr(\base)"] \\
    \bullet \arrow[r,"\idp{\star_{\Sn1\wedge A}}"'] & \bullet
  \end{tikzcd}
  \]
  which, in conjunction with the equality $\projlr:\projr(\base)=\projl(\star_A)$, shows that
  \[\ap g(\ap{\ft(-,a)}(\lloop)) = \projl(a)\inv\concat\ap{\proj(-,a)}(\lloop)\concat\projl(a),\]
  which is what we wanted. Now we have to prove that $g(f(x))=x$ for the other constructors of
  $\Sn1\wedge A$. For the basepoint it’s true by definition. The path $\projl(a)$ is sent to
  $\idp{\star_{\Sn1\wedge A}}$ which is what we want because the equality
  $g(f(\proj(\base,a)))=\proj(\base,a)$ is already given by $\projl(a)$. For $\projr(x)$, we have to
  prove that $\ap g(\ap f(\projr(x)))$ is equal to the composition of the path
  $g(f(\proj(x,\star_A)))=\proj(x,\star_A)$ constructed above with the path $\projr(x)$, and we
  proceed by induction on $x$. For $\base$, the image of $\projr(\base)$ by $g\circ f$ is the
  constant path, and the equality $g(f(\proj(x,\star_A)))=\proj(x,\star_A)$ is
  $\projl(\star_A)\inv$, therefore the result follows from $\projlr$. For $\lloop$, the image by
  $g\circ f$ is a coherence operation (coming from the equality $\varphi_A(\star_A)=\idp{\north}$)
  and the equality $g(f(\proj(x,\star_A)))=\proj(x,\star_A)$ is essentially $\ap\projl(\lloop)\inv$,
  therefore it follows. Finally, for $\projlr$ it follows automatically as well.
\end{proof}

The previous proposition together with associativity of the smash product shows that $\Sn n\wedge\Sn
m\simeq\Sn{n+m}$ for every $n,m:\N$, but we need to prove that this collection of equivalences is
compatible with associativity of the smash product in the following sense.

\begin{proposition}\label{smashspheres}
  For every $n,m:\N^*$, there is an equivalence \[\wedge_{n,m}:\Sn n\wedge\Sn m\simeq\Sn{n+m}\]
  together with a filler of the diagram
  \[
  \begin{tikzcd}[row sep=small]
    (\Sn n\wedge\Sn m)\wedge\Sn k
    \arrow[rr,"{(\wedge_{n,m})\wedge\id_{\Sn k}}"]
    \arrow[dd,"{\alpha_{\Sn n,\Sn m,\Sn k}}"']&&
    \Sn{n+m}\wedge\Sn{k} \arrow[start anchor=south east, end anchor=north west,rd,"{\wedge_{n+m,k}}"]& \\
    &&&\Sn{n+m+k}\\
    \Sn n\wedge(\Sn m\wedge\Sn k)
    \arrow[rr,"{\id_{\Sn n}\wedge(\wedge_{m,k})}"']&&
    \Sn{n}\wedge\Sn{m+k} \arrow[start anchor=north east,end anchor=south west,ru,"{\wedge_{n,m+k}}"']&
  \end{tikzcd}
  \]
\end{proposition}

\begin{proof}
  We define the family of types $(\Sm n)_{n:\N^*}$ by
  \begin{align*}
    \Sm1 \defeq \Sn1,\\
    \Sm{n+1} \defeq \Sn1\wedge\Sm n.
  \end{align*}
  Using proposition \ref{suspsmashs1} we can easily construct by induction on $n$ an equivalence between $\Sm n$
  and $\Sn n$, therefore it is enough to prove the result for the family $(\Sm n)$ instead of
  $(\Sn n)$.

  The equivalence $\wedge_{n,m}$ is defined by induction on $n$.
  \begin{itemize}
  \item For $n=1$ we take the identity equivalence, as $\Sm1\wedge\Sm m$ is equal to $\Sm{m+1}$ by
    definition.
  \item For $n+1$ we take the composition
    \[
    \begin{tikzcd}
      \Sm{n+1}\wedge\Sm m\arrow[d,equals] \arrow[rr,dotted] & & \Sm{n+m+1} \arrow[d,equals]\\
      (\Sn 1\wedge\Sm n)\wedge\Sm m \arrow[r,"\sim"]& \Sn1\wedge(\Sm n\wedge\Sm
      m) \arrow[r,"\sim"]& \Sn1\wedge\Sm{n+m}
    \end{tikzcd}
    \]
    where we used associativity of the smash product and then the smash by $\Sn1$ of $\wedge_{n,m}$.
  \end{itemize}

  The compatibility for $n=1$ holds by definition (as both $\wedge_{n,m}$ and $\wedge_{n,m+k}$ are
  the identity), and for $n+1$ let’s consider the following diagram where we omitted the names of
  the arrows for
  readability.
  \[
  \begin{tikzcd}[row sep=scriptsize, column sep=-30]
    &&&\Sn1\wedge((\Sm{n}\wedge \Sm{m})\wedge \Sm{k}) \arrow[rdd] \arrow[llddddddd,rounded corners,to path={-|
      ([xshift=-1em]Z.west) |- (\tikztotarget)}]&\\
    &(\Sn1\wedge(\Sm{n}\wedge \Sm{m}))\wedge \Sm{k}\arrow[rru] \arrow[rdd] &&&\\
    &&&&\Sn1\wedge(\Sm{n+m}\wedge \Sm{k})\arrow[dd]\\
    |[alias=Z]|((\Sn1\wedge \Sm{n})\wedge \Sm{m})\wedge \Sm{k}\arrow[ruu] \arrow[rr] \arrow[dd]&&(\Sn1\wedge \Sm{n+m})\wedge \Sm{k}\arrow[rru]
    \arrow[rrd]&&\\
    &&&&\Sn1\wedge \Sm{n+m+k}\\
    (\Sn1\wedge \Sm{n})\wedge(\Sm{m}\wedge \Sm{k}) \arrow[rr] \arrow[rdd]&&(\Sn1\wedge \Sm{n})\wedge
    \Sm{m+k}\arrow[rru]\arrow[rdd]&&\\
    \\
    &\Sn1\wedge(\Sm{n}\wedge(\Sm{m}\wedge \Sm{k})) \arrow[rr]&&\Sn1\wedge(\Sm{n}\wedge \Sm{m+k})\arrow[ruuu]
  \end{tikzcd}
  \]
  We want to fill the inner pentagon. The three triangles are the definition of the maps of the form
  $(\Sn1\wedge \Sm i)\wedge \Sm j \to \Sn1\wedge\Sm{i+j}$. The two squares are filled by naturality
  of the associativity of the smash product. The left part of the diagram is filled by the pentagon
  of coherences of associativity. Finally the outer diagram is filled by the smash product by $\Sn1$
  of the induction hypothesis.  Therefore, we get a filler of the inner pentagon.
\end{proof}

The definition of the equivalences $\wedge_{n,m}$ is rather indirect, because it goes through the
family of types $(\Sm n)$. A more direct way to see them is as follows.

\begin{proposition}\label{smashspheresexplicit}
  Given $n,m:\N^*$, the map $\wedge_{n+1,m}$ of type $(\Susp\Sn n)\wedge\Sn m\to\Susp(\Sn{n+m})$
  satisfies, for every $x:\Sn n$ and $y:\Sn m$,
  \[\ap{\wedge_{n+1,m}}(\ap{\proj(-,y)}(\varphi_{\Sn n}(x))) = \varphi_{\Sn{n+m}}(\wedge_{n,m}(\proj(x,y))).\]
\end{proposition}
\begin{proof}
  It follows from the definition of $\wedge_{n+1,m}$ as we defined it by first using associativity
  of the smash product (which corresponds to moving the $\varphi$ to the outside) and then by
  applying $\wedge_{n,m}$.
\end{proof}

The compatibility with commutativity of the smash product is a bit more subtle as there is a sign.

\begin{proposition}\label{commspheres}
  For $n,m:\N^*$, there is a filler of the diagram
  \[
  \begin{tikzcd}
    \Sn n\wedge\Sn m \arrow[r,"\sigma_{\Sn n,\Sn m}"] \arrow[d,"\wedge_{n,m}"'] & \Sn m\wedge\Sn n \arrow[d,"\wedge_{m,n}"]\\
    \Sn{n+m} \arrow[r,"(-1)^{nm}"'] & \Sn{n+m}
  \end{tikzcd}
  \]
  where $(-1)^{nm}$ is
  \begin{itemize}
  \item the identity map, if $nm$ is even,
  \item the map $(-1)$ sending $\merid(a)$ to $\merid(a)\inv$, if $nm$ is odd.
  \end{itemize}
\end{proposition}

\begin{proof}
  As in the proof of the previous proposition, we work with the family $(\Sm n)$ instead of
  $(\Sn n)$, the map $(-1)$ corresponding to the map
  \begin{align*}
    (-1) &: \Sn1\wedge\Sm{n}\to\Sn1\wedge\Sm{n},\\
    (-1) &\defeq i\wedge\id_{\Sm n},
  \end{align*}
  where the map $i:\Sn1\to\Sn1$ is the map sending $\base$ to $\base$ and $\lloop$ to $\lloop\inv$.
  
  The idea is quite simple. The map exchanging $\Sm n$ and $\Sm m$ has to swap every $\Sn1$-factor
  of $\Sm m$ with every $\Sn1$-factor of $\Sm n$. Therefore in total we have $nm$ permutations
  between two $\Sn1$ factors, so when $nm$ is even they all cancel, and otherwise one remains. And it
  turns out that the map swapping the two factors of $\Sn1\wedge\Sn1$ is equal to the map $(-1)$ on
  $\Sm2$.
  
  We first consider the case $n=m=1$.
  \begin{proposition}\label{case11}
    There is a filler of the diagram
    \[
    \begin{tikzcd}
      \Sm1\wedge\Sm1 \arrow[d,equals] \arrow[r,"\sigma_{\Sm1,\Sm1}"] & \Sm1\wedge\Sm1 \arrow[d,equals]\\
      \Sm2 \arrow[r,"(-1)"'] & \Sm2
    \end{tikzcd}
    \]
  \end{proposition}
  \begin{proof}
    We have to construct an equality between the two functions
    \begin{align*}
      f &: \Sm1\wedge\Sm1 \to \Sm1\wedge\Sm1, & g &: \Sm1\wedge\Sm1 \to \Sm1\wedge\Sm1, \\
      f(\proj(x,y)) &\defeq \proj(y,x), & g(\proj(x,y)) &\defeq \proj(-x,y).
    \end{align*}
    We prove that they are equal on points of the form $\proj(x,y)$ by double induction on $x$ and
    $y$. When either $x$ or $y$ is $\base$, both $f(\proj(x,y))$ and $g(\proj(x,y))$ are equal to
    $\star_{\Sm1\wedge\Sm1}$ via either a $\projl$ or a $\projr$. For the case $(\lloop,\lloop)$,
    the idea is that we are given a square and we want to show that flipping it diagonally (using
    $f$) or horizontally (using $g$) gives equal squares along some coherence operations.

    If we start with the square
    \[
    \begin{tikzcd}[sdiag, row sep=huge]
      \bullet \arrow[r,"r"] \arrow[d,"p"'] & \bullet \arrow[d,"q"] \\
      \bullet \arrow[r,"s"'] & \bullet
    \end{tikzcd}
    \]
    then we consider the cube
    \begin{equation}\label{cube}
      \begin{tikzcd}[sdiag, row sep=huge]
        \bullet \arrow[rrr,"r\inv"] \arrow[ddd,"q"'] & & & \bullet \arrow[ddd,"p"] \\
        & \bullet \arrow[lu,"r",near start] \arrow[r,"p"'] \arrow[d,"r"] & \bullet \arrow[d,"s"']
        \arrow[ru,"p\inv",near start]
        & \\
        & \bullet \arrow[ld,"q",near start] \arrow[r,"q"] & \bullet \arrow[rd,"s\inv",near start] & \\
        \bullet \arrow[rrr,"s\inv"'] & & & \bullet
      \end{tikzcd}
    \end{equation}
    The inner and outer faces are the diagonal and horizontal flipping of the square, and the four
    other sides are defined using the equalities
    \begin{align*}
      r\concat r\inv &= p\concat p\inv,\\
      p\inv\concat p &= s\concat s\inv,\\
      q\concat s\inv &= q\concat s\inv,\\
      r\concat q     &= r\concat q.
    \end{align*}
    In order to fill this cube we proceed by induction on the whole square. Another way to see it is
    to do a path induction on $p$, $q$ and $r$, and then a fourth path induction on the square seen
    as the equality $s=p\inv\concat r\concat q$. Then we only have to fill the cube in the case
    where the square is the constant square with $\idpS$ on all four sides. In that case, all of the
    sides of the cube are constant squares so we can take the constant cube. It concludes the proof
    that $f$ and $g$ are equal.
  \end{proof}

  Note that the proof above works because we managed to write the cube \ref{cube} starting from an
  arbitrary square. There is no way to make a similar cube between the original square and its
  diagonal flipping, for instance.
  
  We now prove that the maps $(-1)$ are stable under smash product by $\Sm m$ on the right and by
  $\Sn1$ on the left.
  \begin{proposition}\label{smashm1right}
    For every $n,m:\N$, there is a filler of the square
    \[
    \begin{tikzcd}[sep=large]
      \Sm n\wedge\Sm m \arrow[r,"(-1)\wedge\id_{\Sm m}"] \arrow[d,"\wedge_{n,m}"'] & \Sm n\wedge\Sm m
      \arrow[d,"\wedge_{n,m}"]\\
      \Sm{n+m} \arrow[r,"(-1)"'] &\Sm{n+m}
    \end{tikzcd}
    \]
  \end{proposition}

  \begin{proof}
    We proceed by induction on $n$. For $n=1$ it is the definition of the map $(-1)$ on
    $\Sm{m+1}$. For $n+1$ we take the composition of the two squares
    \[
    \begin{tikzcd}
      (\Sn1\wedge\Sm n)\wedge\Sm m \arrow[d,"\alpha_{\Sn1,\Sm n,\Sm m}"']
      \arrow[r,"\raisebox{0.7em}{$(i\wedge\id_{\Sm n})\wedge\id_{\Sm m}$}"] & (\Sn1\wedge\Sm
      n)\wedge\Sm m \arrow[d,"\alpha_{\Sn1,\Sm n,\Sm m}"]\\
      \Sn1\wedge(\Sm n\wedge\Sm m) \arrow[d,"\id_{\Sn1}\wedge(\wedge_{n,m})"']
      \arrow[r,"\raisebox{-0.8em}{$i\wedge(\id_{\Sm
          n}\wedge\id_{\Sm m})$}"'] & \Sn1\wedge(\Sm n\wedge\Sm m) \arrow[d,"\id_{\Sn1}\wedge(\wedge_{n,m})"]\\
      \Sm{n+m+1} \arrow[r,"i\wedge\id_{\Sm{n+m}}"'] & \Sm{n+m+1}
    \end{tikzcd}
    \]
    where the top square comes from naturality of associativity of the smash product and the bottom
    square comes from the fact that $\id_{\Sm n}\wedge\id_{\Sm m}=\id_{\Sm{n}\wedge\Sm{m}}$.
  \end{proof}
  
  \begin{proposition}
    For every $n:\N$, the map \[\id_{\Sn1}\wedge(-1):\Sn1\wedge\Sm n\to\Sn1\wedge\Sm n\] is equal
    to \[(-1):\Sm{n+1}\to\Sm{n+1}.\]
  \end{proposition}

  \begin{proof}
    We proceed by induction on $n$. For $n=1$ the proof is entirely analogous to the proof of
    proposition \ref{case11}, except that we equate the vertical flipping of a square with its
    horizontal flipping. For $n+1$ we take the composition of squares
    \[
    \begin{tikzcd}
      \Sn1\wedge(\Sn 1\wedge\Sm n) \arrow[d,"\alpha_{\Sn1,\Sn 1,\Sm n}\inv"']
      \arrow[r,"\raisebox{0.7em}{$\id_{\Sn1}\wedge(i\wedge\id_{\Sm n})$}"] & \Sn1\wedge(\Sn1\wedge\Sm n) \arrow[d,"\alpha_{\Sn1,\Sn1,\Sm n}\inv"]\\
      (\Sn1\wedge\Sn1)\wedge\Sm n \arrow[d,"\id_{(\Sn1\wedge\Sn1)\wedge\Sm n}"']
      \arrow[r,"\raisebox{-0.8em}{$(\id_{\Sn1}\wedge i)\wedge\id_{\Sm n}$}"'] & (\Sn1\wedge\Sn1)\wedge\Sm n \arrow[d,"\id_{(\Sn1\wedge\Sn1)\wedge\Sm n}"]\\
      (\Sn1\wedge\Sn1)\wedge\Sm n \arrow[d,"\wedge_{2,n}"']
      \arrow[r,"\raisebox{-0.8em}{$(i\wedge\id_{\Sn1})\wedge\id_{\Sm n}$}"'] & (\Sn1\wedge\Sn1)\wedge\Sm n \arrow[d,"\wedge_{2,n}"]\\
      \Sm{n+2} \arrow[r,"(-1)"'] & \Sm{n+2}
    \end{tikzcd}
    \]
    where the top one comes from naturality of associativity of the smash product, the middle one
    comes from the case $n=1$ and the bottom one comes from proposition \ref{smashm1right}.
  \end{proof}

  We now come back to the proof of proposition \ref{commspheres}. Given $n$ and $m$, we assume by
  induction that we have constructed a filler of the diagram for $(n,1)$ and for $(n,m)$, and we
  construct one for $(n,m+1)$. Let’s consider the diagram \ref{hexagon} where we want to fill the
  upper square.
  The outer hexagon is the coherence hexagon between the commutator and the associator. The three
  curved squares come from the compatibility of the maps $\wedge_{i,j}$ with the associator
  (proposition \ref{smashspheres}). The lower left square is the smash product of the case $(n,1)$
  with $\Sm m$ and the lower right square is the smash product of $\Sn1$ with the case
  $(n,m)$. Finally, the inner diagram has a filler given by the equality
  $(-1)^{nm}\circ(-1)^n=(-1)^{n(m+1)}$ which can easily be verified.
  
  Finally, if the property is true for $(n,m)$, then it is true for $(m,n)$ as well. Combining that
  with the case above and the base case, we obtain that the property is true for every $(n,m)$.
\end{proof}
  
 \begin{figure}
   \centering
    \begin{tikzpicture}[commutative diagrams/every diagram]
      \node (P0) at (0:8em) {$\Sm{n+m+1}$};
      \node (P1) at (60:8em) {$\Sm{n+m+1}$};
      \node (P2) at (120:8em) {$\Sm{n+m+1}$};
      \node (P3) at (180:8em) {$\Sm{n+m+1}$};
      \node (P4) at (240:8em) {$\Sm{n+m+1}$};
      \node (P5) at (300:8em) {$\Sm{n+m+1}$};

      \node (T0) at ($(0:8em)+(330:8em)$) {$\Sn1\wedge(\Sm{m}\wedge\Sm{n})$};
      \node (T1) at ($(60:8em)+(90:8em)$) {$(\Sn1\wedge \Sm{m})\wedge\Sm{n}$};
      \node (T2) at ($(120:8em)+(90:8em)$) {$\Sm{n}\wedge(\Sn1\wedge\Sm{m})$};
      \node (T3) at ($(180:8em)+(210:8em)$) {$(\Sm{n}\wedge \Sn1)\wedge\Sm{m}$};
      \node (T4) at ($(240:8em)+(210:8em)$) {$(\Sn1\wedge \Sm{n})\wedge\Sm{m}$};
      \node (T5) at ($(300:8em)+(330:8em)$) {$\Sn1\wedge(\Sm{n}\wedge \Sm{m})$};

      \path[commutative diagrams/.cd, every arrow, every label]
      (P1) edge node[swap] {$\id$} (P0)
      (P2) edge node[swap] {$(-1)^{n(m+1)}$} (P1)
      (P3) edge node[swap] {$\id$} (P2)
      (P3) edge node {$(-1)^n$} (P4)
      (P4) edge node {$\id$} (P5)
      (P5) edge node {$(-1)^{nm}$} (P0)

      (T1) edge[bend left]  node {$\alpha_{\Sn1,\Sm m,\Sm n}$} (T0)
      (T2) edge             node {\raisebox{0.7em}{$\sigma_{\Sm{n},\Sn1\wedge \Sm{m}}$}} (T1)
      (T3) edge[bend left]  node {$\alpha_{\Sm n,\Sn1,\Sm m}$} (T2)
      (T3) edge             node[swap] {$\sigma_{\Sm{n},\Sn1}\wedge\id_{\Sm m}$} (T4)
      (T4) edge[bend right] node[swap] {$\alpha_{\Sn1,\Sm{n},\Sm m}$} (T5)
      (T5) edge             node[swap] {$\id_{\Sn1}\wedge\sigma_{\Sm n,\Sm m}$} (T0)

      (T1) edge (P1)
      (T2) edge (P2)
      (T3) edge (P3)
      (T4) edge (P4)
      (T5) edge (P5)
      (T0) edge (P0);
    \end{tikzpicture}
    \caption{Case $(n,m+1)$\label{hexagon}}
  \end{figure}
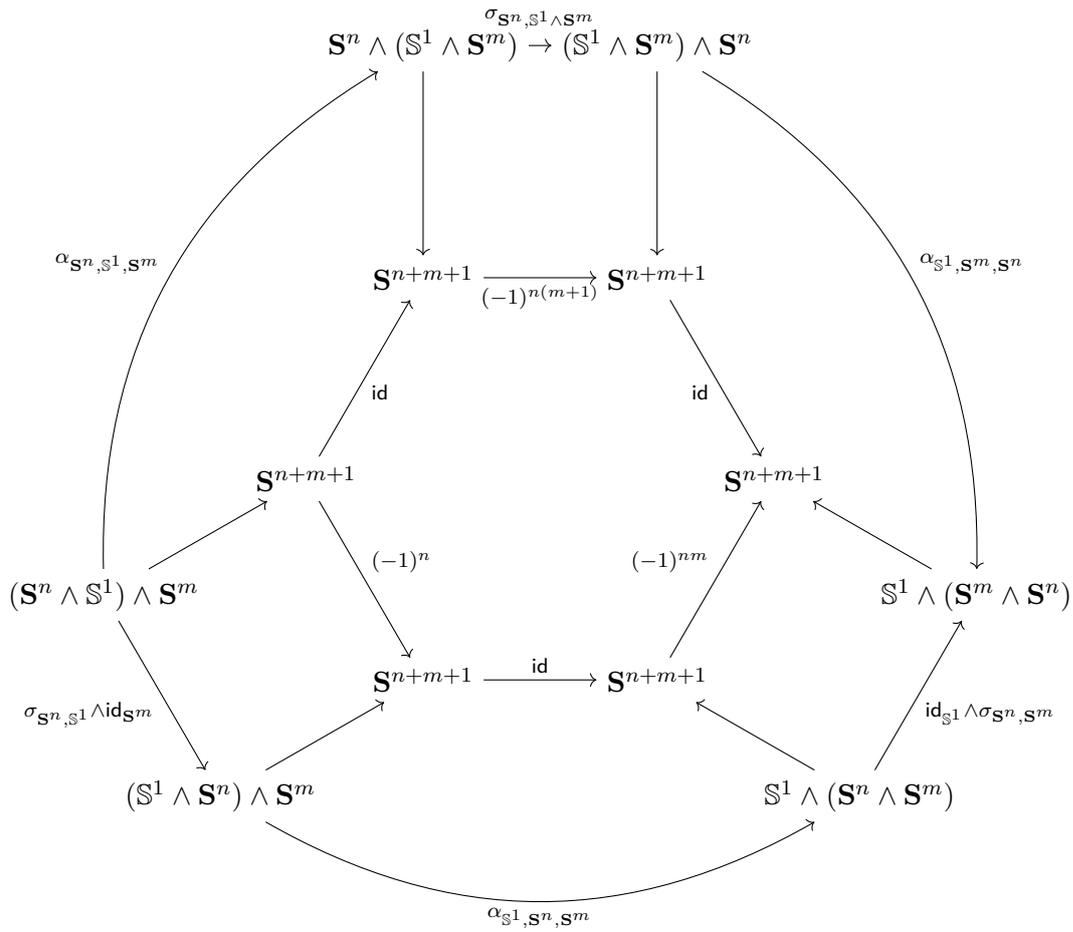

\section{Smash product and connectedness}

The connectedness of a smash product is given by the following proposition.

\begin{proposition}
  Given two pointed types $A$ and $B$, if $A$ is $m$-connected and $B$ is $n$-connected then
  $A\wedge B$ is $(m + n + 1)$-connected.
\end{proposition}

\begin{proof}
  Notice that the map $\iwedge_{A,B}:A\vee B\to A\times B$ is the pushout-product of the two maps
  $\Unit\to A$ and $\Unit\to B$, which are $(m-1)$- and $(n-1)$-connected, therefore proposition
  \ref{pushoutproduct} shows that $\iwedge_{A,B}$ is $(m+n)$-connected. Given that we have a pushout
  square
  \[
  \begin{tikzcd}
    A\vee B \arrow[r,"\iwedge_{A,B}"] \arrow[d] & A\times B \arrow[d] \\
    \Unit \arrow[r] & A\wedge B \arrow[lu,phantom,"\ulcorner",at start]
  \end{tikzcd}
  \]
  and that the top map is $(m+n)$-connected, we conclude that the map $\Unit\to A\wedge B$ is
  $(m+n)$-connected as well and therefore $A\wedge B$ is $(m+n+1)$-connected.
\end{proof}

For the smash product of maps, it’s not as easy because the answer depends not only on the
connectivity of the maps but also on the connectivity of the types. We first prove the following
proposition.
\begin{proposition}
  Given a pointed map $f:A\topt B$ and a pointed type $C$, if $f$ is $n_f$-connected and $C$ is
  $n_C$-connected then the map $f\wedge\id_C:A\wedge C\to B\wedge C$ is $(n_f+n_C+1)$-connected.
\end{proposition}

As a sanity check, note that if $C=\Bool$, then $f\wedge\id_C=f$ and $n_C=-1$, and we do have
$n_f+n_C+1=n_f$.

\begin{proof}
  We consider $P:B\wedge C\to\Type$ a family of $(n_f+n_C+1)$-truncated types and
  \[d : (x : A\wedge C) \to P((f\wedge\id_C)(x)).\]
  According to proposition \ref{inductionconnectedtruncated} and given that $f$ is $n_f$-connected,
  if we take an arbitrary $c:C$ then the function
  \[\lambda s.s\circ f:\prod_{b:B}P(\proj(b,c)) \to \prod_{a:A}P(\proj(f(a),c))\]
  is $(n_C-1)$-truncated. Note also that there is an element in the codomain given by
  $d\circ\proj(-,c)$.  Let’s now consider $Q:C\to\Type$ such that $Q(c)$ is the fiber of the
  function $\lambda s.s\circ f$ above over $d\circ\proj(-,c)$. We just proved that $Q$ is a family
  of $(n_C-1)$-truncated types, therefore in order to construct a section of it it is enough to give
  an element of $Q(\star_C)$. An element of $Q(\star_C)$ consists of a function
  $k:(b:B)\to P(\proj(b,\star_C))$ together with an equality between $k(f(a))$ and
  $d(\proj(a,\star_C))$ for every $a:A$. Therefore the map
  \[\lambda b.\transport^P(\projr(b)\inv,d(\star_{A\wedge C}))\]
  together with the equality
  \[\transport^P(\projr(f(a))\inv,d(\star_{A\wedge C})) = d(\proj(a,\star_C))\]
  given by $\ap d(\projr(a))$ is an element of $Q(\star_C)$.

  Putting everything together, what we obtain is a function
  \[g:(b:B)(c:C)\to P(\proj(b,c))\]
  together with paths, for all $a:A$, $b:B$ and $c:C$,
  \begin{align}
    g(f(a),c) &= d(\proj(a,c)),\label{eq1}\\
    g(b,\star_C) &= \transport^P(\projr(b)\inv,d(\star_{A\wedge C}))\label{eq2}
  \end{align}
  which are equal (via a $2$-dimensional path) in the case $g(f(a),\star_C)$.

  Therefore we can now define $h:(x:B\wedge C)\to P(x)$ as follows. For the basepoint
  $\star_{B\wedge C}$ we take $d(\star_{A\wedge C})$. For an element of the form $\proj(b,c)$ we
  take $g(b,c)$. For $\projr(b)$ it is automatic by the equation \ref{eq2}. For $\projl(c)$ we prove
  that $g(\star_B,c)$ is equal to $d(\star_{A\wedge C})$ over $\projl(c)$ by the composition
  \[
  \begin{tikzcd}[sdiag]
    g(\star_B,c) & g(f(\star_A),c) \arrow[l] \arrow[r,"\ref{eq1}"] & d(\proj(\star_A,c)) \arrow[r] &
    d(\star_{A\wedge C})
  \end{tikzcd}
  \]
  (note that the last equality is a dependent equality over $\projl(c)$).
  Finally for $\projlr$ we have to check that in the case $c=\star_C$, the composition we just
  defined is equal to the one obtained by transporting $d(\star_{A\wedge C})$ along
  $\projr(\star_B)$, and this follows from the $2$-dimensional equality for $g(f(a),\star_C)$
  mentioned above.
\end{proof}

We can now compute the connectedness of a smash product of maps.

\begin{proposition}\label{smashmapsconn}
  Given two maps $f:A\to A'$ and $g:B\to B'$, if $f$ is $n_f$-connected, $g$ is $n_g$-connected,
  $A'$ is $n_{A'}$-connected and $B$ is $n_B$-connected, then $f\wedge g$ is $k$-connected, where
  \[k=\min(n_f + n_B + 1, n_g + n_{A'} + 1).\]
\end{proposition}

\begin{proof}
  We have \[f\wedge g = (f\wedge\id_B)\circ(\id_{A'}\wedge g).\qedhere\]
\end{proof}


\chapter{Cohomology}\label{ch:cohomology}

In this chapter we introduce the cohomology groups with integer coefficients of a space, together
with their additive and multiplicative structure, and we prove that they form a graded ring. This
means that for every space $X$ we define a sequence of abelian groups $(H^n(X))_{n:\N}$ and a
multiplication operation $\cupp:H^i(X)\times H^j(X)\to H^{i+j}(X)$ which is distributive over
addition, associative, and graded-commutative in the sense that for every $x:H^i(X)$ and $y:H^j(X)$
we have
\[x\cupp y=(-1)^{ij}(y\cupp x).\]
We also construct the Mayer--Vietoris sequence, and we compute the cohomology ring of a product of
two spheres. We then define the Hopf invariant, which is an invariant of maps of the form
$\Sn{2n-1}\to\Sn{n}$ constructed using the multiplicative structure of cohomology. Finally we show
that the James construction provides functions of Hopf invariant $2$ for every even $n$, which
proves that all groups of the form $\pi_{4n-1}(\Sn{2n})$ are infinite and that $\pi_4(\Sn3)$ is
equal to either $\Z/2\Z$ or $\Z/1\Z$.

The usual definition of singular cohomology using singular cochains cannot be reproduced in homotopy
type theory because defining the \emph{set} of singular cochains requires taking the underlying set
of a topological space which is not an operation invariant under homotopy equivalence, hence it
cannot exist in homotopy type theory. We define cohomology instead via the Eilenberg--MacLane spaces
$K(\Z,n)$, which can be defined directly very easily by truncating the spheres.  It is also possible
to define the Eilenberg--MacLane spaces $K(G,n)$ for an arbitrary group $G$ (abelian if $n\ge2$), as
was shown by Dan Licata and Eric Finster in \cite{KGn}, and to use them to define cohomology with
coefficients in $G$. However, given that we only use cohomology with integer coefficients in this
work and that the definition of $K(\Z,n)$ is much simpler, we restrict ourselves to this case. For
readability, we only write $K_n$ instead of $K(\Z,n)$.

It is well-known that this definition of cohomology is the right one in homotopy type theory, but
the multiplicative structure and everything which depends on it is new to our knowledge. The
Mayer--Vietoris sequence has already been studied by Evan Cavallo in homotopy type theory in
\cite{cavallo:cohomology}, but with a rather different approach. His construction starts from the
Eilenberg--Steenrod axioms for cohomology and uses extensively the cubical machinery whereas the
construction presented here is more concrete.

\section{The cohomology ring of a space}

\begin{definition}
  For $n:\N$, the type $K_n$ is the $n$-truncated and $(n-1)$-connected pointed type defined by
  \begin{align*}
    K_0 &\defeq \Z,\\
    K_n &\defeq \trunc n{\Sn n}\text{ for }n\ge1.
  \end{align*}
\end{definition}

An important property of the family $(K_n)$ is that $K_n$ is equivalent to $\Omega K_{n+1}$ for
every $n:\N$.

\begin{proposition}
  For every $n:\N$, there is an equivalence $\sigma_n:K_n\simeq\Omega K_{n+1}$ and, when $n\ge1$, a
  filler of the diagram
  \[
  \begin{tikzcd}
    \Sn n \arrow[r,"\varphi_{\Sn n}"] \arrow[d,"|-|"'] & \Omega\Susp\Sn n \arrow[d,"\Omega|-|"]\\
    K_n \arrow[r,"\sigma_n"',"\sim"] & \Omega K_{n+1}
  \end{tikzcd}
  \]
\end{proposition}

\begin{proof}
  We define $\sigma_n$ by
  \begin{align*}
    \sigma_0(k) &\defeq (\ap{|-|}(\lloop))^k,\\
    \sigma_n(|x|) &\defeq \ap{|-|}(\varphi_{\Sn n}(x))\text{ for }n\ge1,
  \end{align*}
  and the diagram is filled by $\idpS$. We now have to prove that $\sigma_n$ is an equivalence.
  For $n=0$, we proved it in section \ref{hgss1}. For $n\ge1$, the composition
  \[
  \begin{tikzcd}
    K_n=\trunc n{\Sn n} \arrow[r,"\trunc n{\varphi_{\Sn{n}}}"]& \trunc{n}{\Omega\Sn{n+1}} \arrow[r,"\sim"]&
    \Omega\trunc{n+1}{\Sn{n+1}}=\Omega K_{n+1}
  \end{tikzcd}
  \]
  is equal to $\sigma_n$, hence it is enough to prove that $\trunc n{\varphi_{\Sn{n}}}$ is an
  equivalence. For $n=1$, we know from section \ref{hgshopf} that the map $d:\Omega\Sn2\to\Sn1$
  sending $p$ to $\transport^{\Hopf}(p,\base)$ has fiber $\Omega\Sn3$ and, hence, is
  $1$-connected. Moreover, for every $x:\Sn1$ we have
  \begin{align*}
    \transport^{\Hopf}(\varphi_{\Sn1}(x),\base)
    &= \transport^{\Hopf}(\merid_{\Sn1}(x)\concat\merid_{\Sn1}(\base)\inv,\base)\\
    &= \transport^{\Hopf}(\merid_{\Sn1}(\base)\inv,\\
    &\qquad\transport^{\Hopf}(\merid_{\Sn1}(x),\base))\\
    &= \transport^{\Hopf}(\merid_{\Sn1}(x),\base)\\
    &= \mu(\base,x)\\
    &= x,
  \end{align*}
  which shows that the composition
  \[
  \begin{tikzcd}
    \Sn1 \arrow[r,"\varphi_{\Sn1}"] & \Omega\Sn2 \arrow[r,"d"] & \Sn1
  \end{tikzcd}
  \]
  is equal to the identity function. Given that the map $\trunc1d$ is an equivalence, the map
  $\trunc1{\varphi_{\Sn1}}$ is an equivalence as well, which concludes the case $n=1$.

  For $n\ge2$, the Freudenthal suspension theorem states that the map
  $\varphi_{\Sn{n}}:\Sn{n}\to\Omega\Sn{n+1}$ is $(2n-2)$-connected, and we have $2n-2\ge n$; hence,
  $\trunc n{\varphi_{\Sn{n}}}$ is an equivalence.
\end{proof}

\begin{definition}
  Given a type $X$ and $n:\N$, the \emph{$n$-th cohomology group} of $X$ is the type
  \begin{align*}
    H^n(X) &\defeq \trunc0{X\to K_n},
  \end{align*}
  and, if $X$ is pointed, the \emph{$n$-th reduced cohomology group} of $X$ is the type
  \[\Ht^n(X) \defeq \trunc0{X \topt K_n}.\]
\end{definition}

Note that for every type $X$ we have $H^n(X)\simeq\Ht^n(X\sqcup\Unit)$, where $X\sqcup\Unit$ is pointed
by $\inr(\ttt)$, given that a map from $X$ to $K_n$ can be seen as a pointed map from $X\sqcup\Unit$
to $K_n$. Elements of $\Ht^n(X)$ can also be seen as elements of $H^n(X)$ by forgetting the
pointedness.

We now construct the group structure on cohomology groups. The idea is to construct the whole
structure on $K_n$ and then lift it to $H^n(X)$. We write $0$ for $\star_{K_n}$.

\begin{proposition}
  There are two maps $+:K_n\times K_n\to K_n$ and $-:K_n\to K_n$, and equalities
  \begin{align*}
    x + 0 &= x, \\
    0 + x &= x, \\
    x + (- x) &= 0,\\
    (- x) + x &= 0,\\
    (x + y) + z &= x + (y + z),\\
    x + y &= y + x,
  \end{align*}
  for every $x,y,z:K_n$.
\end{proposition}

\begin{proof}
  We use the equivalence $K_n\simeq\Omega K_{n+1}$. We define the addition of two elements $x$ and
  $y$ of $K_n$ as the composition of the corresponding loops in $\Omega K_{n+1}$, and the opposite
  of an element of $K_n$ as the inverse of the corresponding loop. Note that $\star_{K_n}$
  corresponds to the constant path of $\Omega K_{n+1}$. The unit laws, inverse laws, and
  associativity law come from the corresponding facts about paths, and commutativity comes from the
  Eckmann--Hilton argument (proposition \ref{eckmannhilton}), given that $\Omega K_{n+1}$ is
  equivalent to the double loop space $\Omega^2 K_{n+2}$.
\end{proof}

\begin{definition}
  Given a type $X$ and a natural number $n$, we equip $H^n(X)$ with the structure of abelian group
  given by
  \begin{align*}
    + &: H^n(X) \times H^n(X) \to H^n(X),\\
    |\alpha|+|\beta| &\defeq |\lambda x.(\alpha(x)+\beta(x))|,
  \end{align*}
  \begin{align*}
    - &: H^n(X) \to H^n(X),\\
    -|\alpha| &\defeq |\lambda x.(-\alpha(x))|,
  \end{align*}
  \begin{align*}
    0 &: H^n(X),\\
    0 &\defeq |\lambda x.0|.
  \end{align*}
  The group laws are proved in a similar fashion. For instance we prove associativity by
  \begin{align*}
    a &: (\alpha, \beta, \gamma : H^n(X)) \to (\alpha + \beta) + \gamma = \alpha
        + (\beta + \gamma),\\
    a(|\alpha|, |\beta|, |\gamma|) &\defeq \ap{|-|}(\funext(\lambda x.a'(\alpha(x),\beta(x),\gamma(x)))),
  \end{align*}
  where $a'$ is associativity of $+$ on $K_n$.  We also equip $\Ht^n(X)$ with the structure of
  abelian group defined in the exact same way, and the forgetful map from $\Ht^n(X)$ to $H^n(X)$ is
  a group homomorphism.
\end{definition}

For the multiplicative structure, we first define a map $\cupp:K_i\wedge K_j\to K_{i+j}$.  It
induces a map $K_i\times K_j\to K_{i+j}$ by composition with $\proj$ which then induces a product on
cohomology groups called the \emph{cup product}:
\[\cupp:H^i(X)\times H^j(X) \to H^{i+j}(X).\]

\begin{definition}\label{squarecupp}
  For every $i,j:\N^*$, there is a map $\cupp:K_i\wedge K_j\to K_{i+j}$ together with a filler of
  the square
  \[
  \begin{tikzcd}
    \Sn i\wedge\Sn j \arrow[r,"\wedge_{i,j}"] \arrow[d,"|-|_i\wedge|-|_j"'] & \Sn{i+j} \arrow[d,"|-|_{i+j}"]\\
    K_i\wedge K_j \arrow[r,"\cupp"]& K_{i+j}
  \end{tikzcd}
  \]
\end{definition}

\begin{proof}
  For every $k:\N^*$, the type $\Sn k$ is $(k-1)$-connected by proposition \ref{sphereconn} and the
  map $|-|_k:\Sn k\to K_k$ is $k$-connected by proposition \ref{truncconn}. Therefore, by
  proposition \ref{smashmapsconn}, the map $|-|_i\wedge|-|_j$ on the left of the diagram is
  $(i+j)$-connected. Moreover, the type $K_{i+j}$ is $(i+j)$-truncated; hence, by proposition
  \ref{inductionconnected}, we obtain the map $\cupp$ and a filler of the
  diagram.
\end{proof}

This definition doesn’t work when either $i$ or $j$ is equal to $0$, because then we have $K_0=\Z$,
which is not equal to $\trunc{0}{\Sn0}$. In those cases we define the cup product by
\begin{align*}
  \cupp &: K_0\wedge K_j \to K_j,\\
  n\cupp \beta &\defeq \beta + \dots + \beta\quad \text{($n$ terms)},
\end{align*}
\begin{align*}
  \cupp &: K_i\wedge K_0 \to K_i,\\
  \alpha\cupp m &\defeq \alpha + \dots + \alpha\quad \text{($m$ terms)}.
\end{align*}
These two definitions are consistent in the case $i=j=0$ because multiplication of integers is
commutative.

\begin{proposition}
  The collection of maps $\cupp:K_i\wedge K_j\to K_{i+j}$ we just defined is distributive with
  respect to addition, graded-commutative, and associative.
\end{proposition}

\begin{proof}  
  \bigskip\noindent
  \textbf{Distributivity} 
  Given $x:K_n$ and $y,z:K_m$ we want to prove that
  \[x\cupp(y+z)=(x\cupp y)+(x\cupp z).\]
  If $m=0$, it’s clear, so we assume now $m>0$. The idea is to show that we can assume that either
  $y$ or $z$ is equal to $0$. Distributivity is equivalent to giving a filler of the diagram
  \[
  \begin{tikzcd}
    K_n\wedge(K_m\times K_m) \arrow[rr,"K_n\wedge+"] \arrow[d] && K_n\wedge K_m \arrow[d,"\cupp"']\\
    (K_n\wedge K_m)\times(K_n\wedge K_m) \arrow[r,"\cupp\times\cupp"'] & K_{n+m}\times
    K_{n+m}\arrow[r,"+"'] & K_{n+m}
  \end{tikzcd}
  \]
  where the map on the left sends $\proj(x,(y,z))$ to $(\proj(x,y),\proj(x,z))$.

  The map $\iwedge_{K_m,K_m}:K_m\vee K_m\to K_m\times K_m$ is $(2m-2)$-connected; therefore, the map
  \begin{align*}
    f &: K_n\wedge(K_m\vee K_m) \to K_n\wedge(K_m\times K_m),\\
    f(\proj(x,\inl(y))) &\defeq \proj(x,(y,0)),\\
    f(\proj(x,\inr(z))) &\defeq \proj(x,(0,z))
  \end{align*}
  is $(n+2m-2)$-connected. Given that what we have to prove is $(n+m-1)$-truncated and that
  $n+2m-2\ge n+m-1$, it is enough to prove commutativity of the diagram for elements given by
  $f$. For an element of the form $\proj(x,\inl(y))$, we have to prove
  \[x\cupp(y+0)=(x\cupp y)+(x\cupp0),\]
  which is true because $y+0=y$ and $x\cupp0=0$. It is true in the same way for elements of the form
  $\proj(x,\inr(z))$.

  \bigskip\noindent
  \textbf{Graded commutativity} Given $x:K_i$ and $y:K_j$, we want to prove that
  \[x\cupp y=(-1)^{ij}y\cupp x.\]
  If either $i$ or $j$ is equal to $0$, it is immediate from the definition of the cup product by an
  element of degree $0$. We assume now that $i$ and $j$ are positive and we consider the diagram
  \[
  \begin{tikzcd}[row sep=scriptsize, column sep=scriptsize]
    & \Sn i\wedge\Sn j \arrow[rr] \arrow[dd,"{\wedge_{i,j}}" near start] \arrow[ld] && \Sn j\wedge \Sn i\arrow[dd,"{\wedge_{j,i}}"] \arrow[ld]\\
    K_i\wedge K_j \arrow[rr,crossing over] \arrow[dd,"\cupp"'] && K_j\wedge K_i\\
    & \Sn{i+j} \arrow[rr,"(-1)^{ij}"',near start] \arrow[ld] && \Sn{i+j} \arrow[ld]\\
    K_{i+j} \arrow[rr,"(-1)^{ij}"'] && K_{i+j} \arrow[from=uu,crossing over,"\cupp" near start]
  \end{tikzcd}
  \]
  where we have to fill the front face.  The left and right faces are filled by definition of the
  cup product. The back face is filled by proposition \ref{commspheres}. The top face is filled by
  naturality of commutativity of the smash product. The bottom face is filled by definition of
  $(-1)$ on $K_{i+j}$. This is not yet enough to fill the front face, it only gives a filler of it
  after composition with the map $\Sn i\wedge\Sn j\to K_i\wedge K_j$. But that map is
  $(i+j)$-connected and we want to construct an equality in $K_{i+j}$ which is
  $(i+j)$-truncated. Therefore, proposition \ref{inductionconnected} gives a filler of the front
  face.

  \bigskip\noindent
  \textbf{Associativity} Given $x:K_i$, $y:K_j$ and $z:K_k$, we want to prove that
  \[(x\cupp y)\cupp z=x\cupp(y\cupp z).\]
  If either $i$, $j$, or $k$ is equal to $0$, it follows from distributivity. For instance we have
  \begin{align*}
    (n\cupp x)\cupp y &= (x+\dots+x)\cupp y\\
                      &= x\cupp y+\dots+x\cupp y\\
                      &= n\cupp(x\cupp y).
  \end{align*}
  We assume now that $i$, $j$, and $k$ are positive and we consider the diagram
  \[
  \begin{tikzcd}[row sep=scriptsize, column sep=small]
    & (\Sn i\wedge\Sn j)\wedge\Sn k \arrow[rr] \arrow[dd] \arrow[ld]&& \Sn{i+j}\wedge\Sn k
    \arrow[rrd] \arrow[ld]&&\\
    (K_i\wedge K_j)\wedge K_k \arrow[rr,"{\cupp}\wedge \id_{K_k}",crossing over, near end] \arrow[dd]&&
    K_{i+j}\wedge K_k &&& \Sn{i+j+k} \arrow[ld]\\
    & \Sn i\wedge(\Sn j\wedge\Sn k) \arrow[rr] \arrow[ld]&& \Sn{i}\wedge\Sn{j+k}
    \arrow[rru] \arrow[ld]& K_{i+j+k} \arrow[from=llu,"\cupp",crossing over]\\
    K_i\wedge(K_j\wedge K_k) \arrow[rr,"\id_{K_i}\wedge{\cupp}"']&& K_i\wedge K_{j+k}
    \arrow[rru,"\cupp"']&&
  \end{tikzcd}
  \]
  We want to fill the front face, and proposition \ref{smashspheres} gives a filler of the back
  face.  The left square is filled by naturality of associativity of the smash product, the
  top-right and bottom-right squares are filled by definition of the cup product, and the top-left
  and bottom-left squares are filled by definition of the cup product and naturality of the smash
  product. The map $(\Sn i\wedge\Sn j)\wedge\Sn k\to(K_i\wedge K_j)\wedge K_k$ is
  $(i+j+k)$-connected, and we want to construct an equality in $K_{i+j+k}$, which is
  $(i+j+k)$-truncated. Therefore, we get a filler of the front face of the diagram.
\end{proof}

This concludes the construction of the structure of graded ring on cohomology. Moreover, as we see
now, the operation associating the cohomology ring to a space is a contravariant functor.

\begin{definition}
  Given $n:\N$ and a map $f:X\to Y$, we define $f^*:H^n(Y)\to H^n(X)$ by
  \begin{align*}
    f^* &: H^n(Y)\to H^n(X),\\
    f^*(|\beta|) &\defeq |\beta\circ f|.
  \end{align*}
  Similarly, if $X$ and $Y$ are pointed and $f:X\topt Y$ is a pointed map, we have
  $f^*:\Ht^n(Y)\to\Ht^n(X)$ defined in the same way.
\end{definition}

The following proposition is straightforward given that all operations are defined first on $K_n$
and then transferred to $H^n(X)$ by composition.
\begin{proposition}
  For every $n,m:\N$, $f:X\to Y$, $g:Y\to Z$, $\beta:H^n(Y)$, $\beta':H^n(Y)$ and $\gamma:H^m(Y)$,
  we have
  \begin{itemize}
  \item $(\id_X)^*=\id_{H^n(X)}$ and $(g\circ f)^*=f^*\circ g^*$,
  \item $f^*(0_{H^n(Y)})=0_{H^n(X)}$, $f^*(-\beta)=-f^*(\beta)$ and $f^*(\beta+\beta')=f^*(\beta)+f^*(\beta')$,
  \item $f^*(\beta\cupp\gamma)=f^*(\beta)\cupp f^*(\gamma)$.
  \end{itemize}
\end{proposition}

\section{The Mayer--Vietoris sequence}

The \emph{Mayer--Vietoris sequence} is a long exact sequence relating the cohomology groups of a pushout
$A\sqcup^CB$ with those of $A$, $B$, and $C$. We prove a “half-reduced” version of the
Mayer--Vietoris sequence which allows us to deduce easily both the unreduced version of the
Mayer--Vietoris sequence and the long exact sequence of a cofiber.

We consider three types $A$, $B$, and $C$, with $A$ pointed, and two functions $f:C\to A$ and
$g:C\to B$. We define $D \defeq A\sqcup^CB$ pointed by $\star_D\defeq\inl(\star_A)$. We define the
maps
\begin{align*}
  i&: \Ht^n(D)\to \Ht^n(A)\times H^n(B), & \Delta &:\Ht^n(A)\times H^n(B)\to H^n(C),\\
  i(\delta) &\defeq (\inl^*(\delta),\inr^*(\delta)),
                                    & \Delta(\alpha,\beta) &\defeq f^*(\alpha)-g^*(\beta),
\end{align*}
and
\begin{align*}
  d &: H^n(C)\to \Ht^{n+1}(D),\\
  d(|\gamma|) &\defeq \widetilde{d}(\gamma),
\end{align*}
where
\begin{align*}
  \widetilde{d} &: (C\to K_n)\to (D \topt K_{n+1}),\\
  \widetilde{d}(\gamma)(\inl(a)) &\defeq 0,\\
  \widetilde{d}(\gamma)(\inr(b)) &\defeq 0,\\
  \ap{\widetilde{d}(\gamma)}(\push(c)) &\defeq \sigma_n(\gamma(c)),
\end{align*}
where in the last line we see $\gamma(c)$ as an element of $\Omega K_{n+1}$ via the equivalence
$\sigma_n:K_n\simeq\Omega K_{n+1}$.  The first thing to note is that $i$, $\Delta$ and $d$ are group
homomorphisms. It’s obvious for $i$ and $\Delta$, and for $d$ it follows from the fact that
\begin{align*}
  \ap{d(|\gamma|+|\gamma'|)}(\push(c)) &= \sigma_n(\gamma(c)+\gamma'(c))\\
                                        &= \sigma_n(\gamma(c))\concat\sigma_n(\gamma'(c))\\
                                        &=\ap{d(|\gamma|)}(\push(c))\concat\ap{d(|\gamma'|)}(\push(c))\\
                                        &=\ap{d(|\gamma|)+d(|\gamma'|)}(\push(c)).
\end{align*}
We then have the following sequence of groups and homomorphisms of groups, which starts at
$\Ht^0(D)$ and then extends infinitely.

\begin{equation}
\begin{tikzcd}
  \Ht^{n+1}(D) \arrow[r,"i"]& \Ht^{n+1}(A)\times H^{n+1}(B) \ar[r,"\Delta"]& |[alias=X]| H^{n+1}(C)\\
  \Ht^n(D) \arrow[r,"i"]& \Ht^n(A)\times H^n(B) \ar[r,"\Delta"]& |[alias=X]| H^n(C)
  \arrow[ull,out=80,in=280,looseness=0.3,"d"']\\
  \Ht^{n-1}(D) \arrow[r,"i"]& \Ht^{n-1}(A)\times H^{n-1}(B) \ar[r,"\Delta"]& |[alias=X]| H^{n-1}(C)
  \arrow[ull,out=80,in=280,looseness=0.3,"d"']
\end{tikzcd}\label{eq:mayervietoris}
\end{equation}

\newcommand{\ploup}[1]{\bigskip\noindent{#1}\smallskip}

\begin{proposition}
  The sequence \ref{eq:mayervietoris} is a long exact sequence of groups.
\end{proposition}

\begin{proof}
\noindent{$\Im(d)\subset\Ker(i)$}
\smallskip

\noindent
The composition $i\circ d$ is equal to $0$ because it only applies $\widetilde{d}$ to elements of
the form $\inl(a)$ or $\inr(b)$.

\ploup{$\Ker(i)\subset\Im(d)$}

\noindent
A map $f$ from $D$ to $K_{n+1}$ is in the kernel of $i$ if and only if both $f\circ\inl$ and
$f\circ\inr$ are equal to $0$. Given such a map we construct a map from $C$ to $\Omega K_{n+1}$ by
sending $c$ to $\ap f(\push(c))$ composed with the equalities from the hypothesis. This gives a map
from $C$ to $K_n$ after composition with $\sigma_n\inv$ and we can easily check that its image by
$f$ is equal to $d$.

\ploup{$\Im(i)\subset\Ker(\Delta)$}

\noindent
Given $\delta:D\to K_n$, we have
$\Delta(i(|\delta|))=f^*(\inl^*(|\delta|))-g^*(\inr^*(|\delta|))$. Moreover we have
$f^*(\inl^*(|\delta|))=\lambda c.\delta(\inl(f(c)))$ and
$g^*(\inr^*(|\delta|))=\lambda c.\delta(\inr(g(c)))$ which are equal via the homotopy
$\lambda c.\ap \delta(\push(c))$. Therefore $\Delta(i(|\delta|))=0$.

\ploup{$\Ker(\Delta)\subset\Im(i)$}

\noindent
Given $\alpha:A\to K_n$, $\beta:B\to K_n$, and $\gamma:f^*(|\alpha|)=g^*(|\beta|)$, i.e.\
$\gamma:(c:C)\to\alpha(f(c))=\beta(g(c))$, we construct $\delta:D\to K_n$ by
\begin{align*}
  \delta(\inl(a)) &\defeq \alpha(a),\\
  \delta(\inr(b)) &\defeq \beta(b),\\
  \ap\delta(\push(c)) &\defeq \gamma(c).
\end{align*}
The image of $|\delta|$ by $i$ is $(|\alpha|,|\beta|)$, which is what we wanted.

\ploup{$\Im(\Delta)\subset\Ker(d)$}

\noindent
Given $\alpha:A\to K_n$ and $\beta:B\to K_n$, we have
\begin{align*}
  d(\Delta(|\alpha|,|\beta|)) &= d(f^*(|\alpha|)-g^*(|\beta|)) \\
                              &= d(f^*(|\alpha|)) - d(g^*(|\beta|))\\
                              &= d(|\alpha\circ f|) - d(|\beta\circ g|)\\
                              &= \widetilde{d}(\alpha\circ f) - \widetilde{d}(\beta\circ g).
\end{align*}
We show that $\widetilde{d}(\alpha\circ f)$ is equal to $0$, and a similar proof apply to
$\widetilde{d}(\beta\circ g)$. We proceed by induction on the argument, which is of type $D$.
\begin{itemize}
\item For an element of the form $\inl(a)$, we use the path $\sigma_n(\alpha(a)):\Omega K_{n+1}$
  between $\widetilde{d}(\alpha\circ f)(\inl(a))$ and $0$ and the fact that
  $\widetilde{d}(\alpha\circ f)(\inl(a))$ is equal to $0$ by definition. It might be tempting to
  take the constant path instead, but then the induction wouldn’t work for the $\push(c)$ case.
\item For an element of the form $\inr(b)$, we use the constant path.
\item For a path of the form $\push(c)$, we have to prove that $\sigma_n(\alpha(f(c)))$ is equal to
  the constant path, along $\ap{\widetilde{d}(\alpha\circ f)}(\push(c))$, which is also equal to
  $\sigma_n(\alpha(f(c)))$, hence it works.
\end{itemize}

\ploup{$\Ker(d)\subset\Im(\Delta)$}

\noindent
Given a map $\gamma:C\to K_n$ the hypothesis $\widetilde{d}(\gamma)=0$ is an equality between two
functions from $D$ to $K_{n+1}$. If we look at what that means for each constructor of $D$, we get
\begin{align*}
  \alpha &: A\to\Omega K_{n+1},\\
  \beta &: B\to\Omega K_{n+1},\\
  c &: (c:C)\to \alpha(f(c))\inv\concat\sigma_n(\gamma(c))\concat \beta(g(c))=\idpS.
\end{align*}
This gives us $\alpha$ and $\beta$, and, moreover, we have $f^*(|\alpha|) - g^*(|\beta|)=|\gamma|$ because
$\alpha(f(c))\concat\beta(g(c))\inv=\sigma_n(\gamma(c))$.
\end{proof}

In the special case where we have a pushout of the form $C_f=\Unit\sqcup^AB$ for some map
$f:A\to B$, we obtain the following long exact sequence relating the cohomology of $A$ and $B$ with
the reduced cohomology of $C_f$, where $i$ is the inclusion $B\to C_f$.

\[
\begin{tikzcd}
  \Ht^{n+1}(C_f) \arrow[r,"i^*"]& H^{n+1}(B) \ar[r,"f^*"]& |[alias=X]| H^{n+1}(A)\\
  \Ht^n(C_f) \arrow[r,"i^*"]& H^n(B) \ar[r,"f^*"]& |[alias=X]| H^n(A)
  \arrow[ull,out=80,in=280,looseness=0.3,"d"']\\
  \Ht^{n-1}(C_f) \arrow[r,"i^*"]& H^{n-1}(B) \ar[r,"f^*"]& |[alias=X]| H^{n-1}(A)
  \arrow[ull,out=80,in=280,looseness=0.3,"d"']
\end{tikzcd}
\]

We also obtain the unreduced Mayer--Vietoris sequence, which is formally identical to the
Mayer--Vietoris sequence \ref{eq:mayervietoris} with the exception that all reduced cohomology
groups are replaced by regular cohomology groups. The idea is to apply the Mayer--Vietoris sequence
\ref{eq:mayervietoris} to $A'\defeq \Unit+A$ and to notice that $\Ht^n(A')\simeq H^n(A)$ and that
$\Ht^n(D')\simeq H^n(D)$ because
\begin{align*}
  D' &\defeq (\Unit+A)\sqcup^CB\\
     &\simeq \Unit+(A\sqcup^CB)\\
     &\simeq \Unit+D.
\end{align*}

\section{Cohomology of products of spheres}

\begin{proposition}
  The cohomology groups of the point are
  \[H^k(\Unit)\simeq
  \begin{cases}
    \Z &\text{if $k=0$},\\
    0 &\text{otherwise}.
  \end{cases}
  \]
\end{proposition}

\begin{proof}
  We have $(\Unit\to K_k)\simeq K_k$, which is equal to $\Z$ for $k=0$ and is connected otherwise.
\end{proof}

The cohomology of spheres is easy to compute using the Mayer--Vietoris sequence.
\begin{proposition}
  For $n>0$ we have
  \[H^k(\Sn n)\simeq
  \begin{cases}
    \Z &\text{if $k=n$ or $k=0$},\\
    0 &\text{otherwise}.
  \end{cases}\]
  We write $\cc_n$ for the generator of $H^n(\Sn n)$.
\end{proposition}

\begin{proof}
  It follows easily, by induction, from the application of the Mayer--Vietoris sequence to the
  pushout $\Sn n\defeq\Unit\sqcup^{\Sn{n-1}}\Unit$.
\end{proof}

Note that the cup product of $\cc_n$ with itself in $H^n(\Sn n)$ vanishes because it is an element
of $H^{2n}(\Sn n)$ which is the trivial group.  We now look at the cohomology of a product of two
spheres.

\begin{proposition}\label{cohomologyproductspheres}
  Given $n,k:\N$, the cohomology groups of $\Sn n\times\Sn k$ are generated (additively) by
  \begin{align*}
    1 &: H^0(\Sn n\times\Sn k),\\
    \xx &: H^n(\Sn n\times\Sn k),\\
    \yy &: H^k(\Sn n\times\Sn k),\\
    \zz &: H^{n+k}(\Sn n\times\Sn k).
  \end{align*}
  Note that $n$ and $k$ may be equal, which is why we state it in this way.
  Moreover in the diagram
  \[
  \begin{tikzcd}[column sep=huge]
    \Sn n \arrow[r,"{i_n:x\mapsto(x,\star_{\Sn k})}"]& \Sn n\times\Sn k \arrow[r,"{p_n:(x,y)\mapsto x}"]& \Sn n,
  \end{tikzcd}
  \]
  the map $p_n^*$ sends $\cc_n$ to $\xx$ and the map $i_n^*$ sends $\xx$ (resp. $\yy$) to $\cc_n$
  (resp. to $0$).
  Similarly, in the diagram
  \[
  \begin{tikzcd}[column sep=huge]
    \Sn k \arrow[r,"{i'_k:y\mapsto(\star_{\Sn n},y)}"]& \Sn n\times\Sn k \arrow[r,"{p'_k:(x,y)\mapsto
      y}"]& \Sn k,
  \end{tikzcd}
  \]
  the map ${p'_k}^*$ sends $\cc_k$ to $\yy$ and the map ${i'_k}^*$ sends $\yy$ (resp. $\xx$) to $\cc_k$
  (resp. to $0$).
  Finally, we have
  \begin{align*}
    \xx\cupp\xx&=0,\\
    \yy\cupp\yy&=0,\\
    \xx\cupp\yy&=\zz.
  \end{align*}
\end{proposition}

\begin{proof}
  For the additive structure, we use the Mayer--Vietoris sequence and the fact that
  \[\Sn n\times\Sn k\simeq\Unit\sqcup^{\Sn{n+k-1}}(\Sn n\vee\Sn k),\]
  which gives directly the first result together with the fact that $i_n^*$ and ${i'_k}^*$ send
  $\xx$ and $\yy$ to $\cc_n$ and $\cc_k$ and the other one to $0$. The two projections $p_n^*$ and
  ${p'_k}^*$ send $\cc_n$ and $\cc_k$ back to $\xx$ and $\yy$ because we have
  $p_n\circ i_n=\id_{\Sn n}$ and $p'_k\circ i'_k=\id_{\Sn k}$. The Mayer--Vietoris sequence also
  shows that the map $\wedge_{n,k}\circ\proj:\Sn n\times\Sn k\to\Sn n\wedge\Sn k\to\Sn{n+k}$ induces
  an isomorphism on $H^{n+k}$.

  By definition, the cup product of $\xx$ and $\yy$ corresponds to the composition
  \[
  \begin{tikzcd}
    \Sn n\times\Sn k \arrow[rr,"{(x,y)\mapsto(x,y)}"]&& \Sn n\times \Sn k \arrow[r,"\proj"]& \Sn n\wedge\Sn k
    \arrow[r,"\sim"]& \Sn{n+k} \arrow[r,"|-|"]& K_{n+k}.
  \end{tikzcd}
  \]
  Note that the first map is the identity function because it’s the pairing of the two projections,
  the first projection coming from $\xx$ and the second from $\yy$. Therefore, we have
  \[\xx\cupp \yy=\zz.\]

  The cup product of $\xx$ with itself corresponds to the composition
  \[
  \begin{tikzcd}
    \Sn n\times\Sn k \arrow[rr,"{(x,y)\mapsto(x,x)}"] \arrow[rd,"{(x,y)\mapsto x}"']&& \Sn n\times \Sn n \arrow[r,"\proj"]& \Sn n\wedge\Sn n
    \arrow[r,"\sim"]& \Sn{2n} \arrow[r,"|-|"]& K_{2n}.\\
    &\Sn n \arrow[ru,"{x\mapsto(x,x)}"']
  \end{tikzcd}
  \]
  But the first map factors through $\Sn n$ which has no cohomology in dimension $2n$, therefore
  \[\xx\cupp \xx=0,\]
  and similarly
  \[\yy\cupp \yy=0.\qedhere\]
\end{proof}

\section{The Hopf invariant}

Given a map $f:\Sn k\to\Sn n$, we can consider the pushout $\Unit\sqcup^{\Sn k}\Sn n$, which has
cohomology $\Z$ in dimensions $0$, $n$ and $k+1$ (using the Mayer--Vietoris sequence). The cup
product structure on this space is often trivial, unless $k=2n-1$ in which case the square of the
generator in dimension $n$ may be a nontrivial multiple of the generator in dimension $2n$. This is
what we study in this section.

\begin{definition}
  Given a pointed map $f:\Sn{2n-1}\to\Sn n$, we define
  \begin{align*}
    C_f &\defeq \Unit\sqcup^{\Sn{2n-1}}\Sn n,\\
    \alpha_f &\defeq (i^*)\inv(\cc_n) : H^n(C_f),\\
    \beta_f &\defeq p^*(\cc_{2n}) : H^{2n}(C_f),
  \end{align*}
  where $i:\Sn n\to C_f$ is the inclusion on the right and $p:C_f\to\Sn{2n}$ is the map collapsing
  the $\Sn{n}$ term in $\Unit\sqcup^{\Sn{2n-1}}\Sn n$ to a point. The Mayer--Vietoris sequence shows
  that the map $i^*$ induces an isomorphism on $H^n$ and that the map $p^*$ induces an isomorphism
  on $H^{2n}$, hence the elements $\alpha_f$ and $\beta_f$ are generators of the respective
  cohomology groups.
\end{definition}

\begin{definition}
  The \emph{Hopf invariant} of a pointed map $f:\Sn{2n-1}\to\Sn n$ is the integer $H(f):\Z$ such
  that
  \[\alpha_f^2=H(f)\beta_f,\]
  where $\alpha_f^2$ is $\alpha_f\cupp\alpha_f$.
\end{definition}

Note that, if $n$ is odd, we have $\alpha_f^2=-\alpha_f^2$ by graded-commutativity; therefore, the
Hopf invariant of any map $f:\Sn{2n-1}\to\Sn{n}$ is $0$.

\begin{proposition}
  The Hopf invariant $H:\pi_{2n-1}(\Sn n)\to\Z$ is a homomorphism of groups.
\end{proposition}

\begin{proof}
  Let $f$ and $g$ be two pointed maps from $\Sn{2n-1}$ to $\Sn{n}$. The sum of $f$ and $g$ when seen
  as elements of $\pi_{2n-1}(\Sn{n})$ is represented by the map
  \[
  f+g:\begin{tikzcd}
    \Sn{2n-1}\arrow[r,"\contreq"]& \Sn{2n-1}\vee\Sn{2n-1}\arrow[r,"f\vee g"]&\Sn{n}.
  \end{tikzcd}
  \]
  Let’s consider the type $C_{f\vee g}$ defined by the pushout
  \[
  \begin{tikzcd}
    \Sn{2n-1}\vee\Sn{2n-1} \arrow[r,"f\vee g"] \arrow[d]& \Sn{n} \arrow[d,dotted,"i"] \\
    \Unit\arrow[r,dotted]& C_{f\vee g}
  \end{tikzcd}
  \]
  The cohomology groups of $C_{f\vee g}$ are $\Z$ in dimensions $0$ and $n$, $\Z^2$ in dimension
  $2n$, and the trivial group otherwise. We denote by $\alpha_{f\vee g}$ the generator in dimension
  $n$ coming from $i$ and by $\beta_{f\vee g}^f$ and $\beta_{f\vee g}^g$ the two generators in
  dimension $2n$ coming from the two $\Sn{2n-1}$ terms.
  The three maps $\contreq$, $\inl$ and $\inr$ of type $\Sn{2n-1}\to\Sn{2n-1}\vee\Sn{2n-1}$ induce
  three maps $q:C_{f+g}\to C_{f\vee g}$, $j_f:C_f\to C_{f\vee g}$ and $j_g:C_g\to C_{f\vee g}$
  satisfying
  \begin{align*}
    q^*(\alpha_{f\vee g}) &= \alpha_{f+g}, & j_f^*(\alpha_{f\vee g}) &= \alpha_f, & j_g^*(\alpha_{f\vee g}) &= \alpha_g,\\
    q^*(\beta_{f\vee g}^f) &= \beta_{f+g}, & j_f^*(\beta_{f\vee g}^f) &= \beta_f, & j_g^*(\beta_{f\vee g}^f) &= 0,\\
    q^*(\beta_{f\vee g}^g) &= \beta_{f+g}, & j_f^*(\beta_{f\vee g}^g) &= 0, & j_g^*(\beta_{f\vee g}^g) &= \beta_g.
  \end{align*}
  Given that the square of $\alpha_{f\vee g}$ is an element of $H^{2n}(C_{f\vee g})$, it is a linear
  combination of $\beta_{f\vee g}^f$ and $\beta_{f\vee g}^g$ i.e.\ there are two natural numbers $x$
  and $y$ such that
  \[\alpha_{f\vee g}^2 = x\beta_{f\vee g}^f + y\beta_{f\vee g}^g.\]
  By applying $j_f^*$ to this equation we obtain $\alpha_f^2=x\beta_f$, hence $x=H(f)$. Similarly
  we get $y=H(g)$.  Therefore, we can compute the square of $\alpha_{f+g}$:
  \begin{align*}
    \alpha_{f+g}^2 &= q^*(\alpha_{f\vee g})^2\\
                   &= q^*(\alpha_{f\vee g}^2)\\
                   &= q^*(H(f)\beta_{f\vee g}^f + H(g)\beta_{f\vee g}^g)\\
                   &= (H(f) + H(g))\beta_{f+g}.
  \end{align*}
  This shows that $H(f+g)=H(f)+H(g)$ and that $H$ is a group homomorphism.
\end{proof}

\begin{proposition}
  If $n$ is even, then the Hopf invariant of the map $\fold_{\Sn n}\circ W_{n,n}:\Sn{2n-1}\to\Sn{n}$
  is equal to $2$.
\end{proposition}

\begin{proof}
  We consider the space $C_{\fold_{\Sn n}\circ W_{n,n}}$ and we write
  $\alpha\defeq\alpha_{\fold_{\Sn n}\circ W_{n,n}}$ and $\beta\defeq\beta_{\fold_{\Sn n}\circ
    W_{n,n}}$ for short.
  We saw in proposition \ref{kerneljames} that $C_{\fold_{\Sn n}\circ W_{n,n}}$ is equivalent to the
  pushout
  \[
  \begin{tikzcd}
    \Sn n\vee\Sn n \arrow[d]\arrow[r]&\Sn n\times\Sn n\arrow[d,dotted,"q"]\\
    \Sn n\arrow[r,dotted,"i",swap]&J_2(\Sn n) \arrow[lu,phantom,"\ulcorner",at start]
  \end{tikzcd}
  \]
  via the sequence of equivalences
  \begin{align*}
    C_{\fold\circ W_{n,n}} &\simeq \Unit\sqcup^{\Sn{2n-1}}\Sn n\\
                           &\simeq (\Unit\sqcup^{\Sn{2n-1}}(\Sn n\vee\Sn n))\sqcup^{\Sn n\vee\Sn n}\Sn n\\
                           &\simeq (\Sn n\times\Sn n)\sqcup^{\Sn n\vee\Sn n}\Sn n\\
                           &\simeq J_2(\Sn n).
  \end{align*}
  This shows that the map $i^*:H^n(J_2(\Sn n))\to H^n(\Sn n)$ sends $\alpha$ to $\cc_n$ and the map
  $q^*:H^{2n}(J_2(\Sn n))\to H^{2n}(\Sn n\times\Sn n)$ sends $\beta$ to $\xx\cupp\yy$, where $\xx$
  and $\yy$ are the two generators of $H^n(\Sn n\times\Sn n)$ given by proposition
  \ref{cohomologyproductspheres}.
  
  The composition of $q:\Sn n\times \Sn n\to J_2(\Sn n)$ with any of the two inclusions
  $\Sn n\to \Sn n\times \Sn n$ is homotopic to the map $i:\Sn n\to J_2(\Sn n)$. Therefore,
  $q^*(\alpha)=\xx+\yy:H^n(\Sn n\times\Sn n)$. Moreover we have $\xx\cupp\yy=\yy\cupp\xx$ by
  graded-commutativity, because we assumed $n$ even, and then
  \begin{align*}
    q^*(\alpha^2) &= q^*(\alpha)^2\\
                  &= (\xx+\yy)\cupp(\xx+\yy)\\
                  &= (\xx\cupp\xx)+(\xx\cupp\yy)+(\yy\cupp\xx)+(\yy\cupp\yy)\\
                  &= 2(\xx\cupp\yy)\\
                  &= 2q^*(\beta)\\
                  &= q^*(2\beta).
  \end{align*}
  Finally, we know that $q^*$ is an equivalence on $H^{2n}$, which shows that $\alpha^2=2\beta$ and
  that the Hopf invariant of the map $\fold_{\Sn n}\circ W_{n,n}:\Sn{2n-1}\to\Sn{n}$ is equal to $2$.
\end{proof}

This allows us to define new non-trivial elements in the homotopy groups of spheres.

\begin{proposition}
  For every $n\ge1$, the group $\pi_{4n-1}(\Sn{2n})$ is infinite.
\end{proposition}

\begin{proof}
  We define a group homomorphism $h$ from $\Z$ to $\pi_{4n-1}(\Sn{2n})$ sending $1$ to the element
  corresponding to the function $\fold\circ W_{2n,2n} : \Sn{4n-1}\to\Sn{2n}$. For every $n:\Z$ we
  have $H(h(n))=2n$, hence all the $h(n)$ are different, which shows that $\pi_{4n-1}(\Sn{2n})$ is
  infinite.
\end{proof}

\begin{proposition}
  The natural number $n$ satisfying $\pi_4(\Sn 3)\simeq\Z/n\Z$ is equal to either $1$ or $2$.
\end{proposition}

\begin{proof}
  By definition of $n$ in proposition \ref{firstpi4s3}, we have $[i_2,i_2]=\pm n\eta$ in $\pi_3(\Sn2)$,
  for $\eta$ a generator of $\pi_3(\Sn2)$. Applying the Hopf invariant to this equality we
  get
  \[2=H([i_2,i_2])=H(\pm n\eta)=\pm nH(\eta),\]
  which shows that $n$ divides $2$ and that, therefore, it is equal to either $1$ or $2$.
\end{proof}

What is missing to prove that $n$ is equal to $2$ is to know whether there is an element of
$\pi_3(\Sn2)$ of Hopf invariant $1$. More exactly we have the following.

\begin{cor}\label{corcohomology}
  We have $\pi_4(\Sn 3)\simeq\Z/2\Z$ if and only if there exists a map $\Sn3\to\Sn2$ of Hopf invariant
  $\pm1$. Otherwise $\pi_4(\Sn 3)$ is the trivial group.
\end{cor}

\begin{proof}
  If $\pi_4(\Sn 3)\simeq\Z/2\Z$, then the computation above shows that any generator of
  $\pi_3(\Sn2)$ has Hopf invariant $\pm1$. Conversely, if there is a map $\eta:\Sn3\to\Sn2$ of Hopf
  invariant $\pm1$, then the Hopf invariant homomorphism $H:\pi_3(\Sn2)\to\Z$ sends $\eta$ to
  $\pm1$, which shows that $\eta$ is a generator of $\pi_3(\Sn2)$. The computation above then shows
  that $\pi_4(\Sn 3)\simeq\Z/2\Z$.
\end{proof}


\chapter{The Gysin sequence}\label{ch:gysin}

We now present a tool called \emph{the Gysin sequence}, which gives some information on the cup
product structure of a space, given a fibration of spheres over it. More precisely, given a
fibration
\[
\begin{tikzcd}
  \Sn{n-1}\arrow[r]& E\arrow[r,"p"]& B
\end{tikzcd}
\]
where $B$ is $1$-connected, we prove that there is an element $e: H^n(B)$ and a long exact
sequence
\[
\begin{tikzcd}
  \dots\arrow[r]& H^{i-1}(E) \arrow[r]& H^{i-n}(B) \arrow[r,"\cupp e"]& H^i(B) \arrow[r,"p^*"]&
  H^i(E) \arrow[r]& \dots,
\end{tikzcd}
\]
where the middle arrow is the operation of cup product by $e$.  We then define the complex
projective plane $\CP2$ using the Hopf map, and we construct a fibration of circles over $\CP2$ with
total space $\Sn5$. The Gysin sequence allows us to compute the cohomology of $\CP2$, to show that
the Hopf invariant of the Hopf map is $\pm1$, and, therefore, that $\pi_4(\Sn3)\simeq\Z/2\Z$.

The projective spaces $\CP n$ have also independently been defined in homotopy type theory by Ulrik
Buchholtz and Egbert Rijke.

\section{The Gysin sequence}

The following proposition is the main ingredient in the construction of the Gysin sequence. It
relates the cup product in dimensions $(n,m)$ with the one in dimensions $(n+1,m)$.

\begin{proposition}\label{stepcupp}
  Given $n,m:\N$, $p:K_n$ and $y:K_m$, we have the equality
  \[\ap{\lambda x.x\cupp y}(\sigma_n(p)) = \sigma_{n+m}(p\cupp y).\]
\end{proposition}
Note that the cup product on the left is the one in dimensions $(n+1,m)$ while the one on the right
is the one in dimensions $(n,m)$.

\begin{proof}
  For $n=0$, we have $p:\Z$ and we write $k\defeq p$ in order to remember that it is an integer. We
  then have $\sigma_0(p)=\ap{|-|}(\lloop)^k$ and what we want to prove is
  \[\ap{\lambda x.x\cupp y}(\ap{|-|}(\lloop)^k) = \sigma_m(ky).\]
  The left-hand side is equal to $\ap{\lambda x.|x|\cupp y}(\lloop)^k$ and the right-hand side is
  equal to $\sigma_m(y)^k$. Therefore it is enough to prove that
  $\ap{\lambda x.|x|\cupp y}(\lloop)=\sigma_m(y)$ for all $y:\trunc{m}{\Sn m}$. What we have to
  prove is an $(m-1)$-type, so we can assume that $y$ is of the form $|z|$ for $z:\Sn m$. Let’s
  consider the diagram
  \[
  \begin{tikzcd}
    \Omega(\Sn1\wedge\Sn m) \arrow[r] \arrow[d] & \Omega\Sn{m+1} \arrow[d]\\
    \Omega(K_1\wedge K_m) \arrow[r]& \Omega K_{m+1}
  \end{tikzcd}
  \]
  which is obtained by looping the diagram defining the cup product in definition \ref{squarecupp}.
  If we start with $\ap{\proj(-,z)}(\lloop)$ on the top left and go down and then right, we obtain
  $\ap{\lambda x.|x|\cupp|z|}(\lloop)$ while going right and then down gives
  $\ap{|-|_m}(\varphi_{\Sn m}(z))$ which is equal to $\sigma_m(|z|)$.

  For $n>0$, we want to fill the front face of the diagram
  \[
  \begin{tikzcd}[row sep=small, column sep=tiny]
    &\Sn n\wedge\Sn m \arrow[dd]\arrow[rr]\arrow[ld] && \Sn{n+m}\arrow[dddd] \arrow[ld] \\
    K_n\wedge K_m\arrow[rr,crossing over,"\cupp" near end] \arrow[dd] && K_{n+m} &\\
    &(\Omega\Sigma\Sn n)\wedge\Sn m \arrow[dd] \arrow[ld] &&\\
    \Omega K_{n+1}\wedge K_m &&&\\
    &\Omega((\Sigma\Sn n)\wedge\Sn m) \arrow[rr] \arrow[ld] && \Omega\Sigma\Sn{n+m} \arrow[ld] \\
    \Omega(K_{n+1}\wedge K_m) \arrow[from=uu,crossing over] \arrow[rr,"\Omega\cupp"] && \Omega
    K_{n+m+1} \arrow[from=uuuu,crossing over]&
  \end{tikzcd}
  \]
  where the two vertical maps of the form $\Omega A\wedge B\to\Omega(A\wedge B)$ send $(p,b)$ to
  $\ap{\proj(-,b)}(p)$.
  The top and bottom squares are filled by definition of the cup product.  The rectangle on the
  right and the upper-left square are filled by definition of $\sigma_n$.  The lower-left square is
  filled by naturality of the family of maps $\Omega A\wedge B\to\Omega(A\wedge B)$.  The back
  rectangle is filled by proposition \ref{smashspheresexplicit}.

  Therefore, given that the map $\Sn n\wedge\Sn m\to K_n\wedge K_m$ is $(n+m)$-connected and that
  what we want to prove is $(n+m)$-truncated, the result follows.  
\end{proof}

The main technical result of this section is the following.
\begin{proposition}\label{thomfst}
  Given a connected pointed type $B$, a family of pointed types $Q:B\to\Type_\star$ (i.e. a family
  of types equipped with a section) such that $Q(\star_B)=\Sn{n}$, and a map
  \[c:(b : B) \to (Q(b)\topt K_n)\]
  such that $c_{\star_B}:\Sn{n}\topt K_n$ is a generator of $(\Sn{n}\topt K_n)\simeq\Z$,
  then the map
  \begin{align*}
    \Phi &: H^i(B) \to \trunc0{(b:B)\to(Q(b)\topt K_{i+n})},\\
    \Phi(|\beta|) &\defeq |\lambda b,x.\beta(b)\cupp c_b(x)|
  \end{align*}
  is an equivalence.
  Moreover, such a $c$ always exists if $B$ is 1-connected.
\end{proposition}
\begin{proof}
  The result is actually true for the untruncated map
  \begin{align*}
    \Phit&:(B \to K_i)\to((b:B)\to(Q(b) \topt K_{i+n})),\\
    \Phit(\beta)&\defeq \lambda b,x. \beta(b)\cupp c_b(x)
  \end{align*}
  and for all the fiberwise maps
  \begin{align*}
    g^i_b&:K_i\to (Q(b)\topt K_{i+n}),\\
    g^i_b(y)&\defeq \lambda x.y\cupp c_b(x).
  \end{align*}
  Let’s prove that all the $g^i_b$ are equivalences. Given that $B$ is connected, it is enough to
  do it for $b=\star_B$. We proceed by induction on $i$.
  For $i=0$, we have $K_0 = \Z$, and
  \begin{align*}
    g^0_{\star_B}:\Z &\to (\Sn n \topt K_n)\\
    k &\mapsto (\lambda y. k\cupp c_{\star_B}(y))\\
                     & \quad= \lambda y. k(c_{\star_B}(y))\\
                     & \quad= k(\lambda y. c_{\star_B}(y))\\
                     & \quad= kc_{\star_B},
  \end{align*}
  which is an equivalence by assumption.
  
  We now assume that $g^i_{\star_B}$ is an equivalence, and we want to prove that
  $g^{i+1}_{\star_B}$ is also an equivalence. Note first that $K_{i+1}$ and $\Sn{n}\topt K_{i+1+n}$
  are both pointed and connected, hence it’s enough to prove that $\Omega g^{i+1}_{\star_B}$ is an
  equivalence in order to show that $g^{i+1}_{\star_B}$ is, by \cite[theorem 8.8.1]{hottbook}.  We
  will also use the fact that given a function $f:A\to B\to C$ and a path $p:b=b'$, there is an
  equality
  \[\ap{\lambda y.(\lambda x.f(x,y))}(p) = \funext(\lambda x.\ap{\lambda y.f(x,y)}(p))\]
  in the type $(\lambda x.f(x,b))=(\lambda x.f(x,b'))$. This equality is proven by path induction on
  $p$.
  
  Given $p:\Omega K_{i+1}$, we have
  \begin{align*}
    (\Omega g^{i+1}_{\star_B})(p) &= \ap{g^{i+1}_{\star_B}}(p)\\
                                  &= \ap{\lambda y.(\lambda x. y\cupp c_{\star_B}(x))}(p)\\
                                  &= \funext(\lambda x.\ap{\lambda y.y\cupp c_{\star_B}(x)}(p))\\
                                  &= \funext(\lambda x.\ap{\lambda y.y\cupp c_{\star_B}(x)}(\sigma_n(\sigma\inv_n(p))))\\
                                  &= \funext(\lambda x.\sigma_{n+m}(\sigma\inv_n(p)\cupp
                                    c_{\star_B}(x)))\quad\text{ by proposition \ref{stepcupp}}\\
                                  &= \funext(\sigma_{n+m}\circ g^i_{\star_B}(\sigma\inv_n(p))).
  \end{align*}
  Therefore, the map $\Omega g^{i+1}_{\star_B}$ is equal to the composition
  \[(\Omega g^{i+1}_{\star_B}) = \funext \circ (\lambda f.\sigma_{n+m}\circ f) \circ g^i_{\star_B} \circ
  \sigma\inv_n.\]
  All those functions are equivalences, hence $\Omega g^{i+1}_{\star_B}$ and then $g^{i+1}_{\star_B}$ are
  equivalences as well.
  We have proved that $g^i_b$ is an equivalence. We now have
  \[\Phit(\beta)= \lambda b. g^i_b(\beta(b)),\]
  therefore $\Phit$ is an equivalence of inverse
  \[\Phit\inv(\beta')= \lambda b. (g^i_b)\inv(\beta'(b)),\]
  and $\Phi$ is an equivalence as well.

  Let’s now prove that $c$ exists whenever $B$ is $1$-connected. The type $(Q(b)\topt K_n)$ is a set
  for $b={\star_B}$, hence it’s a set for every $b:B$, given that $B$ is $0$-connected. Now, if $B$
  is $1$-connected, the map $\Unit\to B$ is $0$-connected, so it’s enough to define $c$ on the base
  point of $B$, and in this case we choose for $c_{\star_B}$ any generator of
  $(\Sn n\topt K_n)\simeq\Z$.
\end{proof}

We now apply this result to the special case of the fibration obtained by taking the fiberwise
suspension of a family of spheres. This gives what is known as the \emph{Thom isomorphism}.

Let $B$ be a 0-connected pointed type and $P:B\to\Type$ a fibration over $B$ such that
$P({\star_B})=\Sn{n-1}$. We define $E\defeq\sum_{x:B}P(x)$ its total space and
$\tildeE\defeq\Unit\sqcup^EB$ the cofiber of the map $E\to B$, and we have a map $i:\Sn n\to\tildeE$
corresponding to the fiber $\Sn{n-1}$ over $\star_B$. We also define a family of pointed types
$Q:B\to\Type_{\star}$ by $Q(b)\defeq\Susp(P(b))$ pointed by the north pole,
$F\defeq\sum_{x:B}Q(x)$ its total space and $\tildeF\defeq\Unit\sqcup^BF$ the cofiber of the map
$B\to F$ sending $b$ to $(b,\star_{Q(b)})$. It is easy to see that $\tildeE$ and $\tildeF$ are
equivalent, with all identified north poles in $\tildeF$ corresponding to the basepoint in $\tildeE$
and all south poles corresponding to the $B$ term in $\tildeE$.  For every pointed type $A$ we have
an equivalence
\[(\tildeF\topt A) \simeq ((b:B)\to (Q(b) \topt A))\]
because a pointed map from $\tildeF$ to $A$ corresponds to a family of pointed maps from all $Q(b)$
to $A$. Therefore, we have an equivalence
\[\iota_{k}:((b:B)\to (Q(b) \topt K_k))\simeq(\tildeE\topt K_k)\]
for every $k:\N$.
\begin{proposition}[Thom isomorphism]
  Given a map $c:(b:B)\to(Q(b)\topt K_n)$ such that $c_{\star_B}$ is a generator of $H^n(\Sn n)$,
  the map
  \begin{align*}
    \Phi &: H^i(B) \to \Ht^{i+n}(\tildeE),\\
    \Phi(|\beta|) &\defeq |\iota_{i+n}(\lambda b,x.\beta(b)\cupp c_b(x))|
  \end{align*}
  is an isomorphism of groups.  Moreover, such a $c$ always exists if $B$ is 1-connected.
\end{proposition}

\begin{proof}
  Applying proposition \ref{thomfst} to $Q$, we obtain that $\Phi$ is an equivalence and that $c$
  always exists if $B$ is $1$-connected. The map $\Phi$ is a group homomorphism by distributivity of
  the cup product and because $\iota_{i+n}$ preserves the group structure.
\end{proof}

We can finally construct the Gysin sequence. Let’s consider a $1$-connected pointed type $B$ and a
fibration $P:B\to\Type$ such that $P(\star_B)=\Sn{n-1}$. We write $E$ for the total space of $P$,
$p:E\to B$ for the projection and $\tildeE$ for the cofiber of $p$. The long exact sequence of the
cofiber of $p$ is
\[
\begin{tikzcd}
  \dots \arrow[r]& \Ht^i(\tildeE) \arrow[r,"j^*"]& H^i(B) \arrow[r,"p^*"]& H^i(E) \arrow[r]&
  \Ht^{i+1}(\tildeE) \arrow[r]& \dots,
\end{tikzcd}
\]
where $j$ is the inclusion $B\to\tildeE$.
The Thom isomorphism gives a map $c:(b:B)\to(Q(b)\topt K_n)$ together with an isomorphism
$\Phi:H^{i-n}(B)\simeq\Ht^i(\tildeE)$, and the induced map $H^{i-n}(B)\to H^i(B)$ is
\begin{align*}
  H^{i-n}(B)&\to H^i(B),\\
  |\beta| &\mapsto j^*(\Phi(|\beta|))\\
            &= j^*(|\iota_{i+n}(\lambda b,x.\beta(b)\cupp c_b(x))|)\\
            &= |\lambda b'.\iota_{i+n}(\lambda b,x.\beta(b)\cupp c_b(x))(j(b'))|\\
            &= |\lambda b'.(\lambda b,x.\beta(b)\cupp c_b(x))(b',\south_{Q(b')})|\\
            &= |\lambda b'.\beta(b')\cupp c_{b'}(\south_{Q(b')})|\\
            &= |\beta|\cupp e,
\end{align*}
where
\begin{align*}
  e &: H^n(B),\\
  e &\defeq |\lambda b.c_b(\south_{Q(b)})|.
\end{align*}
Therefore we get the Gysin sequence.
\begin{proposition}[Gysin sequence]
  There is an element $e:H^n(B)$ and a long exact sequence
  \[
  \begin{tikzcd}
    \dots\arrow[r]& H^{i-1}(E) \arrow[r]& H^{i-n}(B) \arrow[r,"\cupp e"]& H^i(B) \arrow[r,"p^*"]&
    H^i(E) \arrow[r]& \dots
  \end{tikzcd}\qed
  \]
\end{proposition}

\section{The iterated Hopf construction}

We saw in chapter \ref{ch:hopf} that to every H-space $A$ we can associate a fibration over
$\Susp A$ with fiber $A$ and total space $A*A$. We now prove that if we assume additionally that the
H-space structure is associative, then we can iterate this construction once, by constructing a
fibration over $\Unit\sqcup^{A*A}\Susp A$ with fiber $A$ and total space $A*A*A$.

\begin{definition}
  An \emph{associative H-space} is an H-space $(A,\mu)$ equipped with a map
  \[\alpha:(x,y,z:A)\to \mu(\mu(x,y),z) = \mu(x,\mu(y,z))\]
  and a filler of the diagram
  \[
  \begin{tikzcd}
    \mu(\mu(\star_A,y),z) \arrow[rr,"{\alpha_{\star_A,y,z}}"] \arrow[rd,"{\ap{\mu(-,z)}(\mu_l(y))}"'] && \mu(\star_A,\mu(y,z)) \arrow[ld,"{\mu_l(\mu(y,z))}"] \\
    &\mu(y,z)&
  \end{tikzcd}
  \]
\end{definition}
If $A$ is connected, then the three other triangles follow from this one, but we do not need them
here.

\begin{proposition}
  Given a connected associative H-space $A$, there is a fibration over $\Unit\sqcup^{A*A}\Susp A$
  (where the map $A*A\to\Susp A$ is induced by the Hopf construction on $A$) with fiber $A$, and
  whose total space is equivalent to $A*A*A$.
\end{proposition}

\begin{proof}
  We recall that the Hopf construction is the fibration $H:\Susp A\to\Type$ defined by
  \begin{align*}
    H(\north) &\defeq A,\\
    H(\south) &\defeq A,\\
    \ap H(\merid(y)) &\defeq \ua(\mu(-,y)).
  \end{align*}
  Its total space is equivalent to the pushout $A\sqcup^{A\times A}A$ of the span
  \[
  \begin{tikzcd}
    A & A\times A \arrow[l,"\fst"'] \arrow[r,"\mu"] &A
  \end{tikzcd}
  \]
  via the map
  \begin{align*}
    \psi &: A\sqcup^{A\times A}A\to\sum_{x:\Susp A}H(x),\\
    \psi(\inl(x)) &\defeq (\north,x),\\
    \psi(\inr(z)) &\defeq (\south,z),\\
    \ap\psi(\push(x,y)) &\defeq (\merid(y),\psi_{\push}(x,y)),
  \end{align*}
  where $\psi_{\push}(x,y):x=^H_{\merid(y)}\mu(x,y)$ comes from the fact that
  $\transport^H(\merid(y),x)=\mu(x,y)$. The induced map $h:A\sqcup^{A\times A}A\to\Susp A$ is
  defined by
  \begin{align*}
    h(\inl(x)) &\defeq \north,\\
    h(\inr(z)) &\defeq \south,\\
    \ap h(\push(x,y)) &\defeq \merid(y).
  \end{align*}
  We first prove the following proposition.
  \begin{proposition}
    There is a map $\nu:(x:A\sqcup^{A\times A}A)\to A\simeq H(h(x))$ such that for every $a:A$,
    the map
    \begin{align*}
      \nu'_a &: A\sqcup^{A\times A}A \to \sum_{x:\Susp A}H(x),\\
      \nu'_a(x) &\defeq (h(x), \nu(x,a))
    \end{align*}
    is an equivalence.
  \end{proposition}

  \begin{proof}
    We define $\nu$ by
    \begin{align*}
      \nu(\inl(x)) &\defeq \mu(-,x),\\
      \nu(\inr(z)) &\defeq \mu(-,z),\\
      \ap\nu(\push(x,y)) &\defeq \nu_{\push}(x,y),\\
    \end{align*}
    where $\nu_{\push}(x,y)$ is a filling of the square
    \[
    \begin{tikzcd}
      A \arrow[r,"\id_A"] \arrow[d,"{\mu(-,x)}"'] & A \arrow[d,"{\mu(-,\mu(x,y))}"]\\
      A \arrow[r,"{\mu(-,y)}"'] & A
    \end{tikzcd}
    \]
    In other words, $\nu_{\push}(x,y)$ proves that for every $a:A$, we have
    \[\mu(\mu(a,x),y) = \mu(a,\mu(x,y)),\]
    which we have by associativity of the H-space structure.

    In order to prove that $\nu'_a$ is an equivalence for every $a:A$, it is enough to do it in the
    case $a=\star_A$ because $A$ is connected. We now show that $\nu'_{\star_A}$ is equal to the
    function $\psi:A\sqcup^{A\times A}A\to\sum_{x:\Sigma A}H(x)$ by induction on the argument.
    \begin{itemize}
    \item For $\inl(x)$, we have
      \[\nu'_{\star_A}(\inl(x))=(\north,\mu(\star_A,x))\stackrel{\text{via
        }\mu_l(x)}{=}(\north,x)=\psi(\inl(x)).\]
    \item For $\inr(z)$, we have
      \[\nu'_{\star_A}(\inr(z))=(\south,\mu(\star_A,z))\stackrel{\text{via
        }\mu_l(z)}{=}(\south,z)=\psi(\inr(z)).\]
    \item For $\push(x,y)$, using the fact that $\ap H(\merid(y))=\mu(-,y)$ we have to give a
      filling of the square
      \[
      \begin{tikzcd}
        \mu(\mu(\star_A,x),y) \arrow[r] \arrow[d] & \mu(\star_A,\mu(x,y)) \arrow[d] \\
        \mu(x,y) \arrow[r,equals] & \mu(x,y)
      \end{tikzcd}
      \]
      which follows from the fact that $A$ is an associative H-space.\qedhere
    \end{itemize}
  \end{proof}
  
  We now define $P:\Unit\sqcup^{A\sqcup^{A\times A}A}\Susp A\to\Type$ by
  \begin{align*}
    P(\inl(\ttt)) &\defeq A,\\
    P(\inr(y)) &\defeq H(y),\\
    \ap P(\push(x)) &\defeq \nu(x).
  \end{align*}
  The total space of $P$ is equivalent to the pushout of the span
  \[
  \begin{tikzcd}
    A & A \times (A\sqcup^{A\times A}A) \arrow[l,"\fst"'] \arrow[rr,"{(a,x)\mapsto\nu'_a(x)}"] && \displaystyle\sum_{x:\Susp A}H(x),
  \end{tikzcd}
  \]
  and we have the equivalence of spans
  \[
  \begin{tikzcd}
    A \arrow[d,"\id"'] & A \times (A\sqcup^{A\times A}A) \arrow[l,"\fst"']
    \arrow[rr,"{(a,x)\mapsto\nu'_a(x)}"] \arrow[d,"{(a,x)\mapsto(a,\nu'_a(x))}"] &&
    \displaystyle\sum_{x:\Susp A}H(x) \arrow[d,"\id"]\\
    A & A \times \displaystyle\sum_{x:\Susp A}H(x) \arrow[l,"\fst"] \arrow[rr,"\snd"'] && \displaystyle\sum_{x:\Susp A}H(x)
  \end{tikzcd}
  \]
  where the inverse of the middle map is the map $(a,y)\mapsto(a,{\nu'}_a^{-1}(y))$.  Therefore, the
  total space of $P$ is equivalent to the join of $A$ and $\sum_{x:\Susp A}H(x)$. But we already
  know that the total space of $H$ is equivalent to $A*A$; hence, the total space of $P$ is
  equivalent to $A*A*A$.
\end{proof}

\section{The complex projective plane}

\begin{proposition}\label{hspacecircleassoc}
  The H-space structure on the circle is associative.
\end{proposition}

\begin{proof}
  We want to prove that for every $x,y,z:\Sn1$, we have $\mu(\mu(x,y),z)=\mu(x,\mu(y,z))$.
  We proceed first by induction on $x$.

  For $\base$, we have $\mu(\mu(\base,y),z)=\mu(y,z)=\mu(\base,\mu(y,z))$ by definition, so we can
  just use the constant path. For $\lloop$, we have to prove that
  $\ap{\mu(-,z)}(\lloop_y)=\lloop_{\mu(y,z)}$. We proceed again by induction on $y$.

  For $\base$ we have $\ap{\mu(-,z)}(\lloop_{\base})=\ap{\mu(-,z)}(\lloop)=\lloop_z$ and
  $\lloop_{\mu(\base,z)}=\lloop_z$, so it’s true, and for $\lloop$ we have to construct a
  3-dimensional path in $\Sn1$ and $\Sn1$ is $1$-connected, so it’s immediate.
  Moreover the commutativity of the triangle is immediate because both $\alpha_{\base,y,z}$ and
  $\mu_l(-)$ are constant paths.
\end{proof}

We now define the complex projective plane as $\CP2\defeq\Unit\sqcup^{\Sn3}\Sn2$, where the map
$\Sn3\to\Sn2$ is the Hopf map. Its cohomology groups are $\Z$ in dimensions $0$, $2$ and $4$, and
the trivial group otherwise, and the iterated Hopf construction on the circle gives us a fibration
$K$ over $\CP2$ with fiber $\Sn1$ and total space $\Sn1*\Sn1*\Sn1$, which is equivalent to $\Sn5$.

\begin{proposition}
  The Hopf invariant of the Hopf map $\Sn3\to\Sn2$ is equal to $\pm1$.
\end{proposition}

\begin{proof}
  We apply the Gysin sequence to the fibration over $\CP2$ defined by the iterated Hopf construction
  as above. We obtain a cohomology class $e:H^2(\CP2)$ and the two short exact sequences
  \[
  \begin{tikzcd}
    0\simeq H^1(\Sn5) \arrow[r]& H^0(\CP2) \arrow[r,"\cupp e"]& H^2(\CP2) \arrow[r]& H^2(\Sn5)\simeq0,\\
    0\simeq H^3(\Sn5) \arrow[r]& H^2(\CP2) \arrow[r,"\cupp e"]& H^4(\CP2) \arrow[r]& H^4(\Sn5)\simeq0.
  \end{tikzcd}
  \]
  From the first one we see that $e$ is a generator of $H^2(\CP2)$, and from the second one we see
  that $e\cupp e$ is a generator of $H^4(\CP2)$. Therefore, the Hopf invariant of the Hopf map is
  equal to $\pm1$.
\end{proof}

We have found a map $\Sn3\to\Sn2$ of Hopf invariant $\pm1$, so we can conclude.

\begin{cor}\label{pi4s3}
  We have $\pi_4(\Sn3)\simeq\Z/2\Z$ and more generally $\pi_{n+1}(\Sn n)\simeq\Z/2\Z$ for every
  $n\ge3$.
\end{cor}

\begin{proof}
  We apply corollary \ref{corcohomology} to the Hopf map, which is of Hopf invariant $\pm1$, and we
  obtain that $\pi_4(\Sn3)\simeq\Z/2\Z$. We then apply the Freudenthal suspension theorem, more
  precisely corollary \ref{pinsn}, and we obtain that $\pi_{n+1}(\Sn n)\simeq\Z/2\Z$ for every
  $n\ge3$.
\end{proof}


\cleardoublepage
\chaptertoc{Conclusion}

In this thesis we saw that homotopy type theory is powerful enough to prove that
$\pi_4(\Sn3)\simeq\Z/2\Z$. Even though from the point of view of classical homotopy theory this is a
well-known result, it was not obvious that the univalence axiom and higher inductive types would
suffice to prove it and that a constructive and purely homotopy-theoretic proof exists at
all. Moreover, taking into account the fact that it has been only about five years between the
definition of the circle in homotopy type theory and the computation of $\pi_4(\Sn3)$, the progress
has been rather quick, and one can hope that in a few years homotopy type theory will have reached a
level comparable to that of classical homotopy theory and will help to obtain completely new
results.

\paragraph{Comparison with classical proofs}

As mentioned before, the main difference between classical homotopy theory and homotopy type theory
is that, in homotopy type theory, everything is homotopy-invariant. I have used the book
\cite{hatcher}, presenting classical algebraic topology, quite regularly during this research, but I
often had to find completely different definitions, proofs or statements that would work in homotopy
type theory.

In the construction of the universal cover of the circle and of the Hopf fibration, the most obvious
difference is that, instead of defining a map from the total space to the base space, we define
directly the fibers and how to transport along the fibers, and determining the total space is the
non-trivial part. For the universal cover of the circle it is rather transparent, but for the Hopf
fibration it wasn’t clear to me at first that defining it with the multiplication on $\Sn1$ would
indeed give the Hopf fibration.

Proving that the smash product is associative isn’t easy, neither in homotopy type theory nor in
classical homotopy theory, but for very different reasons. In homotopy type theory the problem comes
from the fact that we have many different paths and higher paths to manage and that it quickly
becomes complicated to take care of all coherences between them. In point-set topology, however,
there is already a canonical bijection between the two spaces, but the problem is that it may not be
continuous unless we assume that the spaces are nice enough. The difference is that in point-set
topology we literally identify various points together, which makes the bijection easy to define but
might mess up the topology, whereas in homotopy type theory we add new paths instead of identifying
points.

For cohomology we already mentioned the need to use Eilenberg--MacLane spaces instead of the
classical definition via singular cochains, as the set of singular cochains is not a homotopy
invariant of a space. The notion of truncations gives a very nice definition of the
Eilenberg--MacLane spaces $K(\Z,n)$ and a very nice definition of the cup product. Note that, when
working constructively, there are some unexpected phenomena related to cohomology. Indeed, without
some version of the axiom of choice, it is not possible to prove the additivity axiom of
cohomology. Even the cohomology group $H^1(\N,\Z)$ cannot be proved to be trivial, as explained in
\cite{mikeblog:h1n}. It is nice to see that no such problem arises in the computation of
$\pi_4(\Sn3)$.

Finally, for computing the cohomology of $\CP2$, I couldn’t manage to adapt the very geometrical
proof presented for instance in \cite[theorem 3.19]{hatcher} and I had to use the Gysin sequence,
which is usually considered a more advanced result in classical homotopy theory. Moreover, the
construction of the Gysin sequence in \cite[section 4.D]{hatcher} is based on the Leray--Hirsch
theorem whose proof uses induction on the cells of a CW complex, which is something that cannot be
done in homotopy type theory. Instead, the proof presented here is new and based on proposition
\ref{stepcupp} which directly relates the cup product in dimensions $n$ and $n+1$.

\paragraph{Cubical homotopy type theory}

I first tried to write this thesis using cubical ideas explicitely as much as possible, as was done
for instance in my previous work with Dan Licata in \cite{licatame:cubical} and in
\cite{cavallo:cohomology}, but it turned out not to be such a good idea, much to my disappointment.
Even though many squares and cubes appear naturally in homotopy type theory, like for instance the
naturality squares of homotopies, there are also many other shapes that can appear, and treating
squares differently might not be the wiser choice. For instance trying to write diagram
\ref{eq:deltaipush} on page \pageref{eq:deltaipush} as a Kan composition of squares and cubes is not
very natural or useful. Using instead a general notion of composition of diagrams (as described at
the end of page \pageref{trianglesquare}) seems much more natural in that case. Cubical ideas still
have some benefits, in particular \cite{cubicaltt} uses cubical sets to give a computational
interpretation of homotopy type theory, and in \cite{licatame:cubical} and \cite{cavallo:cohomology}
cubes are used mainly to simplify formalizations in Agda. I believe, however, that for informal
synthetic homotopy theory as done in this thesis, cubical ideas aren’t especially helpful in
general.

\paragraph{Future work}

Maybe the most compelling next target is to formalize the present work in a proof assistant. There
might be some issues in the formalization of the James construction and even more so in the monoidal
structure of the smash product, which are rather heavy in terms of manipulation of paths and higher
paths, but it should be possible to do.

A few results in this thesis have been stated and proved in a restricted form as my main goal was to
arrive at the result $\pi_4(\Sn3)\simeq\Z/2\Z$. For instance one ought to be able to compute the
cohomology ring of all the $J_n(\Sn k)$, including $J_\infty(\Sn k)\simeq\Omega\Sn{k+1}$, and of all
the $\CP n$, including $\CP\infty\simeq K(\Z,2)$. Moreover we only defined cohomology with integer
coefficients but it should be possible to define cohomology with coefficients in an arbitrary group,
ring or spectra and to extend most of the results presented here.

Finally the long-term goal is of course to continue the development of synthetic homotopy theory in
homotopy type theory. There are many concepts that haven’t been much studied yet but seem accessible
like K-theory, Steenrod operations, spectral sequences, Toda brackets, and many others. I will
definitely keep on pursuing this line of research, and I hope this work will inspire other people to
join the study of synthetic homotopy theory as there are so many things waiting to be found.


\appendix

\chapter{A type-theoretic definition of weak
  \texorpdfstring{$\infty$}{∞}-groupoids}
\chaptermark{\slshape\MakeUppercase{A type-theoretic definition of $\infty$-groupoids}}
\label{ch:infgpd}

In this appendix we give a type-theoretic definition of weak $\infty$-groupoids, and we prove that,
using the $\J$ rule of dependent type theory, we can equip every type with a structure of weak
$\infty$-groupoid.  This definition has been partly inspired by Grothendieck’s definition of weak
$\infty$-groupoids as presented by Dimitri Ara and Georges Maltsiniotis, cf for instance
\cite{theseara}. A similar result has also been proved in \cite{vdbgarner:typesinfgpd} where it is
shown that every type in dependent type theory can be equipped with the structure of a weak
$\omega$-groupoid in the sense of Batanin--Leinster. The definition of weak $\infty$-groupoids
presented here is more directly related to type theory. It is not known, however, whether this
definition is in some sense equivalent to the standard one (Kan complexes).

\section{Globular sets}

Just as a category is defined as a graph together with some structure on it, a weak
$\infty$-groupoid is defined as a globular set (or $\infty$-graph) together with some structure on
it. We recall the definition of globular set.

\begin{definition}
  A \emph{globular set} $A$ is a sequence of sets $(A_n)_{n\in\N}$ and maps
  $s_n,t_n:A_{n+1}\to A_n$ (source and target) such that for every $n\in\N$ we
  have:
  \begin{align*}
  s_n\circ s_{n+1} &= s_n\circ t_{n+1},\\
  t_n\circ s_{n+1} &= t_n\circ t_{n+1}.
  \end{align*}
\end{definition}

If $A$ is a globular set, we write $\Ob(A)$ for $A_0$ and we call its elements the \emph{objects} of
$A$. If $x$ and $y$ are two objects of $A$, there is another globular set, called $\Hom_A(x,y)$,
defined by
\begin{align*}
(\Hom_A(x,y))_n=\{\sigma\in
                 A_{n+1}&\,|\,s_0(s_1(\dots(s_n(\sigma))\dots))=x\text{ and }\\
  &\quad t_0(t_1(\dots(t_n(\sigma))\dots))=y\}.
\end{align*}

The (large) set of all globular sets is denoted by $\Glob$.

\section{The internal language of weak \texorpdfstring{$\infty$}{∞}-groupoids}

A weak $\infty$-groupoid is a globular set equipped with a large number of operations, with a few
examples listed in section \ref{sec:infgpd}. In order to describe precisely what are those
operations, the idea is to design a minimalist type theory $\tinfgpd$ from which we can extract the
specification of all the operations. Then we explain how to interpret this type theory into a
globular set, and a weak $\infty$-groupoid will be defined as a globular set together with all the
required operations.

There are four kind of things in $\tinfgpd$: \emph{contexts}, \emph{types}, \emph{terms}, and
\emph{context morphisms} (sequences of terms). We have one base type $\star$, which represents our
ambient weak $\infty$-groupoid, and a type $u\simeq_Tv$ for each type $T$ and terms $u,v:T$, which
represents an iterated hom-set of the ambient weak $\infty$-groupoid. Apart from variables, the only
way to create new terms is to use one of the coherence operations $\coh\Delta T\delta$ that we
describe below. Moreover, there is a special class of contexts called \emph{contractible contexts},
which are used to define the coherence operations.

We use the letters $\Gamma,\Theta$ to represent contexts, $\Delta$ to represent contractible
contexts, $T$, $U$ to represent types, $t$, $u$, $v$ to represent terms, and $\gamma,\theta,\delta$
to represent context morphisms.

The syntax of the language can be summarized by the following grammar.

\begin{center}
  \begin{minipage}[c]{0.3\linewidth}
    \begin{align*}
      \Gamma,\Theta\defeq&\ |\ \varnothing\\
      &\ |\ (\Gamma,x:T),\\\\
      T,U\defeq&\ |\ \star\\
      &\ |\ u \simeq_T v,\\
    \end{align*}
  \end{minipage}
  \begin{minipage}[c]{0.3\linewidth}
    \begin{align*}
      \gamma,\theta\defeq&\ |\ ()\\
      &\ |\ (\gamma,u),\\\\
      t,u,v\defeq&\ |\ x\\
      &\ |\ \coh\Delta T\delta.\\
    \end{align*}
  \end{minipage}
\end{center}

In $\coh\Delta T\delta$, the part $\cohf_{\Delta. T}$ is the name of the coherence operation
(corresponding for instance to “composition of paths” or “exchange law”) while the $\delta$ is the
list of points, paths, and higher paths to which we apply the coherence operation. The part
$\cohf_{\Delta. T}$ should be seen as an atomic symbol.

Given a term, type, or context morphism $X$, a \emph{free variable} of $X$ is a
variable which appears at least once in $X$ in a position which is not a
subscript of the $\cohf$ symbol.  For instance in the term
\[\coh{(x,y:\star,f:x\simeq_\star y)}{y\simeq_\star x}{(y,z,g)},\]
the free variables are $y$, $z$, and $g$, but $x$ and $f$ are not free.

If $\Gamma$ is a context, $\gamma$ is a context morphism having the same length as $\Gamma$ and $X$
is a type, term or context morphism such that every free variable of $X$ is declared in $\Gamma$, we
define $X[\gamma/\Gamma]$ as the type/term/context morphism obtained by replacing every free
variable of $X$ by the corresponding term in $\gamma$.  In particular, for terms of the form
$\coh{\Delta}{T}{\delta}$, substitution must be performed only in $\delta$ and not in $\Delta.T$.

The five judgment forms of $\tinfgpd$ are the following.

\begin{itemize}
\item $\Gamma\ctx$ means that $\Gamma$ is a well-typed context.
\item $\Gamma\vdash\gamma':\Gamma'$ means that $\gamma'$ is a well-typed context morphism from $\Gamma$ to $\Gamma'$.
\item $\Gamma\vdash T\typ$ means that $T$ is a well-typed type in the
  context $\Gamma$.
\item $\Gamma\vdash t:T$ means that $t$ is a well-typed term of type $T$ in the
  context $\Gamma$.
\item $\Delta\contr$ means that the context $\Delta$ is a contractible context.
\end{itemize}

The ten typing rules of $\tinfgpd$ are the following.

\paragraph{Contexts}

\[
\AxiomC{}
\UnaryInfC{$\varnothing\ctx$}
\DisplayProof
\qquad\qquad
\AxiomC{$\Gamma\vdash{}T\typ$}
\UnaryInfC{$(\Gamma,x:T)\ctx$}
\DisplayProof
\quad(x\not\in\Gamma)\]

\paragraph{Context morphisms}

\[
\AxiomC{}
\UnaryInfC{$\Gamma\vdash():\varnothing$}
\DisplayProof
\qquad\qquad
\AxiomC{$\Gamma'\vdash{}T\typ$}
\AxiomC{$\Gamma\vdash\gamma:\Gamma'$}
\AxiomC{$\Gamma\vdash u:T[\gamma/\Gamma']$}
\TrinaryInfC{$\Gamma\vdash(\gamma,u):(\Gamma',x:T)$}
\DisplayProof
\quad(x\not\in\Gamma')\]

\paragraph{Types}

\[
\AxiomC{}
\UnaryInfC{$\Gamma\vdash\star\typ$}
\DisplayProof
\qquad\qquad
\AxiomC{$\Gamma\vdash u:T$}
\AxiomC{$\Gamma\vdash v:T$}
\BinaryInfC{$\Gamma\vdash u\simeq_T v\typ$}
\DisplayProof\]

\paragraph{Terms}

\[
\AxiomC{$\phantom{\vdash}$}
\UnaryInfC{$\Gamma\vdash x:T$}
\DisplayProof
\quad((x:T)\in\Gamma)
\qquad\qquad
\AxiomC{$\Delta\contr$}
\AxiomC{$\Delta\vdash{}T\typ$}
\AxiomC{$\Gamma\vdash\delta:\Delta$}
\TrinaryInfC{$\Gamma\vdash\coh\Delta T\delta:
  T[\delta/\Delta]$}
\DisplayProof\]

\paragraph{Contractible contexts}

\[
\AxiomC{}
\UnaryInfC{$(x:\star)\contr$}
\DisplayProof
\qquad\qquad
\AxiomC{$\Delta\contr$}
\AxiomC{$\Delta\vdash u:T$}
\BinaryInfC{$(\Delta,(y:T),(z:u\simeq_T y))\contr$}
\DisplayProof
\quad(y,z\notin\Delta)
\]

\medskip

All the operations in section \ref{sec:infgpd} give examples of contractible contexts. The idea is
that a coherence operation has a list of arguments, whose shape forms a contractible context
$\Delta$, and a return type $T$, which can be any type in the context $\Delta$ using only coherence
operations, as they are the only terms available in $\tinfgpd$. The definition of weak
$\infty$-groupoid should now be intuitively clear. A weak $\infty$-groupoid is a globular set
together with, for every context $\Delta$ and every type $T$ such that $\Delta\contr$ and
$\Delta\vdash{}T\typ$ hold, an operation taking a list of arguments of types corresponding to
$\Delta$ and returning a result of type corresponding to $T$. Making this definition precise isn’t
as easy as it seems, though.  We first need to prove some syntactic properties of $\tinfgpd$, and
then we will explain how the syntax is translated into actual operations on a globular set.

\begin{lemma}
  [syntactic substitution]

  Let $X$ be a type, a term or a context morphism, $\Gamma$ and $\Theta$ two contexts, $\gamma$ and
  $\theta$ two context morphisms, $T$ a type and $u$ a term. Then the following syntactic equalities
  hold whenever all the objects involved are well-defined (we consider that $X[\gamma/\Gamma]$ is
  not well-defined if there is a variable in $X$ which is not declared in $\Gamma$)
  \begin{align*}(X[\gamma/\Gamma])[\theta/\Theta]&=
    X[\gamma[\theta/\Theta]/\Gamma)]\\
    X[\gamma/\Gamma]&=X[(\gamma,u)/(\Gamma,x:T)]
  \end{align*}
\end{lemma}

\begin{proof}
  We proceed by induction on $X$. If $X$ is a variable it holds by definition of the objects
  involved, otherwise we apply the induction hypothesis.
\end{proof}

\begin{lemma}
  [weakening]

  Let’s assume that $\Gamma\ctx$ holds, for $\Gamma=(\Gamma_0,y:B)$. Then we have
  \begin{itemize}
  \item If $\Gamma_0\vdash T\typ$ holds, then $\Gamma\vdash T\typ$ holds.
  \item If $\Gamma_0\vdash u:T$ holds, then $\Gamma\vdash u:T$ holds.
  \item If $\Gamma_0\vdash\gamma':\Gamma'$ holds, then
    $\Gamma\vdash\gamma':\Gamma'$ holds.
  \end{itemize}
\end{lemma}

\begin{proof}
  We proceed by induction on the typing derivation of the judgment. For the case of a variable we
  use the fact that $\Gamma\ctx$ holds and that if $(x:T)\in\Gamma_0$, then $(x:T)\in\Gamma$. For the
  other cases we apply the induction hypothesis.
\end{proof}

\begin{lemma}\label{typingvars}
  We have

  \begin{itemize}
  \item If $\Gamma\ctx$ holds and $(x:T)\in\Gamma$, then $\Gamma\vdash
    T\typ$ holds.
  \item If $\Gamma'\vdash\gamma:\Gamma$ holds and $(x:T)\in\Gamma$, then
    $\Gamma'\vdash \pi_x^\Gamma(\gamma):T[\gamma/\Gamma]$ where
    $\pi_x^\Gamma$ is the $n$th projection, where $n$ is the position of $x$ in
    $\Gamma$ (or in other words,
    $\pi_x^\Gamma(\gamma)=x[\gamma/\Gamma]$).
  \end{itemize}
\end{lemma}

\begin{proof}
  For the first point, we proceed by induction on the typing derivation of the judgment
  $\Gamma\ctx$. If $\Gamma$ is empty, then $(x:T)\in\Gamma$ cannot be true so there is nothing to
  prove.  If $\Gamma=(\Gamma_0,y:T')$, then if $x$ is equal to $y$, the types $T$ and $T'$ are also
  equal and we have $\Gamma_0\vdash{}T\typ$. If $x$ is not equal to $y$, then $(x:T)\in\Gamma_0$,
  hence $\Gamma_0\vdash{}T\typ$ by the induction hypothesis. In both cases we get
  $\Gamma_0\vdash{}T\typ$, hence $\Gamma\vdash{}T\typ$ holds by the weakening lemma.

  For the second point, we proceed by induction on the typing derivation of the judgment
  $\Gamma'\vdash\gamma:\Gamma$. If $\Gamma$ is empty, then $(x:T)\in\Gamma$ cannot be true so there
  is nothing to prove.  If $\Gamma=(\Gamma_0,y:T')$ and $\gamma=(\gamma_0,u)$, then if $x$ is equal
  to $y$, the terms $\pi_x^\Gamma(\gamma)$ and $u$ are also equal and we have
  $\Gamma'\vdash\pi_x^\Gamma(\gamma):T[\gamma_0/\Gamma_0]$. If $x$ is not equal to $y$, then
  $(x:T)\in\Gamma_0$, hence $\Gamma'\vdash\pi_x^\Gamma(\gamma):T[\gamma_0/\Gamma_0]$. In both cases
  we get $\Gamma'\vdash\pi_x^\Gamma(\gamma):T[\gamma_0/\Gamma_0]$ hence
  $\Gamma'\vdash\pi_x^\Gamma(\gamma):T[\gamma/\Gamma]$ holds.
\end{proof}

\begin{lemma}[substitution]

  Assume $\Gamma'\ctx$ and $\Gamma'\vdash\gamma:\Gamma$ hold. Then:
  \begin{itemize}
  \item If $\Gamma\vdash T\typ$ holds, then $\Gamma'\vdash T[\gamma/\Gamma]\typ$
    holds.
  \item If $\Gamma\vdash u:T$ holds, then
    $\Gamma'\vdash{}u[\gamma/\Gamma]:T[\gamma/\Gamma]$ holds.
  \item If $\Gamma\vdash\theta:\Theta$ holds, then
    $\Gamma'\vdash\theta[\gamma/\Gamma]:\Theta$ holds.
  \end{itemize}
\end{lemma}

\begin{lemma}
  [compatibility]
  Assume $\Gamma\ctx$ holds. Then we have
  \begin{itemize}
  \item If $\Gamma\vdash u:T$ holds, then $\Gamma\vdash
    T\typ$ holds.
  \item If $\Gamma\vdash\theta:\Theta$ holds, then $\Theta\ctx$ holds.
  \end{itemize}
\end{lemma}

\begin{proof}
  The lemmas of substitution and compatibility are proved by a simultaneous induction on the main
  judgment.
  For the case of a variable, we use lemma \ref{typingvars}. For the other cases we use the
  induction hypothesis.
\end{proof}

\section{Syntactic weak \texorpdfstring{$\infty$}{∞}-groupoids}

We can now define weak $\infty$-groupoids. The definition is mutually recursive with the
interpretation of the syntax, and the various lemmas below are needed to make sure that the
interpretation of the syntax is well-defined. Because many coherence operations depend on other
coherence operations which may in turn depend on other coherence operations, and so on, we first
define a notion of $n$-partial weak $\infty$-groupoid and then we define the notion of weak
$\infty$-groupoids.

\begin{definition}
  The \emph{depth} of a type/term/context morphism is defined by
  \begin{align*}
    \depth(\coh \Delta T\delta) &\defeq\max(\depth(\delta),\depth(T)+1,\depth(\Delta)+1),\\
    \depth(u=_Tv) &\defeq\max(\depth(u),\depth(v),\depth(T)),\\
    \depth((\gamma,u)) &\defeq\max(\depth(\gamma),\depth(u)),
  \end{align*}
  and \[\depth(\star) = \depth(x) = \depth(()) = 0.\]
\end{definition}

In other words, the depth is the highest level of nestedness of coherence operations, where a
coherence operation is nested inside $\coh{\Delta}T\delta$ if it appears in $\Delta$ or $T$.

\begin{definition}
  A \emph{$0$-partial weak $\infty$-groupoid} is a globular set, and for $n\ge0$, an
  \emph{$(n+1)$-partial weak $\infty$-groupoid} is an $n$-partial weak $\infty$-groupoid $G$
  together with, for every contractible context $\Delta$ and every type $T$ in $\Delta$ of depth
  $n$, a function
  \[\cohhf G\Delta T{}:(\eta:\lcr\Delta\rcr^G)\to
  \Ob(\lcr{}T\rcr^G_\Delta(\eta)),\]
  where $\lcr\Delta\rcr^G$ and $\lcr{}T\rcr^G_\Delta$ are defined below.  This means that for every
  $\eta\in\lcr\Delta\rcr^G$, we have an element
  $\cohh G\Delta T\eta\in\Ob(\lcr{}T\rcr^G_\Delta(\eta))$. Using the structure of $n$-partial weak
  $\infty$-groupoid on $G$, we can extend $\cohhf G\Delta T{}$ to the case where $T$ is of depth
  less than or equal to $n$.
\end{definition}

\begin{definition}\label{def:intglob}
  Given an $n$-partial weak $\infty$-groupoid $G$, we interpret well-typed contexts, context
  morphisms, types and terms of depth at most $n$ as follows:
  \begin{itemize}
  \item The interpretation of a well-typed context $\Gamma$ of depth at most $n$ is a set
    $\lcr\Gamma\rcr^G$ defined below.
  \item The interpretation of a well-typed context morphism $\theta$ of depth at most $n$ from
    $\Gamma$ to $\Theta$ is a map
    $\lcr\theta\rcr^G_{\Gamma,\Theta}:
    \lcr\Gamma\rcr^G\to\lcr\Theta\rcr^G$ defined below.
  \item The interpretation of a well-typed type $T$ of depth at most $n$ in context $\Gamma$ is a
    map $\lcr{}T\rcr^G_\Gamma:\lcr\Gamma\rcr^G\to\Glob$ defined below.
  \item The interpretation of a well-typed term $u$ of depth at most $n$ of type $T$ in context
    $\Gamma$ is a dependent map
    $\lcr{}u\rcr^G_{\Gamma,T}: (\gamma:\lcr\Gamma\rcr^G)\to\Ob(\lcr{}T\rcr^G_\Gamma(\gamma))$
    defined below.
  \end{itemize}
  We now give the definitions.
  \begin{itemize}
  \item The definition of $\lcr\Gamma\rcr^G$ is
    \begin{align*}
      \lcr\varnothing\rcr^G&=\{()\},\\
      \lcr(\Gamma,x:T)\rcr^G&=\{(\gamma,a)\ |\ \gamma\in\lcr\Gamma\rcr^G,
      a\in\Ob(\lcr{}T\rcr^G_\Gamma(\gamma))\}.
    \end{align*}
  \item The definition of $\lcr\theta\rcr^G_{\Gamma,\Theta}:
    \lcr\Gamma\rcr^G\to\lcr\Theta\rcr^G$ is
    \begin{align*}
      \lcr()\rcr^G_{\Gamma,\varnothing}(\gamma)&=(),\\
      \lcr(\theta,u)\rcr^G_{\Gamma,(\Theta,x:T)}(\gamma)&=
      (\lcr\theta\rcr^G_{\Gamma,\Theta}(\gamma),
      \lcr{}u\rcr^G_{\Gamma,T[\theta/\Theta]}(\gamma)).
    \end{align*}
  \item The definition of $\lcr{}T\rcr^G_\Gamma:\lcr\Gamma\rcr^G\to\Glob$ is
    \begin{align*}
      \lcr\star\rcr^G_\Gamma(\gamma)&=G,\\
      \lcr u\simeq_Tv\rcr^G_\Gamma(\gamma)&= \Hom_{\lcr T\rcr^G_\Gamma(\gamma)}
      (\lcr u\rcr^G_{\Gamma,T}(\gamma),
      \lcr v\rcr^G_{\Gamma,T}(\gamma)).
    \end{align*}
  \item The definition of $\lcr{}u\rcr^G_{\Gamma,T}:
    (\gamma:\lcr\Gamma\rcr^G)\to\Ob(\lcr{}T\rcr^G_\Gamma(\gamma))$ is
    \begin{align*}
      \lcr{}x\rcr^G_{\Gamma,T}(\gamma)&=\pi^\Gamma_x(\gamma),\\
      \lcr{}\coh\Delta T\delta\rcr^G_{\Gamma,
        T[\delta/\Delta]} (\gamma)&= \cohh G\Delta T
      {\lcr\delta\rcr^G_{\Gamma,\Theta}(\gamma)}.
    \end{align*}
  \end{itemize}
  Note that we use the structure of $n$-partial weak $\infty$-groupoid in the last clause.
\end{definition}

The fact that the previous definition is well-typed follows from the following three lemmas which
are straightforward to prove by induction.

\begin{lemma}[semantic weakening]
  Let $\Gamma=(\Gamma_0,y:T')$ be a nonempty context and assume that
  $\vdash\Gamma\ctx$ holds. Let $\gamma=(\gamma_0,b)$ be an element of
  $\lcr\Gamma\rcr^G$. Then we have
  \begin{itemize}
  \item If $\Gamma_0\vdash{}T\typ$, then
    \[\lcr{}T\rcr^G_{\Gamma_0}(\gamma_0)=\lcr{}T\rcr^G_\Gamma(\gamma)\] in $\Glob$.
  \item If $\Gamma_0\vdash{}u:T$, then
    \[\lcr{}u\rcr^G_{\Gamma_0,T}(\gamma_0)=
    \lcr{}u\rcr^G_{\Gamma,T}(\gamma)\] in $\Ob(\lcr{}T\rcr^G_{\Gamma_0}(\gamma_0))$ which is equal to
    $\Ob(\lcr{}T\rcr^G_\Gamma(\gamma))$.
  \item If $\Gamma_0\vdash\theta:\Theta$, then
    \[\lcr\theta\rcr^G_{\Gamma_0,\Theta}(\gamma_0)=
    \lcr\theta\rcr^G_{\Gamma,\Theta}(\gamma)\] in $\lcr\Theta\rcr^G$.
  \end{itemize}
\end{lemma}

\begin{lemma}
  We have
  \begin{itemize}
  \item
    If $\vdash\Gamma\ctx$ and $(x:T)\in\Gamma$ then, for every
    $\gamma\in\lcr\Gamma\rcr^G$,
    \[\pi^\Gamma_x(\gamma)\in\Ob(\lcr{}T\rcr^G_\Gamma(\gamma)).\]
  \item If $\Gamma\vdash\theta:\Theta$ and $(y:T)\in\Theta$ then, for every
    $\gamma\in\lcr\Gamma\rcr^G$,
    \[\lcr\pi_y^\Theta(\theta)\rcr^G_{\Gamma,T[\theta/\Theta]}
    (\gamma)=
    \pi_y^\Theta(\lcr\theta\rcr^G_{\Gamma,\Theta}(\gamma)).\]
  \end{itemize}
\end{lemma}

\begin{lemma}[semantic substitution]
  Assume we have $\Gamma\vdash\theta:\Theta$ and
  $\gamma\in\lcr\Gamma\rcr^G$. Then
  \begin{itemize}
  \item If $\Theta\vdash{}T\typ$, then
    \[\lcr{}T[\theta/\Theta]\rcr^G_\Gamma(\gamma)
    =\lcr{}T\rcr^G_\Theta(\lcr\theta\rcr^G_{\Gamma,\Theta}(\gamma)).\]
  \item If $\Theta\vdash{}u:T$, then
    \[\lcr{}u[\theta/\Theta]
    \rcr^G_{\Gamma,T[\theta/\Theta]}(\gamma)
    =\lcr{}u\rcr^G_{\Theta,T}(\lcr\theta\rcr^G_{\Gamma,\Theta}(\gamma)).
    \]
  \item If $\Theta\vdash{}\theta':\Theta'$, then
    \[
    \lcr\theta'[\theta/\Theta]\rcr^G_{\Gamma,\Theta'}(\gamma)=\lcr\theta'\rcr^G_{\Theta,\Theta'}(\lcr\theta
    \rcr^G_{\Gamma,\Theta}(\gamma)).\]
  \end{itemize}
\end{lemma}

Finally, we can define weak $\infty$-groupoids.
\begin{definition}
  A \emph{syntactic weak $\infty$-groupoid} is a globular set $G$ together with for every $n\in\N$,
  a structure of $(n+1)$-partial weak $\infty$-groupoid over the previous structure of $n$-partial
  weak $\infty$-groupoid. In particular, that means that for every contractible context $\Delta$ and
  every type $T$ in $\Delta$ (without restriction on depth) we have a function
  \[\cohhf G\Delta T{}:(\eta:\lcr\Delta\rcr^G)\to
  \Ob(\lcr{}T\rcr^G_\Delta(\eta)),\]
  where the contexts, context morphisms, types and terms are interpreted as above.
\end{definition}

Let’s give some examples of coherence operations.

\begin{example}
  The \emph{constant path} coherence operation is
  \[\cohhf{}{(x:\star)}{x\simeq_\star x}{}.\]
\end{example}

\begin{proof}
  Given a globular set $G$ we have
  \[\lcr(x:\star)\rcr^G = \Ob(G),\]
  and given $a\in\Ob(G)$ we have
  \[\lcr x\simeq_\star x\rcr^G_{(x:\star)}(a) = \Hom_G(a,a).\]

  Therefore, the operation $\cohhf G{(x:\star)}{x\simeq_\star x}{}$ takes an object $a$ of $G$ and
  returns a morphism from $a$ to $a$. In particular, such an operation exists only when the graph is
  a reflexive graph.
\end{proof}

\begin{example}
  The \emph{inverse} coherence operation is
  \[\cohhf{}{(x,y:\star)(t:x\simeq_\star y)}{y\simeq_\star x}{}.\]
\end{example}

\begin{proof}
  Given a globular set $G$ we have
  \begin{align*}
    &\lcr(x,y:\star)(t:x\simeq_\star y)\rcr^G\\
    &\qquad=\{(a,b,p)\,|\,
      a,b\in\Ob(G), p\in\Ob(\Hom_G(a,b))\},
  \end{align*}
  and given a triple $(a,b,p)$ in this set we have
  \begin{align*}
    \lcr y\simeq_\star x\rcr^G_{(x,y:\star)(t:x\simeq_\star y)}(a,b,p) &= \Hom_G(b,a).
  \end{align*}

  Therefore, the operation $\cohhf G{(x,y:\star)(t:x\simeq_\star y)}{y\simeq_\star x}{}$ takes a
  morphism $p$ in $G$ between two objects $a$ and $b$ and returns a morphism $p\inv$ going from $b$
  to $a$.
\end{proof}

\begin{example}
  The \emph{involutivity of inverse} coherence operation is
  \[\cohhf{}{(x,y:\star)(t:x\simeq_\star y)}{t\simeq_{x\simeq_\star y}(\coh{(x,y:\star)(t:x\simeq_\star y)}{y\simeq_\star x}{y,x,\coh{(x,y:\star)(t:x\simeq_\star y)}{y\simeq_\star x}{x,y,t}})}{}.\]
\end{example}

\begin{proof}
  Here the coherence operation is of depth $1$, therefore we need to start with a $1$-partial weak
  $\infty$-groupoid $G$. In a $1$-partial weak $\infty$-groupoid we have in particular the
  \[\cohhf{}{(x,y:\star)(t:x\simeq_\star y)}{y\simeq_\star x}{(a,b,p)}\]
  coherence operation described in the previous example. We write it $p\mapsto p\inv$ for short. We
  have
  \begin{align*}
    &\lcr(x,y:\star)(t:x\simeq_\star y)\rcr^G\\
    &\qquad=\{(a,b,p)\,|\,
      a,b\in\Ob(G), p\in\Ob(\Hom_G(a,b))\},
  \end{align*}
  and given a triple $(a,b,p)$ in this set we have
  \begin{align*}
    &\lcr t\simeq_{x\simeq_\star y}(\coh{(x,y:\star)(t:x\simeq_\star y)}{y\simeq_\star
    x}{y,x,\coh{(x,y:\star)(t:x\simeq_\star y)}{y\simeq_\star
    x}{x,y,t}})\rcr^G_{(x,y:\star)(t:x\simeq_\star y)}(a,b,p)\\
    &\qquad=\Hom_{\Hom_G(a,b)}(p,(p\inv)\inv).
  \end{align*}

  Therefore, the operation we’re looking at takes a morphism $p$ in $G$ between two objects $a$ and
  $b$ and returns a $2$-morphism from $p$ to $(p\inv)\inv$.
\end{proof}

There are many coherence operations which seem completely superfluous like for instance
$\cohhf{}{(x,y:\star)(t:x\simeq_\star y)}{x\simeq_\star y}{(a,b,p)}$ which gives a morphism from $a$
to $b$ given a morphism $p$ from $a$ to $b$. There are also various ways to order the arguments, for
instance in the diagram for horizontal composition of $2$-morphisms, which give different contexts
hence different coherence operations. But this is not a problem because it is always possible to
prove that they are all equal to one another via another coherence operation.

\section{The underlying weak \texorpdfstring{$\infty$}{∞}-groupoid of a type}

We now prove that in a dependent type theory $\tml$ with identity types, the $\J$ rule, and the
definitional computation rule, we can equip every type with a structure of syntactic weak
$\infty$-groupoid. In this section we write $\vdashi$ for the judgments in the type theory
$\tinfgpd$ described above and $\vdashML$ for the judgments in $\tml$. We write $\Id_A(u,v)$ for the
identity type if $\tml$ and $\jueq$ for definitional equality in $\tml$.

We first define the iterated constant path and its type.

\begin{definition}
  Given a type $\vdashML A:\Type$, a term $\vdashML a:A$ and a natural number $n\in\N$, we define a
  type $\vdashML\IID^{(A,a)}_n:\Type$ and a term $\vdashML\iid^{(A,a)}_n:\IID^{(A,a)}_n$ by
  \begin{align*}
    \IID^{(A,a)}_0 &\defeq A, & \iid^{(A,a)}_0 &\defeq a,\\
    \IID^{(A,a)}_{n+1} &\defeq \left(\iid^{(A,a)}_n=_{\IID^{(A,a)}_n}\iid^{(A,a)}_n\right), & \iid^{(A,a)}_{n+1} &\defeq \mathsf{idp}_{\iid^{(A,a)}_n}.
  \end{align*}
\end{definition}

We now consider a type $A$ of $\tml$ and we want to equip it with a structure of weak
$\infty$-groupoid. The idea is that a coherence operation in a contractible context $\Delta$ is
defined by successive applications of the $\J$ rule until we reach the context $(a:A)$. Then we have
to prove an equality between two terms in this context and it turns out that all terms defined only
using coherence operations in the context $(a:A)$ are definitionally equal to the iterated constant
path. Therefore we can define the coherence operation in the case $(a:A)$ by an iterated constant
path, and it preserves the invariant that a coherence operation applied to constant paths is
definitionally equal to a constant path.

We define an interpretation $\llp-\rrp$ of $\tinfgpd$ into $\tml$, define two terms $\id^\Delta$ and
$\J_\Delta$ for every contractible context $\Delta$, and prove three lemmas stating that, in the
context $(a:A)$, all types are definitionally equal to $\IID^{(A,a)}_n$ and all terms are
definitionally equal to $\iid^{(A,a)}_n$. All of that has to be done simultaneously by induction. We
also need a substitution and a weakening lemma, stating that the interpretation commutes with the
operations of substitution and weakening, but we do not state them here for simplicity. The
\emph{dimension} of a type is defined by $\dim(\star)=0$ and $\dim(u\simeq_Tv)=\dim(T)+1$.

\begin{definition}\label{def:intml}
  The interpretation of $\tinfgpd$ into $\tml$ satisfies the following typing rules
  and is defined below.
  \begin{itemize}
  \item Given $\vdashi\Gamma$, the interpretation $\llp\Gamma\rrp^A$ is a context in $\tml$.
  \item Given $\Gamma\vdashi\gamma:\Gamma'$, the interpretation $\llp\gamma\rrp_{\Gamma,\Gamma'}^A$
    is a context morphism in $\tml$ from $\llp\Gamma\rrp^A$ to $\llp\Gamma'\rrp^A$.
  \item Given $\Gamma\vdashi T$, the interpretation $\llp T\rrp_\Gamma^A$ is a dependent type in
    $\tml$ over $\llp\Gamma\rrp^A$.
  \item Given $\Gamma\vdashi t:T$, the interpretation $\llp t\rrp_{\Gamma,T}^A$ is a term in $\tml$
    of type $(x:\llp\Gamma\rrp^A)\to\llp T\rrp_\Gamma^A$.
  \end{itemize}
\end{definition}

\begin{definition}\label{def2}
  For every contractible context $\Delta$, we define a context morphism
  \[(a:A)\vdashML \id^\Delta_a : \llp\Delta\rrp^A\] and a term
  \[(a:A)(P:(\delta:\llp\Delta\rrp^A)\to\Type) (d:(a:A)\to P(\id^\Delta_a)) \vdashML \J_\Delta(P,d) :
  (\delta:\llp\Delta\rrp^A)\to P(\delta)\] satisfying the definitional equality
  \[\J_\Delta(P,d)(\id^\Delta_a) \jueq d(a).\]
\end{definition}

\begin{lemma}\label{lemmatype}
  For every contractible context $\Delta$ and every type $\Delta\vdashi T$, we have the definitional
  equality
  \[(a:A)\vdashML \llp T\rrp_\Delta^A(\id^\Delta_a) \jueq \IID^{(A,a)}_{\dim(T)}.\]
\end{lemma}

\begin{lemma}\label{lemmactxmor}
  For every contractible contexts $\Delta,\Delta'$ and every context morphism
  $\Delta'\vdashi \delta:\Delta$, we have the definitional equality
  \[(a:A)\vdashML\llp \delta\rrp_{\Delta',\Delta}^A(\id^{\Delta'}_a)\jueq \id^\Delta_a.\]
\end{lemma}

\begin{lemma}\label{lemmaterm}
  For every contractible context $\Delta$ and every term $\Delta\vdashi t:T$, we have the
  definitional equality
  \[(a:A)\vdashML\llp t\rrp_{\Delta,T}^A(\id^\Delta_a)\jueq \iid^{(A,a)}_{\dim(T)}.\]
\end{lemma}

\begin{proof}[Definition of the interpretation]
The interpretation is defined as follows:
\begin{itemize}
\item The definition of $\llp\Gamma\rrp^A$ is
  \begin{align*}
    \llp\varnothing\rrp^A&=\varnothing,\\
      \llp(\Gamma,x:T)\rrp^A&=(y:\llp\Gamma\rrp^A,x:\llp T\rrp^A_{\Gamma}(y)).
    \end{align*}
  \item The definition of $\llp\delta\rrp^A_{\Gamma,\Theta}:
    \llp\Gamma\rrp^A\to\llp\Theta\rrp^A$ is
    \begin{align*}
      \llp()\rrp^A_{\Gamma,\varnothing}(\gamma)&=(),\\
      \llp(\theta,u)\rrp^A_{\Gamma,(\Theta,x:T)}(\gamma)&=
      (\llp\theta\rrp^A_{\Gamma,\Theta}(\gamma),
      \llp{}u\rrp^A_{\Gamma,T[\theta/\Theta]}(\gamma)).
    \end{align*}
  \item The definition of $(\gamma:\llp\Gamma\rrp^A)\vdashML\llp{}T\rrp^A_\Gamma(\gamma):\Type$ is
    \begin{align*}
      \llp\star\rrp^A_\Gamma(\gamma)&=A,\\
      \llp u\simeq_Tv\rrp^A_\Gamma(\gamma)&= \Id_{\llp T\rrp^A_\Gamma(\gamma)}(
                                            \llp u\rrp^A_{\Gamma,T}(\gamma),
                                            \llp v\rrp^A_{\Gamma,T}(\gamma)).
    \end{align*}
  \item The definition of
    $(\gamma:\llp\Gamma\rrp^A)\vdashML \llp{}u\rrp^A_{\Gamma,T}:\llp{}T\rrp^A_\Gamma(\gamma)$ is
    \begin{align*}
      \llp{}x\rrp^A_{\Gamma,x:T}(\gamma)&=\pi^\Gamma_x(\gamma),\\
      \llp{}\coh\Delta T\delta\rrp^A_{\Gamma,
      T[\delta/\Delta]}(\gamma)&= \J_\Delta(\llp T\rrp_\Delta^A, (\lambda
                                 a.\iid^{(A,a)}_{\dim(T)}))(\llp\delta\rrp_{\Gamma,\Delta}^A(\gamma)).
    \end{align*}
  \end{itemize}
  The things to note are that the interpretation of $\star$ is the type $A$ we are interested in,
  that the interpretation of $u\simeq_Tv$ uses the identity type in $\tml$, and that coherence
  operations are implemented using $\J_\Delta$ and $(\lambda a.\iid^{(A,a)}_{\dim(T)})$ in the base
  case, which is well-typed by lemma \ref{lemmatype}.
\end{proof}

\begin{proof}[Definition of $\id^\Delta_a$ and $\J_\Delta$]
  For $a:A$, the sequence of terms $\id^\Delta_a:\llp\Delta\rrp^A$ is defined by induction on $\Delta$ by
  \begin{align*}
    \id^{(x:\star)}_a &\defeq (a),\\
    \id^{(\Delta',y:T,z:u=_Ty)}_a &\defeq (\id^{\Delta'}_a,\llp
                                    u\rrp_{\Delta',T}^A(\id^{\Delta'}_a),\idp{\llp
                                    u\rrp_{\Delta',T}^A(\id^{\Delta'}_a)}).
  \end{align*}
  We remind that $\J$ has type
  \begin{align*}
    &(u:A)(P:(y:A)(z:\Id_A(u,y))\to\Type)(d:P(u,\idp{u}))\\
    &\qquad\vdashML\J(u,P,d):(v:A)(p:\Id_A(u,v))\to P(v,p).
  \end{align*}
  For $a:A$, $P:(\delta:\llp\Delta\rrp^A)\to\Type$, and $d:(a:A)\to P(\id^\Delta_a)$,
  the term \[\J_\Delta(P,d):(\delta:\llp\Delta\rrp^A)\to P(\delta)\] is defined by induction on
  $\Delta$ by
  \begin{align*}
    \J_{(x:\star)}(P,d,a) &\defeq d(a),\\
    \J_{(\Delta',y:T,z:u=_Ty)}(P,d,\delta',y,z)
                          &\defeq \J(\llp u\rrp_{\Delta',T}^A(\delta'), P(\delta'),\\&\qquad\quad
                            \J_{\Delta'}((\lambda\delta'.P(\delta',\llp
                            u\rrp_{\Delta',T}^A(\delta'),\idp{\llp u\rrp_{\Delta',T}^A(\delta')})),
                                                                                       d, \delta'))\\&\qquad\quad
                                                                                                       (y,
                                                                                                       z).
  \end{align*}
  In other words, in order to apply $\J_\Delta$ to $(\delta',y,z)$, we first apply the regular $\J$
  rule to $z$ and then we apply $\J_{\Delta'}$ to $\delta'$. The reduction rule holds because in the
  case $\id^\Delta_a$ we have $\idpS$ for $z$, hence the reduction rule of the regular $\J$ applies,
  and then $\id^{\Delta'}_a$ for $\delta'$ hence the reduction rule for $\J_{\Delta'}$ applies.
\end{proof}

\begin{proof}[Proof of lemma \ref{lemmatype}]
  We want to prove that $\llp T\rrp_\Delta^A(\id^\Delta_a)$ is definitionally equal to
  $\IID^{(A,a)}_{\dim(T)}$. We proceed by induction on $T$. If $T$ is $\star$, then it’s true as they
  are both equal to $A$. If $T$ is of the form $u\simeq_{T'}v$, then it’s true by lemma
  \ref{lemmaterm}.
\end{proof}

\begin{proof}[Proof of lemma \ref{lemmactxmor}]
  We want to prove that $\llp \delta\rrp_{\Delta',\Delta}^A(\id^{\Delta'}_a)$ is definitionally
  equal to $\id^\Delta_a$. We proceed by induction on $\Delta$ and it follows from lemma
  \ref{lemmaterm}.
\end{proof}

\begin{proof}[Proof of lemma \ref{lemmaterm}]
  We want to prove that $\llp t\rrp_{\Delta,T}^A(\id^\Delta_a)$ is definitionally equal to
  $\iid^{(A,a)}_{\dim(T)}$. We proceed by induction on $\Delta$ and $t$. If $t$ is a variable, then
  given the definition of $\id^\Delta_a$, the interpretation of $t$ is either $a$ or of the form
  $\llp u\rrp_{\Delta',T}^A(\id^{\Delta'}_a)$ or $\idp{\llp u\rrp_{\Delta',T}^A(\id^{\Delta'}_a)})$,
  which are definitionally equal to $\iid^{(A,a)}_{\dim(T)}$ by induction hypothesis, as $\Delta'$ is
  strictly smaller that $\Delta$.

  If $t$ is of the form $\coh\Delta T\delta$, then we have
  \begin{align*}
  \llp\coh\Delta T\delta\rrp^A_{\Gamma,T[\delta/\Delta]}(\id^\Gamma_a) &\jueq \J_\Delta(\llp T\rrp_\Delta^A, (\lambda
                                                                         a.\iid^{(A,a)}_{\dim(T)}),\llp\delta\rrp_{\Gamma,\Delta}^A(\id^\Gamma_a))\\
    &\jueq \J_\Delta(\llp T\rrp_\Delta^A, (\lambda
      a.\iid^{(A,a)}_{\dim(T)}),\id^\Delta_a)\\
    &\jueq \iid^{(A,a)}_{\dim(T)},
  \end{align*}
  using lemma \ref{lemmactxmor} and the reduction rule for $\J_\Delta$.
\end{proof}

This concludes the proof that we can interpret all the coherence operations as terms of $\tml$. This
is all we need when using type theory as a foundational system as is done in this thesis, but we can
also get an actual set-based weak $\infty$-groupoid if desired.

\begin{definition}
  Given a type $A$ in $\tml$, we define a globular set $A\glob$ by
  \begin{align*}
    \Ob(A\glob) &\defeq \{t\,|\,\vdashML t:A\},\\
    \Hom_{A\glob}(t,u) &\defeq \left(\Id_A(t,u)\right)\glob.
  \end{align*}
\end{definition}

\begin{proposition}
  Given a type $A$ in $\tml$, the globular set $A\glob$ has a structure of weak $\infty$-groupoid.
\end{proposition}

\begin{proof}
  Given $\llp-\rrp$ the interpretation of $\tinfgpd$ in $\tml$ given in definition \ref{def:intml},
  we define an interpretation of $\tinfgpd$ in the globular set $A\glob$ in the sense of definition
  \ref{def:intglob} by
  \begin{align*}
    \lcr\Gamma\rcr^{A\glob} &\defeq \{\overline{t}\,|\,\vdashML \overline{t}:\llp\Gamma\rrp^A\},\\
    \lcr\theta\rcr_{\Gamma,\Theta}^{A\glob}(\gamma) &\defeq
    \llp\theta\rrp_{\Gamma,\Theta}^A[\gamma/\Gamma],\\
    \lcr T\rcr_{\Gamma}^{A\glob}(\gamma) &\defeq \left(\llp T\rrp_{\Gamma}^A[\gamma/\Gamma]\right)\glob,\\
    \lcr u\rcr_{\Gamma,T}^{A\glob}(\gamma) &\defeq \llp u\rrp_{\Gamma,T}^A[\gamma/\Gamma].
  \end{align*}
  It is easy to check that it satisfies the required properties and it gives an interpretation of
  all the coherence operations in $A\glob$.
\end{proof}


\addtocontents{toc}{\protect\setcounter{tocdepth}{0}}
\chapter{The cardinal of \texorpdfstring{$\pi_4(\Sn3)$}{π₄(S³)}}\label{ch:defn}

The purpose of this appendix is to give a self-contained definition of the natural number $n:\N$
satisfying $\pi_4(\Sn3)\simeq\Z/n\Z$, in order to serve as a sophisticated test case for candidates
for a computational interpretation of univalence and higher inductive types. It doesn’t contain the
proof that $\pi_4(\Sn3)\simeq\Z/n\Z$ nor the proof that $n=2$, only the definition of $n$ based on
theorem \ref{firstpi4s3}.

If you wrote a proof assistant for homotopy type theory giving a computational interpretation of
univalence and higher inductive types, please try to implement the following computation and check
that you do get $2$ as the result.

\section{Preliminaries}

Here are some basic definitions of homotopy type theory that we use, to fix the notations.
\begin{itemize}
\item The universe is $\Type$.
\item Given a type $A$ and two elements $a,b:A$, the type of paths from $a$ to $b$ is $a=_Ab$ (identity
  type).
\item Given an element $a:A$, the constant path at $a$ is $\idp a$.
\item Given a path $p:a=_Ab$, the reverse path of $p$ is $p\inv:b=_Aa$.
\item Given two composable paths $p:a=_Ab$ and $q:b=_Ac$, their composition is $p \concat q:a=_Ac$
  (diagrammatic order).
\item Given a function $f:A\to B$ and a path $p:a=_Ab$, the application of $f$ to $p$ is
  $\ap f(p):f(a)=_Bf(b)$.
\item Given a dependent type $P:A\to\Type$, a path $p:a=_Ab$ and an element $u:P(a)$, the transport
  of $u$ along $p$ in $P$ is $\transport^P(p,u):P(b)$.
\end{itemize}

\subsection{Equivalences}

Given a function $f:A\to B$, there is a type $\isequiv(f)$ of proofs that $f$ is an equivalence. We
do not specify this type here but it should have the following properties:
\begin{itemize}
\item A function $f$ is an equivalence if and only if there exists $g:B\to A$ such that
  $g(f(x))=_Ax$ and $f(g(y))=_By$ for every $x:A$ and $y:B$.
\item Any two elements of $\isequiv(f)$ are equal.
\end{itemize}
We write $A\simeq B$ for the type $\sum_{f:A\to B}\isequiv(f)$.

\subsection{Dependent paths}

Given a dependent type $P:A\to\Type$, a path $p:a=_Ab$ and two elements $u:P(a)$ and $v:P(b)$, we
write $u=^P_pv$ for the type of dependent paths from $u$ to $v$ over $p$. We assume that
$u=^P_{\idp a}v$ is definitionally equal to $u=_{P(a)}v$. The type of dependent paths $u=^P_pv$ is
equivalent to the type $\transport^P(p,u)=_{P(b)}v$. We have the two maps
\begin{align*}
  \innt&:\transport^P(p,u)=_{P(b)}v\to u=^P_pv,\\
  \outt&:u=^P_pv\to\transport^P(p,u)=_{P(b)}v,
\end{align*}
which are inverse to each other.
If $P$ is of the form $P(x)\defeq(f(x)=_Bg(x))$, then we have two maps
\begin{align*}
  \inne&:(u\concat\ap g(p)=_{f(a)=_Bg(b)}\ap f(p)\concat v)\to u=^{\lambda x.(f(x)=_Bg(x))}_pv,\\
  \oute&:u=^{\lambda x.(f(x)=_Bg(x))}_pv\to(u\concat\ap g(p)=_{f(a)=_Bg(b)}\ap f(p)\concat v).
\end{align*}
If $f:(x:A)\to B(x)$ and $p:a=_Ab$, we write $\apd f(p):f(a)=^B_pf(b)$ for the
image of $p$ by $f$.

\subsection{Function extensionality and univalence}

We assume unary function extensionality, which is a term of type
\begin{align*}
  \funext&:\{A:\Type\}\{B:A\to\Type\}\{f,g:(x:A)\to B(x)\}\\&(h:(x:A)\to f(x)=_{B(x)}g(x))\to
  f=_{(x:A)\to B(x)}g.
\end{align*}
Moreover, $\funext\{f\}\{g\}$ is a equivalence. We also assume ternary function extensionality
\begin{align*}
  \funext^3&:\{A,B:\Type\}\{C:A\to\Type\}\{p:a=_Aa'\}\\
           &\quad\{f:(x:C(a))\to B\}\{g:(y:C(a'))\to
             B\}\\
           &\qquad\to(\{x:C(a)\}\{y:C(a')\}(z:x=^C_py)\to f(x)=_Bf(y))\\
           &\qquad\to f=^{\lambda z.((x:C(z))\to B)}_pg,
\end{align*}
and unary dependent function extensionality
\begin{align*}
  \funextd&:\{A,B:\Type\}\{C:A\to B\to\Type\}(p:b=_Bb')\{f:(x:A)\to
            C(x)(b)\}\\&\qquad\{g:(x:A)\to C(x)(b')\}\\
                         &\qquad\to((x:A)\to f(x)=^{C(x)}_pg(x))\to
  f=^{\lambda t.(x:A)\to C(x)(t)}_pg.
\end{align*}
Finally we assume the univalence axiom
\begin{align*}
  \ua:(A\simeq B)\to(A=_{\Type} B).
\end{align*}
The map $\ua$ is itself an equivalence, but we won’t need to give that a name here.

\subsection{Truncation levels}

A type $A$ is \emph{$(-2)$-truncated} (or \emph{contractible}) if the type
\[\sum_{a:A}((x:A)\to a=_Ax)\]
has an element.  For $n\ge-1$, a type $A$ is \emph{$n$-truncated} if for every $x,y:A$, the type
$x=_Ay$ is $(n-1)$-truncated. We have the following properties:
\begin{itemize}
\item The type of proofs that $A$ is $n$-truncated is $(-1)$-truncated.
\item To prove that a type $A$ is $(-1)$-truncated it is enough to construct a
  term of type $(x,y:A)\to x=_Ay$.
\item $n$-truncated types are stable under products.
\end{itemize}

\section{The definition}

\newcommand{\mv}[1]{{#1}\vphantom{\int}}
The number $n$ is defined as the absolute value of the image of $1$ by the following composition of
maps. The types and maps involved are defined below.
\[
\begin{tikzcd}[column sep=1.8em,cramped]
  \Z \arrow[r,"n\mapsto\lloop^n"] & \Omega\Sn1 \arrow[r,"\Omega \varphi_{\Sn1}"] & \Omega^2\Sn2
  \arrow[r,"\Omega^2 \varphi_{\Sn2}"] & \Omega^3\Sn3 \arrow[r,"\Omega^3e"] & \Omega^3(\Sn1*\Sn1)
  \arrow[r,"\Omega^3\alpha"] &
  \Omega^3\Sn2 \arrow[dlllll,out=280,in=80,looseness=0.2,"h"'] \\
  \Omega^3(\Sn1*\Sn1) \arrow[r,"\mv{\Omega^3(e^{-1})}"'] &
  \Omega^3\Sn3 \arrow[r,"e_3"'] &
  \Omega^2\|\Sn2\|_2 \arrow[r,"\mv{\Omega\kappa_{2,\Sn2}}"'] &
  \Omega\|\Omega\Sn2\|_1 \arrow[r,"\kappa_{1,\Omega\Sn2}"'] &
  \|\Omega^2\Sn2\|_0 \arrow[r,"e_2"'] &
  \Omega\Sn1 \arrow[r,"e_1"'] &
  \Z
\end{tikzcd}
\]
We do not repeat the definition of $\Z$ and of the two maps between $\Z$ and $\Omega\Sn1$ here.

\section{Suspensions and spheres}

The circle is defined as the higher inductive type $\Sn1$ with constructors
\begin{align*}
  \base&:\Sn1,\\
  \lloop&:\base=_{\Sn1}\base.
\end{align*}
Given a type $A$, we define the suspension of $A$ as the higher inductive type $\Susp A$
with constructors
\begin{align*}
  \north&:\Susp A,\\
  \south&:\Susp A,\\
  \merid&:A\to\north=_{\Susp A}\south.
\end{align*}
For $n\ge1$, we define the $(n+1)$-sphere as the suspension of the $n$-sphere:
\begin{align*}
  \Sn{n+1}&\defeq\Susp \Sn{n}.
\end{align*}

\section{Pointed types, pointed maps and loop spaces}

A pointed type is a type $A$ together with a point $\star_A:A$ (often omitted when clear from the
context). The circle is pointed by $\base$ and the suspension of any type by $\north$.
A pointed map between two pointed types $A$ and $B$ is a map $f:A\to B$ together with an equality
$\star_f:f(\star_A)=\star_B$ (also often omitted).
The loop space of a pointed type $A$ is the type $\Omega A \defeq (\star_A=\star_A)$ pointed by
$\idp{\star_A}$.
We can iterate this operation: $\Omega^1A\defeq\Omega A$ and $\Omega^{n+1}A\defeq\Omega(\Omega^nA)$.
If $f$ is a pointed map between the pointed types $A$ and $B$, then $\Omega f$ is the map between
$\Omega A$ and $\Omega B$ defined by
\[(\Omega f)(p)\defeq\star_f\inv\concat\ap f(p)\concat \star_f\]
and pointed by the proof that $\star_f\inv\concat\ap f(\idp{\star_A})\concat \star_f=\idp{\star_B}$.
We can again iterate this operation as above.

\section{Loop space of a suspension}

Given a pointed type $A$, we define the map
\begin{align*}
  \varphi_A&:A\to\Omega(\Susp A),\\
  \varphi_A(a)&\defeq\merid(a)\concat\merid(\star_A)\inv,
\end{align*}
pointed by the obvious proof that $\merid(\star_A)\concat\merid(\star_A)\inv=\idp{\north}$.

\section{The 3-sphere and the join of two circles}

\subsection{Join and associativity}

Given two types $A$ and $B$, the join of $A$ and $B$ is the higher inductive type $A*B$ with
constructors
\begin{align*}
  \inl&:A\to A*B,\\
  \inr&:B\to A*B,\\
  \push&:(a:A)(b:B)\to\inl(a)=_{A*B}\inr(b).
\end{align*}
If $A$ is pointed then $A*B$ is pointed by $\inl(\star_A)$.

This operation is associative, we have the pointed map
\begin{align*}
  \alpha_{A,B,C}&:A*(B*C)\to(A*B)*C,\\
  \alpha_{A,B,C}(\inl(a))&\defeq\inl(\inl(a)),\\
  \alpha_{A,B,C}(\inr(\inl(b)))&\defeq\inl(\inr(b)),\\
  \alpha_{A,B,C}(\inr(\inr(c)))&\defeq\inr(c),\\
  \ap{\alpha_{A,B,C}\circ\inr}(\push(b,c))&\defeq\push(\inr(b),c),\\
  \ap{\alpha_{A,B,C}}(\push(a,\inl(b)))&\defeq\ap{\inl}(\push(a,b)),\\
  \ap{\alpha_{A,B,C}}(\push(a,\inr(c)))&\defeq\push(\inl(a),c),\\
  \apd{\ap{\alpha_{A,B,C}}\circ\push(a,-)}(\push(b,c))&\defeq\inne(op(\oute(\apd{\push(-,c)}(\push(a,b))))).
\end{align*}
For the last equation, inside the $\inne$ we need a term of type
\[\ap{\inl}(\push(a,b))\concat\push(\inr(b),c)=\idp{\inl(\inl(a))}\concat\push(\inl(a),c)\]
and the $\oute$ term has type
\[\push(\inl(a),c)\concat\idp{\inr(c)}=\ap{\inl}(\push(a,b))\concat\push(\inr(b),c).\]
Hence the term $op$ is the path algebra mapping from one type to the other, which isn’t complicated
to define.
The map $\alpha^{-1}_{A,B,C}:(A*B)*C\to A*(B*C)$ is defined in a similar way.

Given two maps $f:A\to A'$ and $g:B\to B'$, we define the map
\begin{align*}
  f*g&:(A*B)\to(A'*B'),\\
  (f*g)(\inl a)&\defeq\inl(f(a)),\\
  (f*g)(\inr b)&\defeq\inr(g(b)),\\
  \ap{f*g}(\push(a,b))&\defeq\push(f(a),g(b)).
\end{align*}
The map $f*g$ is pointed as soon as $f$ is pointed.

\subsection{Suspension and join with the booleans}

The type $\Bool$ of booleans, with constructors $\true,\false:\Bool$ is pointed
by $\true$.
For any type $A$ we define the following two pointed maps, inverse to each other:
\begin{align*}
  \psi_A&:\Susp A\to\Bool*A,\\
  \psi_A(\north)&\defeq\inl(\true),\\
  \psi_A(\south)&\defeq\inl(\false),\\
  \ap{\psi_A}(\merid(a))&\defeq\push(\true,a)\concat\push(\false,a)\inv,
\end{align*}
\begin{align*}
  \psi^{-1}_A&:\Bool*A\to\Susp A,\\
  \psi^{-1}_A(\inl(\true))&\defeq\north,\\
  \psi^{-1}_A(\inl(\false))&\defeq\south,\\
  \psi^{-1}_A(\inr(a))&\defeq\south,\\
  \psi^{-1}_A(\push(\true,a))&\defeq\merid(a),\\
  \psi^{-1}_A(\push(\false,a))&\defeq\idp\south.
\end{align*}
We also define the following two pointed maps, again inverse to each other:
\begin{align*}
  c&:(\Bool*\Bool)\to\Sn1,\\
  c(\inl(b))&\defeq\base,\\
  c(\inr(b'))&\defeq\base,\\
  \ap c(\push(\true,\true))&\defeq\lloop,\\
  \ap c(\push(b,b'))&\defeq\idp\base,
\end{align*}
\begin{align*}
  c^{-1}&:\Sn1\to(\Bool*\Bool),\\
  c^{-1}(\base)&\defeq\inl(\true),\\
  \ap{c^{-1}}(\lloop)&\defeq\push(\true,\true)\\
  &\qquad\concat\push(\false,\true)\inv\\
  &\qquad\concat\push(\false,\false)\\
  &\qquad\concat\push(\true,\false)\inv.
\end{align*}

\subsection{Equivalence between $\Sn3$ and $\Sn1*\Sn1$}

We define the pointed map $e:\Sn3\to\Sn1*\Sn1$ as the composition
\[
\begin{tikzcd}[column sep=scriptsize]
  \Sn3\arrow[r,"\psi_{\Sn2}"] & \Bool*\Sn2 \arrow[rr,"\id_{\Bool}*\psi_{\Sn1}"] &&
  \Bool*(\Bool*\Sn1) \arrow[rr,"{\alpha_{\Bool,\Bool,\Sn1}}"] && (\Bool*\Bool)*\Sn1 \arrow[r,"c*\id_{\Sn1}"]
  & \Sn1*\Sn1
\end{tikzcd}
\]
and the pointed map $e^{-1}:\Sn1*\Sn1\to\Sn3$ as the opposite composition
with the inverse maps.

\section{The main map}

The map $\alpha:\Sn1*\Sn1\to\Sn2$ is the main map which appears in the proof that
$\pi_4(\Sn3)\simeq\Z/n\Z$; all the other maps appear already in the construction of the Hopf
fibration or in the Freudenthal suspension theorem. In the notation of chapter \ref{ch:james} of
this thesis, it is the map $\fold_{\Sn2}\circ W_{\Sn1,\Sn1}$.
It is defined by
\begin{align*}
  \alpha&:\Sn1*\Sn1\to\Sn2,\\
  \alpha(\inl(x))&\defeq\north,\\
  \alpha(\inr(y))&\defeq\north,\\
  \ap\alpha(\push(x,y))&\defeq
                          \merid(y)\concat\merid(\base)\inv\concat\merid(x)\concat\merid(\base)\inv
\end{align*}
and pointed by $\idp{\north}$.

\section{The map defining \texorpdfstring{$\pi_3(\Sn2)$}{π₃(S²)}}

We now want to go from $\Omega^3\Sn2$ to $\Omega^3(\Sn1*\Sn1)$. We do that by going \emph{back} one
of the maps in the long exact sequence of the Hopf fibration, so it’s not completely trivial. It’s a
priori a bit surprising that we can do it with actual loop spaces and not just homotopy groups, but
this is because we can use the fact that the double loop space of the fiber $\Sn1$ is contractible
(which is stronger than just having trivial homotopy groups).

\subsection{The Hopf fibration}

Using the fact that the identity function is an equivalence and that an equality between two
equivalences is determined by an equality between the underlying functions, we define the map
\begin{align*}
  \mu&:\Sn1\to(\Sn1\simeq\Sn1),\\
  \mu(\base)&\defeq\lambda x.x,\\
  \ap\mu(\lloop)&\defeq\funext(f),
\end{align*}
where
\begin{align*}
  f&:(x:\Sn1)\to x=_{\Sn1}x,\\
  f(\base)&\defeq\lloop,\\
  \apd f(\lloop)&\defeq\inne(\idp{\lloop\concat\lloop}).
\end{align*}
The Hopf fibration is the dependent type
\begin{align*}
  \Hopf&:\Sn2\to\Type,\\
  \Hopf(\north)&\defeq\Sn1,\\
  \Hopf(\south)&\defeq\Sn1,\\
  \ap\Hopf(\merid(x))&\defeq\ua(\mu(x)).
\end{align*}
The total space of the Hopf fibration is equivalent to $\Sn1*\Sn1$, which
means that in particular we have a map
\begin{align*}
  t&:(x:\Sn2)\to(\Hopf(x)\to\Sn1*\Sn1),\\
  t(\north)&\defeq\inl,\\
  t(\south)&\defeq\inr,\\
  \apd t(\merid(x))& \defeq\funext^3(t^{\merid}(x)),
\end{align*}
where
\begin{align*}
  t^{\merid}&:(x:\Sn1)\{y,y':\Sn1\}(p:y=^{\Hopf}_{\merid(x)}y')\to\inl(y)=_{\Sn1*\Sn1}\inr(y'),\\
  t^{\merid}(x)(p)&\defeq\push(y,\mu(x)(y))\concat\ap\inr(\outt(p)),
\end{align*}
and then we have
\begin{align*}
  t'&:\sum_{x:\Sn2}\Hopf(x)\to\Sn1*\Sn1,\\
  t'(x,y)&\defeq t(x)(y).
\end{align*}

\subsection{Looping a fibration}

Let $P:B\to\Type$ be a dependent type over a pointed type $B$ and such that $F\defeq P(\star_B)$ is
pointed. The total space of $P$ is pointed by $(\star_B,\star_F)$. We define the dependent type
\begin{align*}
  P^\Omega&:\Omega B\to\Type,\\
  P^\Omega(p)&\defeq(\star_F=^P_p\star_F).
\end{align*}
Note that the fiber of $P^\Omega$ (over the basepoint of $\Omega B$, i.e.\ $\idp{\star_B}$) is
definitionally equal to $\Omega F$, and we point it by $\idp{\star_F}$.

The total space of $P^\Omega$ is equivalent to the loop space of the total space of $P$, in
particular we have a (pointed) map $\sum_{p:\Omega B}(P^\Omega(p))\to\Omega(\sum_{x:B}P(x))$
given by 1-dimensional pairing.
We can then iterate this construction, we write $P^{\Omega^2}$ for
$(P^\Omega)^\Omega$ and $P^{\Omega^3}$ for $(P^{\Omega^2})^\Omega$.
Note that if $F$ is $(n+1)$-truncated, then every fiber of $P^\Omega$ is
$n$-truncated because every fiber of $P^\Omega$ is equivalent to an identity
type of $F$ via $\innt$ and $\outt$.

\subsection{Looping the Hopf fibration}

Let’s consider the triple looping of the Hopf fibration:
\[\Hopf^{\Omega^3}:\Omega^3\Sn2\to\Type.\]
The fiber of $\Hopf$ is $\Sn1$ which is $1$-truncated, hence every fiber of $\Hopf^{\Omega^3}$ is
$(-2)$-truncated, i.e.\ contractible. Therefore for every $p:\Omega^3\Sn2$ there is a
$q:\Hopf^{\Omega^3}(p)$, so we get a point $(p,q)$ in
$\sum_{p:\Omega^3\Sn2}(\Hopf^{\Omega^3}(p))$. Using the maps above, we get a point in
$\Omega^3(\sum_{x:\Sn2}(\Hopf(x)))$, and then in $\Omega^3(\Sn1*\Sn1)$ after applying
$\Omega^3t'$.
This gives us the map $h:\Omega^3\Sn2\to\Omega^3(\Sn1*\Sn1)$.

\section{Going back to \texorpdfstring{$\pi_2(\Sn2)$}{π₂(S²)}}

We now start decreasing the dimension by constructing a map
$\Omega^3\Sn3\to\Omega^2\|\Sn2\|_2$. This is the only place where we really need truncations. In the
next subsection we implement truncations using regular higher inductive types. It can safely be
skipped in a proof assistant with already a native support for truncations.

\subsection{Truncations}

In this subsection, we define the spheres slightly differently to make the inductive steps
simpler. We define $\Sn0\defeq\Bool$ and $\Sn1\defeq\Susp \Sn0$. We do not need the other spheres at
all here, so we keep the same notation for simplicity.

Given a type $A$ and $n\ge-1$, the $n$-truncation of $A$ is the higher inductive type $\|A\|_n$ with
constructors
\begin{align*}
  |-|&:A\to\|A\|_n,\\
  \hub&:(\Sn{n+1}\to\|A\|_n)\to\|A\|_n,\\
  \spoke&:(f:\Sn{n+1}\to\|A\|_n)(t:\Sn{n+1})\to\hub(f)=_{\|A\|_n}f(t).
\end{align*}
If $A$ is pointed, then $\|A\|_n$ is pointed by $|\star_A|$.

Given a type $B$ and a map $f:\Sn{n+1}\to B$, a \emph{filler} of $f$ is a pair $(\hub_f,\spoke_f)$
with $\hub_f:B$ and $\spoke_f:(t:\Sn{n+1})\to\hub_f=_Bf(t)$. If we have a filler for all maps
$\Sn{n+1}\to B$, we say that \emph{all $(n+1)$-spheres in $B$ are filled}. The $\hub$ and $\spoke$
constructors of $\|A\|_n$ state exactly that all $(n+1)$-spheres in $\|A\|_n$ are filled.
We prove that $\|A\|_n$ is $n$-truncated and that it has the following universal property:
given an $n$-truncated type $B$, a map $f:A\to B$ can be extended in a unique way to a map
$\widetilde{f}:\|A\|_n\to B$ satisfying $\widetilde{f}(|x|)=f(x)$ definitionally.

We prove first that if a type $B$ has all $(n+1)$-spheres filled, then $B$ is $n$-truncated, by
induction on $n$. It follows that $\|A\|_n$ is $n$-truncated.

For $n=-1$, given $x,y:B$, we define $f:\Sn0\to B$ by
\begin{align*}
  f(\true)&\defeq x,\\
  f(\false)&\defeq y,
\end{align*}
and we assumed that $B$ has all $0$-spheres filled. We then have
\[\spoke_f(\true)\inv\concat\spoke_f(\false):x=_By\] hence $B$ is $(-1)$-truncated.

For the induction step, we assume that $B$ has all $(n+2)$-spheres filled and we
prove that every identity type of $B$ has all $(n+1)$-spheres filled. Let’s take
$x,y:B$. We define the maps
\begin{align*}
  \lift&:(\Sn{n+1}\to x=_By)\to(\Sn{n+2}\to B),\\
  \lift(f)(\north)&\defeq x,\\
  \lift(f)(\south)&\defeq y,\\
  \ap{\lift(f)}(\merid(t))&\defeq f(t),
\end{align*}
\begin{align*}
  \hub'&:(\Sn{n+1}\to x=_By)\to x=_By,\\
  \hub'_f&\defeq\spoke_{\lift(f)}(\north)\inv\concat\spoke_{\lift(f)}(\south),
\end{align*}
\begin{align*}
  \spoke'&:(f:\Sn{n+1}\to x=_By)(t:\Sn{n+1})\to\hub'_f=_{x=_By}f(t),\\
  \spoke'_f&\defeq\spoke^{\lemm}(f,\merid(t)),
\end{align*}
where $\spoke^{\lemm}$ has type
\begin{align*}
  \spoke^{\lemm}&:(f:\Sn{n+1}\to x=_By)\{t,t':\Sn{n+2}\}(p:t=_{\Sn{n+2}}t')\\
  &\qquad\to\spoke_{\lift(f)}(t)\inv\concat\spoke_{\lift(f)}(t')=_{(\lift(f)(t)=_B\lift(f)(t'))}\ap{\lift(f)}(p)
\end{align*}
and is proved by path induction on $p$. This concludes.

We now prove the converse, i.e.\ that every $n$-truncated type has all $(n+1)$-spheres filled. A
consequence of that is the non-dependent elimination rule of truncations: if $B$ is $n$-truncated,
then any map $f:A\to B$ can be extended to a map $\widetilde{f}:\|A\|_n\to B$ such that
$\widetilde{f}(|x|)\defeq f(x)$, because for the application of $\widetilde{f}$ to the other
constructors of $\|A\|_n$ we can just use the $\hub$ and $\spoke$ structure of $B$.

For $n=-1$, we know that $B$ is a proposition and we define $\hub$ and
$\spoke$ as follows: $\hub_f$ is $f(\true)$ and $\spoke_f$ is automatic because
it’s an equality in a proposition.

For $n+1$, by induction hypothesis for every $x,y:B$, we have a $\hub^{x,y}$ and
$\spoke^{x,y}$ for the type $x=_By$. We then define
\begin{align*}
  \hub&:(\Sn{n+2}\to B)\to B,\\
  \hub_f&\defeq f(\north),
\end{align*}
\begin{align*}
  \spoke&:(f:\Sn{n+2}\to B)(t:\Sn{n+2})\to\hub_f=_Bf(t),\\
  \spoke_f(\north)&\defeq\idp{f(\north)},\\
  \spoke_f(\south)&\defeq\ap{f}(\merid(\north)),\\
  \apd{\spoke_f}(\merid(t))&\defeq\inne(\spoke^{f(\north),f(\south)}_{\lambda w.\ap f(\merid(w))}(t)\inv\\
  &\qquad\qquad\concat\spoke^{f(\north),f(\south)}_{\lambda w.\ap f(\merid(w))}(\north)),
\end{align*}
which shows that $B$ has all $(n+2)$-spheres filled.

Finally we prove the uniqueness principle: if $B$ is $n$-truncated and $g,h:\|A\|_n\to B$ are two
functions such that $(a:A)\to g(|a|)=_Bh(|a|)$, then $(x:\|A\|_n)\to g(x)=_Bh(x)$. We prove it by
induction on $x$ (using the induction principle derived from the $\hub$ and $\spoke$ constructors).
For the $|-|$ constructor, it’s true by assumption.
For the $\hub$ and $\spoke$ constructors, let’s take
$f:\Sn{n+1}\to\|A\|_n$. Using $\hub$ and $\spoke$ on $\|A\|_n$ we have a
filler of $f$, hence by applying $g$ and $h$ to it we get a filler of $g\circ
f$ and a filler of $h\circ f$. Now by induction hypothesis, $g$ and $h$ agree
on the image of $f$, hence $g\circ f=h\circ f$ and by composing the filler of
$h\circ f$ with that equality, we get another filler of $g\circ f$, and we have
to prove that those two fillers of $g\circ f$ are equal.
Let’s call $(\hub,\spoke)$ and $(\hub',\spoke')$ those two fillers of $g\circ f$ and let’s prove
that they are equal. The type $B$ is $n$-truncated, hence its identity types are also $n$-truncated,
hence they have all $(n+1)$-spheres filled. We define the map
\begin{align*}
  k&:\Sn{n+1}\to\hub=_B\hub',\\
  k(t)&\defeq\spoke(t)\concat\spoke'(t)\inv.
\end{align*}
That map can be filled, hence we have
\begin{align*}
  \hub_{k}&:\hub=_B\hub',\\
  \spoke_{k}&:(t:\Sn{n+1})\to\hub_{k}=_{\hub=_B\hub'}\spoke(t)\concat\spoke'(t)\inv.
\end{align*}
In order to prove that $(\hub,\spoke)$ and $(\hub',\spoke')$ are equal, we need
\begin{align*}
  \hub^=&:\hub=_B\hub',\\
  \spoke^=&:\spoke=^{\lambda h.(t:\Sn{n+1})\to
    h=_Bk(t)}_{\hub^=}\spoke'.
\end{align*}
We take $\hub^=\defeq\hub_{k}$ and
\begin{align*}
  \spoke^=&\defeq\funextd(\spoke'^=),
\end{align*}
where
\begin{align*}
  \spoke'^=&:(t:\Sn{n+1})\to\spoke(t)=^{\lambda h.h=_Bk(t)}_{\hub^=}\spoke'(t),\\
  \spoke'^=(t)&\defeq\inne(\spoke''^=(t)),
\end{align*}
where
\begin{align*}
  \spoke''^=&:(t:\Sn{n+1})\to\spoke(t)=\hub^=\concat\spoke'(t)
\end{align*}
is deduced from $\spoke_{k}$ by some coherence operation.

\subsection{Truncated higher Hopf fibration}

The fibration we define now is similar to the Hopf fibration with $\Sn2$ instead of $\Sn1$. The main
difference is that, unlike for $\Sn1$, there is no appropriate multiplication operation on $\Sn2$,
there is only one on $\|\Sn2\|_2$ so it is a bit more complicated to define.
We write $|p|^1$ for $\ap{\lambda w.|w|}(p)$.

We define the map
\begin{align*}
  \mu_2&:\Sn2\to\Sn2\to\|\Sn2\|_2,\\
  \mu_2(\north)&\defeq\lambda x.|x|,\\
  \mu_2(\south)&\defeq\mu_2^{\south},\\
  \ap{\mu_2}(\merid(x))&\defeq\funext(\mu_2^{\merid}(x)),
\end{align*}
where
\begin{align*}
  \mu_2^{\south}&:\Sn2\to\|\Sn2\|_2,\\
  \mu_2^{\south}(\north)&\defeq|\south|,\\
  \mu_2^{\south}(\south)&\defeq|\south|,\\
  \ap{\mu_2^{\south}}(\merid(y))&\defeq|\merid(\base)\inv\concat\merid(y)|^1,
\end{align*}
and
\begin{align*}
  \mu_2^{\merid}&:(x:\Sn1)(y:\Sn2)\to|y|=_{\|\Sn2\|_2}\mu_2^{\south}(y),\\
  \mu_2^{\merid}(x)(\north)&\defeq|\merid(x)|^1,\\
  \mu_2^{\merid}(x)(\south)&\defeq|\merid(\base)\inv\concat\merid(x)|^1,\\
  \apd{\mu_2^{\merid}(x)}(\merid(y))&\defeq\inne(\mu_2^{\merid,\merid}(x)(y)),
\end{align*}
and
\begin{align*}
  \mu_2^{\merid,\merid}:(x,y:\Sn1)&\to|\merid(x)|^1\concat|\merid(\base)\inv\concat\merid(y)|^1\\
  &\quad=|\merid(y)|^1\concat|\merid(\base)\inv\concat\merid(x)|^1
\end{align*}
is defined as follows:
\begin{itemize}
\item When $x$ or $y$ is $\base$, using the fact that $|-|^1$ commutes with composition of paths,
  it’s easy to prove the desired equality.
\item When we apply $\mu_2^{\merid,\merid}$ to $\lloop$ and $\lloop$, we need to
  construct a term in a 4-times iterated identity type of $\|\Sn2\|_2$. But
  $\|\Sn2\|_2$ is 2-truncated, hence we need to construct something in a
  $(-2)$-truncated type which is automatic.
\end{itemize}

We have now defined $\mu_2$, but in order to use it to construct a fibration, we
need to turn it into a map
$\widetilde{\mu_2}:\Sn2\to(\|\Sn2\|_2\simeq\|\Sn2\|_2)$. For $x:\Sn2$, the
underlying map of $\widetilde{\mu_2}(x)$ is defined by applying the universal
property of truncation to $\mu_2(x)$. In other words, we have
\[\widetilde{\mu_2}(x)(|y|)\defeq\mu_2(x)(y).\]
We prove that $\widetilde{\mu_2}(x)$ is an equivalence by noticing that $\widetilde{\mu_2}(\north)$
is equal to the identity function, using the uniqueness principle for maps out of truncations, and
then it follows by induction on $x$ that all of them are equivalences, using the fact that being an
equivalence is a proposition.

We can now define the fibration we are interested in by
\begin{align*}
  \tHopf&:\Sn3\to\Type,\\
  \tHopf(\north)&\defeq\|\Sn2\|_2,\\
  \tHopf(\south)&\defeq\|\Sn2\|_2,\\
  \ap{\tHopf}(\merid(x))&\defeq\ua(\widetilde{\mu_2}(x)),
\end{align*}
and we have $\tHopf^{\Omega^2}:\Omega^2\Sn3\to\Type$ with fiber $\Omega^2\|\Sn2\|_2$.
A point $p:\Omega^3\Sn3$ can be seen as a loop in $\Omega^2\Sn3$, and we
can transport along it in $\tHopf^{\Omega^2}$. We define the map
\begin{align*}
  e_3&:\Omega^3\Sn3\to\Omega^2\|\Sn2\|_2,\\
  e_3(p)&\defeq\transport^{\tHopf^{\Omega^2}}(p,\idp{\idp{|\north|}}).
\end{align*}

\section{Loop spaces of truncations}

Let $A$ be a pointed type and $n\ge-1$, the goal here is to construct a pointed map
$\kappa_{n,A}:\Omega\|A\|_{n+1}\to\|\Omega A\|_n$. Let’s first prove that for $n\ge-1$, the type of
$n$-types is $(n+1)$-truncated.

If $A$ and $B$ are two $n$-types, their identity types in the type of $n$-types and in $\Type$ are
equivalent, because being an $n$-type is a proposition. Using univalence, we get that $A=_{\Type}B$
is equivalent to $A\simeq B$. Using the fact that $\isequiv$ is a mere proposition, the identity
type between two equivalences between $A$ and $B$ is the same as the identity type between the
underlying functions, and using function extensionality it is equivalent to a product of identity
types of $B$. But $B$ is $n$-truncated, hence the identity types of $A\simeq B$ are
$(n-1)$-truncated, hence $A\simeq B$ is $n$-truncated, hence the type of $n$-types is
$(n+1)$-truncated.

We can now define the family of $n$-types
\begin{align*}
  P_{n,A}&:\|A\|_{n+1}\to \Type,\\
  P_{n,A}(|x|)&\defeq\|\star_A=_Ax\|_n.
\end{align*}
The function $\kappa_{n,A}$ is then defined by
\begin{align*}
  \kappa_{n,A}&:\Omega\|A\|_{n+1}\to\|\Omega A\|_n,\\
  \kappa_{n,A}(p)&\defeq\transport^{P_{n,A}}(p,|\idp{\star_A}|).
\end{align*}

\section{Down one more dimension}

We’re almost done, we now just need to go from $\|\Omega^2\Sn2\|_0$ to $\Omega\Sn1$. This step is
quite easy using the Hopf fibration. Using the fact that $\Omega\Sn1$ is a set, we only have to
build a map from $\Omega^2\Sn2$ to $\Omega\Sn1$. That map is defined by
\begin{align*}
  e_2'&:\Omega^2\Sn2\to\Omega\Sn1,\\
  e_2'(p)&\defeq\transport^{\Hopf^\Omega}(p,\idp{\base}).
\end{align*}
This concludes the definition of $n$.

\addtocontents{toc}{\protect\setcounter{tocdepth}{1}}


\printbibliography[heading=bibintoc,title={Bibliography},notkeyword=online]

\printbibliography[heading=subbibliography,title={Online Resources},keyword=online]


\begin{otherlanguage}{french}
\cleardoublepage

\chapter*{Introduction (français)}
\addcontentsline{toc}{chapter}{Version française}
\addcontentsline{toc}{section}{Introduction}
\markboth{\slshape\MakeUppercase{Version française}}{\slshape\MakeUppercase{Introduction}}

L’objectif de cette thèse est de démontrer le théorème suivant dont l’énoncé et la démonstration
seront expliqués en temps voulu.

\begin{maintheoremfr}\label{theoremfr}
  On a un isomorphisme de groupes \[\pi_4(\Sn3)\simeq\Z/2\Z.\]
\end{maintheoremfr}

À proprement parler, c’est un théorème bien connu en théorie de l’homotopie classique, démontré par
Freudenthal dans \cite{freud37} (voir aussi \cite[corollary 4J.4]{hatcher}). La principale
différence est que dans cette thèse on se place en \emph{théorie des types homotopiques} (aussi
connu sous le nom de \emph{fondations univalentes}) qui est un nouveau cadre pour faire des
mathématiques, introduit par Vladimir Voevodsky en 2009, et qui est particulièrement bien adapté
pour la théorie de l’homotopie. Du point de vue de la théorie de l’homotopie, la différence la plus
frappante entre la théorie de l’homotopie classique et la théorie des types homotopiques est qu’en
théorie des types homotopiques, \emph{toutes} les constructions sont invariantes par équivalences
d’homotopie. Un des avantages est que toutes les constructions et toutes les démonstrations faites
dans ce cadre sont complètement indépendantes de la définition de la notion d’«~espace~». En
particulier, rien ne dépend de la topologie générale ou de la combinatoire des ensembles
simpliciaux. De plus, les constructions et les démonstrations ont souvent un aspect plus
«~homotopique~», comme nous espérons que le lecteur en sera convaincu après la lecture de cette
thèse.

Cependant, cela pose un certain nombre de défis, étant donné qu’il n’est pas clair a priori quels
sont les concepts qui peuvent ou ne peuvent pas être définis d’une façon purement invariante par
homotopie. Par exemple, même si la cohomologie singulière est invariante par homotopie, la
définition classique utilise l’ensemble des cochaînes singulières qui, lui, n’est pas invariant par
homotopie. Par conséquent, la définition classique ne peut pas être reproduite telle quelle en
théorie des types homotopiques. Un exemple encore plus simple est le revêtement universel du cercle
qui est défini classiquement en utilisant la fonction exponentielle $\R\to\Sn1$, mais il se trouve
que cette fonction est homotope à une fonction constante. La théorie des types homotopiques nous
donne divers outils pour travailler de façon complètement invariante par homotopie et dans cette
thèse on montre comment démontrer le théorème \ref{theoremfr} en théorie des types homotopiques, en
partant essentiellement de zéro.

Un autre avantage de la théorie des types homotopiques par rapport à la théorie de l’homotopie
classique est que les démonstrations écrites en théorie des types homotopiques sont beaucoup plus
appropriées à une vérification formelle sur ordinateur. Même si le travail présenté ici n’a pas
encore été formalisé, beaucoup de résultats intermédiaires (en particulier des deux premiers
chapitres) ont été formalisés par diverses personnes, voir par exemple les bibliothèques
\cite{HoTTCoq} et \cite{Unimath} pour Coq, \cite{HoTTAgda} pour Agda et \cite{HoTTLean} pour Lean.

\paragraph{Contenu de la thèse}

Les deux premiers chapitres de cette thèse rappellent les bases de la théorie des types
homotopiques. Une référence alternative est le livre \cite{hottbook}, mais on a essayé ici d’être
plus concis et de garder en vue notre objectif principal. Cependant, il se peut que le style de
présentation soit similaire entre \cite{hottbook} et l’introduction et les deux premiers chapitres
de cette thèse. La plupart du contenu des quatre derniers chapitres est nouveau en théorie des types
homotopiques, même si les concepts sont bien connus en théorie de l’homotopie classique. La
définition de la notion d’$\infty$-groupoïde faible présentée dans le premier appendice est
également nouvelle.

Dans le chapitre \ref{ch:hott}, on introduit tous les concepts de base de la théorie des types
homotopiques, c’est-à-dire tous les constructeurs de types et en particulier l’axiome d’univalence et
les types inductifs supérieurs. On énonce également le lemme $3\times3$ et le lemme d’aplatissement
dans les sections \ref{threebythree} et \ref{flattening}, qui seront utilisés à plusieurs
endroits. Finalement on parle de types et d’applications $n$-tronqués et de troncations. La notion
de type $n$-tronqué correspond à la notion classique de $n$-type d’homotopie, c’est-à-dire un espace
qui n’a pas de contenu homotopique en dimension supérieure à $n$, et la troncation est une opération
transformant n’importe quel espace en un espace $n$-tronqué d’une façon universelle. À nouveau, tout
ceci est standard en théorie des types homotopiques.

Dans le chapitre \ref{ch:hopf}, on définit les groupes d’homotopie des sphères. Le groupe
$\pi_k(\Sn n)$ est défini comme étant la $0$-troncation (c’est-à-dire l’ensemble des composantes
connexes) de l’espace des lacets itérés de dimension $k$ dans $\Sn n$. On démontre ensuite que
$\pi_1(\Sn1)\simeq\Z$, qui est un résultat initialement dû à Michael Shulman en 2011 et qui
apparaît dans \cite[section 8.1]{hottbook}, voir aussi \cite{mikeblog:pi1s1} et
\cite{mikelicata:pi1s1}. L’idée est qu’en théorie des types homotopique, afin de définir une
fibration on ne donne pas une application de l’espace total vers la base. Au lieu de cela, on donne
directement la fibre au dessus de chaque point de la base. Dans le cas d’une fibration au dessus du
cercle, il est suffisant de donner la fibre au dessus du point base de $\Sn1$ et l’action sur la
fibre du lacet faisant le tour du cercle. Dans le cas qui nous intéresse, la fibre est l’espace des
entiers relatifs $\Z$ et le lacet faisant le tour du cercle agit dessus en ajoutant $1$. Cela donne
une fibration au dessus de $\Sn1$ et on peut montrer que son espace total est contractile, ce qui
donne l’isomorphisme $\pi_1(\Sn1)\simeq\Z$. On définit ensuite la notion de $n$-connectivité et on
démontre diverses propriétés des espaces et des application $n$-connectés, ce qui nous permet de
montrer que $\pi_k(\Sn n)$ est trivial pour tout $k<n$. Ce résultat a déjà été obtenue dans
\cite[section 8.3]{hottbook} avec une démonstration un peu plus compliquée également due à
l’auteur. Finalement, on définit la fibration de Hopf, qui est une fibration au dessus de $\Sn2$,
de fibre $\Sn1$ et d’espace total $\Sn3$. L’idée de la définition de la fibration de Hopf est la
suivante. Afin de définir une fibration au dessus de $\Sn2$, il suffit de donner la fibre $N$ au
dessus du pôle nord, la fibre $S$ au dessus du pôle sud et, pour tout élément $x$ de $\Sn1$, une
équivalence entre $N$ et $S$ qui décrit se qui se passe quand on se déplace dans la fibration du
pôle nord au pôle sud le long du méridien correspondant à $x$. Dans le cas de la fibration de Hopf,
on prend $N,S\defeq\Sn1$ et l’équivalence entre $N$ et $S$ correspondant à $x$ est l’opération de
multiplication par $x$. La fibration de Hopf a été définie pour la première fois par Peter
Lumsdaine, d’une façon légèrement différente, mais sans démonstration du fait que son espace total
est équivalent à $\Sn3$. La construction présentée ici a été également décrite dans \cite[section
8.5]{hottbook}.

Dans le chapitre \ref{ch:james}, on définit la fibration de Hopf en suivant une idée suggérée par
André Joyal. Pour tout type $A$, on définit une famille d’espaces $(J_nA)$ et on démontre que leur
colimite est équivalente à l’espace des lacets de la suspension de $A$. Pour cela, on définit un
autre espace $JA$ et on démontre que $JA$ est équivalent à la fois à la colimite de $(J_nA)$ et à
l’espace des lacets de la suspension de $A$. La construction de James donne une suite
d’approximations de l’espace des lacets de la suspension de $A$ ce qui, avec le théorème de
Blakers--Massey, nous permet de montrer qu’il existe un entier naturel $n$ tel que
$\pi_4(\Sn3)\simeq\Z/n\Z$. Cet entier $n$ est défini en utilisant le produit de Whitehead, plus
précisément c’est l’image du produit de Whitehead $[\id_{\Sn2},\id_{\Sn2}]$, qui est un élément de
$\pi_3(\Sn2)$, par l’équivalence $\pi_3(\Sn2)\simeq\Z$ construite en utilisant la fibration de Hopf.

Dans le chapitre \ref{ch:smash}, on étudie le produit smash et sa structure monoïdale symétrique. En
particulier, on construit une famille d’équivalences $\Sn n\wedge\Sn m\simeq\Sn{n+m}$ compatible, en
un certain sens, avec l’associativité et la commutativité du produit smash. La construction de la
structure monoïdale symétrique sera essentiellement admise, mais on donne l’intuition de la
construction.

Dans le chapitre \ref{ch:cohomology}, on commence par définir, pour tout entier naturel $n$,
l’espace d’Eilenberg--MacLane $K(\Z,n)$ comme étant la $n$-troncation de la sphère $\Sn n$, et le
$n$ième groupe de cohomologie d’un espace $X$ comme étant la $0$-troncation de l’espace de fonctions
$X\to K(\Z,n)$. On définit ensuite le produit cup en tant qu’application
$K(\Z,n)\wedge K(\Z,m)\to K(\Z,n+m)$ en prenant le produit smash des deux applications
$\Sn n\to K(\Z,n)$ et $\Sn m\to K(\Z,m)$ et en montrant, en utilisant les propriétés de connectivité
des applications, qu’on peut essentiellement l’inverser. Les propriétés du produit smash du chapitre
\ref{ch:smash} sont alors utilisées pour montrer que le produit cup est associatif et commutatif
gradué. Finalement, on définit l’invariant de Hopf d’une application $f:\Sn{2n-1}\to\Sn n$ en
utilisant la structure du produit cup sur le pushout $\Unit\sqcup^{\Sn{2n-1}}\Sn n$ et on démontre
que pour tout nombre pair $n$, une certaine application $\Sn{2n-1}\to\Sn n$ provenant de la
construction de James a un invariant de Hopf égal à $2$. Cela montre en particulier que l’entier $n$
défini au chapitre \ref{ch:james} est égal soit à $1$, soit à $2$, et que le groupe
$\pi_{4n-1}(\Sn{2n})$ est infini pour tout entier naturel $n$.

Finalement, dans le chapitre \ref{ch:gysin}, on construit la suite exacte de Gysin, qui est une
suite exacte longue de groupes de cohomologie associée à toute fibration où la base est $1$-connexe
et les fibres sont des sphères. Cette suite exacte décrit une partie de la structure multiplicative
de la cohomologie de la base. On définit ensuite $\CP2$ comme étant le pushout
$\Unit\sqcup^{\Sn3}\Sn2$ pour l’application de Hopf $\Sn3\to\Sn2$, on construit une fibration de
cercles au dessus de $\CP2$ d’une façon similaire à la construction de la fibration de Hopf et on
calcule son anneau de cohomologie en utilisant la suite exacte de Gysin. Ceci démontre que
l’invariant de Hopf de l’application de Hopf est égal à $\pm1$ et que $\pi_4(\Sn3)\simeq\Z/2\Z$.

Dans l’appendice \ref{ch:infgpd}, on présente une définition élémentaire des $\infty$-groupoïdes
faibles, basée sur des idées venant de théorie des types homotopiques, ainsi qu’une démonstration du
fait que tout type en théorie des types homotopiques a une structure d’$\infty$-groupoïde faible.

Dans l’appendice \ref{ch:defn}, on donne une définition complète de l’entier $n$ défini à la fin du
chapitre \ref{ch:james} et qui satisfait $\pi_4(\Sn3)\simeq\Z/n\Z$. La raison est que, comme on le
verra un peu plus tard, calculer cet entier à partir de sa définition est un problème ouvert
important en théorie des types homotopiques et il est donc utile pour les personnes essayant de
résoudre ce problème d’avoir la définition complète en un seul endroit.

\paragraph{Analytique et synthétique}

La différence principale entre la théorie de l’homotopie classique et la théorie des types
homotopiques est que la première est \emph{analytique} alors que la deuxième est
\emph{synthétique}. Afin de comprendre la différence entre théorie de l’homotopie analytique et
synthétique, il est utile de retourner à la géométrie élémentaire.

La géométrie analytique est la géométrie dans le sens de Descartes. L’objet que l’on étudie est
l’ensemble $\R^2$, les points sont définis comme étant des couples $(x,y)$ de nombres réels et les
droites sont définies comme étant les ensembles de points satisfaisant une équation de la forme
$ax+by=c$. Ensuite, afin de démontrer quelque chose, on utilise les propriétés de $\R^2$. Par
exemple, on peut déterminer si deux droites se croisent en résolvant un certain système d’équations.

En revanche, la géométrie synthétique est la géométrie dans sens d’Euclide. Les points et les lignes
ne sont pas définies en fonction d’autres notions, il s’agit de notions primitives, et on
se donne une collection d’axiomes qui spécifie comment elles sont censées se comporter. Ensuite,
afin de démontrer quelque chose, on doit utiliser les axiomes. Par exemple on ne peut pas utiliser
l’équation d’une droite ou les coordonnées d’un point parce que les droites n’ont pas d’équations et
les points n’ont pas de coordonnées.

La géométrie analytique peut être utilisée pour justifier la géométrie synthétique. En effet, la
géométrie analytique donne une signification aux notions de point et de ligne et tous les axiomes de
la géométrie synthétique peuvent être démontrés en géométrie analytique. Ainsi les axiomes sont
consistants et tout ce qui est vrai en géométrie synthétique est également vrai en géométrie
analytique. La réciproque n’est pas vraie, donc on pourrait penser que la géométrie synthétique est
moins puissante que la géométrie analytique étant donné que moins de théorèmes sont
démontrables. Mais d’un autre côté on peut également considérer que la géométrie synthétique est
\emph{plus} puissante que la géométrie analytique, parce que les théorèmes qui peuvent être
démontrés sont plus généraux. Ils sont valides pour n’importe quelle interprétation des concepts
primitifs qui satisfait les axiomes, alors qu’une démonstration en géométrie analytique est par
nature valide uniquement dans $\R^2$. Un autre inconvénient de la géométrie analytique est qu’en
réduisant la géométrie à la résolution d’équations il est facile de perdre de vue l’intuition
géométrique. En résumé, en géométrie analytique on donne une définition explicite des concepts qui
nous intéressent et on peut démontrer beaucoup de choses mais on est restreint à ce modèle
particulier, alors qu’en géométrie synthétique on ne fait qu’axiomatiser les propriétés de base des
concepts qui nous intéressent, moins de théorèmes sont démontrables mais ils sont applicable plus
largement et sont plus proches de l’intuition géométrique.

La situation en théorie de l’homotopie est très similaire. En théorie de l’homotopie analytique (ou
théorie de l’homotopie classique), la sphère $\Sn n$ est définie comme étant l’ensemble
$\{(x_0,\dots,x_n)\in\R^{n+1},x_0^2+\dots+x_n^2=1\}$ équipé de la topologie appropriée, les
fonctions continues sont les fonctions qui préservent la topologie de façon appropriée, et
$\pi_4(\Sn3)$ est défini comme étant le quotient de l’ensemble des fonctions continues pointées
$\Sn4\to\Sn3$ par la relation d’homotopie. On peut ensuite utiliser diverses techniques pour
démontrer que $\pi_4(\Sn3)\simeq\Z/2\Z$, c’est-à-dire que $\pi_4(\Sn3)$ contient exactement deux
éléments.

En théorie de l’homotopie synthétique, qui est ce sur quoi cette thèse porte, la notion d’espace ne
provient \emph{pas} de la topologie. Au lieu de cela, la notion d’espace est considérée comme une
notion primitive (sous le nom \emph{type}) et on a également les notions primitives de \emph{point}
et de \emph{chemin} entre deux points. En particulier, un chemin n’est pas vu comme étant une
fonction continue depuis l'intervalle, c’est simplement une notion primitive. On introduit aussi une
notion primitive de \emph{fonction continue}. Notons qu’en théorie de l’homotopie classique, on doit
d’abord définir ce qu’est une fonction \mbox{possiblement-non-continue} avant de pouvoir définir ce
qu’est une fonction continue, mais ici on prend directement le concept de fonction continue comme
étant primitif. Pour nous, une fonction continue n’est \emph{pas} une fonction
\mbox{possiblement-non-continue} qui a la propriété additionnelle d’être continue, il n’y a pas du
tout de notion de fonction \mbox{possiblement-non-continue}. Ainsi, l’adjectif «~continue~» est
superflu et on utilisera simplement le mot «~fonction~» ou «~application~» pour désigner ce qui
correspondrait à une «~fonction continue~» en théorie de l’homotopie classique.

Divers espaces sont également axiomatisés, par exemple l’espace $\N$ des entiers naturels est
axiomatisé avec un élément $0$, une fonction $S:\N\to\N$ et le principe d’induction/récursion. Le
cercle est axiomatisé avec un point appelé $\base$, un chemin appelé $\lloop$ de $\base$ vers
$\base$, et un principe similaire d’induction/récursion qui dit intuitivement que le cercle est
librement engendré par $\base$ et $\lloop$. On décrit de façon similaire les sphères de dimension
supérieure $\Sn n$ et l’ensemble des composantes connexes d’un espace. En combinant tout ceci avec
la notion de fonction (continue) mentionnée ci-dessus, on peut définir $\pi_4(\Sn3)$ et on verra
qu’on peut encore montrer qu’il est isomorphe au groupe $\Z/2\Z$.

\paragraph{Théorie des types}

La théorie des types homotopiques est une variante de la théorie des types et plus précisément de la
théorie intuitionniste des types de Per Martin--Löf (qu’on appellera simplement \emph{théorie des
  types dépendants} ici), qui a été introduite dans les années 1970 en tant que fondements des
mathématiques constructives (voir \cite{ML75}). Les mathématiques constructives sont une philosophie
des mathématiques basée sur l’idée qu’afin de démontrer qu’un certain objet existe, il faut donner
une méthode permettant de le construire. Cela fonctionne en restreignant les principes logiques que
l’on peut utiliser et en n’autorisant que ceux qui sont constructifs. Une démonstration en
mathématiques constructives n’est pas nécessairement présentée sous la forme d’un algorithme, mais
un algorithme peut toujours en être extrait. Par conséquent, les mathématiques constructives
rejettent les principes comme l’axiome du choix, qui affirme l’existence d’une fonction sans donner
de moyen de la calculer, et le raisonnement par l’absurde qui permet de montrer que quelque chose
existe en montrant simplement qu’il ne peut pas ne pas exister. En particulier, une démonstration du
fait qu’il existe un entier naturel ayant une certaine propriété doit donner (au moins
implicitement) une méthode pour calculer cet entier. Ce n’est pas vrai en mathématiques
classiques. Par exemple considérons l’entier $n\in\N$ défini comme étant le plus petit nombre
parfait impair, ou $0$ s’il n’existe pas de nombre parfait impair. En mathématiques classiques,
c’est une définition valide de $n$ mais qui ne donne aucun moyen de le calculer. En effet, au moment
où j’écris ces lignes personne ne sait si $n$ est égal à $0$ ou pas. En revanche, ce n’est pas
considéré une définition valide en mathématiques constructives parce qu’on a utilisé le principe du
tiers exclu (soit il existe un nombre parfait impair, soit il n’en existe pas) qui n’est pas
constructif. Il y a diverses variantes des mathématiques constructives et notons que celle utilisée
ici, la théorie des types homotopiques, n’est pas incompatible avec la logique classique. Il serait
tout à fait possible de rajouter l’axiome du choix ou le tiers exclu mais l’inconvénient est que la
constructivité, qui est un des principaux avantages de la théorie des types, serait perdue.

En théorie des types dépendants, les notions primitives sont les \emph{types} et les \emph{éléments
  de types} (ou \emph{termes}). On écrit $u:A$ l’énoncé que $u$ est un élément du type
$A$. Intuitivement, on peut imaginer un type comme étant une sorte d’ensemble, mais il y a plusieurs
différences avec la théorie des ensembles traditionnelle. Les éléments de types n’existent pas
isolément, il s’agit toujours d’éléments \emph{d’un certain type} qui fait intrinsèquement partie de
la nature de l’élément. Le type d’un élément est toujours connu et cela n’a pas de sens de
«~démontrer~» qu’un élément $u$ est de type $A$. Ceci est similaire au fait que cela n’a pas de sens
de «~démontrer~» que $x^2+y^2=0$ est une équation. Il suffit de la regarder et on voit que c’est une
équation et pas une matrice. De plus, le type d’un élément est toujours unique (modulo règles de
réduction, comme on le verra un peu plus loin). Par exemple on ne peut pas dire que le nombre $2$ a
à la fois le type $\N$ et le type $\Q$. Au lieu de cela, il y a deux éléments différents, le premier
étant $2_\N$ de type $\N$ et le deuxième étant $2_\Q$ de type $\Q$ (qui peuvent être écrits tous les
deux $2$ s’il n’y a pas de risque de confusion), et qui satisfont $i(2_\N)=2_\Q$ pour $i:\N\to\Q$
l’inclusion canonique. De façon similaire, étant donné un nombre rationnel $q:\Q$, on ne peut pas
demander si $q$ est de type $\N$. Le nombre $q$ est de type $\Q$ par nature et $\Q$ est différent de
$\N$. Par contre, ce qu’on peut demander est s’il existe un entier naturel $k:\N$ tel que
$i(k)=q$. C’est ce qu’il faudrait faire pour démontrer que $q$ est un entier naturel.

Les mathématiques sont traditionellement basées sur un système à deux niveaux~: le niveau logique où
vivent les propositions et les démonstrations, et le niveau mathématique où vivent les objets
mathématiques. Le niveau logique est utilisée pour raisonner sur le niveau mathématique. Par
exemple, construire un certain objet mathématique est une activité qui se passe dans le niveau
mathématique, alors que démontrer un théorème est une activité qui se passe dans le niveau
logique. En théorie des types dépendants, ces deux niveaux sont fusionnées en un seul où vivent les
types et leurs éléments. En plus de représenter des objets mathématiques, les types jouent aussi les
rôle des propositions (logiques) et leurs éléments jouent le rôle des démonstrations de ces
propositions. Démontrer une certaine proposition consiste en la construction d’un élément du type
correspondant. Par exemple, démontrer une implication $A\implies B$ corresponds à la construction
d’un élément dans le type de fonctions $A\to B$, c’est-à-dire c’est une fonction transformant une
démonstration de $A$ en une démonstration de $B$. Démontrer une conjonction $A\wedge B$ correspond à
la construction d’un élément dans le type produit $A\times B$, c’est-à-dire c’est un couple formé
d’une démonstration de $A$ et d’une démonstration de $B$. Cette correspondance entre types et
propositions et entre éléments de types et démonstrations est connue sous le nom de
\emph{correspondance de Curry--Howard}. On distinguera parfois entre les types «~vus comme des
propositions~» et les types «~vus comme des types~» afin d’expliquer l’intuition derrière certaines
constructions, mais la différence entre les deux est souvent floue. Par exemple, le type $A\simeq B$
peut être vu à la fois comme la proposition «~$A$ et $B$ sont isomorphes~» et comme le type de tous
les isomorphismes entre $A$ et $B$. En effet, en mathématiques constructives, démontrer que $A$ et
$B$ sont isomorphes est la même chose que construire un isomorphisme entre les deux.

Le mot «~dépendants~» dans «~théorie des types dépendants~» fait référence au fait que les types
peuvent dépendre d’éléments d’autres types. De tels types sont appelés des \emph{types dépendants}
ou des \emph{familles de types}. Étant donné un type $A$, avoir un type dépendant $B$ au dessus de
$A$ signifie que pour tout élément $a:A$ on a un type $B(a)$. Les types dépendants sont essentiels
pour la représentation des propositions quantifiées comme on le verra dans le chapitre
\ref{ch:hott}. Par exemple, une proposition dépendant d’un entier naturel $n:\N$ est représentée par
un type dépendant de la variable $n$. Un type dépendant $B$ au dessus de $A$ où tous les types
$B(a)$ sont vus comme des propositions est appelé un \emph{prédicat sur $A$}.

La propriété de constructivité de la théorie des types dépendants nous permet de le voir comme un
langage de programmation. En théorie des types dépendants, toutes les constructions primitives ont
des \emph{règles de réduction} qui expliquent essentiellement comment exécuter les programmes du
langage. Tous les éléments de types peuvent ainsi être vus comme des programmes et peuvent être
exécutés, simplement en appliquant répétitivement les règles de réduction. Il n’y a pas de boucle
infinie en théorie des types, tous les programmes terminent et on obtient donc toujours un résultat
quand on exécute un programme. Du point de vue mathématique, les règles de réduction sont les
équation de définition des notion primitives, et appliquer une règle de réduction correspond à
remplacer quelque chose par sa définition. Deux éléments $u$ et $v$ d’un type $A$ sont dits
\emph{définitionnellement égaux} s’ils deviennent syntaxiquement égaux après qu’on ait tout remplacé
par leurs définitions, c’est-à-dire après avoir exécuté $u$ et $v$. Une règle importante de la
théorie des types, connue sous le nom de \emph{règle de conversion}, stipule que si $u$ est de type
$A$ et que $A$ est définitionnellement égal à $A'$, alors $u$ est aussi de type $A'$. En
particulier, les types sont uniques seulement à égalité définitionnelle près, mais l’égalité
définitionnelle est décidable car il s’agit simplement de développer les définitions. De la même
façon que cela n’a pas de sens de démontrer qu’un terme $u$ est de type $A$, cela n’a pas non plus
de sens de démontrer que deux termes ou deux types sont définitionnellement égaux. C’est quelque
chose qui peut simplement être vérifié algorithmiquement.

Étant donné la correspondance entre démonstrations et éléments de types, il en résulte que les
démonstrations elles-mêmes peuvent être exécutées, ce qui donne à la théorie des types dépendants sa
nature constructive. Par exemple, étant donné une démonstration qu’il existe un entier naturel
ayant une certaine propriété, on peut exécuter la démonstration et le résultat que l’on obtient est
un couple de la forme $(n,p)$ où $n$ est un entier naturel de la forme $0$, ou $1$, ou $2$,
etc. (c’est-à-dire qu’on connaît sa valeur), et $p$ est une démonstration du fait que $n$ satisfait
la propriété. Cette relation étroite entre la théorie des types et l’informatique a donné lieu au
développement d’\emph{assistants de preuve} comme Coq, Agda ou Lean (voir \cite{coq}, \cite{agda},
\cite{lean}). Il s’agit essentiellement de compilateurs pour la théorie des types dépendants avec
diverses fonctionnalités pour les rendre plus faciles à utiliser. Dans un assistant de preuve, on
peut énoncer un théorème en définissant le type correspondant et on peut ensuite le démontrer en
construisant un terme (c’est-à-dire en écrivant un programme) ayant ce type. Si l’assistant de
preuve l’accepte, cela signifie que le programme est bien typé et que la démonstration est donc
correcte.

\paragraph{Théorie des types homotopiques}

La théorie des types dépendants est très fructueuse mais elle souffre de quelques problèmes, en
particulier en ce qui concerne le traitement de l’égalité. Étant donné un type $A$ et deux éléments
$u,v:A$, la proposition «~$u$ est égal à $v$~» est réifiée en un type $u=_Av$ qu’on appelle le
\emph{type identité} (ses éléments sont les témoins de l’égalité entre $u$ et $v$). Martin--Löf a
donné plusieurs versions de la théorie des types dépendants avec différentes règles pour les types
identité. Dans une de ces versions, appelée la \emph{théorie des types extensionnelle}, les types
identité se comportent de façon agréable mais le typage n’est pas décidable, c’est-à-dire qu’il n’y
a pas d’algorithme vérifiant qu’un terme a un type donné. C’est une propriété en général indésirable
pour une théorie des types. Dans une autre version, appelée la \emph{théorie des types
  intensionnelle}, les règles des types identité sont différentes et le typage est
décidable. Cependant, le traitement de l’égalité en théorie des types intensionnelle n’est pas tout
le temps satisfaisant. Par exemple, deux fonctions $f,g:A\to B$ peuvent vérifier $f(x)=g(x)$ pour
tout $x:A$ sans être égales elles-mêmes en tant que fonctions. Définir le quotient d’un ensemble par
une relation d’équivalence est aussi plutôt problématique. Un problème différent est le principe
d’\emph{unicité des preuves d’égalités} qui stipule que pour tous $u,v:A$, tous les éléments de
$u=_Av$ sont égaux, n’est plus démontrable, ce qui est contraire à l’intuition qui était derrière
les types identité. En effet, l’idée des types identité dans la théorie des types de Martin--Löf est
que tout type représente un ensemble et que $u=_Av$ représente l’ensemble ayant exactement un
élément si $u$ et $v$ sont égaux ou l’ensemble vide si $u$ et $v$ sont différents.

La théorie des types homotopiques est basée sur la théorie des types intensionnelle et résoud ce
dernier problème en changeant l’intuition derrière les types et les types identité. En théorie des
types homotopiques, les types ne sont plus vus comme des ensembles mais comme des \emph{espaces},
les types dépendants sont vus comme des \emph{fibrations}, et le type identité $u=_Av$ est vu comme
le type des \emph{chemins continus} de $u$ vers $v$ dans l’espace $A$. De façon plutôt surprenante,
on peut montrer que sous cette interprétation toutes les règles de la théorie des types
intensionnelle sont toujours vérifiées. De plus, dans cette interprétation, l’unicité des preuves
d’égalités n’est plus une propriété désirable. Étant donné deux points $u$ et $v$ dans un espace
$A$, il peut y avoir beaucoup de chemins non homotopes entre $u$ et $v$ et beaucoup d’homotopie non
homotopes entre deux chemins, et ainsi de suite.

Cette connection entre la théorie des types et la théorie de l’homotopie a été découverte vers 2006
indépendemment par Vladimir Voevodsky et par Steve Awodey et Michael Warren dans
\cite{awodeywarren}. Ensuite, en 2009, Vladimir Voevodsky a énoncé l’\emph{axiome d’univalence}, a
démontré sa consistance dans le modèle simplicial et a commencé le projet de formalisation des
mathématiques dans ce système, la théorie des types intensionnelle augmentée de l’axiome
d’univalence, sous le nom de \emph{fondations univalentes}. Étant donné un univers $\Type$,
c’est-à-dire un type dont les éléments sont eux-mêmes des types, et deux éléments $A$ et $B$ de
$\Type$, l’axiome d’univalence identifie le type identité $A=_{\Type}B$ avec le type des
équivalences $A\simeq B$. Cet axiome rend précise l’idée que «~deux structures isomorphes ont les
mêmes propriétés~», qui est souvent utilisé implicitement en mathématiques. Notons que l’axiome
d’univalence n’est pas compatible avec le principe d’unicité des preuves d’égalité, il implique par
exemple qu’il y a deux égalités $\Bool=_{\Type}\Bool$ différentes correspondant aux deux bijections
$\Bool\simeq\Bool$ (où $\Bool$ est le type ayant deux éléments). Voevodsky a aussi remarqué que
l’axiome d’univalence implique l’extensionnalité des fonctions, c’est-à-dire que si $f(x)=_Bg(x)$
pour tout $x:A$, alors $f=_{A\to B}g$, et qu’il rend possible la définition des quotients.

En 2011, la notion de \emph{type inductif supérieur} a commencé à émerger. Les types inductifs
ordinaires sont des types définis en donnant des générateurs (les \emph{constructeurs}) et un
principe d’induction qui rend précise l’idée que le type est librement engendré par les
constructeurs. Les types inductifs supérieurs sont une généralisation des types inductifs ordinaires
où on peut non seulement donner des constructeurs de points, mais aussi des constructeurs de
chemins. Par exemple, le cercle a un constructeur de points appelé $\base$ et un constructeur de
chemins $\lloop$ qui est un chemin de $\base$ vers $\base$. En combinaison avec l’axiome
d’univalence, une fibration peut être définie par induction sur l’espace de base, ce qui est un
moyen très puissant pour définir des fibrations. Par exemple, afin de définir une fibration sur le
cercle il suffit de donner la fibre au dessus de $\base$ et l’action de $\lloop$ sur cette fibre
(cette action devant être une équivalence).

Un des inconvénients de la théorie des types homotopiques est qu’en ajoutant l’axiome d’univalence
ou les types inductifs supérieurs on perd la propriété de constructivité qui, comme on l’a déjà
mentionné, est une propriété essentielle de la théorie des types. Cependant, contrairement au cas de
l’axiome du choix ou du tiers exclu, il est généralement admis que l’axiome d’univalence et les
types inductifs supérieurs sont constructifs d’une façon ou d’une autre, et diverses personnes
cherchent à donner une description alternative de la théorie des types homotopiques dans laquelle
on peut calculer avec l’axiome d’univalence et les types inductifs supérieurs, voir en particulier
\cite{cubicaltt}. Une conjecture associée est la conjecture de canonicité homotopique de Voevodsky:
pour tout terme clos $n:\N$ construit en utilisant l’axiome d’univalence, il existe un terme clos
$k:\N$ construit \emph{sans} utiliser l’axiome d’univalence et une démonstration de $k=_{\N}n$.

\paragraph{Constructivité de $\pi_4(\Sn3)$}

Le premier résultat majeur de cette thèse est le corollaire \ref{firstpi4s3}, qui dit qu’il existe
un entier naturel $n$ tel que $\pi_4(\Sn3)\simeq\Z/n\Z$. Cet énoncé est assez étrange parce que
c’est un énoncé de la forme «~il existe un entier naturel $n$ satisfaisant une certaine propriété~»
donc selon la conjecture de constructivité il devrait être possible d’extraire la valeur de $n$ de
la démonstration. Cependant, pour l’instant personne n’a réussi à le faire, principalement parce que
la démonstration est relativement compliquée et que la constructivité de l’axiome d’univalence et
des types inductifs supérieurs n’est pas encore bien comprise. Dans les chapitres \ref{ch:smash},
\ref{ch:cohomology} et \ref{ch:gysin} on présente une démonstration du fait que cet entier est égal
à $2$, mais il s’agit d’une démonstration mathématique et non pas d’un calcul extrait de la
définition de $n$ donc cela n’aborde pas la conjecture de constructivité. En revanche cela montre
que l’on peut définir et travailler avec la cohomologie et la suite exacte de Gysin en théorie des
types homotopiques, ce qui est intéressant en soi.

\paragraph{Modèles de la théorie des types homotopiques}

On ne va pas beaucoup parler de la relation entre la théorie des types homotopiques (théorie de
l’homotopie synthétique) et la théorie de l’homotopie classique (théorie de l’homotopie analytique)
dans cette thèse, mis à part le fait que beaucoup de définitions et de démonstrations sont assez
similaires à leurs homologues classiques. Une construction d’un modèle de la théorie des types
homotopiques (sans types inductifs supérieurs) en théorie de l’homotopie classique est présentée
dans \cite{simplicialmodel} et une démonstration qu’il modèle également les types inductifs
supérieurs est en préparation dans \cite{ls:hit}. Comme on l’a mentionné plus haut, une des
conséquences de travailler de façon synthétique est que tout le travail effectué dans cette thèse
est également valide dans n’importe quel autre modèle de la théorie des types homotopiques, pas
seulement le modèle classique. Michael Shulman a donné dans \cite{mikemodelsunivalence} divers
autres modèles de la théorie des types homotopiques et il est généralement admis que tout
$\infty$-topos au sens de Lurie (voir \cite{htt}) devrait donner un modèle de la théorie des types
homotopiques.

Un autre modèle très important est le modèle de Thierry Coquand et coauteurs décrit dans
\cite{cubical}, qui est un modèle constructif de la théorie des types homotopiques dans les ensemble
cubiques. En théorie, ce modèle devrait nous permettre de calculer le nombre $n$ du chapitre
\ref{ch:james}, mais cela n’a pas encore été fait. Ce modèle suggère également une version
différente de la théorie des types homotopiques, appelée \emph{théorie des types cubiques} (voir
\cite{cubicaltt}), mais dans cette thèse on restera avec la théorie des types décrite dans
\cite{hottbook}. On utilisera néanmoins divers carrés et cubes à chaque fois que cela sera jugé
utile.

\end{otherlanguage}


\begin{otherlanguage}{french}
\cleardoublepage
\chapter*{Résumé substantiel (français)}
\addcontentsline{toc}{section}{Résumé substantiel}
\markboth{\slshape\MakeUppercase{Version française}}{\slshape\MakeUppercase{Résumé substantiel}}

\section*{1\quad La théorie des types homotopiques}

On commence par introduire tous les types de base et le constructeurs de types de la théorie des
types homotopiques. Le premier constructeur de types est le types des \emph{fonctions}. Étant donnés
deux types $A$ et $B$, le type des fonctions de $A$ vers $B$ est noté $A\to B$. On peut appliquer
une fonction $f:A\to B$ à un élément $a:A$, on obtient alors un élément $f(a):B$, et on définit une
fonction $f:A\to B$ en donnant la valeur de $f(x):B$, pour $x$ une variable de type $A$. Le type des
fonctions se généralise au type des \emph{fonctions dépendantes}, noté $(x:A)\to B(x)$, lorsque $B$
est un type dépendant sur $A$.

On introduit ensuite le type des \emph{paires} $A\times B$, pour deux types $A$ et $B$. Étant donné
$a:A$ et $b:B$ on peut considérer l’élément $(a,b)$ de $A\times B$, et étant donné un élément
$u:A\times B$ on peut considérer $\fst(u):A$ et $\snd(u):B$. Comme pour le type des fonctions, le
type des paires se généralise au type des \emph{paires dépendantes} $\sum_{x:A}B(x)$ où le type du
deuxième composant dépend du premier composant.

On introduit ensuite les \emph{types inductifs}. L’idée est qu’un type inductif est librement
engendré par un certain nombre de \emph{constructeurs}, ce qui est exprimé par un \emph{principe
  d’induction}. Par exemple le type vide, le type singleton, l’union disjointe de deux types, le
type des entiers naturels et le type des entiers relatifs peuvent être définis comme des types
inductifs.

Le quatrième constructeur de types que l’on introduit est le \emph{type identité} ou type des
chemins entre deux points. Étant donné un type $A$ et deux éléments $u,v:A$, le type $u=_Av$
représente le type des chemins continus du $u$ vers $v$, et la principale règle de typage gouvernant
les types identité est la règle $\J$. Cette règle permet de munir chaque type d’une structure
d’$\infty$-groupoïde faible, par exemple on peut inverser et composer les chemins, cette composition
est associative, et ainsi de suite. Une définition plus précise de la notion d’$\infty$-groupoïde
est décrite dans l’appendice \ref{ch:infgpd}.

Une fonction $f:A\to B$ est appelée une \emph{équivalence} s’il existe une fonction $g:B\to A$ telle
que les deux composées $g\circ f$ et $f\circ g$ soient homotopes à la fonction identité. Une petite
modification de cette définition nous permet de définir le type $A\simeq B$ des équivalences entre
$A$ et $B$ ainsi qu’une application $(A=_{\Type}B)\to(A\simeq B)$. L’axiome d’univalence est
l’énoncé stipulant que cette application est elle-même une équivalence.

La notion de chemin introduite précédemment est homogène, dans le sens où on ne peut parler d’un
chemin entre deux éléments $u$ et $v$ qu’à condition que $u$ et $v$ aient le même type. On introduit
alors la notion de \emph{chemin dépendant}. Étant donné un type dépendant $B:A\to\Type$, un chemin
$p:x=_Ay$ dans $A$ et deux éléments $u:B(x)$ et $v:B(y)$, le type $u=^B_pv$ représente le type des
chemins de $u$ vers $v$ au dessus de $p$. Cette notion est utilisée en particulier pour définir le
type de l’application d’une fonction dépendante à un chemin. Lorsque $B(x)$ est un type identité, il
est naturel d’introduire le type des remplissages d’un carré, étant donné quatre chemins
correspondant aux quatre côtés. Finalement on a la propriété d’extensionalité des fonctions qui
permet de construire une égalité entre deux fonctions $f$ et $g$ étant donné une égalité entre
$f(x)$ et $g(x)$ pour tout $x$.

Les derniers types que l’on introduit sont les \emph{types inductifs supérieurs}. L’idée est
similaire aux types inductifs, mis à part la possibilité d’avoir des constructeurs de chemins. Ceci
permet de définir de nombreux espace de la théorie de l’homotopie, comme par exemple le cercle
$\Sn1$ qui est engendré par les deux constructeurs
\begin{align*}
  \base &: \Sn1,\\
  \lloop &: \base=_{\Sn1}\base
\end{align*}
et le pushout d’un diagramme de la forme
\[
\begin{tikzcd}
  A & C \arrow[l,"f"'] \arrow[r,"g"] & B
\end{tikzcd}
\]
qui est le type $A\sqcup^C B$ engendré par les trois constructeurs
\begin{align*}
  \inl &:A \to A\sqcup^C B,\\
  \inr &:B \to A\sqcup^C B,\\
  \push &:(c:C)\to \inl(f(c))=_{A\sqcup^C B}\inr(g(c)).
\end{align*}
Les pushouts permettent de définir de nombreux autres types comme la suspension d’un type, les
sphères, le join de deux types, la somme pointée et le produit smash. Comme pour les types inductifs
ordinaires, les types inductifs supérieurs ont un principe d’induction, qui utilise la notion de
chemin dépendant vue ci-dessus.

Étant donné un diagramme $3\times3$ où l’on peut prendre le pushout de chacune des lignes et de
chacune des colonnes (voir le diagramme \ref{diag:3x3} de la page \pageref{diag:3x3}), le pushout du
pushout des lignes est équivalent au pushout du pushout des colonnes. Ce résultat est appelé le
\emph{lemme $3\times3$} et est utilisé entre autres pour montrer que le join de $\Sn n$ et de $\Sn
m$ est équivalent à $\Sn{n+m+1}$.

Étant donné un type inductif $T$, on peut définir un type dépendant $P:T\to\Type$ en utilisant le
principe d’induction de $T$ et l’axiome d’univalence pour les constructeurs de chemin de $T$. Le
\emph{lemme d’aplatissement} est un résultat décrivant l’espace total d’une telle fibration.

On introduit finalement le concept de \emph{type $n$-tronqué}, c’est à dire intuitivement un type
qui n’a pas d’information homotopique en dimension supérieure à $n$. Un type est $(-2)$-tronqué (ou
contractile) lorsque qu’il possède un point égal à tous les autres points, et un type est
$(n+1)$-tronqué lorsque tous ses types identité sont $n$-tronqués. On a également une opération de
troncation qui transforme un type $A$ en un type $n$-tronqué $\trunc nA$ d’une façon universelle,
c’est-à-dire qu’on a une application $|-|:A\to\trunc nA$ et afin de définir une application
$f:\trunc nA\to B$ pour un type $B$ qui est lui-même $n$-tronqué, il suffit de définir $f$ sur les
éléments de la forme $|a|$.

\section*{2\quad Résultats préliminaires sur les groupes d’homotopie des sphères}

Étant donné un type $A$ pointé par $\star_A:A$ et un entier $n\ge1$, on définit l’\emph{espace des
  lacets itérés} de $A$ par
\begin{align*}
  \Omega^0A &\defeq A,\\
  \Omega^{n+1}A &\defeq \Omega(\Omega^nA),
\end{align*}
où $\Omega A\defeq (\star_A=\star_A)$ est l’\emph{espace des lacets} de $A$. On définit ensuite le
\emph{$n$ième groupe d’homotopie} de $A$ par
\[\pi_n(A) \defeq \trunc0{\Omega^nA}.\]

Afin de calculer les groupes d’homotopie du cercle, on définit la fibration
\begin{align*}
  U &: \Sn1\to\Type,\\
  U(\base) &\defeq \Z,\\
  \ap U(\lloop) &\defeq \ua(\succZ).
\end{align*}
où on utilise l’axiome d’univalence $\ua$ et le fait que $\succZ$ est une équivalence de $\Z$. Le
lemme d'aplatissement implique que l’espace total de $U$ est contractile, ce qui implique que
$\pi_1(\Sn1)\simeq\Z$ et que $\pi_n(\Sn1)$ est trivial pour tout $n>1$.

On définit ensuite la notion de \emph{type $n$-connexe} comme étant un type dont la $n$-troncation
est contractile, et une fonction est dite \emph{$n$-connexe} si toutes ses fibres sont
$n$-connexes. Les fonctions $n$-connexes vérifient le principe d’induction suivant. Une fonction
$f:A\to B$ est $n$-connexe si et seulement si pour toute famille $P:B\to\Type$ de types $n$-connexes
et toute fonction $d:(a:A)\to P(f(a))$ il existe une section $s:(b:B)\to P(b)$ de $P$ telle que pour
tout $a:A$ on ait $s(f(a))=d(a)$. Ceci permet de montrer entre autres que la fonction
$|-|:A\to\trunc nA$ est $n$-connexe, que la composée de deux fonctions $n$-connexes est $n$-connexe
et qu’un pushout d’une fonction $n$-connexe est $n$-connexe. On en déduit ensuite que la sphère
$\Sn n$ est $(n-1)$-connexe et qu’en particulier tous les groupes $\pi_k(\Sn n)$ avec $k<n$ sont
triviaux.

On définit enfin la fibration de Hopf par
\begin{align*}
  \Hopf &: \Sn2\to\Type,\\
  \Hopf(\north) &\defeq \Sn1,\\
  \Hopf(\south) &\defeq \Sn1,\\
  \ap\Hopf(\merid(x)) &\defeq \ua(\mu(-,x)),
\end{align*}
où $\mu : \Sn1\times\Sn1\to\Sn1$
est défini par double induction sur le cercle. Les deux fonctions $\mu(-,\base)$ et $\mu(\base,-)$
sont égales à la fonction identité, ce qui montre que $\mu(-,x)$ est une équivalence pour tout
$x:\Sn1$. Le lemme d’aplatissement ainsi que le fait que $\mu(x,-)$ est une équivalence pour tout
$x:\Sn1$ nous permet alors de montrer que l’espace total de la fibration de Hopf est équivalent au
join $\Sn1*\Sn1$ qui est équivalent à $\Sn3$. On obtient alors $\pi_2(\Sn2)\simeq\Z$ et
$\pi_k(\Sn2)\simeq\pi_k(\Sn3)$ pour tout $k\ge3$.

\section*{3\quad La construction de James}

Étant donné un type pointé connexe $A$, la \emph{construction de James} donne une suite
d’approximation du type $\Omega\Sigma A$. Plus précisément, si $A$ est $k$-connexe, on obtient une
suite de types et de fonctions
\[
\begin{tikzcd}
  J_0A \arrow[r,"i_0"] & J_1A \arrow[r,"i_1"] & J_2A \arrow[r,"i_2"] & J_3A \arrow[r,"i_3"] &
  \dots,
\end{tikzcd}
\]
où pour tout $n$ l’application $i_n$ est $(n(k+1)+(k-1))$-connexe, et telle que la colimite
$J_\infty A$ du diagramme est équivalente à $\Omega\Susp A$. Les types $J_nA$ sont définis par
induction sur $n$ par des pushouts. On définit ensuite un type inductif supérieur $JA$ et on
démontre que $JA$ est équivalent à la colimite de $(J_nA)_{n:\N}$ ainsi qu’à $\Omega\Susp A$. Une
première conséquence de la construction de James est le théorème de suspension de Freudenthal, qui
dit que l’application $A\to\Omega\Susp A$ est $2k$-connexe lorsque $A$ est $k$-connexe. En
particulier, cela implique que $\pi_n(\Sn n)\simeq\pi_2(\Sn 2)\simeq\Z$ pour tout $n\ge2$ et que
$\pi_{n+1}(\Sn n)\simeq\pi_4(\Sn 3)$ pour tout $n\ge 3$.

On définit ensuite, pour tout $n,m:\N$, une application $W_{n,m}:\Sn{n+m-1}\to\Sn{n}\vee\Sn{m}$ qui
donne une équivalence
\[\Sn n\times\Sn m\simeq\Unit\sqcup^{\Sn{n+m-1}}(\Sn n\vee\Sn m).\] Cette application induit (par
composition) une application $[-,-]:\pi_n(X)\times\pi_m(X)\to\pi_{n+m-1}(X)$ pour tout type $X$.  En
combinaison avec la construction de James et le théorème de Blakers--Massey, on démontre alors que
$\pi_4(\Sn 3)\simeq\Z/n\Z$, où $n$ est l’image de $[i_2,i_2]$ par l’équivalence
$\pi_3(\Sn2)\simeq\Z$ et $i_2$ est le générateur de $\pi_2(\Sn2)$.

\section*{4\quad Produits smash de sphères}

Le \emph{produit smash} $A\wedge B$ de deux types pointés $A$ et $B$ est le type inductif supérieur
défini par les constructeurs
\begin{align*}
  \star_{A\wedge B} &: A\wedge B,\\
  \proj &: A \to B \to A\wedge B,\\
  \projr &: (a : A) \to \proj(a,\star_B) = \star_{A\wedge B}, \\
  \projl &: (b : B) \to \proj(\star_A,b) = \star_{A\wedge B},\\
  \projlr &: \projr(\star_A) = \projl(\star_B).
\end{align*}
Intuitivement, il s’agit du produit $A\times B$ où l’on a contracté les deux axes $A$ et $B$ sur un
point. Le produit smash est un produit monoïdal symétrique sur les types pointés, en particulier il
est fonctoriel, commutatif, associatif, unitaire et satisfait des hypothèses de naturalité et de
cohérence en dimension $1$. Le produit smash de sphères satisfait l’équivalence
$\Sn n\wedge\Sn m\simeq\Sn{n+m}$. On démontre que cette famille d’équivalences est compatible avec
l’associativité du produit smash, dans le sens où les deux équivalences
$\Sn n\wedge\Sn m\wedge\Sn n\to\Sn{n+m+k}$ provenant des deux parenthésages du codomaine sont
égales. En ce qui concerne la commutativité, on obtient un résultat similaire à celui pour
l’associativité, mis à part qu’il y a un signe lorsque les deux sphères sont de dimension impaire.

\section*{5\quad La cohomologie}

On définit ensuite les groupes de cohomologie (à coefficients entiers) d’un type. La première étape
est la définition des espaces d’Eilenberg--MacLane $K(\Z,n)$ (notés plus simplement $K_n$). Pour
$n=0$, on définit $K_0\defeq\Z$ et pour $n>0$ on définit $K_n\defeq\trunc n{\Sn n}$. La principale
propriété de ces espaces est le fait qu’il existe une équivalence $K_n\simeq\Omega K_{n+1}$.  La
construction de cette équivalence utilise le résultat sur $\pi_1(\Sn1)$ pour $n=0$, la fibration de
Hopf pour $n=1$ et le théorème de suspension de Freudenthal pour $n\ge2$. Les groupes de cohomologie
d’un type $X$ sont alors définis par
\[H^n(X)\defeq\trunc0{X\to K_n}\]
et l’équivalence entre $K_n$ et $\Omega K_{n+1}$ donne une structure de groupe abélien sur $H^n(X)$.

On définit ensuite le produit cup
\[\cupp:H^n(X)\times H^m(X)\to H^{n+m}(X)\]
en utilisant les équivalences $\Sn n\wedge\Sn m\simeq\Sn {n+m}$, et la compatibilité de cette
famille d’équivalences avec l’associativité et la commutativité du produit smash permettent de
montrer que le produit cup est associatif, commutatif gradué et distributif.

La suite exacte longue de Mayer--Vietoris se déduit assez facilement de la définition de la
cohomologie. Cela permet de montrer que $H^k(\Sn n)$ est équivalent à $\Z$ pour $k=0$ et $k=n$, et
trivial sinon. On calcule ensuite les groupes de cohomologie de $\Sn n\times\Sn k$. La suite exacte
longue de Mayer--Vietoris montre qu’ils sont engendrés (additivement) par un élément $\mathbf{1}$ de
degré $0$, un élément $\xx$ de degré $n$, un élément $\yy$ de degré $k$ et un élément $\zz$ de degré
$n+k$. On montre de plus que le produit cup $\xx\cupp\yy$ est égal à $\zz$.

Finalement on définit l’invariant de Hopf d’une application $f:\Sn{2n-1}\to\Sn{n}$ comme étant
l’entier $H(f)$ vérifiant
\[\alpha\cupp\alpha=H(f)\beta,\]
où $\alpha$ et $\beta$ sont les générateurs de la cohomologie en dimensions $n$ et $2n$ du type
$C_f\defeq\Unit\sqcup^{\Sn{2n-1}}\Sn n$. L’invariant de Hopf est un homomorphisme de groupes
$H:\pi_{2n-1}(\Sn n)\to\Z$ et on montre que pour tout $n$, l’invariant de Hopf de
$[i_{2n},i_{2n}]:\pi_{4n-1}(\Sn{2n})$ est égal à $2$. On en conclut que $\pi_{4n-1}(\Sn{2n})$ est
infini, et que $\pi_4(\Sn3)$ est équivalent soit à $\Z/2\Z$ s’il existe une application
$\Sn3\to\Sn2$ d’invariant de Hopf égal à $1$, ou est trivial s’il n’existe pas de telle application.

\section*{6\quad La suite de Gysin}

Étant donné une fibration de sphères $\Sn{n-1}$ sur un type $B$, on construit un élément $e:H^n(B)$
ainsi que la \emph{suite exacte longue de Gysin} reliant le produit cup $-\cupp e$ avec la
cohomologie de l’espace total de la fibration. La principale propriété permettant la construction de
cette suite exacte longue est le fait que pour $n,m:\N$, $p:K_n$ et $y:K_m$, on a
\[\ap{\lambda x.x\cupp y}(\sigma_n(p)) = \sigma_{n+m}(p\cupp y),\]
où $\sigma_n$ dénote l’équivalence $K_n\simeq\Omega K_{n+1}$. On en déduit que pour tout $i,n:\N$,
l’application
\begin{align*}
f&:K_i\to(\Sn n\to K_{i+n})\\
f(x)&\defeq(\lambda y. x\cupp|y|)
\end{align*}
est une équivalence, ce qui est démontré par récurrence sur $i$.

On définit ensuite $\CP2$ comme étant le pushout $\Unit\sqcup^{\Sn3}\Sn2$ pour l’application
$\eta:\Sn3\to\Sn2$ provenant de la fibration de Hopf. D’une façon similaire à la construction de la
fibration de Hopf, on construit une fibration sur $\CP2$ de fibre $\Sn1$ et dont l’espace total est
équivalent à $\Sn5$. En appliquant la suite de Gysin à cette fibration, on en déduit que l’invariant
de Hopf de $\eta$ est égal à $\pm1$, ce qui montre que \[\pi_4(\Sn3)\simeq\Z/2\Z.\]

\section*{A\quad Une définition des $\infty$-groupoïdes faibles par la théorie des
  types}

Dans cet appendice on donne une définition de la notion d’$\infty$-groupoïde faible, inspirée de la
théorie des types. L’idée est de définir une théorie des types minimaliste dans laquelles les termes
que l’on peut définir sont exactement ceux qui existent dans tout $\infty$-groupoïde. On explique
ensuite comment cette théorie des types donne une définition des $\infty$-groupoïdes faibles et on
montre que tout type en théorie des types homotopiques est un $\infty$-groupoïde faible.

\section*{B\quad Le cardinal de $\pi_4(\Sn3)$}

Dans cet appendice on résume la définition de l’entier $n$ défini à la fin du chapitre
\ref{ch:james}, afin de faciliter une implémentation future dans un assistant de preuves muni d’une
interprétation calculatoire de l’axiome d’univalence et des types inductifs supérieurs.
\end{otherlanguage}


\begin{otherlanguage}{french}
\cleardoublepage
\chapter*{Conclusion (français)}
\addcontentsline{toc}{section}{Conclusion}
\markboth{\slshape\MakeUppercase{Version française}}{\slshape\MakeUppercase{Conclusion}}

On a vu dans cette thèse que la théorie des types homotopiques est suffisamment puissante pour
démontrer que $\pi_4(\Sn3)\simeq\Z/2\Z$. Même si du point de vue de la théorie de l’homotopie
classique c’est un résultat bien connu, ce n’était pas évident que l’axiome d’univalence et les
types inductifs supérieurs suffiraient à le démontrer, ni même qu’une démonstration constructive et
purement homotopique existe. De plus, en prenant en compte le fait que seulement cinq années se sont
écoulées entre la définition du cercle en théorie des types homotopiques et le calcul de
$\pi_4(\Sn3)$, le progrès a été plutôt rapide et on peut espérer que d’ici quelques années la
théorie des types homotopiques aura atteint un niveau comparable à celui de la théorie de
l’homotopique classique et va aider à obtenir des résultats complètements nouveaux.

\paragraph{Comparaison avec les démonstrations classiques}

Comme on l’a déjà mentionné, la principale différence entre la théorie de l’homotopie classique et
la théorie des types homotopiques est qu’en théorie des types homotopique tout est invariant par
homotopie. J’ai utilisé le livre \cite{hatcher}, qui présente la topologie algébrique classique,
assez régulièrement pendant cette recherche, mais j’ai souvent dû trouver des définitions, des
démonstrations et des énoncés complètement différents afin de pouvoir les reproduire en théorie des
types homotopiques.

Dans la construction du revêtement universel du cercle et de la fibration de Hopf, la différence la
plus évidente est qu’au lieu de définir une fonction de l’espace total vers la base, on définit
directement les fibres et comment transporter le long des fibres, et déterminer l’espace total est
la partie non triviale. Pour le revêtement universel du cercle c’est assez transparent mais pour la
fibration de Hopf il n’était pas clair a priori pour moi que la définir avec la multiplication de
$\Sn1$ donnerait effectivement la fibration de Hopf.

Démontrer que le produit smash est associatif n’est facile ni en théorie des types homotopiques ni
en théorie de l’homotopie classique mais pour des raisons très différentes. En théorie des types
homotopiques, le problème vient du fait qu’on a beaucoup de chemins et de chemins de dimension
supérieure à gérer et que cela devient vite compliqué de s’occuper de toutes les cohérences entre
eux. En topologie générale, en revanche, il y a une bijection canonique entre les deux espaces mais
le problème est qu’elle peut ne pas être continue sauf si l’on suppose que les deux espaces se
comportent suffisamment bien. La différence est qu’en topologie général on identifie littéralement
divers points entre eux, ce qui fait que la bijection est facile à définir mais peut ne pas
respecter la topologie, alors qu’en théorie des types homotopiques on rajoute de nouveaux chemins au
lieu d’identifier des points.

Pour la cohomologie, on a déjà mentionné le besoin d’utiliser les espaces d’Eilenberg--MacLane au
lieu de la définition classique par les cochaînes singulières, étant donné que l’ensemble des
cochaînes singulières n’est pas invariant par homotopie. La notion de troncation donne une
définition très agréable des espaces d’Eilenberg--MacLane $K(\Z,n)$ et une définition très agréable
du produit cup. Notons que quand on travaille de façon constructive certains phénomènes inattendus
peuvent se produire en cohomologie. En effet, sans l’axiome du choix il n’est pas possible de
démontrer l’axiome d’additivité de la cohomologie. On ne peut même pas démontrer que le groupe de
cohomologie $H^1(\N,\Z)$ est trivial, comme est expliqué dans \cite{mikeblog:h1n}. Il est agréable
de voir que ce genre de problème ne se pose pas dans le calcul de $\pi_4(\Sn3)$.

Finalement, afin de calculer la cohomologie de $\CP2$ je n’ai pas réussi à adapter la démonstration
géométrique présentée par exemple dans \cite[theorem 3.19]{hatcher} and j’ai dû utiliser la suite
exacte de Gysin qui est en général considérée comme étant un résultat plus avancé en théorie de
l’homotopie classique. De plus, la construction de la suite exacte de Gysin présentée dans
\cite[section 4.D]{hatcher} est basée sur le théorème de Leray--Hirsch dont la démonstration utilise
une induction sur les cellules d’un CW complexe, ce qui est impossible à faire en théorie des types
homotopiques. La démonstration présentée ici est nouvelle et basée sur la proposition \ref{stepcupp}
qui relie directement le produit cup en dimensions $n$ et $n+1$.

\paragraph{Théorie des types homotopiques cubique}

J’ai d’abord essayé d’écrire cette thèse en utilisant des idées “cubiques” autant que possible,
comme nous l’avions fait par exemple avec Dan Licata dans \cite{licatame:cubical} and comme a été
fait dans \cite{cavallo:cohomology}, mais cela s’avéra être une mauvaise idée, à ma grande
déception. Même si beaucoup de carrés et de cubes apparaissent naturellement en théorie des types
homotopiques, comme par exemple le carré de naturalité des homotopies, il y a aussi d’autres formes
qui peuvent apparaître, et traiter les carrés différemment n’est peut-être pas la meilleure
idée. Par exemple, essayer d’écrire le diagramme \ref{eq:deltaipush} de la page
\pageref{eq:deltaipush} comme une composition de Kan de carrés et de cubes n’est pas très naturel ni
utile. Il semble beaucoup plus naturel d’utiliser dans ce cas une notion générale de composition de
diagrammes (comme décrite à la fin de la page \pageref{trianglesquare}). Les idées cubiques ont
quand même des avantages, en particulier \cite{cubicaltt} utilise des ensembles cubiques pour donner
une interprétation calculatoire de la théorie des types homotopiques, et dans
\cite{licatame:cubical} et \cite{cavallo:cohomology} les cubes sont utilisés principalement pour
simplifier les formalisations en Agda. Je pense, cependant, que pour la théorie de l’homotopie
synthétique informelle, comme dans cette thèse, les idées cubiques ne sont pas spécialement utiles
en général.

\paragraph{Travaux futurs}

Un des principaux objectifs futurs est la formalisation des résultats présentés ici dans un
assistant de preuves. Il risque d’y avoir des problèmes dans la formalisation de la construction de
James, et encore plus dans la structure monoïdale du produit smash, où il y a beaucoup de
manipulations de chemins et de chemins supérieurs à faire, mais cela devrait être possible.

Quelques résultats de cette thèse ont été énoncés et démontrés sous une forme plutôt restreinte
étant donné que mon objectif principal était d’arriver au résultat $\pi_4(\Sn3)\simeq\Z/2\Z$. Par
exemple on devrait pouvoir calculer l’anneau de cohomologie de tous les $J_n(\Sn k)$, y compris
$J_\infty(\Sn k)\simeq\Omega\Sn{k+1}$, et de tous les $\CP n$, y compris $\CP\infty\simeq
K(\Z,2)$. De plus, on a défini seulement la cohomologie à coefficients entiers mais il devrait être
possible de définir la cohomologie à coefficients dans un groupe, un anneau ou un spectre
arbitraire, et d’étendre la plupart des résultats présentés ici.

Finalement, le projet à long terme est évidemment de continuer le développement de la théorie de
l’homotopie synthétique dans la théorie des types homotopiques. Il y a beaucoup de concepts qui
n’ont pas encore été beaucoup étudiés mais qui semblent accessibles, comme la K-théorie, les
opérations de Steenrod, les suites spectrales, les crochets de Toda et beaucoup d’autres. Je vais
certainement continuer de poursuivre cette ligne de recherche et j’espère que cette thèse va
inspirer d’autres personnes à participer à l’exploration de la théorie de l’homotopie synthétique
étant donné tout ce qui n’attend qu’à être découvert.
\end{otherlanguage}

\end{document}